\documentclass[11pt]{article}
\usepackage[total={7in, 8in}]{geometry}
\usepackage{graphicx}
\usepackage{xcolor}
\usepackage{amsmath, amsthm, latexsym, amssymb, color,cite,enumerate, physics,hyperref,rotating, wrapfig}
\usepackage{cutwin}
\usepackage{floatflt}
\usepackage{tikz} 
\usepackage{circuitikz}
\usetikzlibrary{external} 
\usepackage{adjustbox}
\usepackage{listings}
\usepackage{caption,subcaption}
\usepackage{mathrsfs}
\usepackage{comment}
\theoremstyle{theorem}
\newtheorem{theorem}{Theorem}[section]
\newtheorem{lemma}[theorem]{Lemma}
\newtheorem{definition}[theorem]{Definition}
\newtheorem{corollary}[theorem]{Corollary}
\newtheorem{proposition}[theorem]{Proposition}

\newtheorem{notation}[theorem]{Notation}
\theoremstyle{remark}
\newtheorem{remark}{Remark}
\newtheorem{examplex}{Example}
\newenvironment{example}
  {\pushQED{\qed}\examplex}
  {\popQED\endexamplex}
\newcommand*{\myproofname}{Proof}

\renewcommand\Re{\operatorname{Re}}
\renewcommand\Im{\operatorname{Im}}
\newcommand\MdR{\mbox{M}_d(\mathbb{R})}
\newcommand\GldR{\mbox{Gl}_d(\mathbb{R})}
\newcommand\OdR{\mbox{O}_d(\mathbb{R})}

\newcommand\Exp{\operatorname{Exp}}
\newcommand\diag{\operatorname{diag}}
\newcommand\supp{\operatorname{Supp}}

\renewcommand\det{\operatorname{det}}

\newcommand{\Log}{\operatorname{Log}}
\newcommand{\sgn}{\operatorname{sgn}}
\newcommand{\lcm}{\operatorname{lcm}}
\usepackage{multicol,lipsum,etoolbox}



\newcommand{\evan}[1]{{\color{blue} Evan says: #1}}
\newcommand{\jose}[1]{{\color{red} Jose says: #1}}

\author{Pedro H. Alves\\
\normalsize  Department of Mathematics\\[-0.8ex]
\normalsize Duke University
\and
Evan Randles\thanks{Corresponding author: erandles@colby.edu}\,\,\thanks{Research partially supported by Haynesville Project Grant}\\
\normalsize  Department of Mathematics\\[-0.8ex]
\normalsize Colby College
\and
Jos\'e M. Valdovinos \thanks{Research supported by the ANR project HEAD under grant agreement ANR-24-CE40-3260} \\
\normalsize Univ Toulouse, INSA Toulouse \\[-0.8ex]
\normalsize CNRS, IMT, Toulouse, France\\\\}

\title{Generalized Gaussian Estimates and Local Limit Theorems for Discrete Convolution Powers of Complex Functions\\ \textit{The $d$-dimensional case}}
\date{}
\begin{document}
\maketitle

\begin{abstract}
    We establish generalized Gaussian estimates and local limit theorems with cumulants, up to any order of accuracy, with sharp Gaussian-type error for the convolution powers of certain complex-valued functions on $\mathbb{Z}^d$. These global space-time estimates/error are written in terms of the Legendre-Fenchel transforms of positive-homogeneous polynomials and are mirrored by estimates satisfied by the heat kernels associated to a related class of partial differential operators. The results obtained here enjoy applications to the analysis and stability of numerical difference schemes to partial differential equations. This work extends several recent results, pertaining to one and several dimensions, of P. Diaconis, L. Saloff-Coste, J.-F. Coulombel, G. Faye, L. Coeuret, and the second author.
\end{abstract}



\section{Introduction and Main Results}

In this article, we establish generalized Gaussian bounds and local (central) limit theorems with cumulants, up to any order of accuracy, with Gaussian-type error for the convolution powers of certain complex-valued functions on $\mathbb{Z}^d$. Mirrored by the sharp error satisfied by the ``heat'' kernels which appear as attractors in local limits, our global space-time error is written in terms of the Legendre-Fenchel transform of positive-homogeneous polynomials, which were introduced by L. Saloff-Coste and the second author in \cite{RSC17,RSC17b,RSC20}. As we will see, for a given $\phi:\mathbb{Z}^d\to\mathbb{C}$, these polynomials often appear in series expansions of $\phi$'s Fourier transform about its local maxima.  Our results extend to $d$ dimensions the $1$-dimensional  Gaussian-type estimates of J.-F. Coulombel and G. Faye in \cite{CF22} and extend the local limit theorems of L. Coeuret in \cite{Co25} and J.-F. Coulombel and G. Faye in \cite{CF24}. In the context of $\mathbb{Z}^d$, our results also extend the Gaussian estimates and local limit theorems of the second author and L. Saloff-Coste in \cite{RSC17} by weakening hypotheses and significantly sharpening error.\\

\noindent We denote by $\ell^1(\mathbb{Z}^d)$ the set of absolutely summable complex-valued functions on $\mathbb{Z}^d$. Equipped with its usual norm $\|\cdot\|_1$, $\ell^1(\mathbb{Z}^d)$ is a Banach algebra under the convolution product
\begin{equation*}
    \phi\ast \lambda(x)=\sum_{y\in\mathbb{Z}^d}\phi(x-y)\lambda(y)
\end{equation*}
defined for $x\in\mathbb{Z}^d$. Following the articles \cite{DSC14, RSC15, RSC17, BR22, CF22, R23, CF24,Co25}, for a given $\phi\in\ell^1(\mathbb{Z}^d)$, we study the asymptotic behavior of the convolution powers $\phi^{(n)}\in\ell^1(\mathbb{Z}^d)$ defined iteratively by setting $\phi^{(1)}=\phi$ and, for $n\geq 2$, $\phi^{(n)}=\phi^{(n-1)}\ast \phi$. In the special case that $\phi$ is non-negative with $\sum_x \phi=1$, the convolution powers $\phi^{(n)}$ are well-studied objects in probability theory: they appear as the transition kernels of a random walk on $\mathbb{Z}^d$. In this case, the large-$n$ asymptotic description of $\phi^{(n)}$ is described by the classical local (central) limit theorem and under mild conditions can be shown to satisfy two-sided Gaussian estimates\cite{Woess2000,Spitzer,LawlerLimic2010}. When $\phi$ is signed or complex valued, the convolution powers can exhibit much richer behavior, much of which is not seen in the probabilistic setting. Our study of convolution powers of complex-valued functions is driven by its applications to numerical difference schemes in partial differential equations, especially the stability of such schemes. For a more thorough description of this subject (including history, applications, and recent developments), we point the reader to \cite{DSC14,RSC17,CF22,R23,CF24,Co25}.\\

\noindent Given $\phi\in\ell^1(\mathbb{Z}^d)$, its characteristic function/Fourier transform is defined by
\begin{equation*}
\widehat{\phi}(\xi)=\sum_{x\in\mathbb{Z}^d}\phi(x)e^{ix\cdot\xi}
\end{equation*}
for $\xi\in\mathbb{R}^d$ where $\cdot$ denotes the usual dot product. With this, we can understand $\phi$'s convolution powers through the identity
\begin{equation}\label{eq:FTIdentityReal}
\phi^{(n)}(x)=\frac{1}{(2\pi)^d}\int_{\mathbb{T}^d}\widehat{\phi}(\xi)^n e^{-i x\cdot\xi}\,d\xi
\end{equation}
for $n\in\mathbb{N}_+$ and $x\in\mathbb{Z}^d$; here $\mathbb{T}^d=(-\pi,\pi]^d$. As in \cite{CF22,CF24,Co25}, our Gaussian estimates and local limit theorems will be obtained under several assumptions on $\widehat{\phi}$, including strong regularity. To this end, we shall focus on the subspace $\mathcal{H}_d$ of $\ell^1(\mathbb{Z}^d)$ consisting of those $\phi$'s for which $\widehat{\phi}$ is holomorphic on a neighborhood of $\mathbb{T}^d$ in $\mathbb{C}^d$. Aligned with the perspective taken in \cite{CF24}, we remark that having $\phi\in\mathcal{H}_d$ is equivalent to the assumption that
\begin{equation*}
F_\phi(z)=\sum_{x\in\mathbb{Z}^d}\phi(x)z^x=\sum_{x\in\mathbb{Z}^d}\phi(x)z_1^{x_1}z_2^{x_2}\cdots z_d^{x_d}
\end{equation*}
is holomorphic on a poly-annulus of the form $\{z=(z_1,z_2,\dots,z_d)\in\mathbb{C}^d:1-\epsilon<\abs{z_j}<1+\epsilon,\,\,j=1,2,\dots,d \}$
for $\epsilon>0$. Of course, this equivalence is seen on account of the fact that $F_\phi(e^{i\xi_1},e^{i\xi_2},\dots,e^{i\xi_d})=\widehat{\phi}(\xi)$ for $\xi=(\xi_1,\xi_2,\dots,\xi_d)$. 
We remark that the collection of finitely-supported complex-valued functions on $\mathbb{Z}^d$ is contained in $\mathcal{H}_d$ and every $\phi\in\mathcal{H}_d$ has finite moments of all orders (and hence $\mathcal{H}_d\subseteq \mathcal{S}_d$ in the notation of \cite{RSC17} and \cite{R23}). We shall further restrict our attention to the set $\mathcal{H}_d^*$ consisting of those $\phi\in\mathcal{H}_d$ which are normalized in the sense that
\begin{equation}\label{eq:PhiHatNormalization}
\sup_{\xi\in \mathbb{T}^d}\abs{\widehat{\phi}(\xi)}=\sup_{\xi\in\mathbb{T}^d}\abs{F_\phi\left(e^{i\xi_1},e^{i\xi_2},\dots,e^{i\xi_d}\right)}=1.
\end{equation}
For $\phi\in \mathcal{H}_{d}^*$, we define
\begin{equation*}
    \Omega(\phi)=\left\{\xi\in\mathbb{T}^d:\abs{\widehat{\phi}(\xi)}=1\right\}.
\end{equation*}
Through the identity \eqref{eq:FTIdentityReal}, it is evident that the behavior of convolution powers is determined by the behavior of $\widehat{\phi}$ near to the points $\Omega(\phi)$. To this end, for $\xi_0\in\Omega(\phi)$, we define
\begin{equation}\label{eq:DefOfGamma}
\Gamma_{\xi_0}(\xi)=\Log\left(\frac{\widehat{\phi}(\xi+\xi_0)}{\widehat{\phi}(\xi_0)}\right)
\end{equation}
where $\Log$ denotes the principal branch of the logarithm; $\Gamma_{\xi_0}$ is evidently holomorphic on a neighborhood of $0$ in $\mathbb{C}^d$. Given that $\widehat{\phi}$ is maximized in absolute value at $\xi_0$, it is easy to see that the zeroth and first-order terms of the Maclaurin expansion for $\Gamma_{\xi_0}$ are zero and $i\alpha\cdot\xi$ for some $\alpha\in\mathbb{R}^d$, respectively. As illustrated in \cite{Tho65,DSC14,RSC15,RSC17,CF22,BR22, R23,CF24,Co25}, it is the nature of the next non-vanishing term(s) in the series expansion that determines the asymptotic nature of convolution powers. In this article (and consistent with \cite{RSC17}), we shall assume that these are \textit{positive-homogeneous polynomials} (Definition \ref{def:PosDefPoly}). As we will see, these polynomials generalize to $d$ dimensions the single-variable even-order polynomials $\xi\mapsto \beta\xi^{2m}$ for $\Re(\beta)>0$ and $m\in\mathbb{N}_+$. In the context of one dimension, the assumption that all such expansions begin (beyond the first order) with (the negative of) such even order polynomials are basic hypotheses in \cite{DSC14,CF22,CF24,Co25}.\\

\subsection{Positive-Homogeneous Polynomials}
\noindent To introduce positive-homogeneous polynomials, we first introduce the ``dilations" with which homogeneity is formulated. We denote by $\MdR$ the collection of $d\times d$ real matrices. The general linear group and orthogonal group will be denoted by $\GldR$ and $\OdR$, respectively, and their common identity will be denoted by $I_d$ or just $I$ when the dimension is clear. Correspondingly, their (real) Lie algebras taken equipped with the classical Lie bracket $[\cdot,\cdot]$ are $\operatorname{Lie}(\GldR)=\mathfrak{g}_d(\mathbb{R})=\MdR$ and $\operatorname{Lie}(\OdR)= \mathfrak{o}_d(\mathbb{R})=\{S\in\MdR: S\,\mbox{is skew symmetric}\}$. For $E\in\MdR$, we define
\begin{equation*}
t^E=\exp((\ln t) E)=\sum_{k=0}^\infty\frac{(\ln t)^k}{k!}E^k
\end{equation*}
for $t>0$. As discussed in \cite{RSC17,BR22,R23}, $\{t^E\}$ is a one-parameter subgroup of $\GldR$ in the sense that $\mathbb{R}^*\ni t\mapsto t^E\subseteq \GldR$ is a smooth homomorphism. Given a complex function $P$ defined on $\mathbb{R}^d$ and $E\in\MdR$, we say that \textbf{$P$ is homogeneous with respect to $E$} (or homogeneous with respect to $\{t^E\}$) if
\begin{equation}\label{eq:HomWRTE}
    tP(\xi)=P(t^{E}\xi)
\end{equation}
for all $t>0$ and $\xi=(\xi_1,\xi_2,\dots,\xi_d)\in\mathbb{R}^d$. The set of all such matrices $E$ for which \eqref{eq:HomWRTE} holds is called \textbf{the exponent set of $P$} and denoted by $\Exp(P)$. 
\begin{definition}\label{def:PosDefPoly}
    Let $P$ be a complex-valued multivariate polynomial in $d$ variables and consider the real-valued function $R:\mathbb{R}^d\to\mathbb{R}$ defined by
    \begin{equation*}
        R(\xi)=\Re P(\xi)
    \end{equation*}
    for $\xi\in\mathbb{R}^d$. We say that \textbf{$P$ is a positive-homogeneous polynomial} if $\Exp(P)\neq \varnothing$ and $R$ is positive-definite in the sense that $R(\xi)\geq 0$ and $R(\xi)=0$ only when $\xi=0$. 
\end{definition}

\begin{example}\label{ex:Intro1}
Here, we consider several examples of positive-homogeneous polynomials, including the canonical class of semi-elliptic ones.
    \begin{enumerate}
        \item Consider the polynomial $P$ defined by
        \begin{equation*}
            P(\eta,\zeta)=\eta^2+\zeta^4
        \end{equation*}
        for $(\eta,\zeta)\in\mathbb{R}^2$. We observe easily that $D=\diag(1/2,1/4)\in\Exp(P)$ because
        \begin{equation*}
            P(t^D(\eta,\zeta))=P(t^{1/2}\eta,t^{1/4}\zeta)=(t^{1/2}\eta)^2+(t^{1/4}\zeta)^4=t P(\eta,\zeta)
        \end{equation*}
        for all $t>0$ and $(\eta,\zeta)\in\mathbb{R}^2$. Since $R(\eta,\zeta)=P(\eta,\zeta)$ is clearly positive-definite, we conclude that $P$ is a positive-homogeneous polynomial.
        \item For a positive integer $m$, observe that
        \begin{equation*}
            \xi\mapsto \abs{\xi}^{2m}=(\xi_1^2+\xi_2^2+\cdots\xi_d^2)^m
        \end{equation*}
        is a positive-homogeneous polynomial with
        \begin{equation*}
            \Exp(\abs{\cdot}^{2m})=\frac{1}{2m}I+\mathfrak{o}_d(\mathbb{R})
        \end{equation*}
        where $I$ is identity in $\GldR$. 
        \item Let $\mathbf{n}=(n_1,n_2,\dots,n_d)$ be a $d$-tuple of positive integers (i.e., $\mathbf{n}\in\mathbb{N}_+^d$) and, for a multi-index $\alpha=(\alpha_1,\alpha_2,\dots,\alpha_d)\in \mathbb{N}^d$, define
        \begin{equation*}
            \abs{\alpha:\mathbf{n}}=\sum_{k=1}^d \frac{\alpha_k}{n_k}.
        \end{equation*}
        In the language of L. H\"{o}rmander \cite{Hormander1983}, a polynomial $P$ is \textbf{semi-elliptic} if it can be written in the form
        \begin{equation}\label{eq:SemiEllipticPolyGeneral}
        P(\xi)=\sum_{|\alpha:\mathbf{n}|=1}a_\alpha \xi^\alpha=\sum_{|\alpha:\mathbf{n}|=1}a_\alpha \xi_1^{\alpha_1}\xi_2^{\alpha_2}\cdots\xi_d^{\alpha_d}
        \end{equation}
        for some $\mathbf{n}\in\mathbb{N}_+^d$ and coefficients $\{a_\alpha\}\subseteq\mathbb{C}$. Given $D=\diag(1/n_1,1/n_2,\dots,1/n_d)\in\MdR$, observe that
        \begin{equation*}
            P(t^D\xi)=\sum_{|\alpha:\mathbf{n}|=1}a_\alpha (t^{1/n_1}\xi_1)^{\alpha_1}(t^{1/n_2}\xi_2)^{\alpha_2}\cdots(t^{1/n_d}\xi_d)^{\alpha_d}=\sum_{|\alpha:\mathbf{n}|=1}a_\alpha t^{|\alpha:\mathbf{n}|}\xi^\alpha=tP(\xi) 
        \end{equation*}
        
        for all $t>0$ and $\xi\in\mathbb{R}^d$. Thus, $D\in\Exp(P)$. It is not hard to see that, for a semi-elliptic polynomial $P$ to have positive-definite real part, $\mathbf{n}=2\mathbf{m}$ for $\mathbf{m}\in\mathbb{N}_+^d$ so that
        \begin{equation}\label{eq:PosSemiElliptic1}
            P(\xi)=\sum_{\abs{\alpha:2\mathbf{m}}=1}a_\alpha\xi^\alpha=\sum_{\abs{\alpha:\mathbf{m}}=2}a_\alpha\xi^{\alpha}.
        \end{equation}
        Consequently, any semi-elliptic polynomial of the form \eqref{eq:SemiEllipticPolyGeneral} is a positive-homogeneous polynomial provided that $\mathbf{n}=2\mathbf{m}$ for $\mathbf{m}\in\mathbb{N}_+^d$ and
        \begin{equation}\label{eq:PosSemiElliptic2}
            R(\xi)=\sum_{\abs{\alpha:\mathbf{m}}=2} \Re(a_{\alpha})\xi^{\alpha}
        \end{equation}
        is positive-definite. In the sequel, a \textbf{positive semi-elliptic polynomial} is one of the form \eqref{eq:PosSemiElliptic1} (for $\mathbf{m}\in\mathbb{N}_+^d$) whose real part \eqref{eq:PosSemiElliptic2} is positive definite. As we have just shown, every positive semi-elliptic polynomial $P$ is a positive-homogeneous polynomial with 
        \begin{equation*}
            D=\diag\left(\frac{1}{2m_1},\frac{1}{2m_2},\dots,\frac{1}{2m_d}\right)\in\Exp(P).
        \end{equation*}

       We see that the polynomial $P(\eta,\zeta)=\eta^2+\zeta^4$ of Item 1 is positive semi-elliptic with $\mathbf{m}=(1,2)$. It is not difficult to see that, for $a\in\mathbb{C}$, 
       \begin{equation*}
           P(\eta,\zeta)=\eta^2+a\eta\zeta^2+\zeta^4
       \end{equation*}
       is positive semi-elliptic whenever $-2<\Re a<2$. Looking to Item 2, $\xi\mapsto \abs{\xi}^{2m}$ is also semi-elliptic with $\mathbf{m}=(m,m,\dots,m)$. By the same logic, it is easy to verify that all positive-definite elliptic polynomials of order $2m$ are positive semi-elliptic and have $I/2m$ in their exponent sets. Semi-elliptic polynomials arise as symbols for a class of hypoelliptic partial differential operators which were introduced by F. Browder in \cite{Br57}. For further analysis of such semi-elliptic operators, we refer the reader to \cite{Hormander1983, RSC17b,RSC20}.
       
       \item We consider the polynomial
       \begin{equation*}
       P(\eta,\zeta)=\eta^2 + \eta^4 + 2 \eta \zeta - 4 \eta^3 \zeta + \zeta^2 + 6 \eta^2 \zeta^2 - 4 \eta \zeta^3 + \zeta^4
       \end{equation*}
       for $(\eta,\zeta)\in\mathbb{R}^2$; this is evidently not semi-elliptic. We invite the reader to verify the somewhat tedious computation that $P$ is a positive-homogeneous polynomial with
       \begin{equation*}
       E=\begin{pmatrix}
       3/8 & 1/8\\
       1/8 & 3/8
       \end{pmatrix}\in\Exp(P).
       \end{equation*}
       An easy way to verify this is to first observe that $P$ can be factored as
        \begin{equation*}
           P(\eta, \zeta) = (\eta + \zeta)^2 +(\eta - \zeta)^4
       \end{equation*}
       for $(\eta,\zeta)\in\mathbb{R}^2$. With this, 
       \begin{eqnarray*}
           P(t^E(\eta, \zeta)) & =& P\left (t^{1/2}(\eta+\zeta)/2 + t^{1/4}(\eta-\zeta)/2, t^{1/2}(\eta+\zeta)/2 - t^{1/4}(\eta-\zeta)/2 \right)\\
           & = &\left(t^{1/2}(\eta+\zeta) \right)^2 +  \left(t^{1/4}(\eta-\zeta) \right)^4 \\
           & =& tP(\eta, \zeta)
       \end{eqnarray*} 
       for $t>0$ and $(\eta,\zeta)\in\mathbb{R}^2$ from which we see that $E\in\Exp(P)$. The fact that this computation becomes easy through factoring is directly connected to the observation that $E$ is diagonalizable with $E=ADA^{-1}$ where $D=\diag(1/2,1/4)$ and 
       \begin{equation*}A = \begin{pmatrix}
           1/\sqrt{2} & -1/\sqrt{2} \\
           1/\sqrt{2} & 1/\sqrt{2} \\
       \end{pmatrix}.
       \end{equation*}
       Further, it is essentially the coordinate transformation defined by $A$ that allows us to nicely factor $P$. To see this, one can easily check that
      \begin{equation*}
      P_A(x,y):=P(A(x,y))=2x^2+4y^4
      \end{equation*}
      is semi-elliptic with $D=\diag(1/2,1/4)\in\Exp(P_A)$. With this, we see that
      \begin{equation*}
      P(t^E(\eta,\zeta))=P(At^DA^{-1}(\eta,\zeta))=P_A(t^DA^{-1}(\eta,\zeta))=t P_A(A^{-1}(\eta,\zeta))=tP(\eta,\zeta)
      \end{equation*}
      for all $t>0$ and $(\eta,\zeta)\in\mathbb{R}^2$. This example is closely connected to Example 7.3 of \cite{RSC17} which we revisit in Section \ref{sec:Misaligned} below.
    \end{enumerate}
\end{example}

\noindent In view of the final item in the example above, we see that not all positive-homogeneous polynomials are semi-elliptic. Still, as the example illustrates, it is at least sometimes possible to make a change of coordinates so that a positive-homogeneous polynomial becomes semi-elliptic in the new coordinate system. The following proposition shows that this, in fact, can always be done.

\begin{proposition}\label{prop:PosHomAreSemiElliptic}
    For any positive-homogeneous polynomial $P$ (in $d$ variables), there exists $A\in \GldR$ for which $P_A(\xi)=P(A\xi)$ is positive semi-elliptic. In particular, there exists some $\mathbf{m} = \left (m_1, m_2,\dots, m_d\right)\in\mathbb{N}_+^d$ for which
    \begin{equation*}
    P_A(\xi)=\sum_{\abs{\alpha:\mathbf{m}}=2}a_\alpha\xi^\alpha
    \end{equation*}
    and
    \begin{equation*}
        R_A(\xi)=R(A\xi)=\sum_{|\alpha:\mathbf{m}|=2}\Re(a_\alpha)\xi^\alpha
    \end{equation*}
    is positive-definite. Further
    \begin{equation*}
    D:=\diag\left(1/2m_1,1/2m_2,\dots,1/2m_d\right)\in\Exp(P_A)
    \end{equation*}
    and $E:=ADA^{-1}\in\Exp(P)$. 
\end{proposition}

\noindent The above proposition extends Proposition 2.2 of \cite{RSC17} by removing an unnecessary hypothesis requiring $\Exp(P)$ to contain a element with real spectrum. In fact, this condition was baked into the definition of positive-homogeneous polynomial in \cite{RSC17} and so, in view of the proposition, our more general definition (Definition \ref{def:PosDefPoly}) coincides with the class of positive-homogeneous polynomials studied in that article (and also appear in \cite{RSC17b,RSC20,BR22,R23}). As the proposition's proof is not central to the focus of this article, we have placed it in Appendix \ref{ssec:PHP}.\\

\subsection{Positive-Homogeneous Functions and the Legendre-Fenchel Transform}
\noindent We now discuss a more general concept of positive homogeneity taken from \cite{BR22} (and \cite{R23}) that will be very useful to us for stating and obtaining exponential bounds. For $E\in\MdR$, we say that $\{t^E\}$ is a \textbf{contracting group} if
\begin{equation*}
    \lim_{t\to 0}\|t^E\|=0;
\end{equation*}
here, $\|\cdot\|$ denotes the operator norm on $\MdR$. For continuous, positive-definite, and homogeneous functions (not necessarily polynomials), we have the following characterization from \cite{BR22}.
\begin{proposition}[Proposition 1.2 of \cite{BR22}]\label{prop:PosHomEquiv}
Let $R:\mathbb{R}^d\to\mathbb{R}$ be continuous, positive-definite, and have $\Exp(R)\neq \varnothing$. The following are equivalent:
\begin{enumerate}[(a)]
    \item The unital level set of $R$, 
    \begin{equation*}
        S_R=\{\eta\in \mathbb{R}^d:R(\eta)=1\},
    \end{equation*}
    is compact.
    \item For every $E\in\Exp(R)$, $\{t^E\}$ is contracting.
    \item There exists $E\in\Exp(R)$ for which $\{t^E\}$ is contracting. 
    \item We have
    \begin{equation*}
        \lim_{\xi\to\infty}R(\xi)=\infty.
    \end{equation*}
\end{enumerate}
\end{proposition}
\noindent In light of the above, we make the following definition:
\begin{definition}\label{def:PosDefFunc}
    Let $R$ be a continuous, positive definite function on $\mathbb{R}^d$ with $\Exp(R)\neq \varnothing$. If any (and hence every) of the conditions of Proposition \ref{prop:PosHomEquiv} is fulfilled, we say that $R$ is a positive-homogeneous function.
\end{definition}

\noindent Given a positive-homogeneous function $R$, it follows from Proposition \ref{prop:PosHomEquiv} that
\begin{equation*}
    \tr E=\tr E'>0
\end{equation*}
for every $E,E'\in\Exp(R)$ (see, e.g., Corollary 2.2 of \cite{BR22}). With this, we define the \textbf{homogeneous order} of a positive-homogeneous function $R$ to be the positive number
\begin{equation*}
    \mu_R=\tr E
\end{equation*}
for $E\in\Exp(R)$. As shown in \cite{DSC14,RSC17,BR22,R23} and as we will see in this article, the homogeneous order will determine the so-called on-diagonal asymptotics of convolution powers for $\phi\in\mathcal{H}_d^*$; it plays the role of $1/2m$ in the estimates and asymptotics of \cite{DSC14,RSC15,CF22,CF24,Co25}.  \\

\noindent For a positive-homogeneous polynomial $P$ with real part $R$, Proposition \ref{prop:PosHomAreSemiElliptic} guarantees that, for 
\begin{equation*}
    D=\diag(1/2m_1,1/2m_2,\dots,1/2m_d),
\end{equation*}
$E=ADA^{-1}\in \Exp(P)$ and hence $E=ADA^{-1}\in \Exp(R)$. Since $\|t^E\|\asymp \|t^D\|$ for all\footnote{For two non-negative functions $f$ and $g$ on a set $X$, the notation $f\asymp g$ means that there are positive constants $C$ and $C'$ for which $C f(x)\leq g(x)\leq C'g(x)$ for all $x\in X$.} $t>0$, it follows immediately that $R$ is a positive-homogeneous function in the sense of Definition \ref{def:PosDefFunc} and its homogeneous order is\footnote{We remark that, for a positive-homogeneous polynomial $P$, $\Exp(P)\subseteq\Exp(R)$ and hence it makes sense to define the order of $P$ to be $\mu_P=\mu_R$.}
\begin{equation*}
    \mu_P=\mu_R=\frac{1}{2m_1}+\frac{1}{2m_2}+\cdots+\frac{1}{2m_d}=\abs{\mathbf{1}:2\mathbf{m}}
\end{equation*}
where $\mathbf{1}=(1,1,\dots,1)\in\mathbb{N}_+^d$. In short, the real parts of positive-homogeneous polynomials are always positive-homogeneous functions. In \cite{BR22}, many non-polynomial examples of positive-homogeneous functions are studied, including some fairly exotic ones. In this article, we will consider two main types of positive-homogeneous functions. The first are positive-homogeneous polynomials arising as principal terms in the Maclaurin expansion of $\Gamma_{\xi_0}(\xi)$ for $\xi_0\in\Omega(\phi)$. The second, which will appear naturally as the error/upper bounds in our theorems, are introduced as follows. \\

\noindent Let $P$ be a positive-homogeneous polynomial with $R=\Re(P)$. We define the Legendre-Fenchel transform of $R$ by
\begin{equation*}
R^{\#}(x)=\sup_{\xi\in\mathbb{R}^d}\{x\cdot\xi-R(\xi)\}
\end{equation*}
for $x\in\mathbb{R}^d$. As illustrated in \cite{BD96,RSC17,RSC17b,RSC20} and discussed below, the Legendre-Fenchel transform $R^{\#}$ plays a primary role in off-diagonal estimates for convolution powers and, relatedly, the heat kernels associated to higher-order partial differential operators. As the following proposition guarantees, $R^{\#}$ is itself a (continuous) positive-homogeneous function.

\begin{proposition}[Proposition 8.15 of \cite{RSC17}]\label{prop:ExpSetLegFenchelTransform}
    Let $P$ be a positive-homogeneous polynomial with $R=\Re P$ and exponent set $\Exp(P)$. Then the Legendre-Fenchel transform $R^{\#}$ of $R$ is a continuous positive-homogeneous function with
    \begin{equation*}
        \{F\in\MdR: (I-F)^{\top}=I-F^{\top}\in\Exp(R)\}\subseteq \Exp(R^{\#}).
    \end{equation*}
\end{proposition}

\begin{example}
    We return back to the examples considered in Example \ref{ex:Intro1}.
    \begin{enumerate}
        \item For $P(\eta,\zeta)=R(\eta,\zeta)=\eta^2+\zeta^4$, we have $\mu_P=1/2+1/4=3/4$ and it is readily computed that
        \begin{equation*}
            R^{\#}(x,y)=\frac{1}{4} x^2+ \frac{3}{4^{4/3}} \abs{y}^{4/3}
        \end{equation*}
        for $(x,y)\in\mathbb{R}^2$. Here, it is easy to see that $R^{\#}$ is a positive-homogeneous function with $F=\diag(1/2,3/4)\in\Exp(R^{\#})$.
        \item For a positive semi-elliptic polynomial $P$ of the form \eqref{eq:PosSemiElliptic1}, we previously showed that $\mu_P=\abs{\mathbf{1}:2\mathbf{m}}$. While an explicit formula for $R^{\#}$ is difficult in general to obtain, it is easy to establish that
        \begin{equation*}
            R^{\#}(x)\asymp \abs{x_1}^{2m_1/(2m_1-1)}+\abs{x_2}^{2m_2/(2m_2-1)}+\cdots+\abs{x_d}^{2m_d/(2m_d-1)}
        \end{equation*}
        for $x=(x_1,x_2,\dots,x_d)\in\mathbb{R}^d$ thanks to Proposition \ref{prop:LFCompare} below. In fact, the proposition guarantees that this comparison is available, after a suitable coordinate transformation, to $R^{\#}$ where $R=\Re P$ for any positive-homogeneous polynomial $P$.
        \item For the positive-homogeneous polynomials $\xi\mapsto P(\xi)=R(\xi)=\abs{\xi}^{2m}$ for a positive integer $m$, we have $\mu_P=\tr(I/2m)=d/2m$ and it is readily calculated that
        \begin{equation*}
            R^{\#}(x)=(2m-1)\abs{\frac{x}{2m}}^{2m/(2m-1)}
        \end{equation*}
        for $x\in\mathbb{R}^d$.
    \end{enumerate}
\end{example}

\noindent We will find the following proposition useful throughout this article; a proof is given in Subsection \ref{ssec:PHFAsymp}. 
\begin{proposition}\label{prop:LFCompare}
    Let $P$ be a positive  semi-elliptic polynomial of the form
    \begin{equation*}
        P(\xi)=\sum_{\abs{\alpha:\mathbf{m}}=2}a_\alpha \xi^\alpha
    \end{equation*}
    for $\xi\in\mathbb{R}^d$ where $\mathbf{m}=(m_1,m_2,\dots,m_d)\in\mathbb{N}_+^d$ and let $R=\Re P$. Then
    \begin{equation}\label{eq:LFAsympt}
        R^{\#}(x)\asymp \abs{x_1}^{2m_1/(2m_1-1)}+\abs{x_2}^{2m_2/(2m_2-1)}+\cdots+\abs{x_d}^{2m_d/(2m_d-1)}
    \end{equation}
    for $x=(x_1,x_2,\dots,x_d)\in\mathbb{R}^d$. More generally, if $P$ is a positive-homogeneous polynomial and $A\in\GldR$ is that guaranteed by Proposition \ref{prop:PosHomAreSemiElliptic} which makes $P_A$ semi-elliptic with $\mathbf{m}=(m_1,m_2,\dots,m_d)\in\mathbb{N}_+^d$. Then $R^{\#}(x)=R_A^{\#}(A^{\top}x)$ for all $x\in\mathbb{R}^d$ and $R_A=\Re P_A$ satisfies the asymptotic \eqref{eq:LFAsympt}.
\end{proposition}

\subsection{Hypotheses and Main Results}
\noindent We are now in a position to present the hypotheses under which our theorems are stated. For $\phi\in\mathcal{H}_d^*$, let $\xi_0\in\Omega(\phi)$ and consider $\Gamma_{\xi_0}$ defined by \eqref{eq:DefOfGamma}. We take the following definition from \cite{RSC17}.
\begin{definition}\label{def:PosHomType}
  We say that \textbf{$\xi_0$ is a point of positive-homogeneous type for $\widehat{\phi}$} if
  \begin{equation*}
      \Gamma_{\xi_0}(\xi)=i\alpha_{\xi_0}\cdot\xi-P_{\xi_0}(\xi)+\Upsilon_{\xi_0}(\xi)
  \end{equation*}
where $\alpha_{\xi_0}\in\mathbb{R}^d$, $P_{\xi_0}$ is a positive-homogeneous polynomial (with $R_{\xi_0}=\Re P_{\xi_0}$), and $\Upsilon_{\xi_0}(\xi)=o(R_{\xi_0}(\xi))$
as $\xi\to 0$ (for $\xi\in\mathbb{R}^d$).
\end{definition}
\noindent As $\widehat{\phi}$ is maximized in absolute value at $\xi_0\in\Omega(\phi)$, $\xi\mapsto i\alpha_{\xi_0}\cdot\xi$ must be purely imaginary and so necessarily $\alpha_{\xi_0}\in\mathbb{R}^d$. What is essential to the definition above is that we ask the Maclaurin expansion for $\Gamma_{\xi_0}$ to be dominated at low order by $i\alpha_{\xi_0}\cdot\xi-P_{\xi_0}(\xi)$ where $P_{\xi_0}$ is a positive-homogeneous polynomial. This generalizes to $d$ dimensions the parabolic case of \cite{CF24} (see, Assumption 3 of \cite{CF24} and Assumption 1.2 of \cite{CF22}), and, equivalently, the cases of Type 1 and Type $\gamma$ of \cite{RSC15} and \cite{Tho65}, respectively. In the one-dimensional setting, a result of Thom\'{e}e \cite{Tho65} (generalized by Coulombel and Faye in \cite{CF22}) guarantees that, when $\Omega(\phi)\neq \mathbb{T}$, there are only two possibilities for points $\xi_0\in\Omega(\phi)$ when $\phi\in\mathcal{H}_1^*$. The first is for $\xi_0$ to be of positive-homogeneous type for $\widehat{\phi}$. The second asks that the dominant low-order terms of $\Gamma_{\xi_0}$ be purely imaginary (i.e., Type 2 of \cite{RSC15} coinciding with the dispersive case discussed in \cite{CF24}); the existence of points of the second instance implies that associated finite difference schemes are unstable in the maximum norm \cite[Theorem 3]{Tho65}. In $d$-dimensions, the article \cite{R23} considers an analogy of the second case (called points of imaginary-homogeneous type) and establishes local limit theorems in that setting; however, in $d\geq 2$ dimensions, the two cases are no longer collectively exhaustive (see \cite{RSC23}). We refer the reader to the introduction of \cite{R23} which discusses these cases, in one and several dimensions, and their implications for stability of numerical difference schemes.\\ 

\noindent Beyond the hypothesis that $\phi\in\mathcal{H}_d^*$, our theorems are stated under the assumption that every element of $\Omega(\phi)$ is of positive-homogeneous type for $\widehat{\phi}$.  The fact that, for each $\xi\in\Omega(\phi)$, $R_{\xi}$ is continuous and positive-definite guarantees immediately that $\Omega(\phi)\subseteq\mathbb{T}^d$ consists entirely of isolated points (see \cite[Proposition 4.1]{RSC17}). Since $\mathbb{T}^d$ is compact, it then follows that $\Omega(\phi)$ is finite and, in this case, we write
\begin{equation*}
    \Omega(\phi)=\{\xi_1,\xi_2,\dots,\xi_K\}.
\end{equation*}
For each $k=1,2,\dots,K$, by an abuse of notation, we will write $\Gamma_{k}=\Gamma_{\xi_k}$, $\alpha_k=\alpha_{\xi_k}$, $P_k=P_{\xi_k}$, $R_k=R_{\xi_k}$, $\mu_k=\mu_{P_{\xi_k}}$, and $\Upsilon_k=\Upsilon_{\xi_k}$. In this notation, our first theorem is as follows.

\begin{theorem}\label{thm:GeneralGaussEstimate}
    Let $\phi\in\mathcal{H}_d^*$ and assume that every member of $\Omega(\phi)$ is a point of positive-homogeneous type for $\widehat{\phi}$. With this, we assume the notation of the previous paragraph. There exists a constant $L>0$ and positive constants $M_0,M_1,\dots,M_K, C_1,C_2,\dots, C_K$ for which
    \begin{equation*}
        \abs{\phi^{(n)}(x)}\leq \mathcal{E}(n,x)
        \end{equation*}
       for all $n\in\mathbb{N}_+$ and $x\in\mathbb{Z}^d$ where
        \begin{equation*}
            \mathcal{E}(n,x)=\begin{cases}
            \displaystyle e^{-M_0(n+\abs{x})} & \abs{x}>Ln\\
            \displaystyle\sum_{k=1}^K\frac{C_k}{n^{\mu_k}}\exp\left(-nM_kR^{\#}_k\left(\frac{x-n\alpha_k}{n}\right)\right) & \abs{x}\leq Ln
        \end{cases}.
    \end{equation*}
\end{theorem}

\noindent As was shown in \cite{CF24} in the context of one dimension and as we will show in Section \ref{sec:ProofOfMain}, the previous theorem's dependence on $L$ can be removed when $\phi$ is finitely supported. In this case, we have the following theorem which gives global space-time Gaussian-type bounds.
\begin{theorem}\label{thm:FiniteSupport}
    Let $\phi:\mathbb{Z}^d\to\mathbb{C}$ be finitely supported and normalized so that $\phi\in\mathcal{H}_d^*$. If every point of $\Omega(\phi)$ is of positive-homogeneous type for $\widehat{\phi}$, then there are positive constants $M_1,M_2,\dots,M_K, C_1,C_2,\dots,C_K$ for which
    \begin{equation*}
        \abs{\phi^{(n)}(x)}\leq \sum_{k=1}^K\frac{C_k}{n^{\mu_k}}\exp\left(-nM_k R_k^{\#}\left(\frac{x-n\alpha_k}{n}\right)\right)
    \end{equation*}
    for $n\in\mathbb{N}_+$ and $x\in\mathbb{Z}^d$.
\end{theorem}
\noindent In the context of one-dimension, Theorems \ref{thm:GeneralGaussEstimate} and \ref{thm:FiniteSupport} extend Theorem 1.6 of \cite{CF22}; the first corresponds to the implicit case and the second to the explicit case. Beyond extending these theorems to $d$ dimensions, the theorems also remove multiple hypotheses (Assumptions 1.3-1.5) of Theorem 1.6 of \cite{CF22}.   In the full $d$-dimensional setting, Theorem \ref{thm:FiniteSupport} extends Theorem 1.5 of \cite{RSC17} which asked that every expansion $\Gamma$ begin with the same drift $\alpha$ and positive-homogeneous polynomial $P$, i.e., $\alpha_\xi=\alpha$ and $P_\xi=P$ for all $\xi\in \Omega(\phi)$. The proof of Theorem 1.5 of \cite{RSC17} involved making a single contour deformation (via Cauchy's integral formula) that was then optimized to pick out $R^{\#}$ (which was common to all points $\xi\in\Omega(\phi)$). By contrast, here we extend to $d$ dimensions the argument given in \cite{CF24} that allows for several local contour deformations, each specific to $P_\xi$ for each $\xi\in\Omega(\phi)$ which naturally picks out the relevant Legendre-Fenchel transform $R^{\#}_{\xi}$. The local contour deformations we use are direct extensions to $d$-dimensions of those used in \cite{CF24}.\\

\noindent We next turn our attention to local limit theorems. For fairly complete history on local limit theorems for complex-valued functions (up to the date of publication), we refer the reader to \cite{R23}. That article covers local limit theorems in one and several dimensions treating the cases of positive-homogeneous type and so-called imaginary homogeneous type. It also contains a history of the problem dating back to its initial investigation by E. L. De Forest in the nineteenth century and subsequent investigation by I. J. Schoenberg \cite{Schoenberg1953}, T. N. E. Greville\cite{Greville1966} (and later G. Strang \cite{Strang1962}, V. Thom\'{e}e \cite{Tho65,Thomee1969}). What isn't described in \cite{R23} is the recent progress made by \cite{CF24} (and to some extent \cite{Co25}) in the context of one dimension ($d=1$).  The theorems in those articles give local limits with principle attractors ($H_k$), higher-order corrections (given by \textit{cumulants} up to a specified order of accuracy $M\in\mathbb{N}$), and generalized Gaussian-type error. Their results represent a significant improvement over the results in \cite{DSC14} and \cite{RSC17} (in the case $d=1$) where local limit theorems are presented with uniform error of the form $o(n^{-1/2m})$ and $o(n^{-\mu})$, respectively. Our theorem below, Theorem \ref{thm:LLT}, extends the results of \cite{RSC17} giving Gaussian-type error in place of the uniform error $o(n^{-\mu})$ (see Corollary \ref{cor:LLTCor}). The theorem also extends the one-dimensional results of \cite{CF24} and \cite{Co25} to $\mathbb{Z}^d$ for $d\geq 1$ by giving local limits in terms of attractors and higher-order corrections (via cumulants) to any specified order of accuracy. \\

\noindent To introduce our local limit theorems, we must first introduce the (base) attractors with which they are written. Given a positive-homogeneous polynomial $P$ (with $R=\Re P$ and positive-homogeneous order $\mu_P$), we define
\begin{equation}\label{eq:DefOfHP}
    H_P^t(x)=\frac{1}{(2\pi)^d}\int_{\mathbb{R}^d}e^{-tP(\xi)}e^{-ix\cdot\xi}\,d\xi
\end{equation}
for $t>0$ and $x\in\mathbb{R}^d$. These functions arise as ``heat" kernels corresponding to the operator $\Delta_{P}$ with symbol $P$, i.e., they appear as fundamental solutions to the (generally) higher-order heat-type equation
\begin{equation*}
\partial_t+\Delta_P=0
\end{equation*}
where $\Delta_P=P(i\partial)=P(i\partial_1,i\partial_2,\dots,i\partial_d)$. It is shown in \cite[Proposition 2.6]{RSC17} that, for each $t>0$, $x\mapsto H_P^t(x)$ is a Schwartz function and further satisfies the estimate
\begin{equation*}
\abs{H_P^t(x)}\leq \frac{C}{t^{\mu_P}}\exp\left(-tMR^{\#}(x/t)\right)
\end{equation*}
for $t>0$ and $x\in\mathbb{R}^d$, mirrored by the estimates in Theorems \ref{thm:GeneralGaussEstimate} and \ref{thm:FiniteSupport}. The role that these off-diagonal estimates play in the analysis of variable-coefficient partial differential equations is the subject of the article \cite{RSC20}.  We will not say much more about these base attractors here but refer the reader to \cite{Davies1995,Davies1995a,BD96,Davies1997,DSC14,RSC17,RSC20,RSC23} for further discussion.\\

\begin{remark}\label{rmk:OneDHeatKer}
In the one-dimensional setting, the articles \cite{RSC15,Co25,CF24} use the notation $H_l^\beta=H_{2m}^\beta$ to denote the present article's $H_P=H_P^1$ where $P(\xi)=\beta \xi^l$ where $l=2m$ is a positive even integer and $\beta$ is a complex number with $\Re(\beta)>0$. As these one-variable heat kernels will appear in some of our examples, we will distinguish them by using lowercase font. In particular, for an even integer $l=2m$ and $\beta\in\mathbb{C}$ with $\Re(\beta)>0$, we set
\begin{equation}\label{eq:OneDHeatKer}
    h_l^\beta(x)=\frac{1}{2\pi}\int_{\mathbb{R}}e^{-\beta\xi^l}e^{-ix\xi}\,d\xi
\end{equation}
for $x\in\mathbb{R}$. This is easily seen to enjoy the scaling identity $\beta^{1/2m}h_{2m}^\beta(x)=h_{2m}(x/\beta^{1/2m})$ and $h_{2m}:=h_{2m}^1$ can be computed as
\begin{equation*}
    h_{2m}(x)=\frac{1}{\pi}\int_0^\infty e^{-u^{2m}}\cos(x u)\,du
\end{equation*}
for $x\in\mathbb{R}$. In the special case that $l=2m=2$, $h_2^{\beta}$ is the standard Gaussian/heat kernel given explicitly by
\begin{equation*}
    h_2^\beta(x)=\frac{1}{\sqrt{4\pi\beta}}\exp\left(-\frac{x^2}{4\beta}\right)
\end{equation*}
for $x\in\mathbb{R}$. For $l=2m=4$, $h_4$ is known as the bi-harmonic heat kernel (associated with the bi-Laplacian). A detailed and well-written asymptotic analysis of the bi-harmonic heat kernel (and various perturbations) can be found in \cite{Davies1995a}.\\
\end{remark}

\noindent We now introduce the higher-order corrections that depend on the  base attractors $H_{P_k}$ for $k=1,2\dots,K$ and cumulants appearing in the holomorphic error $\Upsilon_k$ for $k=1,2,\dots,K$. To introduce these precisely, let's first focus on a single pair $(P,\Upsilon)$ where $P$ is a positive-homogeneous polynomial and $\Upsilon$ is holomorphic on a neighborhood of $0$ in $\mathbb{C}^d$ having $\Upsilon(\xi)=o(R(\xi))$ as $\xi\to 0$ in $\mathbb{R}^d$. Thanks to Proposition \ref{prop:PosHomAreSemiElliptic}, we take $A\in \GldR$ for which $P_A$ is positive semi-elliptic with $\mathbf{m}=(m_1,m_2,\dots,m_d)\in\mathbb{N}_+^d$. We set $m=\lcm(\mathbf{m})$ and $\mathbf{\kappa}=(\kappa_1,\kappa_2,\dots,\kappa_d)\in\mathbb{N}_+^d$ for which $m_j\kappa_j=m$ for $j=1,2,\dots,d$. We can\footnote{As we show at the beginning of Section \ref{sec:Technical}, the assumption that $\Upsilon_A(\xi)=o(R_A(\xi))$ as $\xi\to 0$ ensures that all non-zero terms $b_\beta z^{\beta}$ in the series $\Upsilon_A$ must have $\abs{\beta:2\mathbf{m}}>1$.}   write
\begin{equation*}
    \Upsilon_A(z)=\Upsilon(Az)=\sum_{\abs{\beta:2\mathbf{m}}>1}b_\beta z^\beta
\end{equation*}
where this series converges absolutely and uniformly on some neighborhood of $0$ in $\mathbb{C}^d$. The key to our higher-order corrections is to aggregate the terms of this series in the following way. Upon noting that
\begin{equation*}
    \Lambda(\beta):=\beta\cdot\mathbf{\kappa}-2m\in\mathbb{N}_+
\end{equation*}
whenever $\abs{\beta:2\mathbf{m}}>1$, we can write
\begin{equation*}
    \Upsilon_A(z)=\sum_{\lambda=1}^\infty \frac{S_\lambda(z)}{\lambda!}
\end{equation*}
where, for each $\lambda\in\mathbb{N}_+$,
\begin{equation*}
S_\lambda(z)=\lambda!\sum_{\Lambda(\beta)=\lambda}b_\beta z^\beta. 
\end{equation*}
As the exponential $e^{\Upsilon(z)}$ appears essentially in our arguments, we aim to understand the series expansion of the exponential of a series. To this end, we introduce the complete Bell polynomials $B_\lambda$ which are characterized by the formal series identity
\begin{equation}\label{eq:BellGenerating}
\exp(\sum_{l=1}^\infty \frac{s_l}{l!} t^l)=\sum_{\lambda=0}^\infty \frac{B_\lambda(s_1,s_2,\dots,s_\lambda)}{\lambda!}t^\lambda.
\end{equation}
For properties and applications of Bell polynomials\footnote{As O'Sullivan explains in \cite{OSullivan2022}, these polynomials were studied by A. De Moivre and L. F. A. Arbogast prior to E. T. Bell's work. Bell called them ``exponential polynomials" \cite{BellExpPoly}.}, we refer the reader to L. Comtet's classic reference \cite{ComtetBook} and C. O'Sullivan's recent survey \cite{OSullivan2022}. For the reader's convenience, we have listed several Bell polynomials in Table \ref{tab:Bell}.
\begin{center}
\begin{tabular}{|l|l|}
\hline
$\lambda$ & $B_\lambda(s_1, s_2, \dots, s_\lambda)$ \\
\hline
\hline
\rule{0pt}{3ex} 0   & $1$ \\
\hline
\rule{0pt}{3ex} 1   & $s_{1}$ \\
\hline
\rule{0pt}{3ex} 2   & $s_{1}^{2} + s_{2}$ \\
\hline
\rule{0pt}{3ex} 3   & $s_{1}^{3} + 3 s_{1} s_{2} + s_{3}$ \\
\hline
\rule{0pt}{3ex} 4   & $s_{1}^{4} + 6 s_{1}^{2} s_{2} + 4 s_{1} s_{3} + 3 s_{2}^{2} + s_{4}$ \\
\hline
\rule{0pt}{3ex} 5   & $s_{1}^{5} + 10 s_{1}^{3} s_{2} + 10 s_{1}^{2} s_{3} + 15 s_{1} s_{2}^{2} + 5 s_{1} s_{4} + 10 s_{2} s_{3} + s_{5}$ \\
\hline
\rule{0pt}{3ex} 6   & $s_1^6 + 15s_1^4s_2 + 20s_1^3s_3 + 45s_1^2s_2^2 + 15s_2^3 + 15s_1^2s_4 + 60s_1s_2s_3 + 6s_1s_5 + 15s_2s_4 + 10s_3^2 + s_6$ \\
\hline
\rule{0pt}{3ex} 7   & $s_1^7 + 21s_1^5s_2 + 35s_1^4s_3 + 105s_1^3s_2^2 + 35s_1^3s_4 + 210s_1^2s_2s_3 + 21s_1^2s_5 $ \\
\rule{0pt}{3ex} & $\hspace{1cm}+ 105s_1s_2^3 + 105s_1s_2s_4+ 70s_1s_3^2 + 7s_1s_6 + 105s_2^2s_3 + 21s_2s_5 + 35s_3s_4 + s_7$\\
\hline
\end{tabular}
\captionof{table}{The complete Bell polynomials $B_\lambda$ for $\lambda=0,1,2,\dots,7$.}\label{tab:Bell}
\end{center}
\vspace{.5cm}
Momentarily forgoing a discussion of convergence, in Section \ref{sec:Technical} we will find that
\begin{equation*}
    e^{n\Upsilon(\xi)}=\sum_{\lambda=0}^\infty\frac{1}{\lambda!}B_\lambda(nS_1(A^{-1}\xi),nS_2(A^{-1}\xi),\dots,nS_\lambda(A^{-1}\xi))
\end{equation*}
for small enough $\xi$. With this, we now introduce the operators that have the above series coefficients as symbols. Given the pair $(P,\Upsilon)$ and $A\in\GldR$ which makes $P_A$ positive semi-elliptic, for each pair of natural numbers $n$ and $\lambda$, consider the partial differential operator defined by
\begin{equation}\label{eq:QLambdaOpertor}
    Q^n_\lambda=\frac{1}{\lambda!}B_\lambda(nS_1(A^{-1}(i\partial),nS_2(A^{-1}(i\partial),\dots,nS_\lambda(A^{-1}(i\partial))
\end{equation}
where $i\partial=(i\partial_1,i\partial_2,\dots,i\partial_d)$. We will take by convention $Q_0^n$ to be the identity operator (for all $n$) and set $Q_\lambda=Q_\lambda^1$ for all $\lambda$. It should be noted that the derivative vector $A^{-1}(i\partial)$ is essentially that used in the references \cite{RSC17b,RSC20} where it is written as $D_{\mathbf{w}}$ for a basis $\mathbf{w}$ of $\mathbb{R}^d$; the coordinate transformation $A$ characterizes the basis $\mathbf{w}$. We avoid this notation here to avoid confusion with our diagonal matrices.\\

\noindent The partial differential operators $Q_{\lambda}^{n}$ play a key role in the statement of our main local limit theorem. Upon noting that, for a positive homogeneous polynomial $P$, the matrix $A$ which makes $P_A$ semi-elliptic need not be unique (one can simply permute coordinates, e.g.), it is natural to ask: How do the operators $Q_{\lambda}^n$ depend, if at all, on $A$? This question is addressed by the following proposition. For the sake of smoothness in our presentation and to avoid an unnecessary digression, the proof is presented in Appendix \ref{ssec:PHP}. 

\begin{proposition}\label{prop:QLmbdOpIsIndependentOfA}
Let $P$, $\Upsilon$, $A$, and $Q_\lambda^n$ as above. Suppose that $\widetilde{A}\in \GldR$ makes $P_{\widetilde{A}}(\cdot)=P(\widetilde{A}\cdot)$ semi-elliptic with $\widetilde{\mathbf{m}}=(\widetilde{m}_1,\widetilde{m}_2,\ldots, \widetilde{m}_d )\in \mathbb{N}_{+}^{d}$. Following the same procedure described above to define $Q^{n}_{\lambda}$ in \eqref{eq:QLambdaOpertor}, for each $\lambda\in\mathbb{N}$ and $n\in\mathbb{N}_+$, we define $\widetilde{Q}^{n}_{\lambda}$ by 
    \begin{equation*}
       \widetilde{Q}^n_\lambda=\frac{1}{\lambda!}B_\lambda \left(n\widetilde{S}_1\left( \widetilde{A}^{-1}(i\partial)\right),n\widetilde{S}_2\left( \widetilde{A}^{-1}(i\partial)\right),\dots,n\widetilde{S}_\lambda\left(\widetilde{A}^{-1}(i\partial) \right) \right)
    \end{equation*}
    where the polynomials $\widetilde{S}_\lambda$ have been obtained from $\Upsilon_{\widetilde{A}}(\cdot)=\Upsilon(\widetilde{A}\cdot)$. Then 
    \begin{equation*}
        \widetilde{Q}_{\lambda}^n= Q_{\lambda}^n, \quad \forall \lambda\in \mathbb{N}, \,\, n\in \mathbb{N}_+.
    \end{equation*}
    In other words, the partial differential operators defined in \eqref{eq:QLambdaOpertor} are independent of the choice of $A$.
\end{proposition}

 \noindent Let's now consider $\phi\in\mathcal{H}_d^*$ for which every member of $\Omega(\phi)$ is of positive-homogeneous type for $\widehat{\phi}$ and assume the notation of the preceding theorems. For each $P_k$, we take $A_k\in\GldR$ for which $(P_k)_{A_k}(\xi)=P_k(A_k\xi)$ is positive semi-elliptic with $\mathbf{m}_k=(m_{1,k},m_{2,k},\dots,m_{d,k})\in\mathbb{N}_+^d$ and, by a slight abuse of notation, set $m_k=\lcm(\mathbf{m}_k)$. By the preceding construction, we shall write $S_{\lambda,k}$ for the polynomials associated to the holomorphic error $z\mapsto \Upsilon_k(A_kz)$ and denote by $Q_{\lambda,k}^n$ the differential operators defined by \eqref{eq:QLambdaOpertor} (each taking, as ingredients, $S_{\lambda,k}(z)$, $A_k$, and (implicitely) $z\mapsto \Upsilon_k(A_kz)$). In this notation, we have our main local limit theorem.

\begin{theorem}\label{thm:LLT}
    Let $\phi\in\mathcal{H}_d^*$ and suppose that every point of $\Omega(\phi)$ is of positive-homogeneous type for $\widehat{\phi}$. We assume the notation above.  Then, for any $L>0$ and choice of orders of accuracy $\lambda_1,\lambda_2,\dots,\lambda_K\in\mathbb{N}$, there are positive constants $M_1,M_2,\dots,M_k$ and $C_1,C_2,\dots,C_K$ for which
    \begin{eqnarray}\label{eq:LLT}\nonumber
    \lefteqn{\abs{\phi^{(n)}(x)-\sum_{k=1}^K\sum_{\lambda=0}^{\lambda_k}e^{-ix\cdot \xi_k}\widehat{\phi}(\xi_k)^n(Q^n_{\lambda,k} H_{P_k}^n)(x-n\alpha_k)}}\\
    &&\hspace{3cm}\leq \sum_{k=1}^K\frac{C_k}{n^{\mu_k+(\lambda_k+1)/2m_k}}\exp\left(-nM_kR_k^{\#}\left(\frac{x-n\alpha_k}{n}\right)\right)
    \end{eqnarray}
    for all $n\in\mathbb{N}_+$ and $x\in\mathbb{Z}^d$ with $\abs{x}\leq nL$. In the case that $\phi$ is finitely supported, \eqref{eq:LLT} holds uniformly for $n\in\mathbb{N}_+$ and $x\in\mathbb{Z}^d$ for an appropriate choice of constants $C_1,C_2,\dots, C_K, M_1,M_2,\dots,M_K$. 
\end{theorem}
\noindent The above local limit theorem gives precise approximations to convolution powers in terms of the attractors $Q_{\lambda,k}^n H_{P_k}^n$ to any desired order of accuracy (specified by $\lambda_1,\lambda_2,\dots,\lambda_K$); we shall refer to $H_{P_k}^n=Q_{0,k}^nH_{P_k}^n$ as the base attractors and $Q_{\lambda,k}^n H_{P_k}^n$, for $\lambda\geq 1$, the higher-order corrections. Mirroring the estimates of Theorems \ref{thm:GeneralGaussEstimate} and \ref{thm:FiniteSupport}, our Gaussian-type error is written in terms of the Legendre-Fenchel transforms of $R_k=\Re P_k$ for $k=1,\dots,K$. To illustrate the power and utility of the theorem, the remainder of this subsection is dedicated to presenting several consequences of this theorem and making several remarks that help to place the result in context with earlier work. In particular, we will see how the theorem improves upon the local limit theorems of \cite{RSC17,CF24,Co25}. In the subsection to follow, we shall illustrate the results of the theorem for the convolution powers of a simple real-valued function on $\mathbb{Z}^2$. We begin with the following corollary which gives us a local limit theorem with only base attractors.

\begin{corollary}\label{cor:LLTCor}
    Let $\phi\in\mathcal{H}_d^*$ and suppose that every point of $\Omega(\phi)$ is of positive-homogeneous type for $\widehat{\phi}$. Assuming the notation of the preceding theorems, for each $k=1,2,\dots,K$, define\footnote{We note that $\{\gamma\in\mathbb{N}_+:S_{\gamma,k}\neq 0\}$ is necessarily non-empty. For, if $\Upsilon_k\equiv 0$, then $\widehat{\phi}$ would not be $2\pi$ periodic.}
    \begin{equation*}
        \gamma_k=\min\left\{\gamma\in\mathbb{N}_+:S_{\gamma,k}\neq 0\right\}.
    \end{equation*}
    Then, for any $L>0$, there are constants $M_1,M_2,\dots,M_K$ and $C_1,C_2,\dots,C_K$ for which
    \begin{eqnarray*}
        \abs{\phi^{(n)}(x)-\sum_{k=1}^K e^{-ix\cdot\xi_k}\widehat{\phi}(\xi_k)^nH_{P_k}^n(x-n\alpha_k)}&\leq& \sum_{k=1}^K\frac{C_k}{n^{\mu_k+\gamma_k/2m_k}}\exp\left(-nM_k R_k^{\#}\left(\frac{x-n\alpha_k}{n}\right)\right)\\
        &\leq&\sum_{k=1}^K\frac{C_k}{n^{\mu_k+1/2m_k}}\exp\left(-nM_k R_k^{\#}\left(\frac{x-n\alpha_k}{n}\right)\right)
    \end{eqnarray*}
    for all $n\in\mathbb{N}_+$ and $x\in\mathbb{Z}^d$ with $\abs{x}\leq nL$. In the case that $\phi$ is finitely supported, for appropriately chosen constants $C_1,C_2,\dots,C_K$ and $M_1,M_2,\dots,M_K$, the above estimates hold uniformly for all $n\in\mathbb{N}_+$ and $x\in\mathbb{Z}^d$. 
\end{corollary}
\noindent Before proving the corollary, let's compare the statement to known local limit theorems in the context of $\mathbb{Z}^d$ (for $d\geq 1)$. For finitely supported functions, the corollary immediately extends Theorem 1.6 of \cite{RSC17} in replacing the uniform $o(n^{-\mu_\phi})$ error with Gaussian-type error which, on the diagonal, decays as $O(n^{-\mu_k-\gamma_k/2m_k})$. We note that Theorem 1.6 of \cite{RSC17} focuses only on attractors with minimal on-diagonal decay, i.e., those $k$'s for which $\mu_k=\mu_\phi=:\min_{j}\mu_j$, whereas the statement above includes all attractors; we touch further on this in Remark \ref{rmk:refRn}. To our knowledge, the only other such local limit theorem for complex functions on $\mathbb{Z}^d$ with Gaussian-type error appears in the first author's thesis as Theorem 3.3.1 \cite{AlvesThesis}. In that theorem, the on-diagonal decay is given in terms of the following constant: For a holomorphic error term $\Upsilon$ for which $\Upsilon(\xi)=o(R(\xi))$ as $\xi\to 0$ where $R=\Re P$, we define, for any $E\in\Exp(P)$,
    \begin{equation*}
        \lambda(\Upsilon):=\sup\left\{\lambda\geq 0:\lim_{t\to 0}t^{-(1+\lambda)}\sup_{R(\xi)\leq 1}\abs{\Upsilon(t^E\xi)}=0\right\}.
    \end{equation*}
It is straightfoward to verify that this is independent of $E\in\Exp(P)$. In studying the proof of Proposition 2.3.3 of \cite{AlvesThesis}, we find, for $k=1,2,\dots,K$,
\begin{eqnarray*}
   \lefteqn{\hspace{-1cm}\lambda_k= \lambda(\Upsilon_k)=\min\{\abs{\beta:2\mathbf{m}_k}:b_{\beta,k}\neq 0\}-1}\\
   &&\hspace{1cm}=\frac{1}{2m_k}\min\{\beta\cdot\kappa_k-2m_k:b_{\beta,k}\neq 0\}=\frac{1}{2m_k}\min\{\lambda:S_{\lambda,k}\neq 0\}=\gamma_k/2m_k.
\end{eqnarray*}
With this observation, we see that Corollary \ref{cor:LLTCor} recaptures exactly Theorem 3.3.1 of \cite{AlvesThesis}.

\begin{proof}[Proof of Corollary \ref{cor:LLTCor}]
We obtain the result by applying Theorem \ref{thm:LLT} to the orders of accuracy $\lambda_k=\gamma_k-1$ for $k=1,2,\dots,K$. In the case that, $\gamma_k=1$, this yields the attractor $Q_{0,k}^nH_{P_k}^n=H_{P_k}^n$ since $Q_{0,k}$ is the identity. If $\gamma_k>1$, we have $S_{\lambda,k}=0$ for all $1\leq \lambda<\gamma_k$. Consequently
\begin{equation*}
    Q_{\lambda,k}^n=\frac{1}{\lambda!}B_{\lambda}(nS_1(A^{-1}(i\partial),nS_2(A^{-1}(i\partial),\dots,nS_\lambda(A^{-1}(i\partial)))=\frac{1}{\lambda!}B_{\lambda}(0,0,\dots,0)
\end{equation*}
is the zero operator for every $\lambda=1,2,\dots,\gamma_{k}-1$. Thus,
\begin{eqnarray*}
\lefteqn{\hspace{-1cm}\sum_{\lambda=0}^{\gamma_k-1}e^{-ix\cdot\xi_k}\widehat{\phi}(\xi_k)^n (Q_{\lambda,k}^nH_{P_k}^n)(x-n\alpha_k)}\\
\hspace{2cm}&=&e^{-ix\cdot\xi_k}\widehat{\phi}(\xi_k)^n(Q_{0,k}^n H_{P_k}^n)(x-n\alpha_k)+\sum_{\lambda=1}^{\gamma_k-1}e^{-ix\cdot\xi_k}\widehat{\phi}(\xi_k)^n (Q_{\lambda,k}^nH_{P_k}^n)(x-n\alpha_k)\\
&=&e^{-ix\cdot\xi_k}\widehat{\phi}(\xi_k)^n(Q_{0,k}^nH_{P_k}^n)(x-n\alpha_k)+\sum_{\lambda=1}^{\gamma_k-1} 0\\
&=&e^{-ix\cdot\xi_k}\widehat{\phi}(\xi_k)^nH_{P_k}^n(x-n\alpha_k).
\end{eqnarray*}
Appealing to Theorem \ref{thm:LLT}, we obtain
\begin{eqnarray*}
\abs{\phi^{(n)}(x)-\sum_{k=1}^{K}e^{-ix\cdot\xi_k}\widehat{\phi}(\xi_k)^n H_{P_k}^n(x-n\alpha_k)}&=&
    \abs{\phi^{(n)}(x)-\sum_{k=1}^{K}\sum_{\lambda=0}^{\gamma_k-1}e^{-ix\cdot\xi_k}\widehat{\phi}(\xi_k)^n (Q_{\lambda,k}^n H_{P_k}^n)(x-n\alpha_k)}\\
    &\leq&\sum_{k=1}^K\frac{C_k}{n^{\mu_k+\gamma_k/2m_k}}\exp\left(-nM_kR_k^{\#}\left(\frac{x-n\alpha_k}{n}\right)\right)
\end{eqnarray*}
for $n\in\mathbb{N}_+$ and $x\in\mathbb{Z}^d$ with $\abs{x}\leq nL$ since $\lambda_k+1=\gamma_k-1+1=\gamma_k$ for $k=1,2\dots,K$. By the theorem, this result holds uniformly for $n\in\mathbb{N}_+$ and $x\in\mathbb{Z}^d$ in the case of finitely-supported $\phi$.
\end{proof}

\noindent Let us now make two comments to help understand the results in a way that might be useful for numerical simulations. The first aims at elucidating how the attractors $Q_{\lambda,k}^nH_{P_k}^n$ depends on $n$.\\

\noindent As we will see in Section \ref{sec:Technical}, each polynomial $S_\lambda$ enjoys a notion of homogeneity which yields, in particular,
\begin{equation*}
    n^{1+\lambda/2m} S_\lambda(A^{-1}\xi)=S_\lambda(n^DA^{-1}\xi)=S_\lambda (A^{-1}n^E\xi)
\end{equation*}
for all $n\in\mathbb{N}_+$ and $\xi\in\mathbb{R}^d$ where $D=\diag(1/2m_1,1/2m_2,\dots,1/2m_d)$ and $E=ADA^{-1}\in\Exp(P)$. Further, thanks to the homogeneity of Bell polynomials (see \cite[Equation (2.19)]{OSullivan2022}),
\begin{eqnarray*}
\lefteqn{\hspace{-4cm}B_\lambda(nS_1(A^{-1}\xi),nS_2(A^{-1}\xi),\dots,nS_\lambda(A^{-1}\xi))}\\  \hspace{4cm}&=&B_\lambda(n^{-1/2m}S_1(A^{-1}n^E\xi),n^{-2/2m}S_2(A^{-1}n^E\xi),\dots,n^{-\lambda/2m}S_\lambda(A^{-1}n^{E}\xi))\\
    &=&n^{-\lambda/2m}B_\lambda(S_1(A^{-1}n^{E}\xi),S_2(A^{-1}n^{E}\xi),\dots,S_\lambda(A^{-1}n^{E}\xi))
\end{eqnarray*}
for $n\in\mathbb{N}_+$ and $\xi\in\mathbb{R}^d$. Appealing to this identity and making the change of variables $\xi\mapsto n^{-E}\xi$ for $E=ADA^{-1}$, we find
\begin{eqnarray}\label{eq:RewritingQH}\nonumber
    \lefteqn{(Q_\lambda^nH_P^n)(x-n\alpha)}\\\nonumber
    &=&\frac{1}{(2\pi)^d}\int_{\mathbb{R}^d}\frac{1}{\lambda!}B_\lambda(nS_1(A^{-1}\xi),nS_2(A^{-1}\xi),\dots,nS_\lambda(A^{-1}\xi))e^{-nP(\xi)}e^{-i(x-n\alpha)\cdot\xi}\,d\xi\\\nonumber
    &=&\frac{1}{(2\pi)^d}\int_{\mathbb{R}^d}\frac{n^{-\lambda/2m}}{\lambda!}B_{\lambda}(S_1(A^{-1}n^{E}\xi),S_2(A^{-1}n^E\xi),\cdots,S_\lambda(A^{-1}n^E\xi))e^{-P(n^E\xi)}e^{-i(x-n\alpha)\cdot\xi}\,d\xi\\
    &=&\frac{1}{n^{\mu+\lambda/2m}}\frac{1}{(2\pi)^d}\int_{\mathbb{R}^d}\frac{1}{\lambda!}B_\lambda(S_1(A^{-1}\xi),S_2(A^{-1}\xi),\dots,S_\lambda(A^{-1}\xi))e^{-P(\xi)}e^{-i(x-n\alpha)\cdot n^{-E}\xi}\,d\xi\\\nonumber
    &=&\frac{1}{n^{\mu+\lambda/2m}}Q_{\lambda}H_P\left(n^{-E^{\top}}(x-n\alpha)\right)
\end{eqnarray}
for all $n\in\mathbb{N}_+$ and $x\in\mathbb{R}^d$. Also, given that $R^{\#}$ is homogeneous with respect to $I-E^{\top}$, we have
\begin{equation*}
    n R^{\#}\left(\frac{x-n a}{n}\right)=R^{\#}\left(n^{-E^{\top}}(x-n\alpha)\right)
\end{equation*}
for $n\in\mathbb{N}_+$ and $x\in\mathbb{Z}^d$. Thus, in these terms, \eqref{eq:LLT} can be rewritten as
\begin{equation}\label{eq:LLTExplicitTime}
    \abs{\phi^{(n)}(x)-\sum_{k=1}^K\sum_{\lambda=0}^{\lambda_k}\frac{e^{-ix\cdot\xi_k}\widehat{\phi}(\xi_k)^n}{n^{\mu_k+\lambda/2m_k}}(Q_{\lambda,k}H_{P_k})(y_{n,k}(x))}\leq \sum_{k=1}^K\frac{C_k}{n^{\mu_k+(\lambda_k+1)/2m_k}}\exp\left(-M_kR_k^{\#}(y_{n,k}(x))\right)
\end{equation}
where we have put $y_{n,k}(x)=n^{-E_{k}^\top}(x-n\alpha_k)$. 

Let's now rewrite \eqref{eq:LLTExplicitTime} to make everything semi-elliptic by also changing variables by $A_k$. To this end, observe
\begin{eqnarray*}
    (Q_{\lambda,k}H_{P_k})(y)&=&\frac{1}{(2\pi)^d}\int_{\mathbb{R}^d}\frac{1}{\lambda!}B_{\lambda}(S_1(A_k^{-1}\xi),S_2(A_k^{-1}\xi),\dots,S_{\lambda}(A_k^{-1}\xi))e^{-P_k(\xi)}e^{-iy\cdot\xi}\,d\xi\\
    &=&\frac{\abs{\det(A_k)}}{(2\pi)^d}\int_{\mathbb{R}^d}\frac{1}{\lambda!}B_{\lambda}(S_1(\xi),S_2(\xi),\dots,S_\lambda(\xi))e^{-P_{k}(A_k\xi)}e^{-iy\cdot A_k\xi}\,d\xi\\
    &=&\frac{\abs{\det(A_k)}}{\lambda!}(B_{\lambda}(S_1(i\partial),S_2(i\partial),\dots,S_{\lambda}(i\partial))H_{P_{A_k}})(A_k^{\top}y)
\end{eqnarray*}
for all $y\in\mathbb{R}^d$ where, by an abuse of notation, we have written $P_{A_k}(\xi)=P_k(A_k\xi)$ instead of $(P_k)_{A_k}$. Now, thanks to Proposition \ref{prop:LFCompare}, we have $R_{k}^{\#}(\cdot)=R_{A_k}^{\#}(A_{k}^{\top}\cdot)$ and therefore
\begin{equation*}
    R_k^{\#}(y_{n,k}(x))=R_{A_k}^{\#}(A_k^{\top}y_{n,k}(x))=R_{A_k}^{\#}(w_{n,k}(x))
\end{equation*}
where
\begin{equation*}
    w_{n,k}(x)=A_k^{\top}y_{n,k}(x)=(n^{-E_k}A_k)^{\top}(x-n\alpha_k)=(A_kn^{-D_k})^{\top}(x-n\alpha_k)=n^{-D_k}(A_k^{\top}(x-n\alpha_k))
\end{equation*}
for $D_k=\diag(1/2m_{1,k},1/2m_{2,k},\dots,1/2m_{d,k})$. Employing \eqref{eq:LFAsympt}, we obtain
\begin{equation*}
    \exp\left(-M_kR_k^{\#}(y_{n,k})\right)=\exp\left(-M_kR_{A_k}^{\#}(w_{n,k})\right)\leq\exp\left(-M_k\sum_{j=1}^d\abs{w_{n,k}(x)_j}^{2m_{j,k}/(2m_{j,k}-1)}\right)
\end{equation*}
where we have (if necessary) adjusted the constant $M_k$. In these terms, \eqref{eq:LLT} (and \eqref{eq:LLTExplicitTime}) is
\begin{eqnarray}\label{eq:LLTSemiElliptic}\nonumber
   \lefteqn{\hspace{-4.5cm} \abs{\phi^{(n)}(x)-\sum_{k=1}^K\sum_{\lambda=0}^{\lambda_k}\frac{\abs{\det(A_k)}e^{-ix\cdot\xi_k}\widehat{\phi}(\xi_k)^n}{\lambda! \,n^{\mu_k+\lambda/2m_k}}(B_{\lambda}(S_{1,k}(i\partial),S_{2,k}(i\partial),\dots,S_{\lambda,k}(i\partial))H_{P_{A_k}})(w_{n,k}(x))}}\\\nonumber
   \hspace{4cm}&&\leq\sum_{k=1}^K\frac{C_k}{n^{\mu_k+(1+\lambda_k)/2m_k}}\exp(-M_kR_{A_k}^{\#}(w_{n,k}(x)))\\
   \hspace{4cm}&&\leq\sum_{k=1}^K\frac{C_k}{n^{\mu_k+(1+\lambda_k)/2m_k}}\exp(-M_k\sum_{j=1}^d\abs{w_{n,k}(x)_j}^{2m_{k,j}/(2m_{k,j}-1)}).
\end{eqnarray}
With \eqref{eq:LLTSemiElliptic}, let's compare our results with the $1$-dimensional results of J.-F. Coulombel and G. Faye in \cite{CF24}. First, as was pointed out in Remark 1.5 of \cite{RSC17}, when $d=1$, every positive homogeneous polynomial is of the form $P(\xi)=\beta\xi^{2m}$ where $\Re(\beta)>0$ and $m$ is a positive natural number. Thus, given $\phi\in\mathcal{H}_1^*$, for $\xi_k\in\Omega(\phi)$ to be of positive-homogeneous type for $\widehat{\phi}$, we must have
\begin{equation*}
    P_k(\xi)=\beta_k\xi^{2m_k}
\end{equation*}
where $\Re(\beta_k)>0$, $m_k\in\mathbb{N}_+$ and, according to our notation, $\mathbf{m}_k=m_k$, $\kappa_k=1$, $A_k=I$, $\mu_k=1/2m_k$,  and $\Lambda_k(\beta)=\beta-2m_k$. Comparing with \cite{CF24}, our holomorphic error term is
\begin{equation*}
    \Upsilon_k(z)=\sum_{\nu\geq 2m_k+1}\frac{\gamma_{k,\nu}}{\nu!}(iz)^\nu=\sum_{\lambda=1}^\infty \frac{1}{\lambda!}S_{\lambda,k}(z)
\end{equation*}
so that
\begin{equation*}
    S_{\lambda,k}(z)=\frac{\lambda!\,\gamma_{k,2m_k+\lambda }}{(2m_k+\lambda)!}(iz)^{2m_k+\lambda}
\end{equation*}
for each $\lambda$. Thus, in view of Equations (5) of \cite{CF24} and \eqref{eq:BellGenerating} above,
\begin{eqnarray*}
    \lefteqn{1+\sum_{\lambda=1}^\infty P_{k,\lambda}(Y)Z^\lambda=\exp\left(\sum_{\nu\geq 1}\frac{\gamma_{k,2m_k+\nu}}{(2m_k+\nu)!}Y^{2m_k+\nu}Z^\nu\right)}\\
    &&\hspace{2cm}=\exp\left(\sum_{\lambda=1}^\infty \frac{S_{\lambda,k}(-iY)}{\lambda!}Z^\lambda\right)=1+\sum_{\lambda=1}^\infty B_{\lambda}(S_{1,k}(-iY),S_{2,k}(-iY),\dots,S_{\lambda,k}(-iY))\frac{Z^\lambda}{\lambda!}.
\end{eqnarray*}
Therefore, with $\partial=\partial_1=d/dx$ and using the notation of \cite{CF24} (and ours here), we have
\begin{equation*}
    P_{k,\lambda}(-d/dx)=\frac{B_\lambda(S_{1,k}(i\partial),S_{2,k}(i\partial),\dots,S_{\lambda,k}(i\partial))}{\lambda!}=Q_{\lambda,k}
\end{equation*}
since $A_k^{-1}=I=1$. Upon noting that $\mu_k+\lambda/2m_k=(1+\lambda)/2m_k$ and $w_{k,j}(x)=(x-n\alpha_k)/n^{1/2m_k}$, \eqref{eq:LLTSemiElliptic} is
\begin{eqnarray*}
    \lefteqn{\abs{\phi^{(n)}(x)-\sum_{k=1}^K\sum_{\lambda=0}^{\lambda_k}\frac{e^{-ix\cdot\xi_k}\widehat{\phi}(\xi_k)^n}{n^{(1+\lambda)/2m_k}}(P_{k,\lambda}(-d/dx)H_{P_k})\left(\frac{x-n\alpha_k}{n^{1/2m_k}}\right)}}\\
    &&\hspace{5cm}\leq \sum_{k=1}^K\frac{C_k}{n^{(2+\lambda_k)/2m_k}}\exp\left(-M_k\abs{\frac{x-n\alpha_k}{n^{1/2m_k}}}^{2m_k/(2m_k-1)}\right)
\end{eqnarray*}
for $n\in\mathbb{N}_+$ and $x\in\mathbb{Z}$ with $\abs{x}\leq nL$ (or unrestricted provided $\phi$ is finitely supported). In the special case that the order of accuracy at each point is chosen to be the same, i.e., $\lambda_1=\lambda_2=\cdots=\lambda_K=M$, we recapture Theorem 4 of \cite{CF24}. \\

\noindent We will conclude this subsection with a remark and a corollary. The remark allows us to distinguish attractors in local limit theorems relative to ``minimal decay". The corollary describes the speed of convergence of convolution powers to attractors and has applications to numerical difference schemes. For simplicity, let's first set some notation that will be used in our examples throughout this article.
\begin{notation}[Real Error]\label{not:RealError}
    Let $\phi\in\mathcal{H}_d^*$ satisfy the hypotheses of Theorem \ref{thm:LLT} and assume the notation of the theorem. For a choice of orders of accuracy $\lambda_1,\lambda_2,\dots,\lambda_K\in\mathbb{N}$, set
    \begin{equation*}
    \mathcal{R}_{\lambda_1,\lambda_2,\dots,\lambda_K}^n(x)=\phi^{(n)}(x)-\sum_{k=1}^K\sum_{\lambda=0}^{\lambda_k}e^{-ix\cdot\xi_k}\widehat{\phi}(\xi_k)^n (Q_{\lambda,k}^nH_{P_k}^n)(x-n\alpha_k)
    \end{equation*}
    for $n\in\mathbb{N}_+$ and $x\in\mathbb{Z}^d$. In view of the calculations we used to obtain \eqref{eq:LLTExplicitTime} and \eqref{eq:LLTSemiElliptic}, this expression has several equivalent formulations. In the special case that $\lambda_1=\lambda_2=\cdots\lambda_K=\lambda$, we will write $\mathcal{R}_\lambda^n$ in place of $\mathcal{R}_{\lambda,\lambda,\dots,\lambda}^n$. 
\end{notation}

\begin{remark}[Paying Attention to Minimal Decay]\label{rmk:refRn}
Let us consider $\phi$  with finite support and satisfying the hypotheses of Theorem \ref{thm:LLT}. From the theorem (in particular, \eqref{eq:LLTExplicitTime}), for $\lambda_1,\lambda_2,\dots,\lambda_K\in\mathbb{N}$, we obtain positive constants $C_k$ and $M_k$, $k=1,\ldots, K$, for which the estimate 
\begin{equation}\label{eq:In-Est-RLn}
\begin{aligned}
\vert \mathcal{R}_{\lambda_1,\ldots,\lambda_K}^{n}(x)\vert &\leq \sum_{k=1}^{K}\frac{C_k}{n^{\mu_k+(\lambda_k+1)/2m_k}}\exp\left( -M_k R_{k}^{\#}(y_{n,k}(x)) \right) \\ & \leq \frac{C}{n^{\nu}}\sum_{k=1}^{K}\exp\left( -M_k R_{k}^{\#}(y_{n,k}(x)) \right),
\end{aligned}
\end{equation}
holds for all $n\in \mathbb{N}_+$ and $x\in\mathbb{Z}^d$ where $y_{n,k}(x)=n^{-E_k}(x-n\alpha_k)$ and $\nu$ is the minimal decay of error given by
\begin{equation*}
    \nu =\min_{k=1,2,\dots,K}\{ \mu_k + (\lambda_{k}+ 1)/2m_{k} \}
\end{equation*}
and we have set $C=\max_k C_k $. With these constants identified, our goal in this remark is to identify those terms in the attractor (appearing in $\mathcal{R}_{\lambda_1,\lambda_2,\dots,\lambda_K}^n$) which decay on the order of the error and thus can be disregarded.\\

\noindent To this end, we first observe the following: Along the same lines as the proof of iii) of Proposition 2.6 in \cite{RSC17}, using the fact that $e^{-P_{k}}$ is a Schwartz function, one can easily show that, for $k=1,\ldots,K$ and any $\lambda\in \mathbb{N}$, there holds the estimate
\begin{eqnarray}\label{eq:CumulantBounds}\nonumber
    \abs{(Q_{\lambda,k}H_{P_k})(y)}&=&\frac{1}{\lambda! (2\pi)^d}\abs{\int_{\mathbb{R}^d}B_{\lambda}(S_{1,k}(A^{-1}\xi),S_{2,k}(A^{-1}\xi),\dots,S_{\lambda,k}(A^{-1}\xi))e^{-P_k(\xi)}e^{-i y\cdot\xi}\,d\xi}\\
    &\leq & C_k'\exp(-M_k' R^{\#}_k(y))
\end{eqnarray}
for all $y\in\mathbb{R}^d$ where $C_k'$ and $M_k'$ are positive constants.  Setting
\begin{equation*}
    U_\nu=\{k=1,2,\dots,K:\mu_k<\nu\}\hspace{1cm}\mbox{and}\hspace{1cm}V_\nu=\{k=1,2,\dots,K:\mu_k\geq\nu\}
\end{equation*}
so that $\{1,2,\dots,K\}=U_{\nu}\cup V_\nu$, we focus our attention on the attractors corresponding to the $k$s in these two sets.

For $k\in U_{\nu}$, let $\underline{\lambda_k}\in\{0,1,\dots,\lambda_k\}$ be the maximal accuracy (relative to $\nu$) having
\begin{equation*}
    \mu_k+\underline{\lambda}_k/2m_k<\nu\hspace{1cm}\mbox{but}\hspace{1cm}\nu\leq \mu_k+(\underline{\lambda}_k+1)/2m_k,
\end{equation*}
we notice that the attractor terms
\begin{equation*}
    (Q_{\lambda,k}^nH_{P_k}^n)(x-n\alpha_k)=\frac{1}{ n^{\mu_k+\lambda/2m_k}}(Q_{\lambda,k}H_{P_k})(y_{n,k})
\end{equation*}
for $\underline{\lambda_k}<\lambda\leq\lambda_k$ will decay at least as fast as the error in \eqref{eq:In-Est-RLn} and so can be disregarded. For $k\in V_{\nu}$, the estimate \eqref{eq:CumulantBounds} shows that every term $Q_{\lambda,k}^n H_{P_k}$ for $\lambda=0,1,\dots,\lambda_k$ decays at least as fast as the error in \eqref{eq:In-Est-RLn} and so every attractor corresponding to $k\in V_\nu$ can be disregarded.

With these two observations in mind, if we consider the refinement
\begin{equation*}
    \widetilde{\mathcal{R}}_{\lambda_1,\lambda_2,\dots,\lambda_K}^n(x)=\phi^{(n)}(x)-\sum_{k\in U_{\nu}}\sum_{\lambda=0}^{\underline{\lambda_k}}e^{-ix\cdot\xi_k}\widehat{\phi}(\xi_k)^n (Q_{\lambda,k}^n H_{P_k}^n)(x-n\alpha_k),
\end{equation*}
then we still have
\begin{equation*}
\abs{\widetilde{\mathcal{R}}_{\lambda_1,\lambda_2,\dots,\lambda_K}^n(x)}\leq \frac{C}{n^{\nu}}\sum_{k=1}^K\exp(-M_kR_k^{\#}(y_{n,k}(x)))
\end{equation*}
for every $n\in\mathbb{N}_+$ and $x\in\mathbb{Z}^d$. Here, by ``refinement", we mean that in the initial expression for $\mathcal{R}_{\lambda_k,\ldots,\lambda_K}^{n}(x)$ we can neglect the terms that decay at the same rate or faster than $n^{-\nu}$.

As a final note, for the special case that $\lambda_1=\lambda_2=\cdots\lambda_K=0$, we have
\begin{equation*}
    \nu>\mu_\phi=\min_k\mu_k. 
\end{equation*}
In this case, the above estimate yields, in particular,
\begin{equation*}
   \widetilde{\mathcal{R}}_{0}^n(x)=\phi^{(n)}(x)-\sum_{k\in U_\nu}e^{-ix\cdot\xi_k}\widehat{\phi}(\xi_k)^n H_{P_k}^n(x-n\alpha_k)= O(n^{-\nu})=o(n^{-\mu_\phi})
\end{equation*}
uniformly for $x\in\mathbb{Z}^d$ as $n\to\infty$. This is precisely how the attractors in the local limit theorem \cite[Theorem 1.6]{RSC17} were identified (i.e., paying only attention to those $k\in U_\nu$). Relevant to this is Example 7.4 of \cite{RSC17} which is revisited in Subsection \ref{ssec:NonMinDecay} below.\\
\end{remark}

\noindent Our final corollary gives quantitative bounds for the speed of convergence of difference schemes in $\ell^p(\mathbb{Z}^d)$ ($1\leq p\leq\infty)$ which we take equipped with its usual norm $\|\cdot\|_p$. In the case that $d=1$, $p=q$, and $\lambda_1=\lambda_2=\cdots=\lambda_K=M$, it recaptures Corollary 1.5 of \cite{CF24}.
\begin{corollary}\label{cor:EllpqEst}
    Let $\phi\in\mathcal{H}_d^*$ and assume that every point of $\Omega(\phi)$ is of positive homogeneous type for $\widehat{\phi}$. We assume the notation of Theorem \ref{thm:LLT}. Given $1\leq p\leq q\leq \infty$ and orders of accuracy $\lambda_1,\lambda_2,\dots,\lambda_K\in\mathbb{N}$, set
    \begin{equation*}
        \gamma=\min_k\left\{\mu_k\left(\frac{1}{p}-\frac{1}{q}\right)+\frac{\lambda_k+1}{2m_k}\right\}.
    \end{equation*}
    Then, there is a uniform constant $C=C_{p,q,\lambda_1,\lambda_2,\dots,\lambda_K}>0$ for which 
    \begin{equation*}
        \left\|\mathcal{R}_{\lambda_1,\lambda_1,\dots,\lambda_K}^n\ast u\right\|_{q}\leq \frac{C}{n^{\gamma}}\|u\|_{p}
    \end{equation*}
    for all $n\in\mathbb{N}_+$ and $u\in\ell^p(\mathbb{Z}^d)$. 
\end{corollary}
\begin{proof}
We shall first prove the corollary in the special case that $\phi$ is finitely supported. Writing $\mathcal{R}^n=\mathcal{R}^n_{\lambda_1,\lambda_2,\dots,\lambda_K}$, by virtue of Young's inequality for convolution, it suffices to prove that
\begin{equation*}
\|\mathcal{R}^n\|_{r}\leq\frac{C}{n^{\gamma}}
\end{equation*}
for a uniform constant $C$ where $1-1/r=1/p-1/q$. If $r=\infty$, $1/p-1/q=1$ and so the desired estimate follows immediately from \eqref{eq:LLTExplicitTime}. In the case that $r<\infty$, the estimate \eqref{eq:LLTSemiElliptic} and the $\ell^r(\mathbb{Z}^d)$ triangle inequality guarantee that
\begin{eqnarray*}
    \|\mathcal{R}^n\|_r&\leq &\sum_{k=1}^{K} \left( \sum_{x\in \mathbb{Z}^d} \frac{C_k^r}{n^{r\mu_k+r(\lambda_k+1)/2m_k}} \exp\left(-rM_k R_{A_k}^{\#}(w_{n,k}(x) \right) \right)^{1/r} \\ &=& \sum_{k=1}^{K} \left( \frac{C_k^r}{n^{(r-1)\mu_k+r(\lambda_k+1)/2m_k}} \sum_{x\in \mathbb{Z}^d} \frac{1}{n^{\mu_k}}\exp\left(-rM_k R_{A_k}^{\#}(w_{n,k}(x) \right) \right)^{1/r} \\ &\leq& \frac{C}{n^{\gamma}}\sum_{k=1}^{K} \| n^{-\mu_k} \exp\left(-rM_k R_{A_k}^{\#}(w_{n,k}(x) \right)\|_{1}^{1/r}
\end{eqnarray*}
for $n\in\mathbb{N}_+$ where $C$ is some positive constant. Let us assume, momentarily, that for each $k=1,\dots,K$, there is a constant $T_k$ for which
\begin{equation}\label{eq:UniformEll1Bound}
    \left\|n^{-\mu_k}\exp(-rM_k R_{A_k}^{\#}(w_{n,k}(\cdot)))\right\|_{1}\leq T_k
\end{equation}
for every $n\in\mathbb{N}_+$, i.e., these majorants are uniformly bounded in $\ell^1(\mathbb{Z}^d)$. Then, we have
\begin{equation*}
    \left\|\mathcal{R}^n\right\|_{r}\leq \frac{C}{n^\gamma}\sum_{k=1}^K T_k^{1/r}
\end{equation*}
for every $n\in\mathbb{N}_+$ which is the desired estimate.

Thus, to complete the proof for this finitely-supported case, it remains to prove the uniform estimate \eqref{eq:UniformEll1Bound}. To this end let's suppress the $k$ dependence and observe that
\begin{eqnarray*}
    \left\|n^{-\mu}\exp(-rM R_{A}^{\#}(w_{n}(\cdot)))\right\|_{1}&=&\sum_{x\in\mathbb{Z}^d}n^{-\mu}\exp(-r M_k R_{A}^{\#}(w_n(x)))\\
    &\leq &\sum_{l=0}^\infty n^{-\mu}\#(\{x\in\mathbb{Z}^d:l\leq R_{A}^{\#}(w_n(x))< l+1\}) e^{-l r M}\\
    &\leq &\sum_{l=0}^\infty \#(B_{n,l})\, n^{-\mu} e^{-lrM}
\end{eqnarray*}
where
\begin{equation*}
    B_{n,l}=\{x\in\mathbb{Z}^d:R_A^{\#}(w_n(x))\leq l+1\}=\{x\in\mathbb{Z}^d:R_{A}^{\#}(n^{-D}A^{\top}(x-n\alpha))\leq l+1\}
\end{equation*}
and, by a slight abuse of notation, we have denoted by $\# (B)$ the number of elements in $B$. By an appeal to Proposition \ref{prop:LFCompare}, we can find positive constants $c_1,c_2,\dots,c_d\geq 1$ such that
\begin{eqnarray*}
    B_{n,l}&\subseteq& \left\{x\in\mathbb{Z}^d: \abs{n^{-1/2m_j}(A^{\top}(x-n\alpha))_j}^{2m_j/(2m_j-1)}\leq c_j (l+1) \,\,\,\forall \,j=1,2,\dots,d\right\}\\
    &&\hspace{2.5cm}\subseteq\left\{x\in\mathbb{Z}^d:\abs{(A^{\top}(x-n\alpha))_j}\leq n^{1/2m_j}c_j (l+1)\,\,\,\forall j=1,2,\dots,d\right\}=\mathbb{Z}^d\cap \mathcal{P}_{n,l}
\end{eqnarray*}
where $\mathcal{P}_{n,l}$ is the parallelepiped
\begin{equation*}
    \mathcal{P}_{n,l}=n\alpha+A^{-\top}n^D\left(\prod_{j=1}^d[-c_j(l+1),c_j(l+1)]\right)\subseteq\mathbb{R}^d.
\end{equation*}
Now, we would like to bound the number of lattice points in $\mathcal{P}_{n,l}$ by the volume/measure of $\mathcal{P}_{n,l}$. Thanks to a result by M. Widmer\footnote{As Widmer notes and M. Henk and J. M. Wills explain in \cite{HW08}, this result (or one closely related) was originally known to H. F. Blichfeldt and presented at the 1921 April 9th meeting of the American Mathematical Society in San Francisco\cite{AMSMeeting}. At the same meeting, Blichfeldt's younger colleague, E. T. Bell, presented several papers foreshadowing Bell's later study of ``exponential polynomials"\cite{BellAhar,BellExpPoly}.} (Corollary 2.10 of \cite{Wid12}), we are able to obtain such a bound after noting that, because $A\in\GldR$ and $n^D$ is a dilation, $\mathcal{P}_{n,0}\subseteq\mathcal{P}_{n,l}$ must contain a shifted (by $n\alpha$) Euclidean ball of radius one for all sufficiently large $n$. Consequently, for sufficiently large $n$ and all $l\in\mathbb{N}$, $\mathcal{P}_{n,l}\cap \mathbb{Z}^d$ cannot be contained in a proper affine subspace of $\mathbb{R}^d$. An appeal to \cite[Corollary 2.10]{Wid12} gives
\begin{equation*}
    \#(\mathbb{Z}^d\cap \mathcal{P}_{n,l})\leq C\operatorname{Vol}(\mathcal{P}_{n,l})\leq C(l+1)^d \det(n^D)=C(l+1)^d n^{\mu}
\end{equation*}
for sufficiently large $n$ where $C>0$. Consequently,
\begin{equation*}
    \sum_{l=0}^\infty\#(B_{n,l})e^{-\mu}e^{-lrM}\leq\sum_{l=0}^\infty C(l+q)^d n^\mu n^{-\mu} e^{-lrM}=\sum_{l=0}^\infty C(1+l)^d e^{-lrM}=T<\infty
\end{equation*}
for all sufficiently large $n$ and so we conclude that $n^{-\mu}\exp(-r M R_{A}^{\#}(w_n(x)))$ is uniformly bounded in $\ell^1(\mathbb{Z}^d)$ and the proof in the case where $\phi$ has finite support is complete.

In the general case that $\phi\in\mathcal{H}_d^*$ satisfies the hypotheses of the corollary but is not finitely supported, we take $L$ as guaranteed by Theorem \ref{thm:GeneralGaussEstimate}.  In this case, one bounds $\|\mathcal{R}^n\|_r$ by three components. The first component is the Gaussian-type local limit bound for $\abs{x}\leq n L$ and is treated exactly as we have done above. The second two components are the $\ell^r(\mathbb{Z}^d)$ norms of $e^{-M_0(n+\abs{x})}$, coming from Theorem \ref{thm:GeneralGaussEstimate}, and the sum of attractors $\abs{(Q_{\lambda,k}^n H_{P_k}^n)(x-n\alpha_k)}$ which are all restricted to the domain $\abs{x}>nL$. When restricted to these domains, the norms of these functions can be easily seen to decay exponentially in $n$ (see Section 4.1 in \cite{CF24}). We leave the details for the interested reader.
\end{proof}

\subsection{Example of a Simple Real-Valued Function on $\mathbb{Z}^2$}\label{ssec:IntroExample}

Consider $\phi: \mathbb{Z}^2\to \mathbb{R}$ given by
\begin{equation*}
\phi(x,y)=\frac{1}{16}\times
\begin{cases}
4 & (x, y) = (0, \pm 1),(\pm 1, 0)\\
1 & (x, y) = (\pm 2, 0)\\
-1 & (x, y) = (0, \pm 2)\\
0 & \mbox{otherwise}
\end{cases}
\end{equation*}
for $(x,y)\in\mathbb{Z}^2$. Figure \ref{fig:Ex1ConvPower} illustrates $\phi^{(n)}(x, y)$ on $-50\leq x,y\leq 50$ for $n= 100$ and $n=1,000$.

\begin{table}[h!]
  \centering
  \begin{tabular}{  | c | c | }
    \hline
     $n=100$ & $n=1000$ \\ \hline

    \begin{minipage}{.48\textwidth}
      \includegraphics[width=\linewidth]{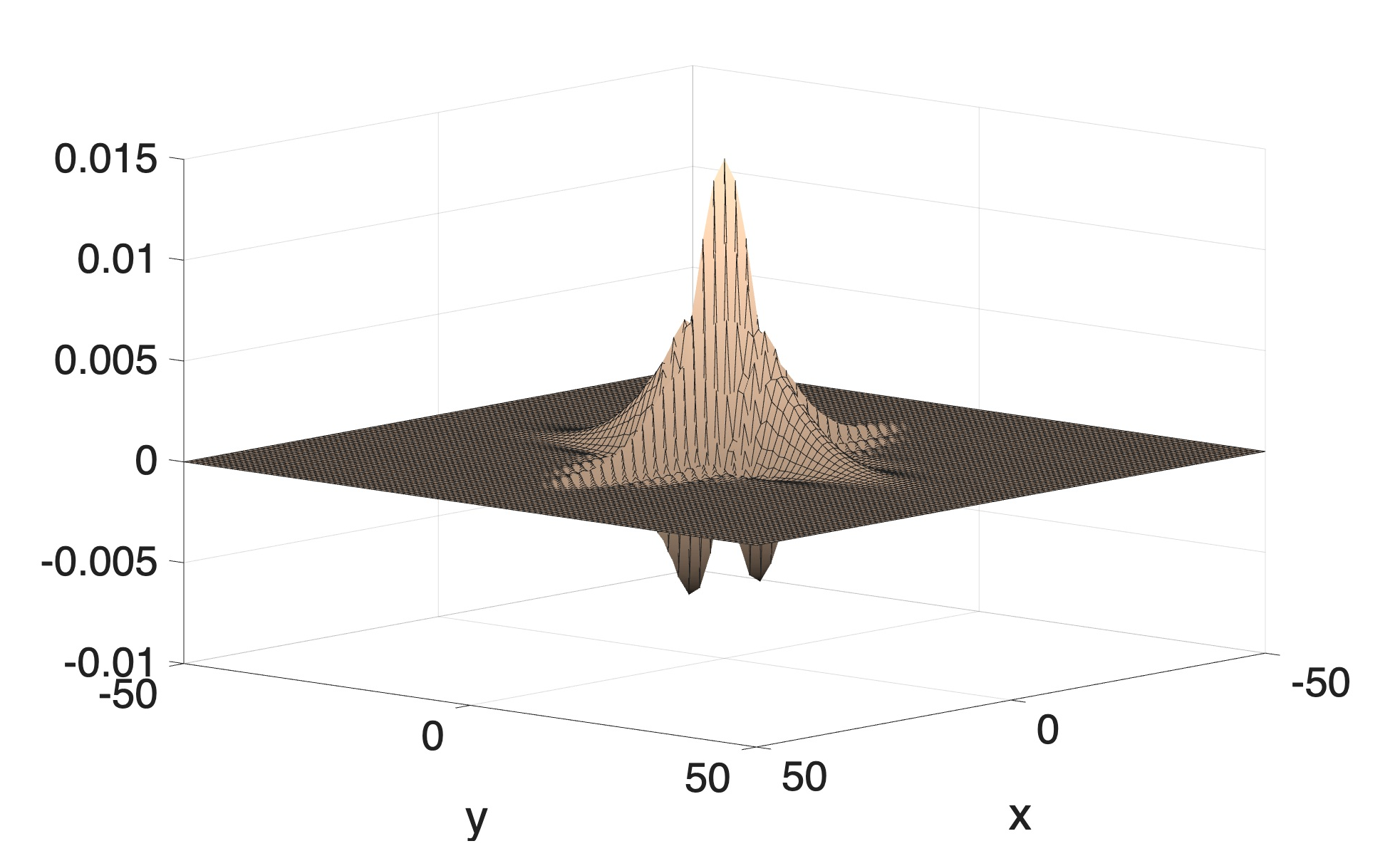}
    \end{minipage}
	&
      \begin{minipage}{.48\textwidth}
      \includegraphics[width=\linewidth]{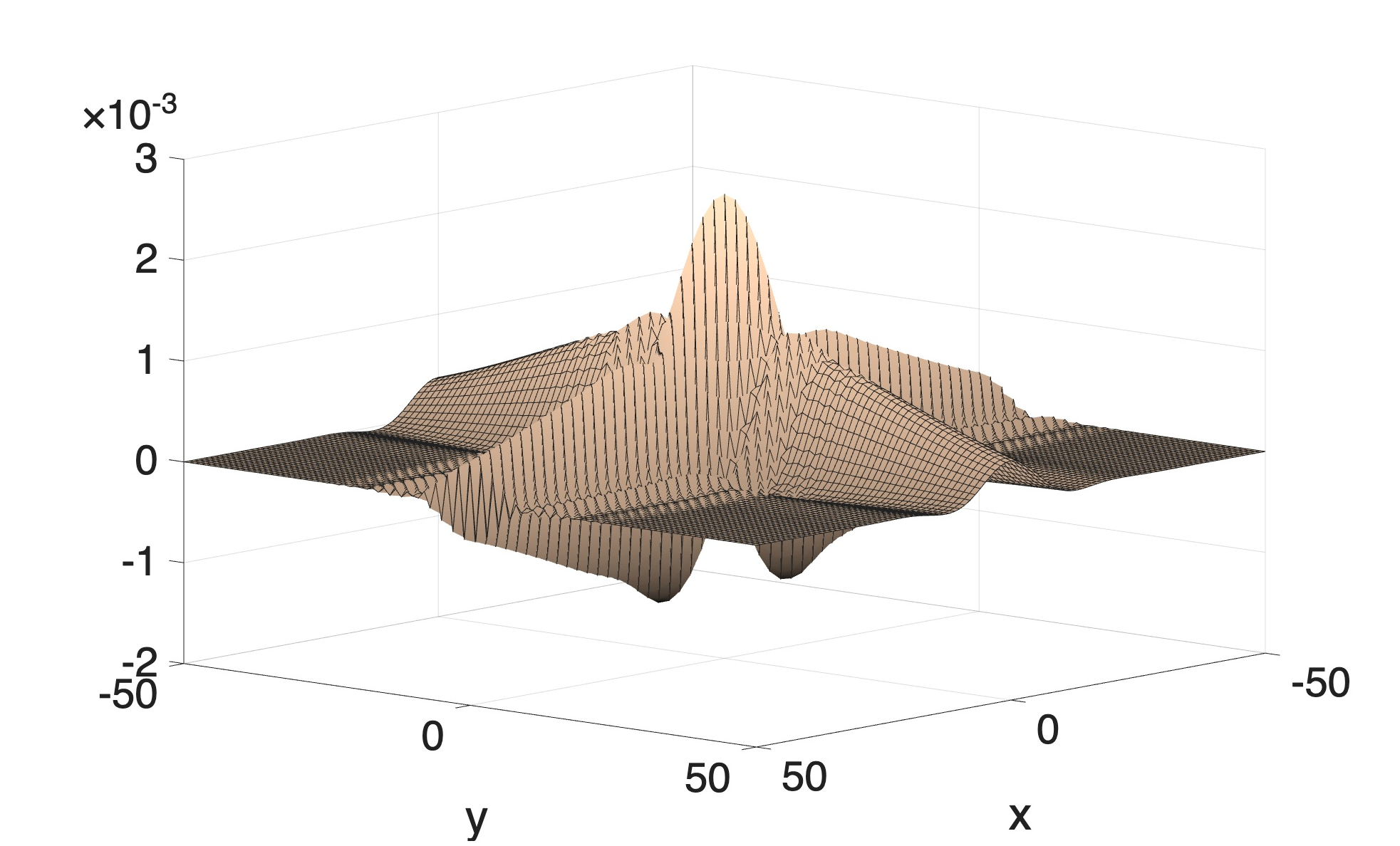}
    \end{minipage}\\ \hline
  \end{tabular}
  \captionof{figure}{The graphs of $\phi^{(n)}$ for $n=100$ and $n=1,000$ are displayed on the grid $-50\leq x,y\leq 50$.}\label{fig:Ex1ConvPower}
  \end{table}

\begin{figure}[h!]
\begin{center}
\includegraphics[width=.6\linewidth]{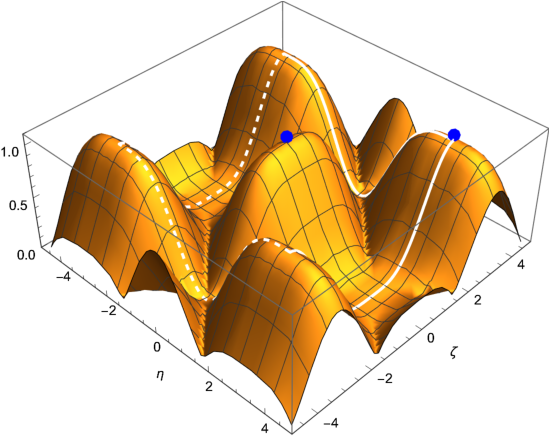}
    \caption{The graph of $|\widehat{\phi}(\eta,\zeta)|$ on the rectangle $[-3\pi/2, 3\pi/2]^2$. The white curve indicates $\partial\mathbb{T}^2$ with the solid portion of the curve indicating $\partial\mathbb{T}^2\cap\mathbb{T}^2$ and the dotted portion indicating $\partial\mathbb{T}^2\setminus\mathbb{T}^2$. The blue points sit above $\xi_1=(0,0)$ and $\xi_2=(\pi,\pi)$ in $\Omega(\phi)$.}
    \label{fig:Ex1FT_Plot}
    \end{center}
\end{figure}

\noindent A direct computation shows that 
\begin{equation*}
    \widehat{\phi}(\eta,\zeta)=\frac{1}{2}\cos(\eta)+\frac{1}{2}\cos(\zeta)+\frac{1}{8}\cos(2\eta)-\frac{1}{8}\cos(2\zeta)
\end{equation*}
for $(\eta,\zeta)\in\mathbb{R}^2$. Figure \ref{fig:Ex1FT_Plot} illustrates $(\eta,\zeta)\mapsto |\widehat{\phi}(\eta,\zeta)|$ on the domain $(-3\pi/2,3\pi/2]^2\supseteq \mathbb{T}^2$. One can easily verify that $\sup_{(\eta,\zeta)}|\widehat{\phi}(\eta,\zeta)|=1$ and $\Omega(\phi)=\{\xi_1,\xi_2\}\subseteq\mathbb{T}^2$ where $\xi_1=(0,0)$ with $\widehat{\phi}(\xi_1)=1$ and $\xi_2=(\pi,\pi)$ with $\widehat{\phi}(\xi_2)=-1$. Since $\phi$ is finitely supported, we have $\phi\in\mathcal{H}_d^*$ and so we may consider the analytic expansions for $\Gamma_1=\Gamma_{\xi_1}$ and $\Gamma_2=\Gamma_{\xi_2}$ about $(0,0)$.  A close look at Figure \ref{fig:Ex1FT_Plot} shows that the decay of $\widehat{\phi}$ away from these maxima  (and hence the nature of $\Gamma_1$ and $\Gamma_2$) are both anisotropic in nature but not identical.

For $\xi_1\in\Omega(\phi)$, we have 
\begin{eqnarray*}
\Gamma_1(\eta, \zeta)&=& -\frac{1}{2}\eta^2-\frac{1}{16}\zeta^4-\frac{1}{48}\eta^4-\frac{1}{32}\eta^2\zeta^4+\frac{1}{96}\zeta^6 -\frac{7}{2560}\zeta^{8} +O(\eta^6)+O(\eta^4\zeta^2)+O(\eta^2\zeta^6)+O(\zeta^{10})\\
&=& -P_1(\eta, \zeta) + \Upsilon_1(\eta, \zeta)
\end{eqnarray*}
where
\begin{equation*}P_1(\eta,\zeta)=R_{1}(\eta, \zeta)=\eta^2/2+\zeta^4/16,
\end{equation*} which is a positive semi-elliptic polynomial with $\mathbf{m}_1=(1,2)$, and $\Upsilon_1$ is given by

\begin{equation*}
 \Upsilon_{1}(\eta,\zeta)=\sum_{\abs{\beta:2\mathbf{m}_1}>1}b_{\beta,1}\xi^{\beta}
 =-\frac{1}{48}\eta^4-\frac{1}{32}\eta^2\zeta^4+ \frac{1}{96}\zeta^6- \frac{7}{2560}\zeta^8 +O(\eta^6+\eta^4\zeta^2+\eta^2\zeta^6+\zeta^{10})
\end{equation*}
 and consequently has $\Upsilon_{1}(\xi)=o(R_1(\xi))$ as $\xi \to 0$ in $\mathbb{R}^2$ (see \cite[Lemma A.5]{R23}). Hence $\xi_1=(0,0)$ is of positive-homogeneous type for $\widehat{\phi}$ with $\alpha_1=(0,0)$, indicating no drift. 
 
 Now, by completely analogous reasoning, we see that $\xi_2=(\pi,\pi)$ is also of positive-homogeneous type for $\widehat{\phi}$ with $\alpha_2=(0,0)$, with 
\begin{equation*}
\Gamma_{2}(\eta,\zeta)=-P_2(\eta,\zeta)+\Upsilon_{2}(\eta,\zeta),
\end{equation*}
where $P_2(\xi)=R_2(\xi)=\eta^4/16+\zeta^2/2$, $\xi=(\eta,\zeta)$,
is a positive semi-elliptic polynomial with $\mathbf{m}_2=(2,1)$, and $\Upsilon_{2}(\xi)=o(R_2(\xi))$ as $\xi\to 0 $ in $\mathbb{R}^2$ is given by
\begin{equation*}
\Upsilon_2(\xi)=\sum_{\abs{\beta:2\mathbf{m}_2}>1}b_{\beta,2}\xi^\beta=-\frac{1}{48}\zeta^4-\frac{1}{32}\eta^4\zeta^2+\frac{1}{96}\eta^6-\frac{7}{2560}\eta^{8}+O(\eta^{10}+\zeta^6+\eta^4\zeta^4+\eta^6\zeta^2)
\end{equation*}
for $\xi=(\eta,\zeta)$. It is readily computed that 
\begin{equation*}
    R_1^{\#}(x,y)=\frac{1}{2}x^2+\frac{3}{4^{2/3}}\abs{y}^{4/3}\hspace{1cm}\mbox{and}\hspace{1cm}R_2^{\#}(x,y)=\frac{3}{4^{2/3}}\abs{x}^{4/3}+\frac{1}{2}y^{2}
\end{equation*}
for $(x,y)\in\mathbb{R}^2$. Thus, an appeal to Theorem \ref{thm:FiniteSupport} gives us positive constants $C_1,C_2,M_1$, and $M_2$ for which
\begin{eqnarray}\label{eq:Ex1Bound}\nonumber
    \abs{\phi^{(n)}(x,y)}&\leq & \sum_{k=1}^2\frac{C_k}{n^{\mu_k}}\exp\left(-nM_k R_k^{\#}\left(\frac{(x,y)-n\alpha_k}{n}\right)\right)\\\nonumber
    &\leq& \frac{C}{n^{3/4}}\left[\exp\left(-nMR_1^{\#}\left(\frac{x}{n},\frac{y}{n}\right)\right)+\exp\left(-nMR_2^{\#}\left(\frac{x}{n},\frac{y}{n}\right)\right)\right]\\
    &=&\frac{C}{n^{3/4}}\left[\exp\left(-\frac{M}{2n}x^2-\frac{3M}{4^{2/3}n^{1/3}}\abs{y}^{4/3}\right)+\exp\left(-\frac{3M}{4^{2/3}n^{1/3}}\abs{x}^{4/3}-\frac{M}{2n}y^2\right)\right]
\end{eqnarray}
for all $n\in\mathbb{N}_+$ and $(x,y)\in\mathbb{Z}^2$; here $C=\max\{C_1,C_2\}$ and $M=\min\{M_1,M_2\}$. We remark that, as $P_1\neq P_2$, this example does not meet the hypotheses of Theorem 1.5 of \cite{RSC17} and so this Gaussian-type estimate cannot be obtained by the results there. This estimate, with $M=0.5$ and $C=0.3$, is illustrated in Figure \ref{fig:Ex1AbsBounds} for $n=100$ and $n=1,000$.

\begin{table}[h!]
  \centering
  \begin{tabular}{  | c | c | }
    \hline
     $n=100$ & $n=1000$ \\ \hline

    \begin{minipage}{.45\textwidth}
      \includegraphics[width=1.05\linewidth]{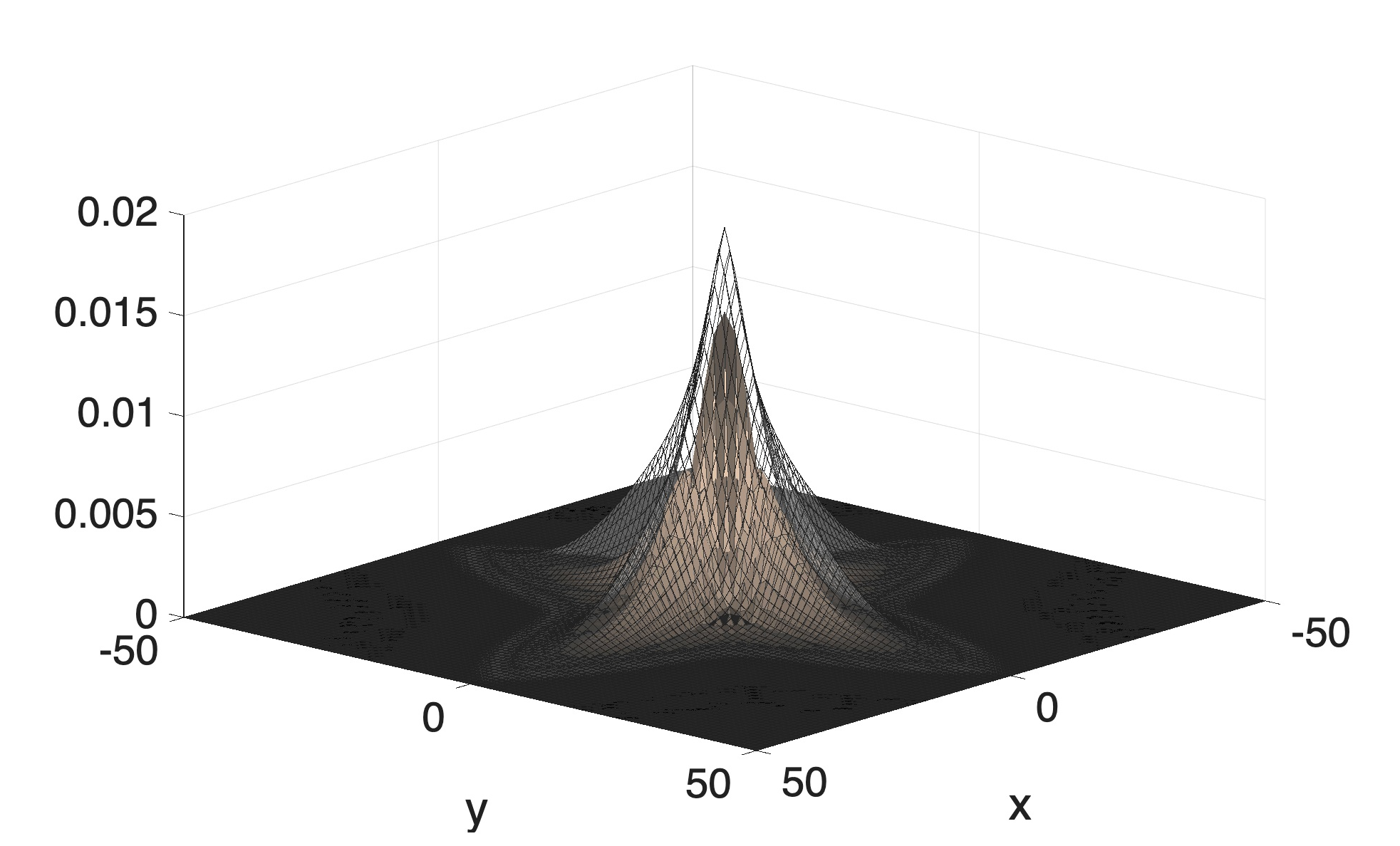}
    \end{minipage}
	&
      \begin{minipage}{.45\textwidth}
      \includegraphics[width=1.05\linewidth]{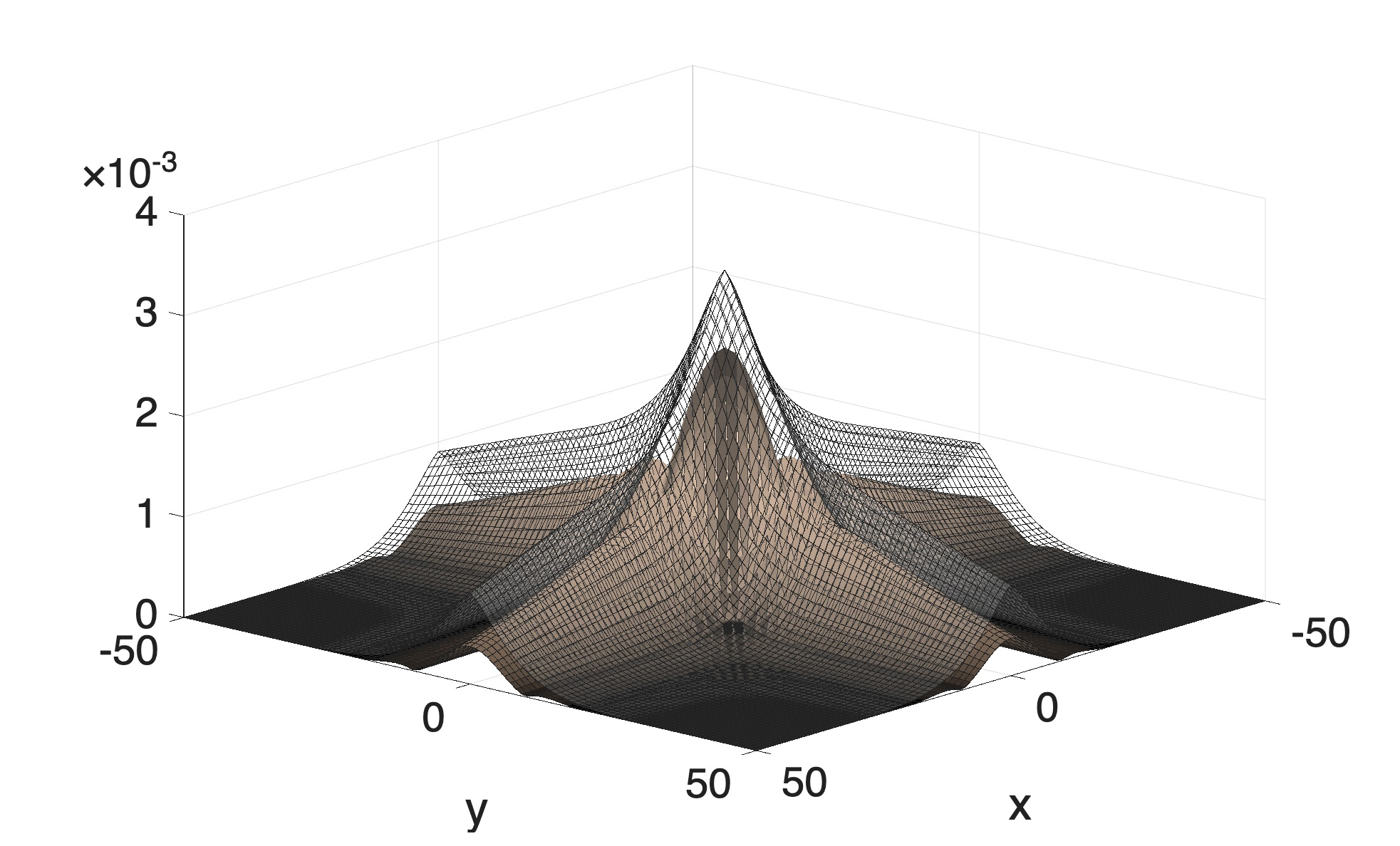}
    \end{minipage}\\ \hline
  \end{tabular}
  \captionof{figure}{On the grid $-50\leq x\leq 50$, $\abs{\phi^{(n)}}$ are illustrated by the solid light brown surfaces (for $n=100$ and $1,000$). The Gaussian-type error appearing on the right-hand side of \eqref{eq:Ex1Bound} is illustrated by the transparent ``nets" above.}\label{fig:Ex1AbsBounds}
\end{table}

Our next task is to state our local limit theorem with base attractors only. In other words, we aim to apply Corollary \ref{cor:LLTCor}. First, from our work above, $P_1(\cdot)$ and $P_2(\cdot)$ are semi-elliptic with $D_1=\diag(1/2,1/4)\in\Exp(P_1)$, $D_2=\diag(1/4,1/2)\in\Exp(P_2)$, and $A_k=I_2$ for $k=1,2$. Correspondingly, we have
\begin{equation*}
\begin{aligned}
    H_{P_1}(x,y)&=\frac{1}{(2\pi)^2}\int_{\mathbb{R}^2}e^{-P_1(\xi)}e^{-i(x,y)\cdot\xi}\,d\xi\\
    &=\frac{1}{(2\pi)^2}\int_{\mathbb{R}^2}e^{-(\eta^2/2+\zeta^4/16)}e^{-i(x\eta+y\zeta)}\,d\eta\,d\zeta
    =h_2^{1/2}(x)h_4^{1/16}(y)
\end{aligned}    
\end{equation*}
and
\begin{equation*}
\begin{aligned}
    H_{P_2}(x,y)&=\frac{1}{(2\pi)^2}\int_{\mathbb{R}^2}e^{-P_2(\xi)}e^{-i(x,y)\cdot\xi}\,d\xi\\
    &=\frac{1}{(2\pi)^2}\int_{\mathbb{R}^2}e^{-(\eta^4/16+\zeta^2/2)}e^{-i(x\eta+y\zeta)}\,d\eta\,d\zeta 
    =h_4^{1/16}(x)h_2^{1/2}(y)
\end{aligned}
\end{equation*}
for $(x,y)\in\mathbb{R}^2$ where we have used \eqref{eq:OneDHeatKer}. In the present simple situation, both attractors can be factored into the product of single-variable functions: the standard Gaussian/heat kernel and bi-harmonic heat kernel on $\mathbb{R}$. As illustrated in the introductory section of \cite{RSC17}, this factoring is not always possible.

With our aim to apply Corollary \ref{cor:LLTCor}, it remains to compute $\gamma_1$ and $\gamma_2$. For this, we compute the polynomials $S_{\lambda,1}(\cdot)$ and $S_{\lambda,2}(\cdot)$ for small integers $\lambda$. Since $m_1=\lcm(\mathbf{m}_1)=2$, $\kappa_1=(2,1)$ and therefore
\begin{equation*}
    \Lambda_1(\beta)=\beta\cdot\kappa_1-2m_1=2\beta_1+\beta_2-4
\end{equation*}
for multi-indices $\beta=(\beta_1,\beta_2)\in\mathbb{N}^2$. In looking at $\Upsilon_1$ above, we see the lowest order non-zero terms have 
\begin{equation*}
\Lambda_1((0,6))=2\hspace{1cm}\mbox{and}\hspace{1cm}\Lambda_1((2,4))=\Lambda_1((4,0))=\Lambda_1((0,8))=4
\end{equation*}
corresponding, respectively, to $\zeta^6/96$, $-\eta^2\zeta^4/32$, $-\eta^4/48$, and $-7\zeta^{8}/2560$. For all higher-order terms, we have $\Lambda_1(\beta)> 4$. The first four polynomials $S_{\lambda,1}$ are thus given by
\begin{equation*}
    S_{1,1}(\xi)=\sum_{\Lambda_1(\beta)=1}b_{\beta,1}\xi^\beta=0,\hspace{3cm}S_{2,1}(\xi)=2!\sum_{\Lambda_1(\beta)=2}b_{\beta,1}\xi^\beta=\frac{2}{96}\zeta^6,
\end{equation*}
\begin{equation*}
    S_{3,1}(\xi)=3!\sum_{\Lambda_1(\beta)=3}b_{\beta,1}\xi^\beta=0,\hspace{1cm}\mbox{and}\hspace{1cm}S_{4,1}(\xi)=4!\sum_{\Lambda_1(\beta)=4}b_{\beta,1}\xi^\beta=-\frac{24}{48}\eta^4-\frac{24}{32}\eta^2\zeta^4-\frac{24\cdot 7}{2560}\zeta^8
\end{equation*}
for $\xi=(\eta,\zeta)$. This shows, in particular, that
\begin{equation*}
    \gamma_1=\min\{\lambda:S_{\lambda,1}\neq 0\}=2.
\end{equation*}
A completely analogous computation gives $\gamma_2=2$ as well. Of course, to simply pick out this constant with the aim of applying Corollary \ref{cor:LLTCor}, it was not necessary to explicitly write down the polynomials above. We have done so for illustrative purposes and so that, if it is of interest, the reader can form the operators $Q_{\lambda,1}^n$ for $\lambda=1,2,3,4$ appearing in Theorem \ref{thm:LLT}. In a separate (and perhaps more interesting) example presented in Section \ref{sec:Misaligned}, we will compute the attractors $Q_{\lambda}^n H_P^n$ for $\lambda=0,1,2,\dots,7$.

Since $\gamma_1=\gamma_2=2$, $Q_{1,1}=Q_{1,2}=0$ and $\mathcal{R}_0=\mathcal{R}_1:=\mathcal{R}$ in terms of Notation \ref{not:RealError}. With this,
\begin{eqnarray*}
    \mathcal{R}^n(x,y) &=& \phi^{(n)}(x,y)-\sum_{k=1}^2 e^{-i(x,y)\cdot\xi_k} \widehat{\phi}(\xi_k)^{n}H_{P_k}^{n}((x,y)-n\alpha_k) 
    \\ 
    &=&\phi^{(n)}(x,y)- \sum_{k=1}^2\frac{e^{-i(x,y)\cdot\xi_k}\widehat{\phi}(\xi_k)^n}{n^{\mu_k}}H_{P_{k}}\left(\left(n^{-D_k}(x,y)\right) \right) \\ 
    &=&\phi^{(n)}(x,y)-\frac{1}{n^{3/4}}\left[h_2^{1/2}\left(\frac{x}{n^{1/2}}\right)h_4^{1/16}\left(\frac{y}{n^{1/4}}\right)+(-1)^{x+y+n}h_4^{1/16}\left(\frac{x}{n^{1/4}}\right)h_2^{1/2}\left(\frac{y}{n^{1/2}}\right)\right]\\
    &=&\phi^{(n)}(x,y)-\frac{2\sqrt{2}}{n^{3/4}}\left[h_2\left(\frac{2x}{(2n)^{1/2}}\right)h_4\left(\frac{2y}{n^{1/4}}\right)+(-1)^{x+y+n}h_4\left(\frac{2x}{n^{1/4}}\right)h_2\left(\frac{2y}{(2n)^{1/2}}\right)\right]
\end{eqnarray*}
for $n\in\mathbb{N}_+$ and $(x,y)\in\mathbb{Z}^2$. An appeal to Corollary \ref{cor:LLTCor} gives us positive constants $C_1,C_2,M_1$, and $M_2$ for which the estimate
\begin{eqnarray}\label{eq:Ex1LLT}\nonumber
    \abs{\mathcal{R}^{n}(x,y)}
    &\leq&\sum_{k=1}^2\frac{C_k}{n^{\mu_k+\gamma_k/(2m_k)}}\exp\left(-nM_kR_k^{\#}\left(\frac{(x,y)-n\alpha_k}{n}\right)\right)\\
   &\leq& \frac{C}{n^{5/4}} \left[\exp\left(-\frac{M}{2n}x^2-\frac{3M}{4^{2/3}n^{1/3}}\abs{y}^{4/3}\right)+\exp\left(-\frac{3M}{4^{2/3}n^{1/3}}\abs{x}^{4/3}-\frac{M}{2n}y^2\right)\right]
\end{eqnarray}
holds for all $n\in\mathbb{N}_+$ and $(x,y)\in\mathbb{Z}^2$; here $C=\max\{C_1,C_2\}$ and $M=\min\{M_1,M_2\}$. The absolute value of the error and its generalized Gaussian bounds (with $C = 0.042$ and $M = 0.15$) are illustrated in Figure \ref{fig:Ex1LLT}. As the previous estimate holds for all $n\in\mathbb{N}_+$ and $(x,y)\in \mathbb{Z}^2$, we obtain
\begin{equation*}
    \| \mathcal{R}^{n}\|_{\infty} \leq \frac{C}{n^{5/4}}
\end{equation*}
for all $n \in \mathbb{N}_+$. This behavior is illustrated by Figure \ref{fig:Ex1_R0_LogLog}.
\begin{table}[h!]
  \centering
  \begin{tabular}{  | c | c | }
    \hline
     $n=100$ & $n=1000$ \\ \hline

    \begin{minipage}{.45\textwidth}
      \includegraphics[width=1.05\linewidth]{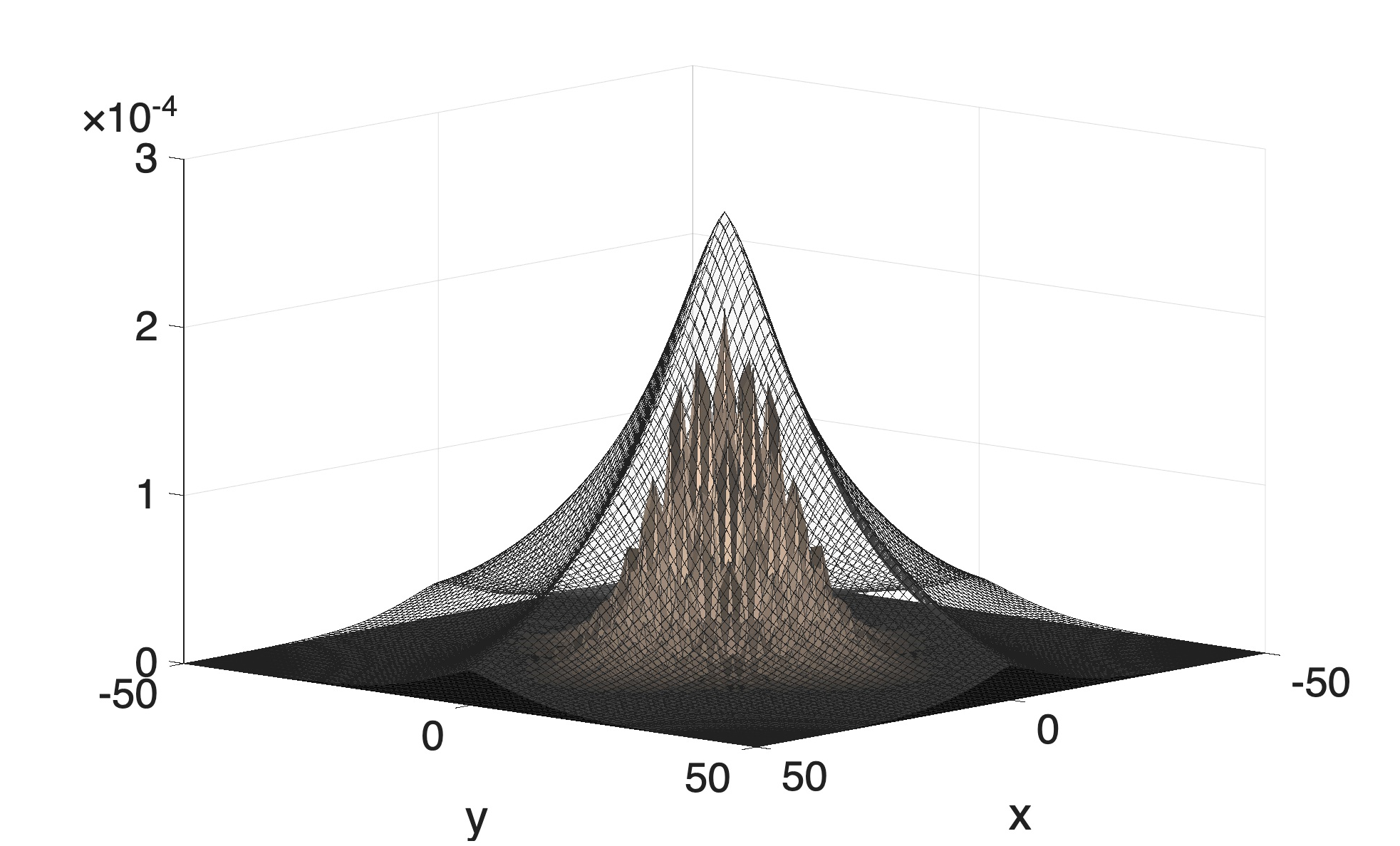}
    \end{minipage}
	&
      \begin{minipage}{.45\textwidth}
      \includegraphics[width=1.05\linewidth]{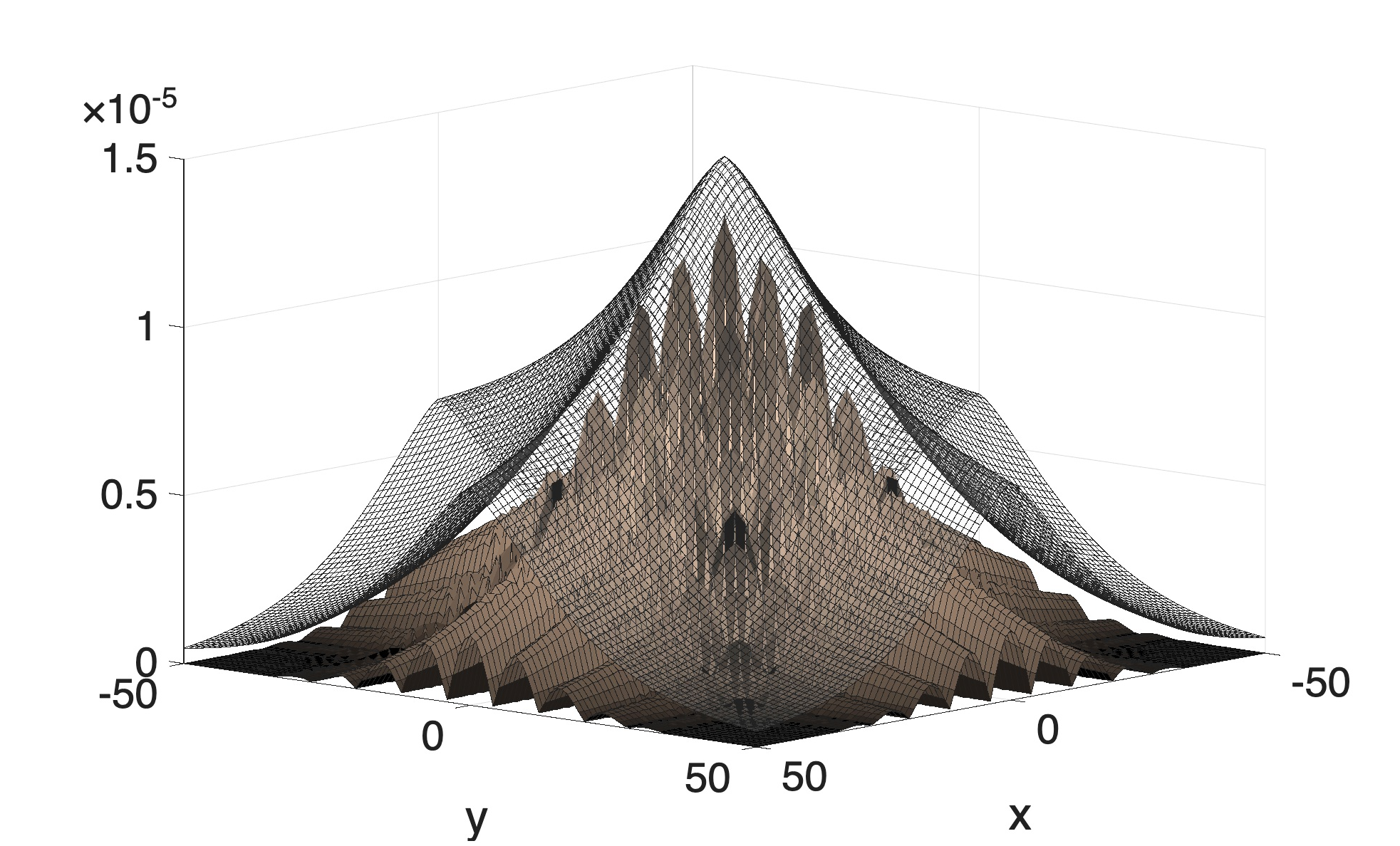}
    \end{minipage}\\ \hline
  \end{tabular}\captionof{figure}{The real error $\abs{\mathcal{R}^{n}}$ is illustrated by the solid light brown surfaces and the Gaussian-type error from the right-hand side of \eqref{eq:Ex1LLT} is illustrated by the transparent ``nets" above for $n=100$ and $n=1,000$.}\label{fig:Ex1LLT} 
\end{table}

\begin{figure}[h!]
    \centering
    \includegraphics[width=0.6\linewidth]{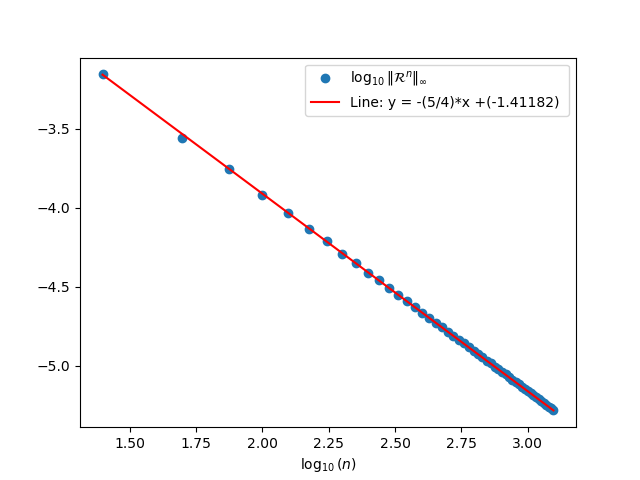}
    \caption{A graph of the $\log_{10}\Vert \mathcal{R}^{n}\Vert_{\infty}$ values (blue points) as a function of $\log_{10}(n)$, for values of $n$ ranging from $25$ to $1250$ in increments of $25$. The line in red is the best linear fit of fixed slope $-5/4$ for the data.} 
    \label{fig:Ex1_R0_LogLog}
\end{figure}

\subsection{Article Structure}

\noindent This article is organized as follows: In Section \ref{sec:ProofOfMain}, we prove our main results, Theorems \ref{thm:GeneralGaussEstimate}, \ref{thm:FiniteSupport}, and \ref{thm:LLT}; these proofs make use of two key lemmas which we take up in Section \ref{sec:Key}. In Section \ref{sec:Technical}, we establish several technical estimates used in the subsequent section. In Section \ref{sec:Key}, we prove the main technical estimates used in the proofs of our main theorems. These estimates are ``complex" in nature, making essential use of the holomorphy  of $\widehat{\phi}$ for $\phi\in\mathcal{H}_d^*$. In Section \ref{sec:Examples}, we revisit three examples from \cite{RSC17} that our results here substantially improve upon. The appendix is split into three subsections focusing, respectively, on the proofs of Propositions \ref{prop:PosHomAreSemiElliptic} and \ref{prop:QLmbdOpIsIndependentOfA}, asymptotics of positive-homogeneous functions, and the useful notion of subhomogeneity.

\section{Proofs of Theorems \ref{thm:GeneralGaussEstimate}, \ref{thm:FiniteSupport}, and \ref{thm:LLT}}\label{sec:ProofOfMain}
Let $\phi\in\mathcal{H}_d^*$ satisfy the hypotheses of Theorem \ref{thm:GeneralGaussEstimate} and assume the theorem's notation. With the aim of invoking the identity \eqref{eq:FTIdentityReal}, we shall assume that the finite set $\Omega(\phi)$ lives in the interior of $\mathbb{T}^d$, $\operatorname{Int}(\mathbb{T}^d)$, otherwise we replace $\mathbb{T}^d$ by a shifted copy of itself as it is done explicitly in \cite[Remark 4.2]{RSC17}. For each $k=1,2,\dots, K$, we write $\Gamma_k=\Gamma_{\xi_k}$ and correspondingly $\Upsilon_k=\Upsilon_{\xi_k}$ (in view of Definition \ref{def:PosHomType}). By virtue of Proposition \ref{prop:PosHomAreSemiElliptic}, we shall choose an element in $\GldR$ that makes $P_k$ semi-elliptic and denote it by $A_k$. With this, set
\begin{equation*}
\mathscr{D}_k=A_k([-\delta,\delta]^d)+\xi_k\subseteq\operatorname{Int}(\mathbb{T}^d)
\end{equation*}
where $\delta>0$ is yet to be specified but small enough so that $z\mapsto \Gamma_k(A_k z)$ is holomorphic on the polysquare
\begin{equation*}
[-\delta,\delta]^d+i[-\delta,\delta]^d=\{z\in\mathbb{C}^d:\max\{\abs{\Re z_j},\abs{\Im z_j}\}\leq \delta\,\,\mbox{for}\,\,j=1,2,\dots,d\}.
\end{equation*}
Our exponential estimates will be obtained by performing contour integration over the boundary of poly-rectangles living inside these poly-squares. Setting $\mathscr{B}=\mathbb{T}^d\setminus \cup_{k=1}^K \mathscr{D}_k$, the identity \eqref{eq:FTIdentityReal} gives
\begin{equation*}
    \phi^{(n)}(x)=\sum_{k=1}^K\int_{\mathscr{D}_k}\widehat{\phi}(\xi)^n e^{-ix\cdot\xi}\,d\xi+\int_{\mathscr{B}}\widehat{\phi}(\xi)^n e^{-ix\cdot\xi}\,d\xi
\end{equation*}
for $x\in\mathbb{Z}^d$ and $n\in\mathbb{N}_+$. Observe that, for each $k=1,2,\dots,K$,
\begin{eqnarray*}
\int_{\mathscr{D}_k}\widehat{\phi}(\xi)^n e^{-ix\cdot\xi}\,d\xi&=& e^{-ix\cdot\xi_k}\int_{A_k([-\delta,\delta]^d)}\widehat{\phi}(\xi+\xi_k)^n e^{-ix\cdot\xi}\,d\xi\\
&=&e^{-ix\cdot\xi_k}\widehat{\phi}(\xi_k)^n\int_{A_k([-\delta,\delta]^d)}e^{n\Gamma_k(\xi)}e^{-ix\cdot\xi}\,d\xi\\
&=&e^{-i x\cdot \xi_k}\widehat{\phi}(\xi_k)^n\int_{A_k([-\delta,\delta]^d)}\left(e^{-P_k(\xi)+\Upsilon_k(\xi)}\right)^n e^{-i(x-n\alpha_k)\cdot\xi}\,d\xi\\
&=&e^{-ix\cdot\xi_k}\widehat{\phi}(\xi_k)^n I_k(n,x-n\alpha_k)
\end{eqnarray*}
for $n\in\mathbb{N}_+$ and $x\in\mathbb{Z}^d$ where, for $f_k(\xi)=\exp(-P_k(\xi)+\Upsilon_k(\xi))$,
\begin{equation*}
I_k(n,y):=\int_{A_k([-\delta,\delta]^d)}f_k(\xi)^n e^{-iy\cdot\xi}\,d\xi
\end{equation*}
for $n\in\mathbb{N}_+$ and $y\in\mathbb{R}^d$. Consequently,
\begin{equation}\label{eq:ConvPowerSumWithI}
\phi^{(n)}(x)=\sum_{k=1}^K e^{-ix\cdot\xi_k}\widehat{\phi}(\xi_k)^n I_k(n,x-n\alpha_k)+\int_{\mathscr{B}}\widehat{\phi}(\xi)^n e^{-ix\cdot\xi}\,d\xi
\end{equation}
for $n\in\mathbb{N}_+$ and $x\in\mathbb{Z}^d$. Given that $\Omega(\phi)$ lives in the interior of $\cup_{k}\mathscr{D}_k$, it is evident that
\begin{equation*}
    \rho:=\sup_{\xi\in \mathscr{B}}\abs{\widehat{\phi}(\xi)}<1
\end{equation*}
and therefore
\begin{equation}\label{eq:BIntEst}
    \abs{\int_{\mathscr{B}} \widehat{\phi}(\xi)^n e^{-ix\cdot\xi}\,d\xi}\leq \int_{\mathscr{B}}\abs{\widehat{\phi}(\xi)}^n\,d\xi\leq \rho^n=e^{-\epsilon_0 n}
\end{equation}
for all $n\in\mathbb{N}_+$ and $x\in\mathbb{Z}^d$; here $\epsilon_0=\ln(1/\rho)>0$. To estimate the terms in our sum \eqref{eq:ConvPowerSumWithI}, we make use of the following key lemma which is proven in Section \ref{sec:Key}.
\begin{lemma}\label{lem:IntegralExponentialEst}
Let $P$ be a positive-homogeneous polynomial with real part $R$ and homogeneous order $\mu$. Let $\Upsilon$ be holomorphic on a neighborhood of $0$ in $\mathbb{C}^d$ with $\Upsilon(\xi)=o(R(\xi))$ as $\xi\to 0$ and set $f(\xi)=\exp(-P(\xi)+\Upsilon(\xi))$. Finally, from Proposition \ref{prop:PosHomAreSemiElliptic}, take $A\in\GldR$ for which $P_A$ is semi-elliptic. Then there exist $\delta>0$ and positive constants $\epsilon$, $M$, and $C$ for which
\begin{equation}\label{eq:IntegralExponentialEst1}
I(n,y)=\int_{A([-\delta,\delta]^d)} f(\xi)^n e^{-iy\cdot\xi}\,d\xi
\end{equation}
satisfies
\begin{equation}\label{eq:IntegralExponentialEst2}
\abs{I(n,y)}\leq \frac{C}{n^{\mu}}\left(e^{-n\epsilon}+\exp\left(-nM R^{\#}(y/n)\right)\right)
\end{equation}
for all $n\in\mathbb{N}_+$ and $y\in\mathbb{R}^d$.
\end{lemma}

Armed with the above lemma and in view of \eqref{eq:ConvPowerSumWithI}, we combine estimates \eqref{eq:BIntEst} and \eqref{eq:IntegralExponentialEst2} to obtain positive constants $\epsilon_0,\epsilon_1,\dots,\epsilon_K$, $C_1,C_2,\dots,C_K$, and $M_1,M_2,\dots,M_K$ for which
\begin{equation*}
\abs{\phi^{(n)}(x)}\leq e^{-\epsilon_0 n}+\sum_{k=1}^K \frac{C_k}{n^{\mu_k}}\left[e^{-n\epsilon_k}+\exp\left(-nM_k R_k^{\#}\left(\frac{x-n\alpha_k}{n}\right)\right)\right]
\end{equation*}
for all $n\in\mathbb{N}_+$ and $x\in\mathbb{Z}^d$. By, if necessary, modifying $\epsilon_0$ and $C_1$, the term $e^{-\epsilon_0 n}$ can be absorbed into the first term of the sum giving
\begin{equation}\label{eq:GoldenSum}
\abs{\phi^{(n)}(x)}\leq\sum_{k=1}^K \frac{C_k}{n^{\mu_k}}\left[e^{-n\epsilon_k}+\exp\left(-nM_k R_k^{\#}\left(\frac{x-n\alpha_k}{n}\right)\right)\right]
\end{equation}
which holds uniformly for $n\in\mathbb{N}_+$ and $x\in\mathbb{Z}^d$. We are now in a position to prove Theorem \ref{thm:FiniteSupport}.

\begin{proof}[Proof of Theorem \ref{thm:FiniteSupport}]
Given that $\phi:\mathbb{Z}^d\to\mathbb{C}$ is finitely supported, we have a uniform constant $L=\max\{\abs{y}:y\in \supp(\phi)\}$ for which $\abs{x}\leq nL$ for every $x\in\supp(\phi^{(n)})$ and $n\in\mathbb{N}_+$. By the triangle inequality, for each $k=1,2,\dots,K$, we obtain a constant $L_k\geq 0$ for which
\begin{equation*}
\abs{\frac{x-n\alpha_k}{n}}\leq L_k
\end{equation*}
for every $x\in\supp(\phi^{(n)})$ and $n\in\mathbb{N}_+$. Also, for each $k$, the fact that $R_k^{\#}$ is continuous on $\mathbb{R}^d$, it must be bounded on bounded sets and so we may find a  positive constant $M_k'$ for which
\begin{equation*}
M_k' R_k^{\#}\left(\frac{x-n\alpha_k}{n}\right)\leq \epsilon_k
\end{equation*}
for all $n\in\mathbb{N}_+$ and $x\in\supp(\phi^{(n)})$. By virtue of \eqref{eq:GoldenSum}, we have
\begin{eqnarray*}
\abs{\phi^{(n)}(x)}&\leq& \sum_{k=1}^K \frac{C_k}{n^{\mu_k}}\left[e^{-\epsilon_k n}+\exp\left(-nM_k R_k^{\#}\left(\frac{x-n\alpha_k}{n}\right)\right)\right]\\
&\leq &\sum_{k=1}^K\frac{C_k}{n^{\mu_k}}\left[\exp\left(-nM_k' R_k^{\#}\left(\frac{x-n\alpha_k}{n}\right)\right)+\exp\left(-nM_k R_k^{\#}\left(\frac{x-n\alpha_k}{n}\right)\right)\right]\\
&=&\sum_{k=1}^K\frac{C_k'}{n^{\mu_k}}\exp\left(-nM_k''R_k^{\#}\left(\frac{x-n\alpha_k}{n}\right)\right)
\end{eqnarray*}
for every $x\in\supp(\phi^{(n)})$ and $n\in\mathbb{N}_+$ where we have put $C_k'=2C_k$ and $M_k''=\min\{M_k,M_k'\}$. By renaming the constants $C_k$ and $M_k$, the theorem follows immediately from the above inequality and the observation that the inequality holds trivially outside of $\phi^{(n)}$'s support.
\end{proof}

\begin{proof}[Proof of Theorem \ref{thm:GeneralGaussEstimate}] As in Subsection 3.1 of \cite{CF24}, we first look at the far-field regime to obtain $L$. Because $\phi\in\mathcal{H}_d^*$, we take $\epsilon>0$ for which $F_\phi(z)$ is holomorphic on the poly-annulus $\{z\in\mathbb{C}^d:1-\epsilon<\abs{z_j}\leq 1+\epsilon,\,\,j=1,2,\dots,d\}$. With this, we define $\gamma=\ln(1+\epsilon/2)$ and take $\delta> 0$ for which
\begin{equation*}
    e^\delta=\sup\{\abs{F_\phi(z)}:1-\epsilon/2\leq \abs{z_j}\leq 1+\epsilon/2,\,\,j=1,2,\dots,d\}.
\end{equation*}
For $x\in\mathbb{Z}^d$ and $n\in\mathbb{N}_+$, we invoke Cauchy's formula to find
\begin{equation*}
    \phi^{(n)}(x)=\int_{\mathbb{T}^d}\widehat{\phi}(\xi-i\theta)^n e^{-ix\cdot(\xi-i\theta)}\,d\xi=e^{-x\cdot\theta}\int_{\mathbb{T}^d}\widehat{\phi}(\xi-i\theta)^ne^{-ix\cdot\xi}\,d\xi
\end{equation*}
for $\theta=\gamma\sgn(x)=\gamma(\sgn(x_1),\sgn(x_2),\dots,\sgn(x_d))$ (which keeps $\abs{e^{i(\xi_k-i\theta_k)}}\in [1-\epsilon/2,1+\epsilon/2]$ regardless of the values of the coordinates of $x$). We note that no boundary terms are produced in this application of Cauchy's formula thanks to the periodicity of $\widehat{\phi}$. Since
\begin{equation*}
x\cdot\theta=\gamma(\abs{x_1}+\abs{x_2}+\cdots+\abs{x_d})\geq \gamma\abs{x},
\end{equation*}
we obtain
\begin{equation*}
\abs{\phi^{(n)}(x)}\leq e^{-\gamma\abs{x}}\int_{\mathbb{T}^d}\abs{\widehat{\phi}(\xi-i\theta)}^n\,d\xi\leq e^{-\gamma\abs{x}}e^{n\delta}
\end{equation*}
for all $x\in\mathbb{Z}^d$ and $n\in\mathbb{N}_+$. Now, upon selecting $L=3\delta/\gamma$, it follows that
\begin{equation*}
    \abs{\phi^{(n)}(x)}\leq \exp\left(-\frac{\gamma}{2}\abs{x}-\frac{\delta}{2}n\right)\leq e^{-M_0(n+\abs{x})}
\end{equation*}
whenever $\abs{x}>nL$ where $M_0=\min\{\gamma,\delta\}/2$. This gives us the estimate for the far-field regime $\abs{x}>nL$. With this $L$ fixed in hand, we can easily obtain the near-field estimate (for $\abs{x}\leq nL$) by repeating exactly the same argument used in the proof of Theorem \ref{thm:FiniteSupport}. We leave these details to the reader.
\end{proof}

\noindent We now focus on the proof of Theorem \ref{thm:LLT}. Similar to the preceding arguments, we make use of the following lemma which bounds the error between $I(n,y)$ and the sum of attractors $Q_{\lambda}^nH_P^n(y)$. The lemma's proof can be found in Section \ref{sec:Key}.

\begin{lemma}\label{lem:IntegralLLTEst}
    Let $P$ be a positive-homogeneous polynomial with real part $R$ and homogeneous order $\mu$. Let $\Upsilon$ be holomorphic on a neighborhood of $0$ in $\mathbb{C}^d$ with $\Upsilon(\xi)=o(R(\xi))$ as $\xi\to 0$ and, as in Lemma \ref{lem:IntegralExponentialEst}, set $f(z)=\exp(-P(z)+\Upsilon(z))$. From Proposition \ref{prop:PosHomAreSemiElliptic}, take $A\in\GldR$ for which $P_A$ is semi-elliptic with $\mathbf{m}=(m_1,m_2,\dots,m_d)$ so that $\Upsilon_{A}$ has the representation
    \[
    \Upsilon_{A}(z)=\sum_{\abs{\beta : 2\mathbf{m}}>1}b_{\beta}z^{\beta}= \sum_{\lambda=1}^{\infty}\frac{S_{\lambda}(z)}{\lambda!},
    \]
    with $S_{\lambda}(z)=\lambda! \sum_{\Lambda(\beta)=\lambda}b_{\beta}z^{\beta}$, where we have set $m = \lcm(\mathbf{m})$, $\kappa=(\kappa_1,\kappa_2,\ldots,\kappa_d)\in \mathbb{N}_{+}^d$ such that $m_j \kappa_j =m$ for $j=1,2,\ldots,d$, and $\Lambda(\beta)=\beta\cdot\kappa- 2m$. With this, we consider the operators $Q_\lambda^n$ defined by \eqref{eq:QLambdaOpertor} and the associated attractors $Q_\lambda^n H_P^n$ defined for $\lambda,n\in\mathbb{N}$ with $n\geq 1$. For any $\lambda_0 \in \mathbb{N}$ there is $\delta>0$ and positive constants $\epsilon$, $M$, and $C$ for which $I(n,y)$ defined by \eqref{eq:IntegralExponentialEst1} satisfies
        \begin{equation*}
    \abs{I(n,y)-\sum_{\lambda=0}^{\lambda_0}(Q_{\lambda}^nH_P^n)(y)}\leq\frac{C}{n^{\mu+(\lambda_0+1)/2m}}\left(e^{-n\epsilon}+\exp(-nMR^{\#}(y/n))\right)
    \end{equation*}
for all $n\in\mathbb{N}_+$ and $y\in\mathbb{R}^d$.
  \end{lemma}

\begin{proof}[Proof of Theorem \ref{thm:LLT}]
Starting with the identity \eqref{eq:ConvPowerSumWithI}, we have
\begin{equation*}
\begin{aligned}
    &\abs{\phi^{(n)}(x)-\sum_{k=1}^K\sum_{\lambda=0}^{\lambda_k} e^{-ix\cdot\xi_k}\widehat{\phi}(\xi_k)^n(Q_{\lambda,k}^nH_{P_k}^n)(x-n\alpha_k)}\\
    &\leq \abs{\int_{\mathscr{B}}\widehat{\phi}(\xi)^ne^{-ix\cdot\xi}\,d\xi}+\abs{\sum_{k=1}^K e^{-ix\cdot\xi_k}\widehat{\phi}(\xi_k)^n\left(I_k(n,x-n\alpha_k)-\sum_{\lambda=0}^{\lambda_k} (Q_{\lambda,k}^nH_{P_n}^n)(x-n\alpha_k) \right)} \\
    &\leq \abs{\int_{\mathscr{B}}\widehat{\phi}(\xi)^ne^{-ix\cdot\xi}\,d\xi}+\sum_{k=1}^K\abs{I_k(n,x-n\alpha_k)-\sum_{\lambda=0}^{\lambda_k} (Q_{\lambda,k}^nH_{P_k}^n)(x-n\alpha_k)}
\end{aligned}
\end{equation*}
for all $n\in\mathbb{N}_+$ and $x\in\mathbb{Z}^d$ (so as long as $\delta>0$ is small enough). Taking $\delta>0$ to be that guaranteed by Lemma \ref{lem:IntegralLLTEst}, the results of the lemma and the estimate \eqref{eq:BIntEst} give positive constants constants $C_0,C_2,\dots,C_K$, $\epsilon_0,\epsilon_1,\dots,\epsilon_K$, and $M_1,M_2,\dots,M_K$, for which
\begin{equation*}
\begin{aligned}
    &\abs{\phi^{(n)}(x)-\sum_{k=1}^K\sum_{\lambda=0}^{\lambda_k} e^{-ix\cdot\xi_k}\widehat{\phi}(\xi_k)^n(Q_{\lambda,k}^nH_{P_k}^n)(x-n\alpha_k)}\\
    &\leq C_0e^{-\epsilon_0 n}+\sum_{k=1}^K\frac{C_k}{n^{\mu_k+(\lambda_k+1)/(2m_k)}}\left(e^{-n\epsilon_k }+\exp\left(-nM_k R_k^{\#}\left(\frac{x-n\alpha_k}{n}\right)\right)\right)
\end{aligned}
\end{equation*}
$n\in\mathbb{N}_+$ and $x\in\mathbb{Z}^d$. By adjusting the constants $C_0,C_1$ and $\epsilon_0,\epsilon_1$, we can absorb the zeroth term into the first term of the sum and from this it follows that
\begin{equation*}
\begin{aligned}
    &\abs{\phi^{(n)}(x)-\sum_{k=1}^K\sum_{\lambda=0}^{\lambda_k} e^{-ix\cdot\xi_k}\widehat{\phi}(\xi_k)^n(Q_{\lambda,k}^nH_{P_k}^n)(x-n\alpha_k)} \\
    &\hspace{4cm}\leq \sum_{k=1}^K\frac{C_k}{n^{\mu_k+(\lambda_k+1)/(2m_k)}}\left(e^{-n\epsilon_k }+\exp\left(-nM_k R_k^{\#}\left(\frac{x-n\alpha_k}{n}\right)\right)\right)
\end{aligned}
\end{equation*}
for all $n\in\mathbb{N}_+$ and $x\in\mathbb{Z}^d$. For a fixed $L>0$, we may now repeat the argument given for the proof of Theorem \ref{thm:FiniteSupport} to obtain the desired estimate \eqref{eq:LLT} for $\abs{x}\leq nL$. 

In the case of finitely-supported $\phi$, we argue along the lines of Section 3.1 of \cite{Co25} and Section 4.1 of \cite{CF24}. First, we apply the above argument to obtain positive constants $C_k$ and $M_k$ (for $k=1,2,\dots,K$) for which \eqref{eq:LLT} holds for when $\abs{x}\leq nL$ where 
\begin{equation*}
    L:=1+\max\{\abs{y},\abs{\alpha_k}:y\in\supp(\phi),k=1,2,\dots,K\}.
\end{equation*}
Though we do not need it here, we suspect that a Bernstein-type inequality akin to Lemma 3.1 of \cite{CF22} is true; such an inequality would give $\abs{\alpha_k}\leq \max\{\abs{y}:y\in\supp(\phi)\}$ so that $L=1+\max\{\abs{y}:y\in\supp(\phi)\}$. For $\abs{x}>nL$, we have $\phi^{(n)}(x)=0$ and so, to obtain a global estimate, we must find constants $C_k'\geq C_k$ and $0<M_k'\leq M_k$ for which
\begin{eqnarray}\label{eq:FarFieldGauss}\nonumber
    \abs{\sum_{k=1}^K\sum_{\lambda=0}^{\lambda_k}e^{-ix\cdot\xi_k}\widehat{\phi}(\xi_k)^n(Q_{\lambda,k}^n H_{P_k}^n)(x-n\alpha_k)}&\leq& \sum_{k=1}^K\sum_{\lambda=0}^{\lambda_k}\abs{(Q_{\lambda,k}^n H_{P_k}^n)(x-n\alpha_k)}\\
    &\leq&\sum_{k=1}^K\frac{C'_k}{n^{\mu_k+(\lambda_k+1)/2m_k}}\exp
    \left(-nM'_kR_k^{\#}\left(\frac{x-n\alpha_k}{n}\right)\right)
\end{eqnarray}
whenever $\abs{x}>nL$. To this end, for each $k$, using the fact that $e^{-P_k}$ is a Schwartz function, one can apply the method of proof of Proposition 2.6 of \cite{RSC17} to find positive constants $\rho_k$ and $B_{k}$ for which
\begin{equation*}
    \abs{Q_{\lambda,k}^n H_{P_k}^n(y)}\leq B_{k}e^{-n (2\rho_k) R_k^{\#}(y/n)}
\end{equation*}
for all $\lambda=0,1,\dots,\lambda_k$, $y\in\mathbb{R}^d$ and $n\in\mathbb{N}_+$. By virtue of Proposition \ref{prop:ExpSetLegFenchelTransform}, $R_k^{\#}$ is continuous and vanishes only at $0$ and therefore
\begin{equation*}
\epsilon_k:=\rho_k\inf_{y\in\mathbb{R}^d,\abs{y}\geq 1} R_k^{\#}(y)>0.
\end{equation*}
Upon noting that $\abs{(x-n\alpha_k)/n}\geq \abs{x/n}-\abs{\alpha_k}\geq 1
$ whenever $\abs{x}>nL$, we have
\begin{equation*}
    \abs{(Q_{\lambda,k}^n H_{P_k}^n)(x-n\alpha_k)}\leq B_{k}e^{-n\epsilon_k}\exp\left(-n\rho_kR_k^{\#}\left(\frac{x-n\alpha_k}{n}\right)\right)
\end{equation*}
for every $\lambda=0,1,\dots,\lambda_k$ and $n\in\mathbb{N}_+$ and $x\in\mathbb{Z}^d$ with $\abs{x}> nL$. With this, let's take $M_k'=\min\{\rho_k,M_k\}>0$ and $C_k'\geq C_k$ so that
\begin{equation*}
    (\lambda_k+1)B_ke^{-n\epsilon_k}\leq \frac{C_k'}{n^{\mu_k+(\lambda_k+1)/2m_k}}
\end{equation*}
for all $n\in\mathbb{N}$. Putting the two preceding estimates together, for each $k=1,2,\dots, K$, we obtain
\begin{equation*}
    \sum_{\lambda=0}^{\lambda_k}\abs{(Q_{\lambda,k}^n H_{P_k}^n)(x-n\alpha_k)}\leq \frac{C_k'}{n^{\mu_k+(\lambda_k+1)/2m}}\exp\left(-nM_k'R_k^{\#}\left(\frac{x-n\alpha_k}{n}\right)\right)
\end{equation*}
for $n\in\mathbb{N}_+$ and $x\in\mathbb{Z}^d$ with $\abs{x}>nL$. Summing these inequalities over $k$ immediately yields \eqref{eq:FarFieldGauss} and the proof is complete.
 
\end{proof} 

\section{Preliminary Technical Estimates}\label{sec:Technical}
In this section, we establish several technical estimates that are used to obtain Lemmas \ref{lem:IntegralExponentialEst} and \ref{lem:IntegralLLTEst}, which are key to the proofs of our main results. Before getting started, let's make some observations and fix some notation. Throughout this section, $P$ is a positive semi-elliptic polynomial of the form
\begin{equation*}
    P(\xi)=\sum_{\abs{\beta:\mathbf{m}}=2}a_\beta \xi^\beta
\end{equation*}
(with $\mathbf{m}=(m_1,m_2,\ldots,m_d)\in \mathbb{N}_{+}^d$ and $R= \Re P$) and $\Upsilon$ is a holomorphic function on a neighborhood of $0$ in $\mathbb{C}^d$ for which $\Upsilon(\xi) = o(R(\xi))$ for $\xi \to 0$ in $\mathbb{R}^d$. Throughout this section, we will deal primarily with three nested open neighborhoods of $0$ in $\mathbb{C}^d$: 
\begin{equation*}
    \mathcal{U}\subseteq\mathcal{U}_0\subseteq\mathcal{U}_{00}
\end{equation*}
The largest of these neighborhoods, $\mathcal{U}_{00}$, we fix here so that its closure $\overline{\mathcal{U}_{00}}$ is a compact subset of the domain of holomorphy of $\Upsilon$. With this, we can write
\begin{equation*}
\Upsilon(z)=\sum_{\beta\in\mathbb{N}^d}b_\beta z^\beta
\end{equation*}
where this series converges absolutely and uniformly on $\overline{\mathcal{U}_{00}}$. In view of Lemma \ref{lem:SubHomLittleO} and the fact that $\Upsilon(\xi)=o(R(\xi))$ as $\xi\to 0$, observe that
\begin{equation*}
0=\lim_{t\to 0}t^{-1}\Upsilon(t^D\xi)=\lim_{t\to 0}\sum_{\beta\in\mathbb{N}^d}b_\beta t^{-1}(t^D\xi)^\beta=\lim_{t\to 0}\sum_{\beta\in\mathbb{N}^d}b_\beta t^{\abs{\beta:2\mathbf{m}}-1}\xi^\beta
\end{equation*}
for every $\xi\in\mathbb{R}^d$ where $D=\diag\left(1/(2m_1),1/(2m_2),\ldots,1/(2m_d) \right)$. Consequently, $b_\beta=0$ for all $\beta$ for which $\abs{\beta:2\mathbf{m}}\leq 1$, and thus $\Upsilon$ is given by
\begin{equation*}
\Upsilon(z) = \sum_{\abs{\beta: 2\mathbf{m}}>1}b_{\beta}z^{\beta},
\end{equation*}
which converges absolutely and uniformly on $\overline{\mathcal{U}_{00}}$. \\

\noindent Looking back to $\mathbf{m}=(m_1,m_2,\dots,m_d$), we set $m= \lcm (\mathbf{m})$ and take $\kappa=(\kappa_1,\kappa_2,\ldots,\kappa_d) \in \mathbb{N}_{+}^d$ such that $m_j \kappa_j =m$ for $j=1,2,\ldots,d$. With this we put
\begin{equation*}
D' = 2mD=\diag (\kappa_1, \kappa_2,\ldots, \kappa_d). 
\end{equation*}
Throughout this section, we will work with both contracting groups $\{t^{D}\}$ and $\{t^{D'}\}$; the one we choose will depend on the application at hand. It is worth mentioning that the $\{t^{D'}\}$ can be evaluated in complex time and the map $\mathbb{C}^{d+1}\ni (r,z)\mapsto r^{D'}z=(r^{\kappa_1}z_1,r^{\kappa_2}z_2,\dots,r^{\kappa_d}z_d)\in\mathbb{C}^d$ is entire. Further, we have
\begin{equation*}
P(r^{D'}z)=r^{2m}P(z)
\end{equation*}
for every $(r,z)\in\mathbb{C}^{d+1}$ (with an analogous result holding for $R=\Re P$). Now, for a multi-index $\beta\in\mathbb{N}^d$, we define
\begin{equation*}
    \Lambda(\beta)=\beta\cdot\kappa-2m
\end{equation*}
and see easily that $\Lambda(\beta)\in \mathbb{N}_+$ precisely when $\abs{\beta:2\mathbf{m}}>1$. As discussed in the introduction, we aggregate the terms in the series for $\Upsilon$ so that
\begin{equation}\label{eq:Up-Slam}
\Upsilon(z)=\sum_{\lambda=1}^\infty\sum_{\Lambda(\beta)=\lambda}b_\beta z^\beta=\sum_{\lambda=1}^\infty \frac{1}{\lambda!}S_\lambda(z)
\end{equation}
where, for each $\lambda\in\mathbb{N}_+$, $S_\lambda$ is the polynomial given by
\begin{equation}\label{Slamb}
S_\lambda(z)=\lambda!\sum_{\Lambda(\beta)=\lambda} b_\beta z^\beta
\end{equation}
for $z\in\mathbb{C}^d$.
It is straightforward to see that the polynomials $S_\lambda$ enjoy the following two homogeneity identities:
\begin{equation}\label{eq:DPrimeScale}
S_\lambda(r^{D'}z)=r^{2m+\lambda}S_\lambda(z)\hspace{2cm}\forall r\in\mathbb{C}\,\,\mbox{ and }\,\,\forall z\in\mathbb{C}^d
\end{equation}
and
\begin{equation}\label{eq:DScale}
S_\lambda(r^Dz)=r^{1+\lambda/2m}S_\lambda(z)\hspace{2cm}\forall r\geq 0\,\,\mbox{ and }\,\,\forall z\in\mathbb{C}^d.
\end{equation}
Finally, given all of the above, we shall fix the compact set
\begin{equation*}
    \mathbb{S}=\{w=\eta-i\zeta\in\mathbb{C}^d:R(w):=R(\eta)+R(\zeta)= 1\}.
\end{equation*}
Since $\{t^D\}_{t>0}$ is contracting, there is some $t_0$ for which $t^D(\mathbb{S})\subseteq\mathcal{U}_{00}$ for all $0\leq t\leq t_0$. With this, fix our second largest neighborhood of $0$ as
\begin{equation*}
    \mathcal{U}_0:=\mathcal{U}_{00}\cap \{z=\xi-i\nu:R(z)=R(\xi)+R(\nu)<t_0\}\subseteq\mathcal{U}_{00}
\end{equation*}
of $0$ in $\mathbb{C}^d$. We have the following basic result that will be used throughout this section.
\begin{lemma}\label{lem:UNaught}
Fix $t_0>0$, $\mathbb{S}$, and $\mathcal{U}_0$ as above. Then, for every $0\leq t\leq t_0$, $t^D(\mathbb{S})\subseteq \mathcal{U}_{00}$ and, for every $z\in \mathcal{U}_0$, $z=t^Dw$ where $w\in \mathbb{S}$ and $t=R(z)=R(\xi)+R(\nu)<t_0$.
\end{lemma}
\begin{proof}
    By the construction above, $t^{D}(\mathbb{S})\subseteq\mathcal{U}_{00}$ for all $0\leq t\leq t_0$ by our choice of $t_0$. If $z\in\mathcal{U}_0$ is zero, then clearly $z=0^{D}(w)$ where $t=R(0)=0$ for any $w\in\mathbb{S}$. If $z=\xi-i\nu\in\mathcal{U}_0$ is non-zero, we set $t=R(z)=R(\xi)+R(\nu)<t_0$ and $w=t^{-D}z=\eta+i\zeta$ so that $z=t^D(w)$, 
    \begin{equation*}
        R(w)=t^{-1}R(z)=\frac{R(\xi)+R(\nu)}{R(\xi)+R(\nu)}=1,
    \end{equation*}
    and therefore $w\in \mathbb{S}$. 
\end{proof}
\begin{remark}
    If $t_0$ is further restricted so that $\{z=\xi-i\nu:R(z)=R(\xi)+R(\nu)<t_0\}\subseteq\mathcal{U}_{00}$. Then, the above lemma gives
    \begin{equation*}
        \mathcal{U}_0=\bigcup_{0\leq t<t_0}t^{D}(\mathbb{S})\subseteq\mathcal{U}_{00}.
    \end{equation*}
    These sets were key to the construction of the surface-carried measures in \cite{BR22}.
\end{remark}
\noindent With our notation now set, we have the following lemma. 

\begin{lemma}\label{lem:PhiEst}
Let $\mathcal{U}_0\subseteq\mathcal{U}_{00}$, $\Upsilon$, and $S_{\lambda}$ for $\lambda=1,2,\dots,$ be as above and set $S_0=0$. For $\tau\in\mathbb{N}$, define
\begin{equation*}
    \Phi(z)=\sum_{\lambda=0}^\infty\frac{S_{\lambda+\tau}(z)}{\lambda!}
\end{equation*}
so that, in particular, $\Phi(z)=\Upsilon(z)$ when $\tau=0$. Then $\Phi$ is holomorphic on $\mathcal{U}_{00}$ and there is a constant $C=C_{\tau}>0$ for which
\begin{equation*}
    \abs{\Phi(z)}\leq C(R(\xi)+R(\nu))^{1+\max\{1,\tau\}/2m}
\end{equation*}
for all $z=\xi-i\nu\in\mathcal{U}_0$. In particular, $\Phi(z)=o(R(\xi)+R(\nu))$ as $z=\xi-i\nu\to 0$. 
\end{lemma}
\begin{proof}
We first observe that
\begin{equation*}
    \Phi(z)=\sum_{\lambda=0}^\infty \frac{(\lambda+\tau)!}{\lambda!}\sum_{\beta\cdot\kappa=2m+(\lambda+\tau)}b_\beta z^\beta=\sum_{\beta\cdot\kappa-2m\geq \tau}\frac{(\beta\cdot\kappa-2m)!}{(\beta\cdot\kappa-2m-\tau)!} b_\beta z^\beta.
\end{equation*}
Upon noting that $\kappa_j\leq m$ (with equality if and only if $m_j=1$), for any multi-index $\beta\in\mathbb{N}^d$ for which $\beta\cdot\kappa-2m\geq \tau$, we have
\begin{equation*}
  1\leq \frac{(\beta\cdot\kappa-2m)!}{(\beta\cdot\kappa-2m-\tau)!}\leq(\beta\cdot\kappa)^\tau \leq m^\tau\abs{\beta}^\tau
\end{equation*}
where $\abs{\beta}=\beta_1+\beta_2+\cdots+\beta_d$. Consequently
\begin{equation*}
    \limsup_{\abs{\beta}\to\infty}\abs{b_\beta z^\beta}^{1/\abs{\beta}}\leq \limsup_{\abs{\beta}\to\infty}\abs{\frac{(\beta\cdot\kappa-2m)!}{(\beta\cdot\kappa-2m-\tau)!}b_\beta z^\beta}^{1/\abs{\beta}}\leq\limsup_{\abs{\beta}\to\infty}(m\abs{\beta})^{\tau/\abs{\beta}}\abs{b_\beta z^\beta}^{1/\abs{\beta}}=\limsup_{\abs{\beta}\to\infty}\abs{b_\beta z^\beta}^{1/\abs{\beta}}
\end{equation*}
since $\lim_{n\to\infty} (mn)^{\tau/n}=1$. In other words, by virtue of the root test, the only place where the convergence of the series for $\Phi$ and $\Upsilon$ can differ is on the boundary of the domain of holomorphy of $\Upsilon$. In particular, the series for $\Phi$ converges absolutely and uniformly on $\overline{\mathcal{U}_{00}}$ and $\Phi$ is holomorphic on $\mathcal{U}_{00}$. 

With the above, set
\begin{equation*} M=\sup_{z\in\overline{\mathcal{U}_{00}}}\sum_{\lambda=0}^\infty\frac{\abs{S_{\lambda+\tau}(z)}}{\lambda!}<\infty.
\end{equation*}
Given $z\in\mathcal{U}_0$, we write $z=t^{D}w$ for $0\leq t=R(\xi)+R(\nu)<t_0$ and $w\in\mathbb{S}$ by virtue of Lemma \ref{lem:UNaught}. In view of \eqref{eq:DScale}, we have
\begin{eqnarray*} \hspace{-1cm}\abs{\Phi(z)}=\abs{\Phi(t^Dw)}&\leq&\sum_{\lambda=0}^\infty\frac{\abs{S_{\lambda+\tau}(t^Dw)}}{\lambda!}\\
&=&\sum_{\lambda=0}^\infty\frac{\abs{(t/t_0)^{1+(\lambda+\tau)/2m}S_{\lambda+\tau}(t_0^Dw)}}{\lambda!}\\
   &\leq&(t/t_0)^{1+\max\{1,\tau\}/2m}M=C (R(\xi)+R(\nu))^{1+\max\{1,\tau\}/2m}
\end{eqnarray*}
where we have set $C=M/t_0^{1+\max\{1,\tau\}/2m}$.
\end{proof}

\noindent As we will see, Lemma \ref{lem:IntegralExponentialEst} relies on the following estimate.

\begin{proposition}\label{prop:gBound}
Let $P$ be a semi-elliptic polynomial (with $R=\Re P$ and $\mathbf{m}=(m_1,m_2,\dots,m_d)\in\mathbb{N}_+^d$) and suppose that $\Upsilon$ is holomorphic on a neighborhood of $0$ in $\mathbb{C}^d$ and has $\Upsilon(\xi)=o(R(\xi))$ as $\xi\to 0$ for $\xi\in\mathbb{R}^d$. We set
\begin{equation*}
f(z)=e^{-P(z)+\Upsilon(z)}
\end{equation*}
for any $z\in\mathbb{C}^d$ in the domain of holomorphy of $\Upsilon$. Given $D=\diag(1/2m_1,1/2m_2,\dots,1/2m_d)\in \Exp(P)$, define $g(n,z)=f(n^{-D}z)^n$ for those $n\in\mathbb{N}_+$ and $z\in\mathbb{C}^d$ for which $n^{-D}z$ lives in the domain of holomorphy  of $\Upsilon$. Then, there exists an open neighborhood $\mathcal{U}$ of $0$ in $\mathbb{C}^d$ and positive constants $\epsilon$ and $M$ for which
\begin{equation*}
\abs{g(n,z)}\leq\exp\left(-\epsilon(\xi_1^{2m_1}+\xi_2^{2m_2}+\cdots+\xi_d^{2m_d})+M(\nu_1^{2m_1}+\nu_2^{2m_2}+\cdots+\nu_d^{2m_d})\right)
\end{equation*}
whenever $z=\xi-i\nu\in n^D(\mathcal{U})$.
\end{proposition}
\begin{proof}
We first appeal to Proposition 8.13 of \cite{RSC17} to find positive constants $\epsilon'$ and $M'$ for which
\begin{equation*}
- \Re P(z)\leq -\epsilon' R(\xi)+M'R(\nu)
\end{equation*}
for $z=\xi-i\nu\in\mathbb{C}^d$. Since $\abs{\Upsilon(z)}=o(R(\xi)+R(\nu))$ as $z=\xi-i\nu\to 0$ by virtue of the preceding lemma, we can find an open neighborhood $\mathcal{U}\subseteq\mathcal{U}_0$ of $0$ in $\mathbb{C}^d$ for which
\[
\abs{\Upsilon(z)} \leq \frac{\epsilon'}{2}\left( R(\xi) + R(\nu) \right),
\]
for all $z=\xi -i\nu\in \mathcal{U}$. Thus, since $R$ is homogeneous with respect to $D$, we have
\begin{eqnarray*}
\abs{g(n,z)}&=&e^{n\Re(-P(n^{-D}z)+\Upsilon(n^{-D}z))}\\
&\leq &e^{-n \Re(P(n^{-D}z))+n\abs{\Upsilon(n^{-D}z)}}\\
&\leq& e^{-n\epsilon' R(n^{-D}\xi)+nM'R(n^{-D}\nu)+\tfrac{\epsilon'}{2}(nR(n^{-D}\xi)+nR(n^{-D}\nu))}\\
&=&e^{-\epsilon R(\xi)+M R(\nu)}
\end{eqnarray*}
for $z\in n^{D}(\mathcal{U})$ where we have set $\epsilon=\epsilon'/2$ and $M=M'+\epsilon'/2$. Upon noting that $R(x)\asymp x_1^{2m_1}+x_2^{2m_2}+\cdots+x_d^{2m_d}$ by virtue of Proposition \ref{prop:PureLambdaCompare}, the desired estimate follows by, if necessary, adjusting the constants $\epsilon$ and $M$.
\end{proof}

\noindent Our next task is to obtain an estimate, analogous to that of the previous proposition, that will be used in the proof of Lemma \ref{lem:IntegralLLTEst}.  To this end,we are interested in the expression $n\Upsilon\left( n^{-D}z \right)$ for $n\in \mathbb{N}_{+}$ which, at least formally, can be expressed as
\begin{equation*}
    n\Upsilon(n^{-D}z)=n\sum_{\lambda=1}^\infty \frac{S_{\lambda}(n^{-D}z)}{\lambda!}=n\sum_{\lambda=1}^\infty n^{-1-\lambda/2m}\frac{S_\lambda(z)}{\lambda!}=\sum_{\lambda=1}^\infty \frac{S_\lambda(z)}{\lambda!}\left(\frac{1}{n^{1/2m}}\right)^\lambda.
\end{equation*}
This motivates us to consider the function

\begin{equation}\label{eq:PsiDef}
\psi(r,z):= \sum_{\lambda=1}^{\infty} \frac{S_{\lambda}(z)}{\lambda !}r^{\lambda},
\end{equation}
which  converges, at least, when $0< r<1$ and $z\in\mathcal{U}_0$ and satisfies
\begin{equation}\label{eq:PsiRelation}
    \Upsilon(r^{D'}z)=r^{2m}\psi(r,z).
\end{equation}
In fact, the following is true.

\begin{lemma}\label{lem:PsiProperties}
    Let $\mathcal{U}_0$ be the open neighborhood of $0$ in $\mathbb{C}^d$ as guaranteed by Lemma \ref{lem:UNaught} whose closure lives inside of the domain of holomorphy of $\Upsilon$. Define
    \begin{equation*}
    \mathcal{D}_0=\left\{(r,z)\in \mathbb{C}\times\mathbb{C}^d: r^{D'}z\in\mathcal{U}_0\right\}\hspace{1cm}\mbox{and}\hspace{1cm}\mathcal{D}_0^+=\left\{(r,z)\in [0,\infty)\times\mathbb{C}^d:r^{D'}z\in\mathcal{U}_0\right\}
    \end{equation*}
    so that $\mathcal{D}_0^+\subseteq\mathcal{D}_0$. Then
    \begin{equation*}
        \psi(r,z)=\sum_{\lambda=1}^\infty \frac{S_{\lambda}(z)}{\lambda!} r^\lambda
    \end{equation*}
    is holomorphic on $\mathcal{D}_0$. Moreover, for each $\tau\in\mathbb{N}_+$,
    \begin{equation*}
        \psi_{\tau}(r,z):=(\partial_r^{\tau}\psi)(r,z)=\sum_{\lambda=\tau}^\infty \frac{S_{\lambda}(z)}{(\lambda-\tau)!} r^{\lambda-\tau}=\sum_{\lambda=0}^\infty\frac{S_{\lambda+\tau}(z)}{\lambda!}r^\lambda
    \end{equation*}
    for every $(r,z)\in\mathcal{D}_0$ and there is a constant $C=C_\tau>0$ for which
    \begin{equation*}
        \abs{\psi_\tau(r,z)}\leq C (R(\xi)+R(\nu))^{1+\tau/2m}
    \end{equation*}
    for every $(r,z)\in\mathcal{D}_0^+$.
\end{lemma}
\begin{proof}
To show that $\psi(r,z)$ is holomorphic on $\mathcal{D}_0$, we will first show that the series defining $\psi$ converges uniformly on every compact subset of $\mathcal{D}_0$. To this end, let $K\subseteq\mathcal{D}_0$ be compact and take $\epsilon>0$. Since the projection map $\Pi_2(r,z)=z$ is continuous, $\Pi_2(K)$ is a compact subset of $\mathbb{C}^d$ and so we can easily choose $r_0>0$ for which $r^{D'}z\in\mathcal{U}_0$ whenever $z\in\Pi_2(K)$ and $0\leq r\leq r_0$. By our construction of $\mathcal{U}_0$, the series defining $\Upsilon$ must converge absolutely and uniformly on $\overline{\mathcal{U}_0}\subseteq\overline{\mathcal{U}_{00}}$. Consequently, we may choose some $N\in\mathbb{N}_+$ such that
\begin{equation*}
    \sum_{\lambda=l}^n\frac{\abs{S_\lambda(w)}}{\lambda!}<r_0^{2m}\epsilon
\end{equation*}
whenever $w\in\mathcal{U}_0$ and $n\geq l\geq N$. Now, for any $(r,z)\in K$ and $n\geq l\geq N$, we have two cases: If $\abs{r}<r_0$,
\begin{equation*}
    \sum_{\lambda=l}^n\frac{\abs{S_{\lambda}(z)}}{\lambda!}\abs{r}^\lambda=\sum_{\lambda=l}^n\frac{\abs{S_{\lambda}(r_0^{D'}z)}}{\lambda!}r_0^{-2m-\lambda}\abs{r}^\lambda=r_0^{-2m}\sum_{\lambda=l}^n\frac{\abs{S_{\lambda}(r_0^{D'}z)}}{\lambda!}\abs{r/r_0}^\lambda<r_0^{-2m}\sum_{\lambda=l}^n\frac{\abs{S_{\lambda}(r_0^{D'} z)}}{\lambda!}<\epsilon
\end{equation*}
where we have used \eqref{eq:DPrimeScale}. If, $\abs{r}\geq r_0$ instead, then
\begin{equation*}
    \sum_{\lambda=l}^n\frac{\abs{S_{\lambda}(z)}}{\lambda!}\abs{r}^\lambda=\abs{r}^{-2m}\sum_{\lambda=l}^n\frac{\abs{S_{\lambda}(r^{D'}z)}}{\lambda!}\leq r_0^{-2m}\sum_{\lambda=l}^n\frac{\abs{S_{\lambda}(r^{D'}z)}}{\lambda!}<\epsilon.
\end{equation*}
Consequently, the series for $\psi$ converges absolutely and uniformly on every compact subset of $\mathcal{D}_0$.  Making an appeal to the theorem of Weierstrass (Theorem 1.75 of \cite{ScheidemannCASVBook}), we conclude that $\psi(r,z)$ is holomorphic on $\mathcal{D}_0$ and can be differentiated, term-by-term, ad infinitum. In particular, for any $\tau\in\mathbb{N}_+$,
\begin{equation*}
    \psi_{\tau}(r,z)=\sum_{\lambda=\tau}^\infty\frac{S_{\lambda}(z)}{(\lambda-\tau)!} r^{\lambda-\tau}=\sum_{\lambda=0}^\infty\frac{S_{\lambda+\tau}(z)}{\lambda!}r^\lambda
\end{equation*}
for all $(r,z)\in\mathcal{D}_0$. Now, for $(r,z)\in\mathcal{D}_0^+\subseteq\mathcal{D}_0$ with $r>0$, we make an appeal to Lemma \ref{lem:PhiEst} to find $C=C_\tau>0$ for which
\begin{equation*}
    \abs{\psi_{\tau}(r,z)}=\abs{r^{-2m-\tau}\sum_{\lambda=\tau}^\infty\frac{S_\lambda(r^{D'}z)}{(\lambda-\tau)!}}\leq r^{-2m-\tau}C(R(r^{D'}\xi)+R(r^{D'}\nu))^{1+\tau/2m}=C(R(\xi)+R(\nu))^{1+\tau/2m}.
\end{equation*}
For $r=0$, we have
\begin{equation*}
    \psi_{\tau}(0,z)=S_\tau(z)=S_\tau(t^Dt^{-D}z)=t^{1+\tau/2m}S_\tau(t^{-D}z)
\end{equation*}
by \eqref{eq:DScale} for every $t>0$ and $z=\xi-i\nu\in\mathbb{C}^d$. Of course, upon taking $t=R(z)=R(\xi)+R(\nu)$ so that $t^{-D}z\in\mathbb{S}$, we obtain the desired estimate by updating $C$, if necessary, so that $C=C_\tau\geq\sup_{w\in \mathbb{S}}\abs{S_\tau(w)}$. 
\end{proof}

\begin{lemma}\label{lem:BellEst}
Let $\mathcal{U}_0$ and $\mathcal{D}_0^+$ be as in Lemma \ref{lem:PsiProperties}. For any $\lambda_0\in\mathbb{N}$ and $\epsilon>0$, there exists an open neighborhood $\mathcal{U}\subseteq\mathcal{U}_0$ of $0$ in $\mathbb{C}^d$ and a positive constant $C$ for which
\begin{equation}\label{eq:BellEst}
    \abs{e^{\psi(r,z)}-\sum_{\lambda=0}^{\lambda_0}B_\lambda(S_1(z),S_2(z),\dots,S_{\lambda}(z))\frac{r^\lambda}{\lambda!}}\leq C r^{\lambda_0+1}e^{\epsilon (R(\xi)+R(\nu))}
\end{equation}
for all
\begin{equation*}
    (r,z)\in\mathcal{D}^+:=\{(r,z)\in [0,\infty)\times\mathbb{C}^d:r^{D'}z\in\mathcal{U}\}\subseteq\mathcal{D}_0^+.
\end{equation*}
\end{lemma}
\begin{proof}
For $(r,z)\in\mathcal{D}_0^+$, Taylor's theorem gives
\begin{equation*}
    e^{\psi(r,z)}-\sum_{\lambda=0}^{\lambda_0}(\partial_r^\lambda e^{\psi})(0,z)\frac{r^{\lambda}}{\lambda!}=\frac{r^{\lambda_0+1}}{(\lambda_0+1)!}\int_0^1(1-t)^{\lambda_0}(\partial_r^{\lambda_0+1}e^{\psi})(rt,z)\,dt.
\end{equation*}
From the generating formula for Bell polynomials \eqref{eq:BellGenerating}, we have
\begin{equation*}
    (\partial_r^{\lambda}e^{\psi(r,z)})(0,z)=B_{\lambda}(S_1(z),S_2(z),\dots,S_\lambda(z))
\end{equation*}
for every $\lambda\in\mathbb{N}$ and $z\in\mathbb{C}^d$. Consequently, for every $(r,z)\in\mathcal{D}_0^+$,
\begin{equation}\label{eq:BellEst1}
    \abs{e^{\psi(r,z)}-\sum_{\lambda=0}^{\lambda_0}B_{\lambda}(S_1(z),S_2(z),\dots,S_{\lambda}(z))\frac{r^\lambda}{\lambda!}}\leq\frac{r^{\lambda_0+1}}{(\lambda_0+1)!} \abs{\int_0^1(1-t)^{\lambda_0}(\partial_r^{\lambda_0+1}e^{\psi})(rt,z)\,dt}.
\end{equation}
Using the Fa\'{a} di Bruno formula, it is easy to see that the derivative $\partial_r^{\lambda_0+1}e^{{\psi}(r,z)}$ is simply a product of $e^{{\psi}(r,z)}$ and a polynomial\footnote{In the context of one dimension, \cite{CF24} describes this polynomial explicitly. See the proof of Lemma 7 in that article.} in $\psi_\tau=\partial_r^\tau \psi$ for $0\leq \tau\leq \lambda_0+1$. Thus, by the previous proposition, we can find $C>0$ and $M\in\mathbb{N}$ for which
\begin{equation}\label{eq:BellEst2}
\abs{(1-t)^{\lambda_0}(\partial_r^{\lambda_0+1}e^{\psi})(rt,z)}\leq C(1-t)^{\lambda_0}(R(\xi)+R(\nu))^{M}\abs{e^{\psi(tr,z)}}
\end{equation}
for $0\leq t\leq 1$ and $(r,z)\in\mathcal{D}_0^+$. Because $\Upsilon(z)=o(R(\xi)+R(\nu))$ by virtue of Lemma \ref{lem:PhiEst}, we select an open neighborhood $\mathcal{U}\subseteq\mathcal{U}_0$ of $0$ in $\mathbb{C}^d$ for which 
\begin{equation*}
\abs{\Upsilon(r^{D'}z)}\leq \frac{\epsilon}{2}\left(R(r^{D'}\xi)+R(r^{D'}\nu)\right)=\frac{\epsilon}{2}r^{2m}(R(\xi)+R(\nu))
\end{equation*}
whenever $r\geq 0$ and $z=\xi-i\nu$ have $r^{D'}z\in\mathcal{U}$, i.e., for $(r,z)\in \mathcal{D}^+$. Thus, because $\Upsilon(r^{D'}z)=r^{2m}\psi(r,z)$ for all $(r,z)\in\mathcal{D}$, we obtain the estimate
\begin{equation}\label{eq:BellEst3}
    \abs{\psi(r,z)}\leq\frac{\epsilon}{2}(R(\xi)+R(\nu))
\end{equation}
for $(r,z)\in\mathcal{D}^+$ after noting that the estimate is immediate when $r=0$. Combining \eqref{eq:BellEst2} and \eqref{eq:BellEst3}, we can adjust the constant $C$ if necessary to obtain
\begin{equation*}
    \abs{\int_0^1 (1-t)^{\lambda_0}(\partial_r^{\lambda_0+1}e^{\psi})(rt,z)}\leq C e^{\epsilon (R(\xi)+R(\nu))}
\end{equation*}
for all $(r,z)\in\mathcal{D}^+$ and, in view of \eqref{eq:BellEst1}, the proof is complete.
\end{proof}

\noindent Our final result of this section is used in the proof of Lemma \ref{lem:IntegralLLTEst}, a key step in the proof of Theorem \ref{thm:LLT}. This result is as follows.

\begin{proposition}\label{prop:HBound}
Let $P$ be a positive semi-elliptic polynomial with $\mathbf{m}=(m_1,m_2,\dots,m_d)\in\mathbb{N}_+^d$ so that $D=\diag(1/2m_1,1/2m_2,\dots,1/2m_d)\in\Exp(P)$ and set $m=\lcm(\mathbf{m})$. Suppose that $\Upsilon$ is holomorphic on a neighborhood of $0$ in $\mathbb{C}^d$ and has $\Upsilon(\xi)=o(R(\xi))$ as $\xi\to 0$ in $\mathbb{R}^d$. With the pair $(P,\Upsilon)$, let $S_1(z),S_2(z),\dots$ be the sequence of polynomials constructed in the introduction (and again in this section) and, for each $\lambda\in\mathbb{N}$, let $B_{\lambda}$ be the complete Bell polynomial of order $\lambda$ (with $B_0\equiv 1$). Given $\lambda_0\in\mathbb{N}$, set
\begin{equation*}
h(n,z)=e^{-nP(n^{-D}z)}\left( e^{n\Upsilon(n^{-D}z)}-\sum_{\lambda=0}^{\lambda_0}B_{\lambda}(S_1(z),S_2(z),\ldots,S_{\lambda}(z))\frac{1}{\lambda!n^{\lambda/2m}} \right),
\end{equation*}
for $n\in \mathbb{N}_+$ and all $z$ for which $n^{-D}z$ liven in the domain of holomorphy  of $\Upsilon$. Then, there exists an open neighborhood $\mathcal{U}$ of $0$ in $\mathbb{C}^d$ and positive constants $C$, $\epsilon$, and $M$ for which the estimate
\begin{equation*}
\abs{h(n,z)} \leq \frac{C}{n^{(\lambda_0+1)/2m}}\exp\left( -\epsilon \left(\xi_{1}^{2m_1}+\xi_{2}^{2m_2}+\cdots + \xi_{d}^{2m_d} \right) + M\left( \nu_{1}^{2m_1}+\nu_{2}^{2m_2}+ \cdots + \nu_{d}^{2m_d} \right)\right)
\end{equation*}
holds whenever $z=\xi-i\nu \in n^{D}(\mathcal{U})$. 
\end{proposition}
\begin{proof}
We first appeal to Proposition 8.13 of \cite{RSC17} to find positive constants $\epsilon'$ and $M'$ for which there holds
\begin{equation}\label{prt-RP}
-\Re P(z) \leq -\epsilon' R(\xi) + MR(\nu),    
\end{equation}
for all $z=\xi-i\nu \in \mathbb{C}^d$. Appealing to Lemma \ref{lem:BellEst} with $\lambda_0\in\mathbb{N}$ and $\epsilon'/2>0$ gives an open neighborhood $\mathcal{U}$ of $0$ in $\mathbb{C}^d$ for which \eqref{eq:BellEst} holds whenever $r\geq 0$ and $r^{D'}z\in \mathcal{U}$. Now, for $r=n^{-1/2m}>0$, $n^{-D}z=r^{D'}z$ and $\psi(n^{-1/2m},z)=n\Upsilon(n^{-D}z)$ and therefore
\begin{equation}\label{Bell-app-n}
\abs{e^{n\Upsilon(n^{-D}z)}-\sum_{\lambda=0}^{\lambda_0}B_{\lambda}(S_1(z),S_2(z),\ldots,S_{\lambda}(z))\frac{1}{\lambda!n^{\lambda/2m}}} \leq C (n^{-1/2m})^{\lambda_0 +1}e^{\epsilon'(R(\xi) + R(\nu))/2},
\end{equation}
whenever $n\in\mathbb{N}_+$ and $n^{-D}z\in\mathcal{U}$. Using \eqref{prt-RP} and \eqref{Bell-app-n}, we obtain
\begin{eqnarray*}
\lefteqn{\hspace{-6cm}\abs{e^{-nP(n^{-D}z)}\left(e^{n\Upsilon(n^{-D}z)}-\sum_{\lambda=0}^{\lambda_0}B_{\lambda}(S_1(z),S_2(z),\ldots,S_{\lambda}(z))\frac{1}{\lambda!n^{\lambda/2m}} \right) }}\\
\hspace{6cm}&\leq& \frac{C}{n^{(\lambda_0+1)/2m}}e^{-n\Re P(n^{-D}z)}e^{\epsilon'(R(\xi) + R(\nu))} \\ 
& \leq& \frac{C}{n^{(\lambda_0+1)/2m}}e^{-\epsilon'R(\xi)+M'R(\nu)}e^{\epsilon'(R(\xi) + R(\nu))/2} \\  
&=& \frac{C}{n^{(\lambda_0+1)/2m}}e^{-\epsilon R(\xi)+MR(\nu)},
\end{eqnarray*}
whenever $n\in \mathbb{N}_+$ and $n^{-D}z\in \mathcal{U}$; here, we have set $\epsilon=\epsilon'/2$ and $M=M'+\epsilon'/2$. Upon noting that $R(x) \asymp x_{1}^{2m_1}+x_{2}^{2m_2}+\cdots + x_{d}^{2m_d}$ by virtue of Proposition \ref{prop:PureLambdaCompare}, the desired estimate follows by, if necessary, adjusting the constants $\epsilon$ and $M$
\end{proof}

\section{Proofs of Lemmas \ref{lem:IntegralExponentialEst} and \ref{lem:IntegralLLTEst}}\label{sec:Key}
In this section, we prove the two key lemmas used in the proofs of Theorems \ref{thm:GeneralGaussEstimate}, \ref{thm:FiniteSupport}, and \ref{thm:LLT}. As the reader will see, essential to our arguments is an application (well, several) of Cauchy's integral formula with well-chosen contours which easily pick out the generalized Gaussian error written in terms of the Legendre-Fenchel transform (of $R$). The specific choice of contours we use are direct $d$-dimensional analogues of the contours used in \cite{CF24}.

\begin{proof}[Proof of Lemma \ref{lem:IntegralExponentialEst}]
As in the statement of the lemma, let's write $f(z)=e^{-P(z)+\Upsilon(z)}$ and, given $A\in\GldR$ for which $P_A$ is semi-elliptic (with $\mathbf{m}=(m_1,m_2,\dots,m_d)\in\mathbb{N}_+^d$ and $R_A=\Re P_A$), we set
\begin{equation*}
f_A(z)=f(Az)=e^{-P_A(z)+\Upsilon_A(z)}
\end{equation*}
and $g_A(n,z)=f_A(n^{-D}z)^n$ where $D=\diag(1/2m_1,1/2m_2,\dots,1/2m_d)\in\Exp(P_A)$. Upon noting that $\Upsilon_A(\xi)=o(R_A(\xi))$ as $\xi\to 0$, we make an appeal to Proposition \ref{prop:gBound} to find an open neighborhood $\mathcal{U}$ of $0$ in $\mathbb{C}^d$ and positive constants $\epsilon$ and $M$ for which 
\begin{equation}\label{eq:gAest}
\abs{g_A(n,z)}\leq e^{-\epsilon(\xi_1^{2m_1}+\xi_2^{2m_2}+\cdots+\xi_d^{2m_d})+M(\nu_1^{2m_1}+\nu_2^{2m_2}+\cdots+\nu_d^{2m_d})}
\end{equation}
whenever $z\in n^D(\mathcal{U})$. With this, select $\delta>0$ for which $[-\delta,\delta]^d+i[-\delta,\delta]^d\subseteq\mathcal{U}$ and make the change of variables $\xi\mapsto An^{-D}\xi$ to find
\begin{eqnarray}\label{eq:change_of_vars_eq}\nonumber
I(n,y)&=&\int_{A([-\delta,\delta]^d)}f(\xi)^n e^{-iy\cdot\xi}\,d\xi\\\nonumber
&=&\int_{n^{D}([-\delta,\delta]^d)}f(An^{-D}\xi)^ne^{-iy\cdot (An^{-D}\xi)}\,\abs{\det(An^{-D})}\,d\xi\\\nonumber
&=&\frac{\abs{\det(A)}}{\det(n^D)}\int_{n^D([-\delta,\delta]^d)}f_A(n^{-D}\xi)^ne^{-i(n^{-D}A^{\top}y)\cdot\xi}\, d\xi\\
&=&\frac{C_A}{n^{\mu}}\int_{n^D([-\delta,\delta]^d)}g_A(n,\xi)e^{-iw\cdot\xi}\,d\xi
\end{eqnarray}
where $w=n^{-D}A^{\top}y$, $C_A=\abs{\det(A)}$, and we have noted that $\mu_P=\mu_{P_A}=\tr D$. Consequently
\begin{equation*}
I(n,y)=\frac{C_A}{n^{\mu}}J(n,w)
\end{equation*}
where
\begin{equation*}
J(n,w):=\int_{\mathscr{R}(n)}g_A(n,\xi)e^{-iw\cdot\xi}\,d\xi
\end{equation*}
for $n\in\mathbb{N}_+$, $w\in\mathbb{R}^d$, and we have put
\begin{equation*}
\mathscr{R}(n)=n^{D}([-\delta,\delta]^d)=[-\delta_1,\delta_1]\times[-\delta_2,\delta_2]\times\cdots\times[-\delta_d,\delta_d]
\end{equation*}
for $\delta_j=n^{1/2m_j}\delta$ for $j=1,2,\dots,d$.

We claim that there are uniform positive constants $C$, $\epsilon$, and $M$ for which
\begin{equation}\label{eq:JBoundClaim}
\abs{J(n,w)}\leq Ce^{-\epsilon n}+Ce^{-MR_A^{\#}(w)}
\end{equation}
for all $n\in\mathbb{N}_+$ and $w\in\mathbb{R}^d$. Momentarily taking \eqref{eq:JBoundClaim} for granted and upon noting that $I-D=(I-D)^{\top}\in\Exp(R_A^{\#})$, an appeal to Proposition \ref{prop:LFCompare} yields
\begin{equation*}
R_A^{\#}(n^{-D}A^{\top}y)=nR_A^{\#}(A^{\top}y/n)=nR^{\#}(y/n).
\end{equation*}
With this, we invoke \eqref{eq:JBoundClaim} and the relationship between $I$ and $J$ to obtain
\begin{equation*}
\abs{I(n,y)}\leq \frac{C}{n^{\mu}}\left(e^{-\epsilon n}+e^{-nM R^{\#}(y/n)}\right)
\end{equation*}
for all $n\in\mathbb{N}_+$ and $y\in\mathbb{R}^d$; this is precisely the conclusion of the lemma. 

Thus, to complete the proof of the lemma, it remains to prove \eqref{eq:JBoundClaim}. To this end, we will apply the Cauchy integration formula successively to each of the integration coordinates $\xi_j$. For simplicity, we shall write $\mathscr{R}=\mathscr{R}(n)$ for the remainder of the proof. Taking our cues from \cite{CF24}, we set 
    \begin{equation}\label{eq:theta_defn}
        \theta_j = \begin{cases}
            \left ( \frac{\abs{w_j}}{2Mm_j} \right )^{1/(2m_j-1)} & \text{ if } \frac{|w_j|}{2m_jM} \leq \delta_j^{2m_j-1}\\
            \delta_j & \text{ otherwise}
        \end{cases}
    \end{equation}
    for $j=1,2,\dots,d$.  It is readily verified that the above choice of $\theta_j$ gives $0\leq M\nu^{2m_j-1}\leq \abs{w_j}/2m_j$ for every $\nu \in [0, \theta_j]$ and, further, $\theta_j \leq \delta_j$ for each $j=1,2,\dots,d$.
Also, for any $j=1,2,\dots,d$, we will write $\hat{\xi}_j=(\xi_1,\xi_2,\dots,\xi_{j-1},\xi_{j+1},\dots,\xi_d)$ to denote the $d-1$-tuple formed by removing the $j$th coordinate from $\xi$. In this spirit, we set
\begin{equation*}
 \widehat{\mathscr{R}}_j=\prod_{\substack{l=1\\ l\neq j}}^d [-\delta_l,\delta_l]
\end{equation*}
for $j=1,2,\dots,d$. With our notation set up, let us apply the Cauchy integration formula in the coordinate $\xi_d$ (going around the final rectangle) as illustrated in Figure \ref{fig:Contour}.

\begin{figure}[!ht]
\centering
\resizebox{0.8\textwidth}{!}{
\begin{circuitikz}

\draw[line width=0.72pt, fill opacity=1.00] (7.75, 14.5) circle [radius=6pt];
\draw[line width=0.72pt, fill opacity=1.00] (22.75,14.5) circle [radius=6pt];
\draw[line width=0.72pt, fill opacity=1.00] (22.75, 8.25) circle [radius=6pt];
\draw[line width=0.72pt, fill opacity=1.00] (7.75, 8.25) circle [radius=6pt];
\node at (7.75,14.5) [circ] {};
\draw [ line width=1.5pt](7.75,14.5) to[short] (24.25,14.5);
\draw [ line width=5.5pt](7.75,14.5) to[short] (5.75,14.5);
\draw [ line width=5.5pt](22.75,14.5) to[short] (24.75,14.5);
\node at (22.75,14.5) [circ] {};
\draw [ line width=5.5pt](7.75,14.5) to[short] (7.75,8.25);
\node at (7.75,8.25) [circ] {};
\draw [ line width=5.5pt](7.75,8.25) to[short] (22.75,8.25);
\draw [ line width=5.5pt](22.75,14.5) to[short] (22.75,8.25);
\node at (22.75,8.25) [circ] {};
\draw [line width=1.1pt, ->, >=Stealth] (15.25,8.25) -- (15.25,16.5);
\draw [line width=5.5pt, ->, >=Stealth] (16.4,8.25) -- (16.6,8.25);
\node [font=\LARGE] at (15.6,15) {$0$};
\node [font=\LARGE] at (26.25,14.6) {$\xi_d \in \mathbb{R}$};
\draw [line width=5.5pt, ->, >=Stealth] (7.75,11.30) -- (7.75,11.25);
\draw [line width=5.5pt, ->, >=Stealth] (22.75,11.45) -- (22.75,11.5);
\node [font=\LARGE] at (7.75,15.25) {$-\delta_d$};
\node [font=\LARGE] at (6.75,7.5) {$-\delta_d - i\theta_d$};
\node [font=\LARGE] at (23.75,7.5) {$\delta_d- i\theta_d$};
\node [font=\LARGE] at (6,11.5) {$J^{(1)}_1$};
\node [font=\LARGE] at (24.5,11.5) {$J^{(1)}_3$};
\node [font=\LARGE] at (14.5,7.5) {$J^{(1)}_2$};
\node [font=\LARGE] at (22.75,15.25) {$\delta_d$};
\end{circuitikz}
}
\caption{Contour deformation for the coordinate $\xi_d$ (assuming $\sigma_d>0$.)}\label{fig:Contour}
\end{figure}

We have
    \begin{equation}\label{eq:first_CIT_application}
        \begin{split}
        J(n,w) = \int_{\mathscr{R}}g_A(n,\hat{\xi}_d, \xi_d)e^{-iw\cdot \xi}d\xi = &  J^{(1)}_1 + J^{(1)}_2 + J^{(1)}_3 \\
        \end{split}
    \end{equation}
    where
    \begin{equation*}
    J^{(1)}_1 = -i\sigma_d\int_{\widehat{\mathscr{R}}_d}\int_{0}^{\theta_d}g_A(n,\hat{\xi}_d, -\delta_d-i\sigma_d\nu)e^{-iw\cdot (\hat{\xi}_d, -\delta_d-i\sigma_d\nu)}\,d\nu\,d\hat{\xi}_d,
    \end{equation*}
    \begin{equation*}
    J^{(1)}_3 = -i\sigma_d\int_{\widehat{\mathscr{R}}_d}\int_{\theta_d}^{0}g_A(n,\hat{\xi}_d, \delta_d-i\sigma_d\nu)e^{-iw\cdot (\hat{\xi}_d, \delta_d-i\sigma_d\nu)}\,d\nu\,d\hat{\xi}_d,
    \end{equation*}
    and
    \begin{eqnarray*}
    J^{(1)}_2 &=& \int_{\widehat{\mathscr{R}}_d}\int_{-\delta_d}^{\delta_d}g_A(n,\hat{\xi}_d, \xi_d-i\sigma_d\theta_d)e^{-iw\cdot (\hat{\xi}_d, \xi_d-i\sigma_d\theta_d)}d\xi_dd\hat{\xi}_d\\
    &=&\int_{\mathscr{R}} g_A(n,\hat{\xi}_d,\xi_d-i\sigma_d\theta_d)e^{-iw\cdot(\hat{\xi}_d,\xi_d-i\sigma_d\theta_d)}\,d\xi
    \end{eqnarray*}
    where (and in what follows) $\sigma_j=\sgn(w_j)$ for $j=1,2,\dots,d$. In our notation, the subscripts keep track of the segments of the rectangle over which we're integrating and the superscripts track the number of times the Cauchy's formula has been applied. 

   We claim that there are positive constants $C$ and $\epsilon$ for which 
   \begin{equation}\label{eq:J1ExpEst}
   \abs{J_1^{(1)}}\leq Ce^{-n\epsilon}
   \end{equation}
   for all $n\in\mathbb{N}_+$ and $w\in\mathbb{R}^d$. By virtue of \eqref{eq:gAest}, we have
   \begin{eqnarray*}
   \abs{J_1^{(1)}}&\leq &\int_{\widehat{\mathscr{R}}_d}\int_0^{\theta_d}e^{-\epsilon(\xi_1^{2m_1}+\cdots+\xi_{d-1}^{2m_{d-1}}+\delta_d^{2m_d})+M\nu^{2m_d}}\abs{e^{-iw\cdot(\hat{\xi}_d,-\delta_d-i\sigma_d\nu)}}\,d\nu\,d\hat{\xi}_d\\
   &=&\int_{\widehat{\mathscr{R}}_d}e^{-\epsilon(\xi_1^{2m_1}+\cdots+\xi_{d-1}^{2m_{d-1}}+\delta_d^{2m_d})}\,d\hat{\xi}_d\int_0^{\theta_d}e^{M\nu^{2m_d}}e^{-w_d\sigma_d\nu}\,d\nu\\
   &=& e^{-\epsilon\delta_d^{2m_d}}\int_{\widehat{\mathscr{R}}_d}e^{-\epsilon(\xi_1^{2m_1}+\cdots+\xi_{d-1}^{2m_{d-1}})}\,d\hat{\xi}_d\int_0^{\theta_d}\exp(M\nu^{2m_d}-\abs{w_d}\nu)\,d\nu
   \end{eqnarray*}
   for $n\in\mathbb{N}_+$ and $w\in\mathbb{R}^d$. Now, the first integral is bounded above by the $L^1(\mathbb{R}^{d-1})$ norm of its integrand. Observe that, by our choice of $\theta_d$, we have
   \begin{equation*}
   M\nu^{2m_d}-\abs{w_d}\nu\leq \frac{\abs{w_d}}{2m_d}\nu-\abs{w_d}\nu\leq 0
   \end{equation*}
   and therefore
   \begin{equation*}
   \int_0^{\theta_d}\exp(M\nu^{2m_d}-\abs{w_d}\nu)\,d\nu\leq \theta_d\leq\delta_d.
   \end{equation*}
   Upon recalling that $\delta_d=n^{1/2m_d}\delta$, these two observations yield
   \begin{equation*}
   \abs{J_1^{(1)}}\leq e^{-\epsilon\delta_d^{2m_d}}C\delta_d=n^{1/2m_d}\delta C \exp(-\epsilon n\delta^{2m_d} )
   \end{equation*}
  and so, by adjusting the values of $\epsilon$ and $C$, we obtain the desired estimate \eqref{eq:J1ExpEst}. The same estimate for $\abs{J_3^{(1)}}$ is easily obtained by analogous reasoning.

With the estimates for $J_1^{(1)}$ and $J_3^{(1)}$ in hand, our goal is to bound $J^{(1)}_2$ by successive applications of the Cauchy integral formula in the remaining $d-1$ coordinates. Indeed, thanks to the holomorphy  of $g_A$, another application of Cauchy's formula yields

    \begin{equation}\label{eq:second_CIT_application}
        J^{(1)}_2 =  J_1^{(2)}+J_2^{(2)}+J_3^{(2)}
    \end{equation} 
    
    where

    \begin{eqnarray*}
            \lefteqn{J_1^{(2)} =-i\sigma_{d-1}\int_{\widehat{\mathscr{R}}_{d-1}}\int_{0}^{\theta_{d-1}}g_A(n,\xi_1, \cdots,\xi_{d-2}, -\delta_{d-1}-i\sigma_{d-1}\nu, \xi_d-i\sigma_d\theta_d)}\\
            &&\hspace{6cm}\times e^{-iw\cdot (\xi_1, \cdots,\xi_{d-2}, -\delta_{d-1}-i\sigma_{d-1}\nu, \xi_d-i\sigma_d\theta_d)} d\nu d\hat{\xi}_{d-1},
            \end{eqnarray*} 
            \begin{eqnarray*}
            \lefteqn{J_3^{(2)} =  -i\sigma_{d-1}\int_{\widehat{\mathscr{R}}_{d-1}}\int_{\theta_{d-1}}^{0}g_A(n,\xi_1, \cdots,\xi_{d-2}, \delta_{d-1}-i\sigma_{d-1}\nu, \xi_d-i\sigma_d \theta_d)}\\
           &&\hspace{6cm} \times e^{-iw\cdot (\xi_1, \cdots, \xi_{d-2},\delta_{d-1}-i\sigma_{d-1}\nu, \xi_d-i\sigma_d\theta_d)} \,d\nu \,d\hat{\xi}_{d-1}
            \end{eqnarray*}
            and
            \begin{equation*}
            J_2^{(2)} = \int_{\mathscr{R}}g_A(n,\xi_1, \cdots,\xi_{d-2}, \xi_{d-1}-i\sigma_{d-1}\theta_{d-1}, \xi_d-i\sigma_d\theta_d)e^{-iw\cdot (\xi_1, \cdots, \xi_{d-2},\xi_{d-1}-i\sigma_{d-1}\theta_{d-1}, \xi_d-i\sigma_d\theta_d)} d\xi\\
        \end{equation*}
 An analogous argument to that done for $J_1^{(1)}$ shows that $J_1^{(2)}$ and $J_3^{(2)}$ satisfy the estimate \eqref{eq:J1ExpEst}. Combining \eqref{eq:first_CIT_application}, \eqref{eq:second_CIT_application}, and the estimates for $J_k^{(j)}$ for $k=1,3$ and $j=1,2$, we obtain
 \begin{equation*}
 \abs{J(n,w)}\leq Ce^{-\epsilon n}+\abs{J_1^{(2)}}+\abs{J_2^{(2)}}+\abs{J_3^{(2)}}\leq 3Ce^{-\epsilon n}+\abs{J_2^{(2)}}
 \end{equation*}
 which holds uniformly for $n\in\mathbb{N}_+$ and $w\in\mathbb{R}^d$. We continue this procedure inductively ($d-2$ more times) while successively applying Cauchy's integral formula to obtain the estimate
 \begin{equation}\label{eq:JBound}
 \abs{J(n,w)}\leq Ce^{-\epsilon n}+\abs{J_2^{(d)}}
 \end{equation}
 which holds uniformly for $n\in\mathbb{N}_+$ and $w\in\mathbb{R}^d$ (for uniform positive constants $C$ and $\epsilon$) where
 \begin{eqnarray*}
 J_2^{(d)}&=&\int_{\mathscr{R}}g_A(n,\xi_1-i\sigma_1\theta_1,\xi_2-i\sigma_2\theta_2,\dots,\xi_d-i\sigma_d\theta_d)e^{-iw\cdot(\xi_1-i\sigma_1\theta_1,\xi_2-i\sigma_2\theta_2,\dots,\xi_d-i\sigma_d\theta_d)}\,d\xi\\
 &=&\int_{\mathscr{R}}g_A(n,\xi-i\sigma\theta)e^{-iw\cdot(\xi-i\sigma\theta)}\,d\xi
 \end{eqnarray*}
 for $n\in\mathbb{N}_+$ and $\sigma\theta=(\sigma\theta)(w)=(\sigma_1\theta_1,\sigma_2\theta_2,\dots,\sigma_d\theta_d)$ as defined by \eqref{eq:theta_defn} with $\sigma_j=\sgn(w_j)$ for $j=1,2,\dots,d$. By an appeal to Proposition \ref{prop:gBound}, we have
 \begin{eqnarray}\label{eq:JdEst}\nonumber
 \abs{J_2^{(d)}}&\leq&\int_{\mathscr{R}}\abs{g_A(n,\xi-i\sigma\theta)}e^{-w\cdot \sigma\theta}\,d\xi\\ \nonumber
&\leq &\exp \left(\sum^d_{j=1} M\theta_j^{2m_j} -w_j\sigma_j \theta_j  \right ) \int_{\mathbb{R}^d} \exp \left(-\epsilon \sum^d_{j=1} \xi_j^{2m_j} \right) d\xi\\
&=&C\exp\left(\sum_{j=1}^d M\theta_j^{2m_j}-\abs{w_j}\theta_j\right).
 \end{eqnarray}
 If, for $j=1,2,\dots,d$, $\abs{w_j}/2m_j M\leq \delta_j^{2m_j-1}$, the definition \eqref{eq:theta_defn} gives
\begin{eqnarray*}
M\theta_j^{2m_j}-\abs{w_j}\theta_j&=&(M\theta_j^{2m_j-1}-\abs{w_j})\theta_j\\
&=&-\left(1-\frac{1}{2m_j}\right)\abs{w_j}\left(\frac{\abs{w_j}}{2m_jM}\right)^{1/(2m_j-1)}\\
&=&-M_j\abs{w_j}^{2m_j/2m_j-1}
\end{eqnarray*}
where $M_j=(1-1/2m_j)/(2m_jM)^{1/2m_j-1}>0$. On the other hand, if $\abs{w_j}/2m_jM>\delta_j^{2m_j-1}$,
\begin{equation*}
M\theta_j^{2m_j}-\abs{w_j}\theta_j=M\delta_j^{2m_j}-\abs{w_j}\delta_j\leq -(2m_j-1)M\delta_j^{2m_j}=-\epsilon_j n
\end{equation*}
where $\epsilon_j=(2m_j-1)M\delta^{2m_j} >0.$ Correspondingly, we have two cases: First, if $\abs{w_j}/2m_jM\leq \delta_j^{2m_j-1}$ for all $j=1,2,\dots,d$, then
\begin{eqnarray*}
\abs{J_2^{(d)}}&\leq& C\exp\left(-(M_1\abs{w_1}^{2m_1/(2m_1-1)}+M_2\abs{w_2}^{2m_2/(2m_2-1)}+\cdots+M_d\abs{w_d}^{2m_d/(2m_d-1)})\right)\\
&\leq & Ce^{-MR_A^{\#}(w)}
\end{eqnarray*}
for some $M>0$ by virtue of Proposition \ref{prop:LFCompare}. On the other hand, if there exists at least one $j$ for which $\abs{w_j}/2m_jM>\delta_j^{2m_j-1}$, we still find that all terms in the exponent of \eqref{eq:JdEst} are non-positive and at least one must be less than or equal to $-\epsilon n$ for some $\epsilon>0$. Consequently,
\begin{equation*}
\abs{J_2^{(d)}}\leq Ce^{-\epsilon n}.
\end{equation*}
Putting these two bounds together, we obtain uniform positive constants $C$, $\epsilon$ and $M$ for which
\begin{equation*}
\abs{J_2^{(d)}}\leq Ce^{-\epsilon n}+Ce^{-MR_A^{\#}(w)}
\end{equation*}
which holds for all $n\in\mathbb{N}_+$ and $w\in\mathbb{R}^d$. In view of \eqref{eq:JBound} and upon adjusting the constants $C$, $\epsilon$ and $M$ if necessary, we immediately obtain \eqref{eq:JBoundClaim} as we set out to do.
\end{proof}
\begin{proof}[Proof of Lemma \ref{lem:IntegralLLTEst}]
As in the statement of the lemma, let's write $f(z)=\exp(-P(z)+\Upsilon(z))$ and take $A\in \GldR$ for which $P_A$ is semi-elliptic with $\mathbf{m}=(m_1,m_2,\dots,m_d)\in\mathbb{N}_+^d$, that is
\[
P_{A}(z) = \sum_{\abs{\beta: 2 \mathbf{m}}=1}a_{\beta}z^{\beta},
\]
so that (as $\Upsilon_A(\xi)=\Upsilon(A\xi)=o(\Re P_A(\xi))$ as $\xi\to 0$ in $\mathbb{R}^d$) $\Upsilon_{A}$ has the representation
\[
\Upsilon_{A}(z) = \sum_{\abs{\beta: 2\mathbf{m}}>1}b_{\beta}z^{\beta} = \sum_{\lambda=1}^{\infty}\frac{S_{\lambda}(z)}{\lambda!},
\]
with $S_{\lambda}(z)=\lambda! \sum_{\Lambda(\beta)=\lambda}b_{\beta}z^{\beta}$, where we have set $m = \lcm(\mathbf{m})$, $\kappa=(\kappa_1,\kappa_2,\ldots,\kappa_d)\in \mathbb{N}_{+}^d$ such that $m_j \kappa_j =m$ for $j=1,2,\ldots,d$, and $\Lambda(\beta)=\beta\cdot\kappa- m$. In addition, we define $f_A(z)=f(Az)=\exp(-P_A(z)+\Upsilon_A(z))$ and for $\lambda_0\in\mathbb{N}$ we set 
\begin{equation*}
    h_A(n,z)=e^{-P_A(z)}\left( e^{n\Upsilon_A(n^{-D}z)} -\sum_{\lambda=0}^{\lambda_0}B_{\lambda}(S_1(z), S_2(z),\ldots,S_{\lambda}(z))\frac{1}{\lambda!n^{\lambda/2m}} \right)
\end{equation*}
where $D=\diag(1/2m_1,1/2m_2,\dots,1/2m_d)\in\Exp(P_A)$. We make an appeal to Proposition \ref{prop:HBound} to find an open neighborhood $\mathcal{U}$ of $0$ in $\mathbb{C}^d$, and positive constants $C$, $\epsilon$, and $M$ for which
\begin{equation}\label{eq:hAEst}
    \abs{h_A(n,z)}\leq \frac{C}{n^{(\lambda_0+1)/(2m)}}\exp\left (-\epsilon(\xi_1^{2m_1}+\xi_2^{2m_2}+\cdots+\xi_d^{2m_d})+M(\nu_1^{2m_1}+\nu_2^{2m_2}+\cdots+\nu_d^{2m_d})\right)
\end{equation}
for $z=\xi-i\nu\in n^{D}(\mathcal{U})$. With this, we select $\delta>0$ for which $[-\delta,\delta]^d+i[-\delta,\delta]^d\subseteq\mathcal{U}$ and make a change of variables $\xi\mapsto An^{-D}\xi$ as in the proof of Lemma \ref{lem:IntegralExponentialEst} to find
\begin{equation*}
    I(n,y)=\frac{\abs{\det(A)}}{n^\mu}\int_{n^{D}([-\delta,\delta]^d)}f_A(n^{-D}\xi)^n e^{-iw\cdot\xi}\,d\xi=\frac{C_A}{n^{\mu}}\int_{\mathscr{R}(n)}e^{-P_A(\xi)}e^{n\Upsilon_A(n^{-D}\xi)}e^{-iw\cdot\xi}\,d\xi
\end{equation*}
for $n\in\mathbb{N}_+$ and $y\in\mathbb{R}^d$, where we have set $w=w(n,y)=n^{-D}A^{\top}y$, $\mathscr{R}(n)=n^{D}([-\delta,\delta]^d)$, and $C_A=\abs{\det(A)}$. Also, by making a similar change of variables (as illustrated in \eqref{eq:RewritingQH}), we have
\begin{eqnarray*}
\sum_{\lambda=0}^{\lambda_0}(Q_{\lambda}^n H_P^n)(y) &=&\sum_{\lambda=0}^{\lambda_0}\frac{1}{n^{\lambda/2m}\lambda!(2\pi)^d}\int_{\mathbb{R}^d}e^{-nP(\xi)}e^{-iy\cdot \xi}B_{\lambda}\left( S_{1}(n^{D}A^{-1}\xi),S_{2}(n^{D}A^{-1}\xi),\ldots, S_{\lambda}(n^{D}A^{-1}\xi) \right)\,d\xi \\ 
    &=& \frac{C_A}{n^{\mu}} \sum_{\lambda=0}^{\lambda_0}\frac{1}{n^{\lambda/(2m)}\lambda!(2\pi)^d}\int_{\mathbb{R}^d}e^{-P_{A}(\xi)}e^{-iw\cdot \xi}B_{\lambda}\left( S_{1}(\xi),S_{2}(\xi),\ldots, S_{\lambda}(\xi) \right)d\xi
\end{eqnarray*}
for $n\in\mathbb{N}_+$ and $y\in\mathbb{R}^d$, and $w$ as before. Combining  the two preceding equations gives
\begin{eqnarray*}
\lefteqn{ I(n,y)- \sum_{\lambda=0}^{\lambda_0} (Q_\lambda^n H_P^n)(y)}\\
&=& \frac{C_A}{n^{\mu}(2\pi)^d} \left( \int_{\mathscr{R}(n)}e^{-nP_{A}(n^{-D}\xi)}e^{n\Upsilon_{A}(n^{-D}\xi)}e^{-iw \cdot \xi}d\xi + \right. \\ 
&& \left. - \sum_{\lambda=0}^{\lambda_0}\frac{1}{n^{\lambda/2m}\lambda!}\int_{\mathbb{R}^d}e^{-nP_A(n^{-D}\xi)}e^{-i w\cdot \xi}B_{\lambda}(S_{1}(\xi),S_{2}(\xi),\ldots, S_{\lambda}(\xi))d\xi  \right) \\
&=& \frac{C_A}{n^{\mu}(2\pi)^d}\int_{\mathscr{R}(n)}h_{A}(n,\xi)e^{-iw \cdot \xi}d\xi + \\
&& \,\, - \frac{C_A}{n^{\mu}(2\pi)^d}\sum_{\lambda=0}^{\lambda_0}\frac{1}{n^{\lambda/2m}\lambda!}\int_{\mathbb{R}^d \setminus{\mathscr{R}(n)}}e^{-P_{A}(\xi)}e^{-i w\cdot \xi}B_{\lambda}(S_{1}(\xi),S_{2}(\xi),\ldots, S_{\lambda}(\xi))d\xi \\
&=& \frac{C_A}{n^{\mu}(2\pi)^d}K(n,w) - \frac{C_A}{n^{\mu}(2\pi)^d}\sum_{\lambda=0}^{\lambda_0}\frac{1}{n^{\lambda/2m}\lambda!}\int_{\mathbb{R}^d \setminus{\mathscr{R}(n)}}e^{-P_{A}(\xi)}e^{-i w\cdot \xi}B_{\lambda}(S_{1}(\xi),S_{2}(\xi),\ldots, S_{\lambda}(\xi))d\xi
\end{eqnarray*}
for $n\in\mathbb{N}_+$ and $y\in\mathbb{R}^d$, where we have set
\begin{equation*}
K(n,w)=\int_{\mathscr{R}(n)}h_A(n,\xi)e^{-iw\cdot\xi}\,d\xi
\end{equation*}
for $n\in\mathbb{N}_+$ and $w\in\mathbb{R}^d$. By repeating the arguments in the proof of Lemma \ref{lem:IntegralExponentialEst} which make use of Cauchy's integration formula and the estimate \eqref{eq:hAEst} (in place of \eqref{eq:gAest}), we obtain positive constants $C$, $\epsilon$, and $M$ for which
\begin{equation*}
    \abs{K(n,w)}\leq \frac{C}{n^{(\lambda_0+1)/(2m)}}\left(e^{-\epsilon n}+Ce^{-MR_A^{\#}(w)}\right)
\end{equation*}
for all $n\in\mathbb{N}_+$ and $w\in\mathbb{R}^d$. Now, combining the facts that $B_{\lambda}(S_1(\xi),S_{2}(\xi),\ldots,S_{\lambda}(\xi))$ is a polynomial in $\xi$, $e^{-R_{A}(\cdot)/2}$ is a element of the Schwartz space $\mathcal{S}(\mathbb{R}^d)$, $\mathcal{S}(\mathbb{R}^d)$ is stable under multiplication by polynomials, and $\mathcal{S}(\mathbb{R}^d)\subset L^{1}(\mathbb{R}^d)$ we can conclude that
\begin{equation*}
C'= \sup_{0\leq \lambda \leq \lambda_0}\|e^{-R_A(\cdot)/2} B_{\lambda}(S_1(\cdot),S_{2}(\xi),\ldots,S_{\lambda}(\cdot))\|_{L^1(\mathbb{R}^d)} < \infty.
\end{equation*}
Thus by setting $\epsilon'=\inf\{R_A(\xi)/2:\xi\in \mathbb{R}^d\setminus [-\delta,\delta]^d\}>0,$ we obtain
\begin{eqnarray*}
    \lefteqn{\abs{\int_{\mathbb{R}^d\setminus \mathscr{R}(n) }e^{-P_A(\xi)}e^{-iw\cdot\xi}B_{\lambda}(S_1(\xi),S_{2}(\xi),\ldots,S_{\lambda}(\xi)) \,d\xi}\leq \int_{\mathbb{R}^d\setminus \mathscr{R}(n)}e^{-R_A(\xi)}\abs{B_{\lambda}(S_1(\xi),S_{2}(\xi),\ldots,S_{\lambda}(\xi))} \,d\xi}\\
    &&\hspace{5.5cm}\leq e^{-n\epsilon'}\int_{\mathbb{R}^d\setminus \mathscr{R}(n)}e^{-R_A(\xi)/2}\abs{B_{\lambda}(S_1(\xi),S_{2}(\xi),\ldots,S_{\lambda}(\xi))} \,d\xi\leq C'e^{- n\epsilon'}
\end{eqnarray*}
for all $n\in\mathbb{N}_+$ and $y\in\mathbb{R}^d$, and $0 \leq \lambda \leq \lambda_0$. Consequently,
\begin{eqnarray*}
    \abs{I(n,y)- \sum_{\lambda=0}^{\lambda_0}(Q_\lambda^n H_P^n)(y)}& \leq& \frac{C_A}{n^\mu(2\pi)^d}\abs{K(n,w)}\\
    &&\hspace{-.2cm}+\frac{C_A}{n^{\mu}(2\pi)^d}\abs{\sum_{\lambda=0}^{\lambda_0}\frac{1}{n^{\lambda/2m}\lambda!}\int_{\mathbb{R}^d \setminus{\mathscr{R}(n)}}e^{-P_{A}(\xi)}e^{-i w\cdot \xi}B_{\lambda}(S_{1}(\xi),S_{2}(\xi),\ldots, S_{\lambda}(\xi))d\xi}\\
    &\leq&  \frac{C_A}{n^{\mu}}\frac{C}{n^{(\lambda_0+1)/(2m)}}\left(e^{-\epsilon n}+e^{-MR_A^{\#}(w)}\right)+C_A Ce^{-\epsilon'n}\\
    &\leq& \frac{C}{n^{\mu+(\lambda_0+1)/(2m)}}\left(e^{-\epsilon n}+e^{-MR_A^{\#}(w)}\right)
\end{eqnarray*}
for all $n\in\mathbb{N}_+$ and $y\in\mathbb{R}^d$, where $w=n^{-D}A^{\top}y$; here, we have adjusted uniform positive constants from line to line without explicit mention. As we did in the proof of Lemma \ref{lem:IntegralExponentialEst}, an appeal to Proposition \ref{prop:ExpSetLegFenchelTransform} gives $R_A^{\#}(w)=nR_A^{\#}(A^\top y/n)=nR^{\#}(y/n)$ and therefore
\begin{equation*}
\abs{I(n,y)- \sum_{\lambda=0}^{\lambda_0} (Q_\lambda^n H_P^n)(y)} \leq \frac{C}{n^{\mu+(\lambda_0+1)/(2m)}}\left(e^{-\epsilon n}+e^{-n MR^{\#}(y/n)}\right)
\end{equation*}
which holds for all $n\in\mathbb{N}_+$ and $y\in\mathbb{R}^d$.
\end{proof}

\section{Examples}\label{sec:Examples}

Given the generality in which we have worked, many of the examples considered in the previous works \cite{DSC14,RSC17,CF22,CF24} are ripe for the application of our theorems. For generalized Gaussian bounds in $d$ dimensions, we remark that some of the examples (namely the introductory example and those of Subsections 7.1, 7.3, and 7.5) in \cite{RSC17} are able to be treated by Theorem 1.8 of that article; of course, Theorem \ref{thm:FiniteSupport} recaptures them here. However, for all finitely-supported examples in \cite{RSC17}, our local limit theorem (Theorem \ref{thm:LLT}) not only significantly improves upon the error of the corresponding results in that article, but also gives precise higher order approximation up to any desired order or accuracy. With this noted, we focus this section on revisiting three of the examples presented in Section 7 of \cite{RSC17}: \textit{Two drifting packets} (7.2); \textit{Supporting lattice misaligned with $\mathbb{Z}^2$} (7.3); and \textit{Contribution from non-minimal decay exponent} (7.4).

\subsection{Two Drifting Packets (revisiting Example 7.2 of \cite{RSC17})}\label{ssec:rev-Ex7.2}
In this example, we revisit Example 7.2 of \cite{RSC17} and illustrate how our results improve upon the results there. Consider $\phi:\mathbb{Z}^2\to\mathbb{C}$ defined by
\begin{equation*}
\phi(x,y)=\frac{1}{a}\times
\begin{cases}
\frac{1+i}{4} & (x, y) = (-1, \pm 1)\\
-\frac{1+i}{4} & (x, y) = (1, \pm 1)\\
\pm \frac{1}{\sqrt{2}} & (x, y) = (0, \pm 1)\\
0 & \mbox{otherwise}.
\end{cases}
\end{equation*} where $a = \sqrt{2 + \sqrt{2}}$. A direct computation shows that $\sup_{\xi}|\widehat{\phi}(\xi)|=1$ (so that $\phi\in\mathcal{H}_2^*$) and
\begin{equation*}
    \Omega(\phi) = \{\xi_1, \xi_2, \xi_3, \xi_4\} = \{(\pi/2,3\pi/4),(\pi/2,-\pi/4),(-\pi/2,-3\pi/4),(-\pi/2,\pi/4)\}
\end{equation*} where \begin{equation*}
    \widehat{\phi}(\xi_1) = \widehat{\phi}(\xi_4) = (i)^{5/4}\hspace{1cm}\mbox{ and }\hspace{1cm}\widehat{\phi}(\xi_2) = \widehat{\phi}(\xi_3) = -(i)^{5/4}.
\end{equation*} 
As in \cite{RSC17}, we compute
\begin{equation*}
    \Gamma_{k}(\xi) = i\alpha_k\cdot \xi - P_k(\xi) + \Upsilon_k(\xi)
\end{equation*} 
to find, for $\rho = \sqrt{2}-1$, $\alpha_1 = \alpha_2 = (0, \rho)$, $\alpha_3 = \alpha_4 = (0, -\rho)$, and
\begin{equation}\label{eq:TwoDrift1}
    P_k(\xi) = P(\xi) = \frac{1+i\rho}{4}\eta^2 + \rho \zeta^2
\end{equation} 
for $\xi=(\eta, \zeta)\in \mathbb{R}^2$ and $k=1,2,3,4$ . In particular, each expansion has the same positive semi-elliptic polynomial $P$ with $\mathbf{m}=(1,1)$ and 
\begin{equation*}
    R(\xi)=\frac{1}{4}\eta^2+\rho\zeta^2
\end{equation*}
for $\xi=(\eta,\zeta)\in\mathbb{R}^2$. Consequently, $\mu_k=\abs{\mathbf{1}:2\mathbf{m}_k}=\abs{\mathbf{1}:(2,2)}=1$, and
\begin{equation*}
    R_k^{\#}(x,y)=R^{\#}(x,y)=x^2+\frac{1}{4\rho}y^2
\end{equation*}
for $k=1,2,3,4$ and $(x,y)\in\mathbb{R}^2$. Though the holomorphic terms $\Upsilon_j$ are not all identical, they all are $o(R(\xi))$ as $\xi\to 0$ and so we may apply Theorems \ref{thm:FiniteSupport} and \ref{thm:LLT}. It should be noted that, while the leading polynomials are identical, the drifts are not and so generalized Gaussian estimates cannot be obtained directly from Theorem 1.8 of \cite{RSC17}. Applying Theorem \ref{thm:FiniteSupport} of the present article gives
\begin{eqnarray*}
    \abs{\phi^{(n)}(x,y)}&\leq & \sum_{k=1}^4\frac{C_k}{n^{\mu_k}}\exp\left(-nM_k R_k^{\#}\left(\frac{(x,y)-n\alpha_k}{n}\right)\right)\\
    &\leq& \sum_{k=1}^2\frac{C_k}{n}\exp\left(-nM_k R^{\#}\left(\frac{x,y-n\rho)}{n}\right)\right)+\sum_{k=3}^4\frac{C_k}{n}\exp\left(-nM_k R^{\#}\left(\frac{(x,y+n\rho)}{n}\right)\right)\\
    &\leq &\frac{C}{n}\left[\exp\left(-M\left(\frac{x^2}{n}+\frac{1}{4\rho}\frac{(y-n\rho)^2}{n}\right)\right)+\exp\left(-M\left(\frac{x^2}{n}+\frac{1}{4\rho}\frac{(y+n\rho)^2}{n}\right)\right)\right]\\
    &=&\frac{C}{n}e^{-Mx^2/n}\left[\exp\left(-M\frac{(y-n\rho)^2}{4\rho n}\right)+\exp\left(-M\frac{(y+n\rho)^2}{4\rho n}\right)\right]
\end{eqnarray*}
for $n\in\mathbb{N}_+$ and $(x,y)\in\mathbb{Z}^2$ where $C$ and $M$ are positive constants. This illustrates Gaussian decay in $x$ away from $0$ and also in $y$ away from the two peaks at $y\approx\pm n\rho$. 

Next, we present two local limit theorems. The first we will obtain from Corollary \ref{cor:LLTCor} with base attractors (only). The second will be gotten from Theorem \ref{thm:LLT} with base attractors and higher-order corrections to the first degree of accuracy: $\lambda=\lambda_1=\lambda_2=\lambda_3=\lambda_4=1$. We first observe that, as a consequence of relation \eqref{eq:TwoDrift1}  the base attractors are all given by
\begin{equation}\label{eq:TwoDrift2}
H_{P}^n(x,y) = \frac{1}{(2\pi)^2}\int_{\mathbb{R}^2}e^{-n P(\eta,\zeta)}e^{-i(x,y)\cdot(\eta,\zeta)}\,d\zeta d \eta = \frac{1}{2\pi n \sqrt{\rho(1+i\rho)}} \exp\left( -\frac{x^2}{n(1+i\rho)}-\frac{y^2}{4n\rho} \right), 
\end{equation} 
for $(x,y)\in\mathbb{R}^2$. We note that, since $P$ is (already) semi-elliptic, we have $A_k=I_2$ for $k=1,2,3,4$. Also, we have $\mathbf{m}=\mathbf{m}_1=\cdots\mathbf{m}_4=(1,1)$ so that $m=\lcm(\mathbf{m})=1$, $\kappa=\kappa_1=\cdots\kappa_4=(1,1)$, and $\Lambda=\Lambda_1=\cdots=\Lambda_4$ where
\begin{equation*}
\Lambda(\beta)=\beta\cdot\kappa-2m=\beta_1+\beta_2-2
\end{equation*}
for $\beta=(\beta_1,\beta_2)\in\mathbb{N}^2$.
As we did for the introductory example of Subsection \ref{ssec:IntroExample}, in order to apply Corollary \ref{cor:LLTCor} we must compute $\Upsilon_k$ to obtain $\gamma_k$ (and later $Q_{\lambda,k}$) for $k=1,2,3,4$. We find
\begin{equation*}
\Upsilon_1(\eta, \zeta) =   b_{(2,1),1}\eta^2\zeta +  b_{(0,3),1}\zeta^3 + b_{(4,0),1}\eta^4 +b_{(2,2),1}\eta^2\zeta^2+b_{(0,4),1}\zeta^4  +O(\eta^4\zeta+\eta^2\zeta^3+\zeta^5)
\end{equation*}
where, in particular,
\begin{equation*}
    b_{(2,1),1}=-\frac{1}{2}(1-\rho)\hspace{1cm}\mbox{and}\hspace{1cm}b_{(0,3),1}=\frac{2}{3}i(1-2\rho).
\end{equation*}
In the case of $k=2,3,4$, the structure of $\Upsilon_k$ is similar to that of $\Upsilon_1$ and it is readily verified that
\begin{equation*}
 b_{(2,1),1} = b_{(2,1),2},\quad b_{(0,3),1}= b_{(0,3),2} , 
\end{equation*}
\begin{equation*}
b_{(2,1),3} = b_{(2,1),4} = -b_{(2,1),1}=\frac{1}{2}(1-\rho),\hspace{1cm}\mbox{and}\hspace{1cm} b_{(0,3),3}= b_{(0,3),4} = -b_{(0,3),1}=-\frac{2}{3}i(1-2\rho).
\end{equation*}
With the above coefficients in hand, we see that, for each $k=1,2,3,4$, $\beta=(0,3)$ and $\beta=(2,1)$ are the only multi-index for non-zero coefficients for which $\Lambda(\beta)=\Lambda_k(\beta)=1$. Consequently, we have
\begin{equation}\label{eq:TwoDrift2}
    S_{1,1}(\xi)=S_{1,2}(\xi)=-\frac{1}{2}(1-\rho)\eta^2\zeta+\frac{2}{3}i(1-2\rho)\zeta^3
\end{equation}
and
\begin{equation}\label{eq:TwoDrift3}
    S_{1,3}(\xi)=S_{1,4}(\xi)=-S_{1,1}(\xi)=\frac{1}{2}(1-\rho)\eta^2\zeta-\frac{2}{3}i(1-2\rho)\zeta^3.
\end{equation}
With these, we immediately see that
\begin{equation*}
    \gamma_k=\min\{\lambda:S_{\lambda,k}\neq 0\}=1
\end{equation*}
for every $k=1,2,3,4$ and so we are in a position to use Corollary \ref{cor:LLTCor} (but now with the knowledge that two majorants appearing in the corollary coincide, i.e., we didn't gain anything from computing the $\gamma_k$s). We have 
\begin{eqnarray*}
\mathcal{R}_{0}^n(x,y) &=&\phi^{(n)}(x,y)-\sum_{k=1}^4 e^{-i(x,y)\cdot\xi_k} \widehat{\phi}(\xi_k)^{n}H_{P_k}^{n}((x,y)-n\alpha_k)\\
&= & \phi^{(n)}(x,y)-(i)^{5n/4} \left[\left((-1)^y + (-1)^n \right) (i)^{-x+y/2} H^n_P(x,y-\rho n)\right.\\
&&\hspace{5cm}\left.+ \left((-1)^{y+n} + 1 \right) (i)^{x-y/2}H^n_P(x,y+\rho n)\right]
\end{eqnarray*}
for $(x,y)\in \mathbb{Z}^2$. In this notation, an appeal to Corollary \ref{cor:LLTCor} gives
\begin{equation}\label{eq:TwoPacketsLLT1}
    \abs{\mathcal{R}^{n}_{0}(x,y)}\leq \frac{C}{n^{3/2}}e^{-Mx^2/n}\left[\exp\left(-M\frac{(y-n\rho)^2}{4\rho n}\right)+\exp\left(-M\frac{(y+n\rho)^2}{4\rho n}\right)\right]
\end{equation}
for all $n\in\mathbb{N}_+$ and $ (x,y)\in\mathbb{Z}^2$, and some positive constants $C$ and $M$. Similar to Figure 6 of \cite{RSC17}, for $n=30$ and $n=60$, we illustrate the real parts of $\phi^{(n)}$ in the first row in Figure \ref{fig:TwoPackets}. In the second row, we have illustrated the actual error $\abs{\mathcal{R}_{0}^n}$ and its Gaussian upper bound as given in \eqref{eq:TwoPacketsLLT1} for $n=30$ and $n=60$; here, we have used the values $C = 0.15$, $M=0.3$. In Figure \ref{fig:TwoPackets_R0_Loc}, we have focused around the peak $(x,y)\approx (0,n\rho)$ to illustrate this error for $n=1,000$. From\eqref{eq:TwoPacketsLLT1}, we have the $\ell^\infty$ estimate
\begin{equation*}
\|\mathcal{R}^{n}_0\|_\infty \leq \frac{C}{n^{3/2}}, 
\end{equation*}
for $n\in\mathbb{N}_+$. This is illustrated by the Figure \ref{fig:TwoPackets_R0_LogLog} for $20\leq n\leq 1,000$.
\begin{table}[!h]
  \centering
  \begin{tabular}{  |c| c | c | }
    \hline
     & $n=30$ & $n=60$ \\ \hline
     $\Re(\phi^{(n)})$ & 
    \begin{minipage}{.4\textwidth}
      \includegraphics[width=1\linewidth]{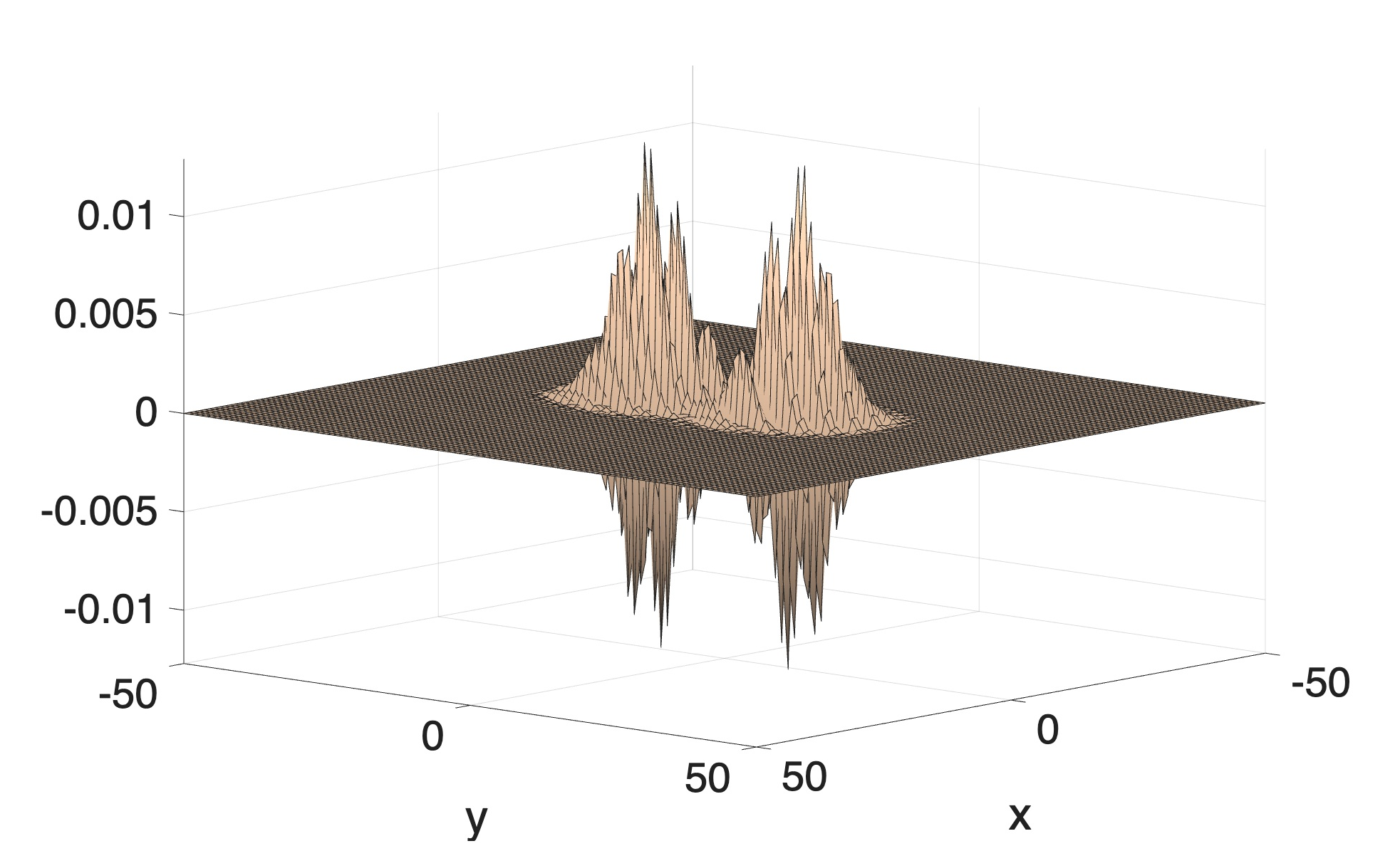}
    \end{minipage}
	&
      \begin{minipage}{.44\textwidth}
      \includegraphics[width=1\linewidth]{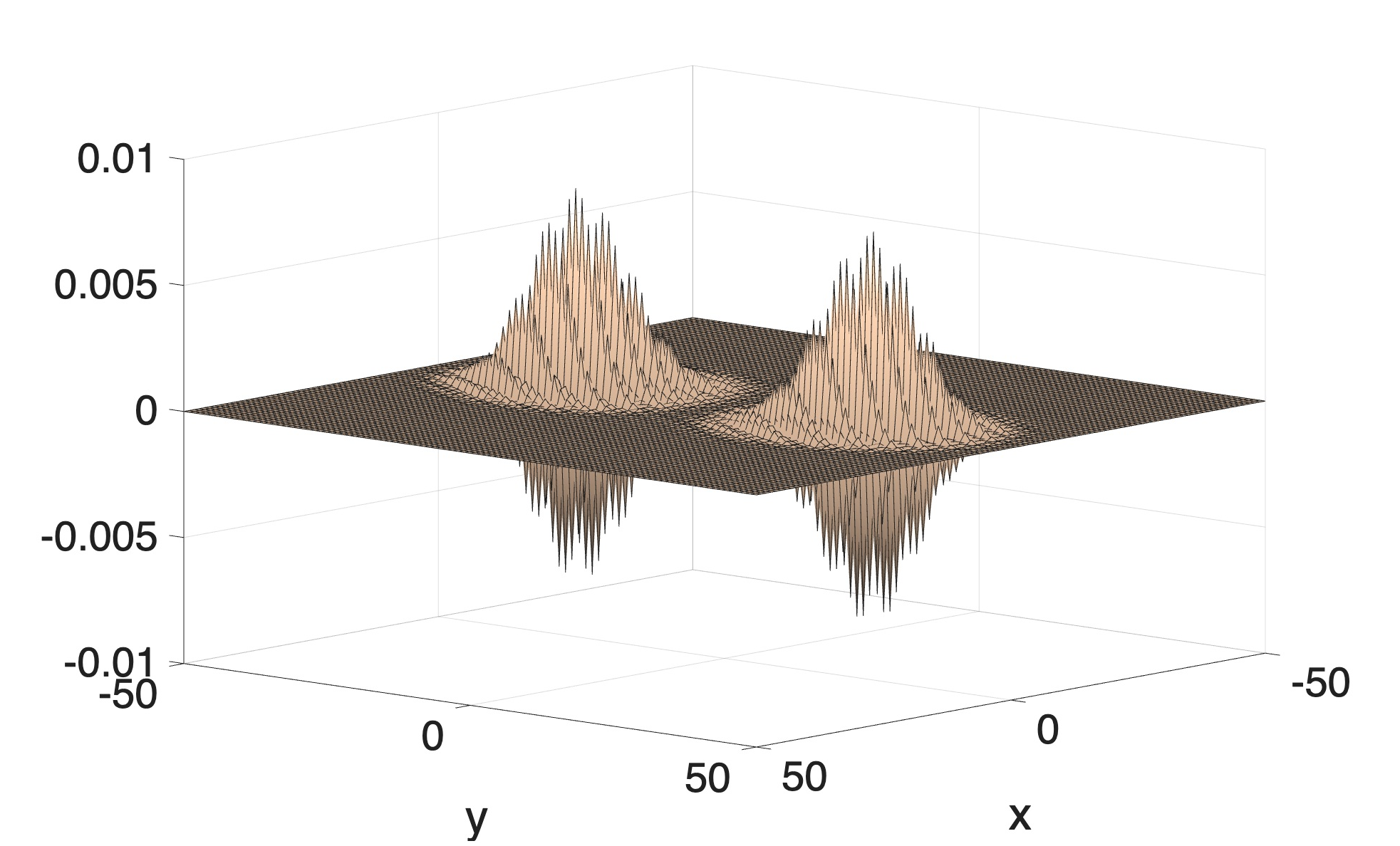}
    \end{minipage}\\ \hline
    \begin{turn}{270}\hspace{-0.5cm}Error\end{turn} & 
    \begin{minipage}{.4\textwidth}
      \includegraphics[width=1\linewidth]{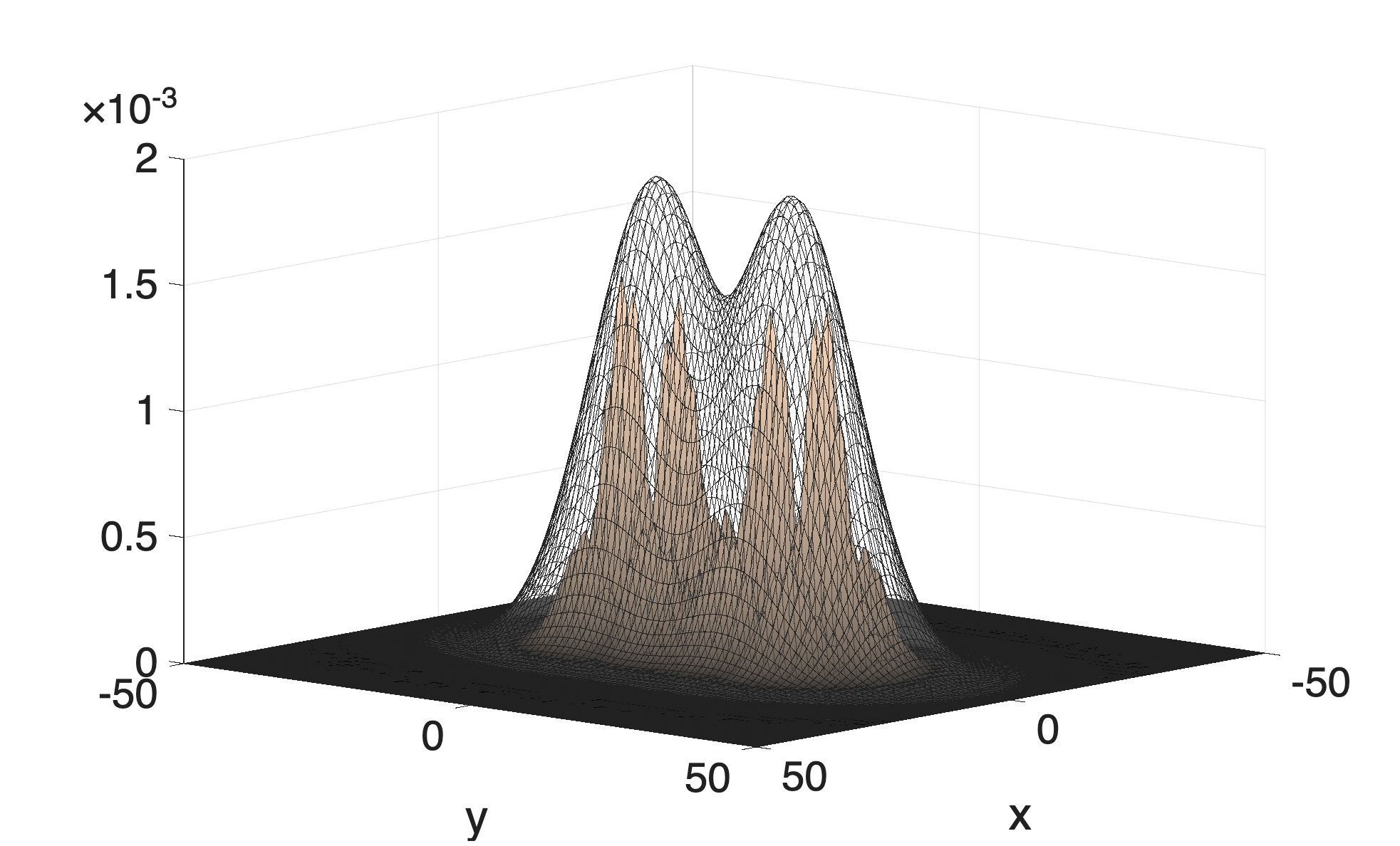}
    \end{minipage}
	&
      \begin{minipage}{.4\textwidth}
      \includegraphics[width=1\linewidth]{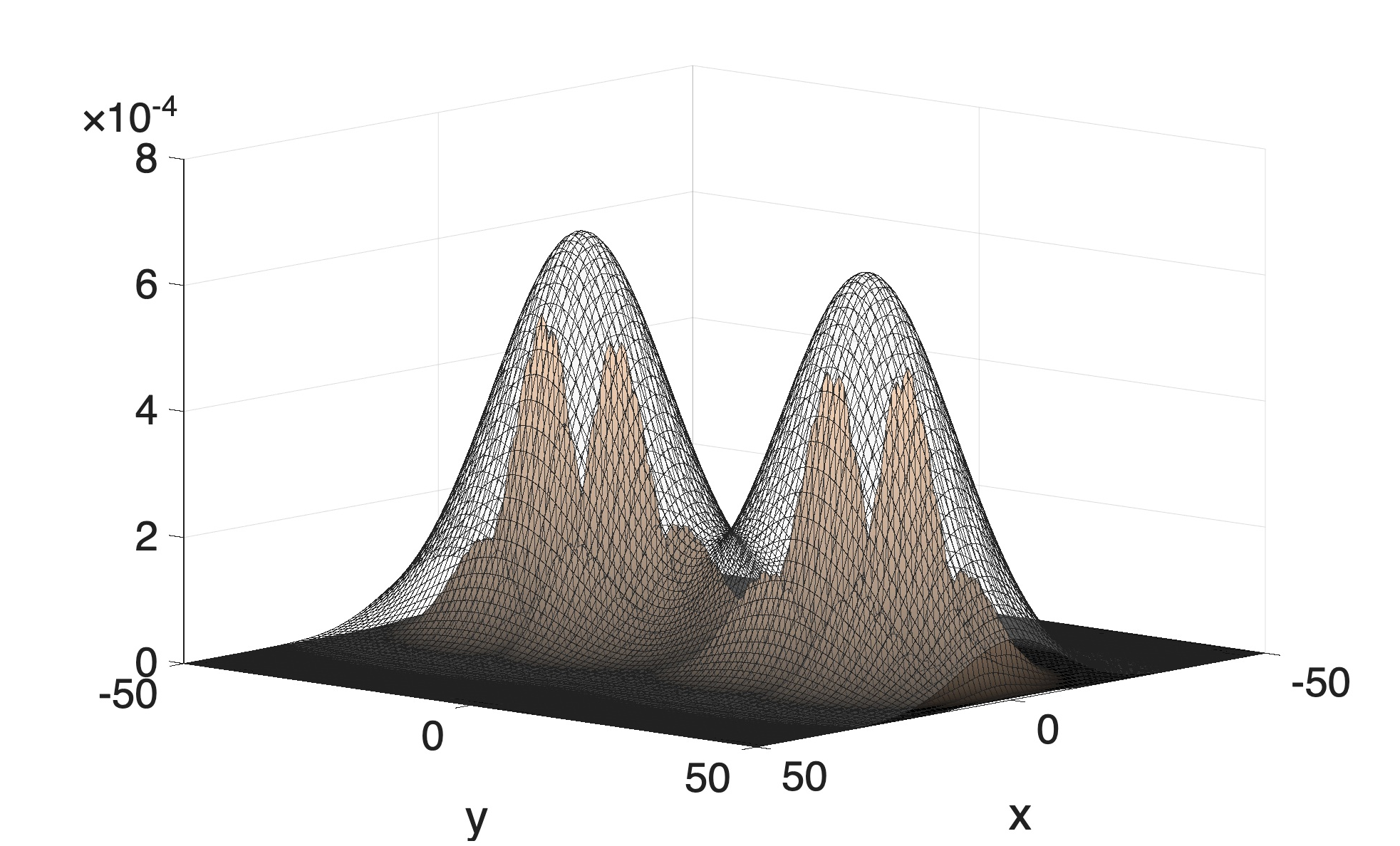}
    \end{minipage}\\ \hline
  \end{tabular}
  \captionof{figure}{For $n=30$ and $n=60$, the graphs of $\Re(\phi^{(n)})$ appear in the first row. In the second row, the absolute errors $\abs{\mathcal{R}_{0}^{n}}$ are illustrated by the solid light brown surfaces and the Gaussian-type error in \eqref{eq:TwoPacketsLLT1} are illustrated by the transparent ``nets" above.}\label{fig:TwoPackets} 
\end{table}

\begin{figure}[h!]
    \centering
    \includegraphics[width=0.75\linewidth, trim = {0 6cm 0 6cm}, clip]{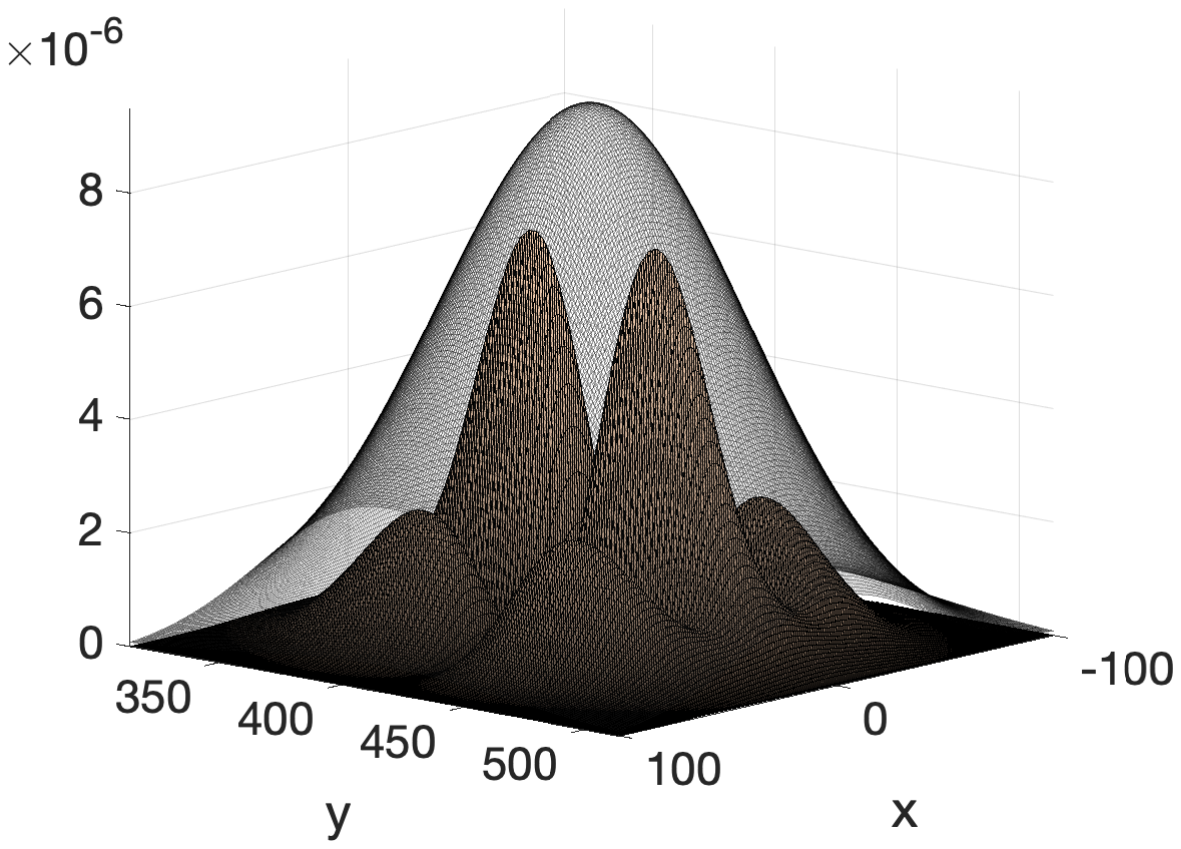}
    \caption{The absolute error $\abs{\mathcal{R}_{0}^n}$ and the Gaussian-type estimate in \eqref{eq:TwoPacketsLLT1} are shown near the drifted peak $(0,n\rho)$ for $n=1,000$, $C=0.3$ and $M = 0.3$.}
    \label{fig:TwoPackets_R0_Loc}
\end{figure}

\begin{figure}[h!]
    \centering
    \includegraphics[width=0.6\linewidth]{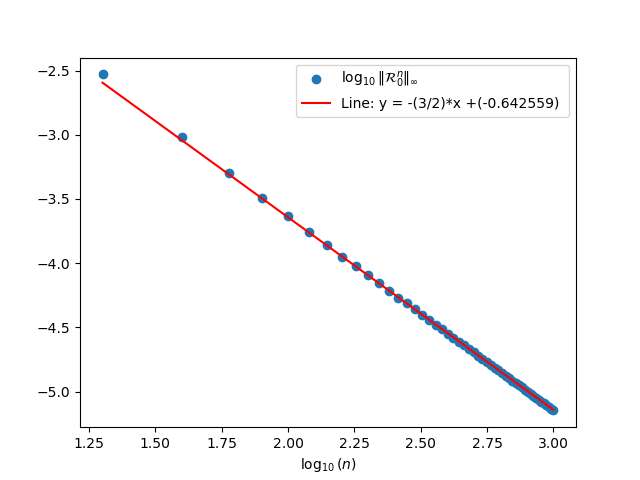}
    \caption{A graph of the values of $\Vert \mathcal{R}^{n}_{0}\Vert_{\infty}$ (blue points) as a function of $n$ in a $\log_{10}-\log_{10}$ plot, for values of $n$ ranging from $20$ to $1,000$ in increments of $20$, along with the best linear fit of slope $-3/2$ (line in red).}
    \label{fig:TwoPackets_R0_LogLog}
\end{figure}

With our goal of deducing a local limit theorem from Theorem \ref{thm:LLT} for $\lambda=\lambda_1=\cdots\lambda_4=1$, let's write down the first-order corrections. Thanks to \eqref{eq:TwoDrift2} and \eqref{eq:TwoDrift3},
\begin{equation*}
    Q_{1,1}=Q_{1,2}=B_1(S_{1,1}(i\partial))=S_{1,1}(i\partial)=\frac{i}{2}(1-\rho)\partial_x^2\partial_y+\frac{2}{3}(1-2\rho)\partial_y^3
\end{equation*}
and
\begin{equation*}
    Q_{1,3}=Q_{1,4}=-Q_{1,1}=-\frac{i}{2}(1-\rho)\partial_x^2\partial_y-\frac{2}{3}(1-2\rho)\partial_y^3.
\end{equation*}
As the base attractors are identical, we compute
\begin{equation*}
\partial_{x}^2\partial_{y} H_{P}(x,y) = \left( \frac{y}{\rho(1+i\rho)} -\frac{2x^2 y}{\rho(1+i\rho)^2} \right) H_{P}(x,y)
\quad \text{and} \quad
\partial_{y}^3H_{P}(x,y) = \left( \frac{3 y}{4\rho^2} - \frac{y^3}{8\rho^3} \right)H_{P}(x,y)
\end{equation*}
so that
\begin{equation*}
    (Q_{1,1}H_P)(x,y)=\frac{y}{4}\left((3+(2+\rho)i)-(3+\rho)(1+i)x^2-\frac{2+\rho}{3}y^2\right)H_P(x,y)
\end{equation*}
for $(x,y)\in\mathbb{R}^2$. Thus, in view of \eqref{eq:LLTExplicitTime}, we have
\begin{eqnarray}\label{eq:TwoDrift4}\nonumber
   (Q_{1,1}^n H_P^n)(x,y-n\rho)&=&(Q_{1,2}^n H_P^n)(x,y-n\rho)\\\nonumber
   &=&\frac{1}{n^{3/2}}(Q_{1,1}H_P)\left(\frac{x}{\sqrt{n}},\frac{y-n\rho}{\sqrt{n}}\right)\\\nonumber
    &=&\frac{1}{4n^{3/2}}\left(\frac{y-n\rho}{\sqrt{n}}\right)\left[3+(2+\rho)i-(3+\rho)(1+i)\frac{x^2}{n}\right.\\
    &&\hspace{4.6cm}\left.-\left(\frac{2+\rho}{3}\right)\frac{(y-n\rho)^2}{n}\right]H_P\left(\frac{x}{\sqrt{n}},\frac{y-n\rho}{\sqrt{n}}\right)
\end{eqnarray}
and, similarly,\begin{eqnarray}\label{eq:TwoDrift5}\nonumber
   (Q_{1,3}^n H_P^n)(x,y+n\rho)&=&(Q_{1,4}^n H_P^n)(x,y+n\rho)\\\nonumber
   &=&\frac{1}{n^{3/2}}(Q_{1,3}H_P)\left(\frac{x}{\sqrt{n}},\frac{y+n\rho}{\sqrt{n}}\right)\\\nonumber
    &=&\frac{-1}{4n^{3/2}}\left(\frac{y+n\rho}{\sqrt{n}}\right)\left[3+(2+\rho)i-(3+\rho)(1+i)\frac{x^2}{n}\right.\\
    &&\hspace{4.6cm}\left.-\left(\frac{2+\rho}{3}\right)\frac{(y+n\rho)^2}{n}\right]H_P\left(\frac{x}{\sqrt{n}},\frac{y+n\rho}{\sqrt{n}}\right)
\end{eqnarray}
for $n\in\mathbb{N}_+$ and $(x,y)\in\mathbb{R}^2$. Using Notation \ref{not:RealError}, we have
\begin{eqnarray*}
    \mathcal{R}_1^n(x,y)&=&\phi^{(n)}(x,y)-\sum_{k=1}^4\sum_{\lambda=0}^1 e^{-i(x,y)\cdot\xi_k}\widehat{\phi}(\xi_k)^n (Q_{\lambda,k}^n H_P^n)((x,y)-n\alpha_k)\\
    &=&\phi^{(n)}(x,y)-(i)^{5n/4}\sum_{\lambda=0}^1\left[\left(((-1)^y+(-1)^n)(i)^{-x+y/2}\right)(Q_{\lambda,1}^n H_P^n)(x,y-n\rho)\right.\\
    &&\left.\hspace{6cm}+\left((-1)^{y+n}+1)(i)^{x-y/2}\right)(Q_{\lambda,3}^n H_P^n)(x,y+n\rho)\right]\\
    &=&\phi^{(n)}(x,y)-(i)^{5n/4}\left(((-1)^y+(-1)^n)(i)^{-x+y/2}\right)\left((Q_{1,1}^n H_P^n)(x,y-n\rho)+H_P^n(x,y-n\rho)\right)\\
    &&\hspace{2.5cm}-(i)^{5n/4}\left(((-1)^{y+n}+1)(i)^{x-y/2}\right)\left((Q_{1,3}^n H_P^n)(x,y+n\rho)+H_P^n(x,y+n\rho)\right)
\end{eqnarray*}
for $n\in\mathbb{N}_+$ and $(x,y)\in\mathbb{Z}^2$ where $H_P^n(x,y\pm n\rho)$, $(Q_{1,1}^n H_P^n)(x,y-n\rho)$, and $(Q_{1,3}^n H_P^n)(x,y+n\rho)$ are given, respectively, by \eqref{eq:TwoDrift2}, \eqref{eq:TwoDrift4}, and \eqref{eq:TwoDrift5}. An appeal to Theorem \eqref{thm:LLT} gives positive constants $C$ and $M$ for which the estimate
\begin{equation}\label{eq:TwoPacketsLLT-wCum}
\abs{\mathcal{R}^{n}_{1}(x,y)}\leq \frac{C}{n^{2}}e^{-Mx^2/n}\left[\exp\left(-M\frac{(y-n\rho)^2}{ n}\right)+\exp\left(-M\frac{(y+n\rho)^2}{ n}\right)\right]
\end{equation}
is satisfied for all $n\in\mathbb{N}_+$ and $ (x,y)\in\mathbb{Z}^2$, as shown in Figure \ref{fig:ex1_LLT_lambda_1}. This yields, in particular, the $\ell^\infty$ estimate
\begin{equation*}
\| \mathcal{R}^{n}_{1} \|_{\infty}\leq \frac{C}{n^2},
\end{equation*}
for $n\in\mathbb{N}_+$ which is supported by the results of numerical simulations shown in Figure \ref{fig:TwoPackets_R1_LogLog}.

\begin{figure}[h!]
    \centering
    \includegraphics[width=0.75\linewidth, trim = {0 6cm 0 6cm}, clip]{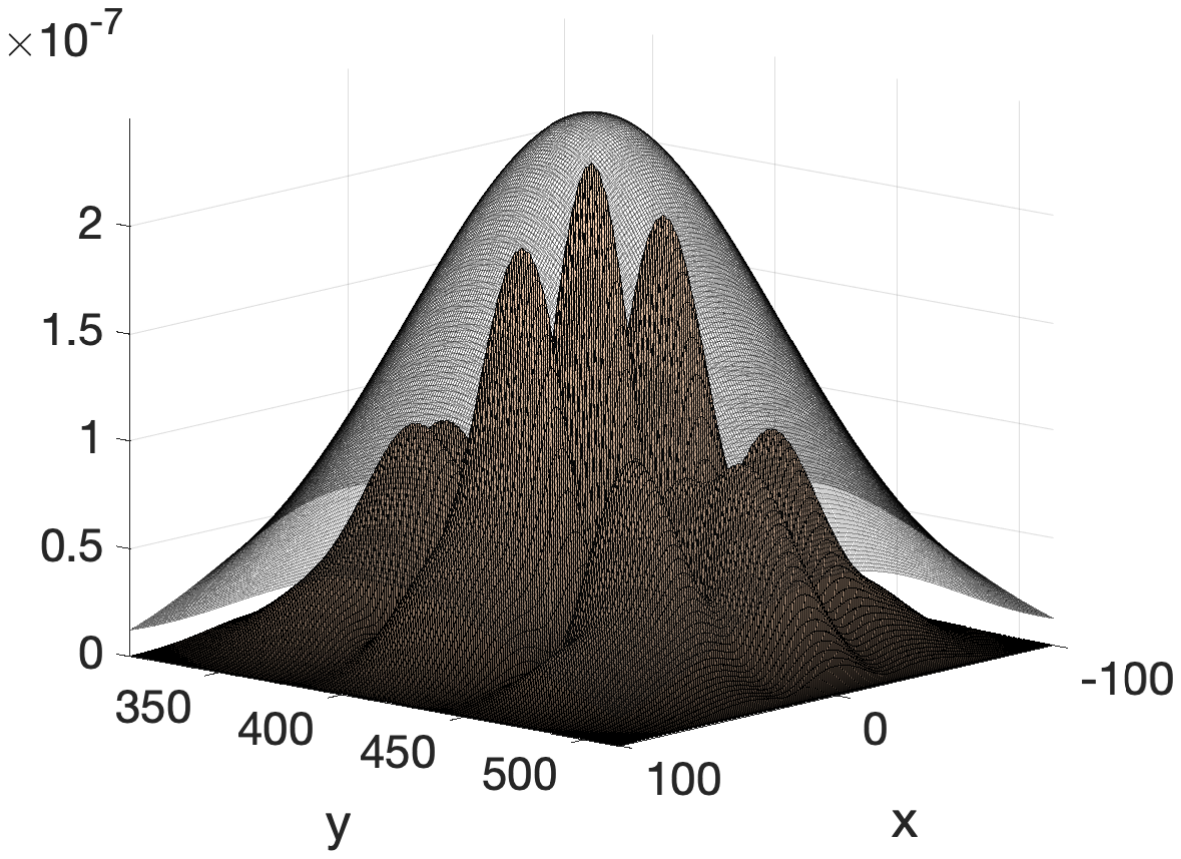}
    \caption{The error $\abs{\mathcal{R}^n_1}$ and the local limit/ theorem bound from Equation \eqref{eq:TwoPacketsLLT-wCum}, with $n=1000$, $C = 0.25$, and $M = 0.15$.}
    \label{fig:ex1_LLT_lambda_1}
\end{figure}

\begin{figure}[h!]
    \centering
    \includegraphics[width=0.6\linewidth]{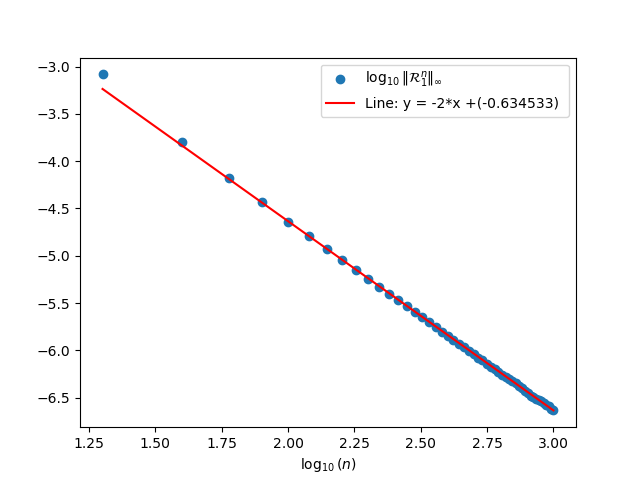}
    \caption{A graph showing the values of $\log_{10}\Vert \mathcal{R}^{n}_{1}\Vert_{\infty}$ (points in blue) as a function of $\log_{10}(n)$, for values of $n$ ranging from $20$ to $1,000$ in increments of $20$, and the best linear fit of slope $-2$ (red line).} 
    \label{fig:TwoPackets_R1_LogLog}
\end{figure}

\subsection{A supporting lattice misaligned with $\mathbb{Z}^2$.}\label{sec:Misaligned}
In this subsection, we revisit Example 7.3 of \cite{RSC17}. Consider $\phi:\mathbb{Z}^2\to\mathbb{C}$ defined by
\begin{equation*}
    \phi(x,y)=\begin{cases}
        3/8 & (x,y)=(0,0)\\
        1/8 & (x,y)=\pm (1,1)\\
        1/4 & (x,y)=\pm (1,-1)\\
        -1/16 & (x,y)=\pm (2,-2)\\
        0 &\,\mbox{otherwise}
    \end{cases}.
\end{equation*}
As shown in \cite{RSC17}, this finitely-supported function is normalized so that $\phi\in\mathcal{H}_2^*$ with
$\Omega(\phi)=\{\xi_1,\xi_2\}=\{(0,0),(\pi,\pi)\}$. For $k=1,2$, we will see that $\xi_k$ is of positive-homogeneous type for $\widehat{\phi}$ with associated positive homogeneous polynomial $P_k=P$ given\footnote{In \cite[Page 1093]{RSC17}, the second coefficient of $P$ was mistakenly given as $23/384$. To the second author's memory, it isn't clear if this error was more than typographical but it doesn't seem to have markedly affected the analysis. Given the high level of accuracy in the present article, getting this coefficient right is critical to our analysis here.} by
\begin{equation*}
    P(\xi)=P(\eta,\zeta)=\frac{1}{8}(\eta+\zeta)^2+\frac{1}{16}(\eta-\zeta)^4
\end{equation*}
for $\xi=(\eta,\zeta)\in\mathbb{R}^2$ and $\mu_k=\mu=3/4$. In particular, this example is ripe for the applications of Theorems \ref{thm:FiniteSupport} and \ref{thm:LLT} and, in contrast to the other examples in this article, this presents a canonical case where the leading polynomials $P=P_1=P_2$ are positive-homogeneous but not semi-elliptic. Correspondingly and as we will see, $A=A_1=A_2\in\operatorname{Gl}_2(\mathbb{R})$ (as guaranteed by Proposition \ref{prop:PosHomAreSemiElliptic} and appearing in \eqref{eq:QLambdaOpertor}) is non-trivial. As discussed in \cite{RSC17}, $\phi$ does meet the hypotheses of \cite[Theorem 1.8]{RSC17} which gives the Gaussian-type estimate
\begin{equation*}
    \abs{\phi^{(n)}(x,y)}\leq\frac{C}{n^{3/4}}\exp\left(-nMR^{\#}\left(\frac{x}{n},\frac{y}{n}\right)\right)
\end{equation*}
for $n\in\mathbb{N}_+$ and $(x,y)\in\mathbb{Z}^2$ where $C$ and $M$ are positive constants and 
\begin{equation*}
    R^{\#}(x,y)\asymp \abs{x+y}^2+\abs{x-y}^{4/3}
\end{equation*}
for $(x,y)\in\mathbb{R}^2$. As Theorem \ref{thm:FiniteSupport} recaptures this estimate exactly, we won't pursue this known result here. Instead, we will focus our attention on local limits. In particular, we will apply Corollary \ref{cor:LLTCor} and Theorem \ref{thm:LLT}, both of which improve upon the local limit theorem presented in Section 7.3 of \cite{RSC17}. We have 
\begin{equation*}
    \widehat{\phi}(\xi)=\frac{3}{8} +\frac{1}{4}\cos(\eta+\zeta) + \frac{1}{2}\cos(\eta-\zeta)-\frac{1}{8}\cos(2(\eta-\zeta))
\end{equation*}
for $\xi=(\eta,\zeta)\in\mathbb{R}^2$. Correspondingly, we see that $\Omega(\phi)=\{(0,0),(\pi,\pi)\}=\{\xi_1,\xi_2\}$ where $\widehat{\phi}(\xi_1)=\widehat{\phi}(\xi_2)=1$. In fact, $\widehat{\phi}(\xi)=\widehat{\phi}(\xi+\xi_1)=\widehat{\phi}(\xi+\xi_2)$, and from this it follows that $\Gamma_1=\Gamma_2=:\Gamma$. Consequently, the ingredients/attractors for local limit theorems are the same for $k=1$ and $2$ and so, when the distinction is unnecessary, we suppress the $k$ subscripts. The Maclaurin expansion of $\Gamma$ to (isotropic) order $6$ is
\begin{eqnarray*}
    \Gamma(\xi)&=&-\frac{1}{8}\eta^2-\frac{1}{4}\eta\zeta-\frac{1}{8}\zeta^2-\frac{23}{384}\eta^4+\frac{25}{96}\eta^3\zeta-\frac{23}{64}\eta^2\zeta^2+\frac{25}{96}\eta\zeta^3-\frac{23}{384}\zeta^4\\
    &&+\frac{67}{23040}\eta^6-\frac{173}{3840}\eta^5\zeta+\frac{259}{1536}\eta^4\zeta^2-\frac{269}{1152}\eta^3\zeta^3+\frac{259}{1536}\eta^2\zeta^4-\frac{173}{3840}\eta\zeta^5+\frac{67}{23040}\zeta^6+o(\abs{\xi}^6)\\
    &=&-P(\eta,\zeta)+\frac{1}{384}\eta^4+\frac{1}{96}\eta^3\zeta+\frac{1}{64}\eta^2\zeta^2+\frac{1}{96}\eta\zeta^3+\frac{1}{384}\eta^4\\
    &&+\frac{67}{23040}\eta^6-\frac{173}{3840}\eta^5\zeta+\frac{259}{1536}\eta^4\zeta^2-\frac{269}{1152}\eta^3\zeta^3+\frac{259}{1536}\eta^2\zeta^4-\frac{173}{3840}\eta\zeta^5+\frac{67}{23040}\zeta^6+o(\abs{\xi}^6)\\
\end{eqnarray*}
as $\xi=(\eta,\zeta)\to (0,0)$. In looking at the first line for $\Gamma$ above, while it is clear that the drift $\alpha$ is zero, it is far from clear which leading terms should be aggregated to form $P$ nor why what remains (especially the forth-order polynomial appearing in the third line) is $o(P(\xi))$ as $\xi\to (0,0)$. We remark that Proposition 3.3 of \cite{RSC17} outlines a strategy to identify leading polynomials in such non-semi-elliptic cases. Still, as it is perhaps obvious to the reader, this is a difficult task. If we, instead, make the change of variables by the matrix
\begin{equation*}
    A=\begin{pmatrix}
        1/\sqrt{2} & -1/\sqrt{2}\\
        1/\sqrt{2} & 1/\sqrt{2}
    \end{pmatrix},
\end{equation*}
things become transparent. Writing $\Gamma_{A}(\xi)=\Gamma(A\xi)$, we find 
\begin{equation*}
    \Gamma_A(\xi)=-\left(\frac{1}{4}\eta^2+\frac{1}{4}\zeta^4\right)+\Upsilon_A(\xi)=-P_A(\xi)+\Upsilon_A(\xi)
\end{equation*}
for $\xi=(\eta,\zeta)\in\mathbb{C}^2$ where
\begin{eqnarray*}
    \Upsilon_A(\xi)&=&\frac{1}{96}\eta^4+\frac{7}{2880}\eta^6+\frac{1}{12}\zeta^6+\frac{53}{322560}\eta^8-\frac{1}{16}\eta^2\zeta^4-\frac{1}{192}\eta^4\zeta^4+\frac{1}{48}\eta^2\zeta^6-\frac{7}{160}\zeta^8+o(\abs{\xi}^8)\\
    &=&\underbrace{0 \vphantom{\frac{1}{2}}}_{S_1(\xi)/1!}+\underbrace{\frac{1}{12}\zeta^6}_{S_2(\xi)/2!}+\underbrace{0 \vphantom{\frac{1}{2}}}_{S_3(\xi)/3!}+\underbrace{\frac{1}{96}\eta^4-\frac{1}{16}\eta^2\zeta^4-\frac{7}{160}\zeta^8}_{S_4(\xi)/4!}+\underbrace{0\vphantom{\frac{1}{2}}}_{S_5(\xi)/5!}+\underbrace{\frac{1}{48}\eta^2\zeta^6+\frac{83}{3780}\zeta^{10}}_{S_6(\xi)/6!}+\cdots.
\end{eqnarray*}

Consistent with the introduction, we have aggregated terms according to $\Lambda(\beta)=\beta\cdot\kappa-2m=(2,1)\cdot\beta-4=\lambda$ for $\lambda=1,2,\dots,6$ where $\mathbf{m}=(1,2)$, $m=\lcm(\mathbf{m})=2$, and $\kappa=(2,1)$. From this, it is evident that $\xi_1$ and $\xi_2$ are of positive-homogeneous type for $\widehat{\phi}$ as the terms of $\Upsilon_A$ all have multi-indices $\beta$ with $\abs{\beta:2\mathbf{m}}\geq 3/2>1$ thanks to \cite[Lemma 8.9]{RSC17}. We observe that, in this anisotropic setting, terms with lower ``isotropic order" often appear in much higher anisotropic homogeneous order terms. For example, the term $(7/2880)\eta^6$ in the above expansion appears in $S_8$ since $\Lambda(6,0)=12-4=8$ but, of course, it is absent from $S_\lambda$ for $1\leq \lambda\leq 7$. For $\lambda=1,2,\dots,7$, Table \ref{tab:Misaligned} lists the polynomials $S_\lambda$. We note that, in looking to Corollary \ref{cor:LLTCor}, $\gamma=2$ since $S_1=0$ and $S_2\neq 0$. 

\begin{center}
\begin{tabular}{|l|l|}
\hline
$\lambda$ & $S_\lambda(\eta,\zeta)$ \\
\hline
\hline
\rule{0pt}{3ex} 1   & $0$ \\
\hline
\rule{0pt}{3ex} 2   & $\frac{1}{6}\zeta^6$ \\[1ex]
\hline
\rule{0pt}{3ex} 3   & $0$ \\
\hline
\rule{0pt}{3ex} 4   & $\frac{1}{4}\eta^4-\frac{3}{2}\eta^2\zeta^4-\frac{21}{20}\zeta^8$   \\[1ex]
\hline
\rule{0pt}{3ex} 5   & $0$ \\
\hline\rule{0pt}{3ex} 6   & $15\eta^2\zeta^6+\frac{332}{21} \zeta^{10}$ \\
\hline
\rule{0pt}{3ex} 7& $0$ \\
\hline
\hline
\end{tabular}
\captionof{table}{The polynomials $S_{\lambda}$ for $\lambda=0,1,2,\dots,7$.}\label{tab:Misaligned}
\end{center}

With \eqref{eq:QLambdaOpertor}, a careful study of Tables \ref{tab:Bell} and \ref{tab:Misaligned} shows that $Q_1^n=Q_3^n=Q_5^n=Q_7^n\equiv 0$,
\begin{equation*}
    Q_{2}^n=\frac{1}{2!}\left(0^2+nS_2(A^{-1}i\partial)\right)=\frac{n}{12}\left(\frac{-i}{\sqrt{2}}\partial_1+\frac{i}{\sqrt{2}}\partial_2\right)^6=-\frac{n}{3\cdot 2^5}(\partial_1-\partial_2)^6,
\end{equation*}
\begin{eqnarray*}
    Q_{4}^n&=&\frac{1}{4!}\left(0+0+0+3(nS_2(A^{-1}i\partial))^2+nS_4(A^{-1}i\partial)\right)\\
    &=&\frac{3n^2}{4!(48)^2}(\partial_1-\partial_2)^{12}\\
    &&+\frac{n}{4!}\left(\frac{1}{4(\sqrt{2})^4}(\partial_1+\partial_2)^4+\frac{3}{2(\sqrt{2})^6)}(\partial_1+\partial_2)^2(\partial_1-\partial_2)^4-\frac{21}{20(\sqrt{2})^8}(\partial_1-\partial_2)^8\right)\\
    &=&\frac{n^2}{3^2 \cdot 2^{11}}(\partial_1-\partial_2)^{12}+\frac{n}{3\cdot 2^7}\left((\partial_1+\partial_2)^4+3(\partial_1+\partial_2)^2(\partial_1-\partial_2)^4-\frac{21}{20}(\partial_1-\partial_2)^8\right),
\end{eqnarray*}
and
\begin{eqnarray*}
    Q_6^n&=&\frac{1}{6!}\left(4\cdot 0 +15 (nS_2(A^{-1}i\partial))^3+3\cdot 0+15(nS_2(A^{-1}i\partial))(nS_4(A^{-1}i\partial)+0+nS_6(A^{-1}i\partial)\right)\\
    &=&\frac{n^3}{3^4\cdot 2^{16}}\left(\partial_1-\partial_2\right)^{18}-\frac{ n^2}{ 3^2\cdot 2^{12}}\left((\partial_1+\partial_2)^4(\partial_1-\partial_2)^6+3(\partial_1+\partial_2)^2(\partial_1-\partial_2)^{10}-\frac{21}{20}(\partial_1-\partial_2)^{14}  \right)\\
    &&\hspace{5cm}+\frac{n}{3 \cdot 2^8}\left((\partial_1+\partial_2)^2(\partial_1-\partial_2)^6-\frac{166}{315}(\partial_1-\partial_2)^{10} \right)
\end{eqnarray*}
for $n\in\mathbb{N}_+$. With the above, we have $Q_{1}^n H_P^n=Q_{3}^n H_P^n=Q_{5}^n H_P^n=0$,
\begin{eqnarray*}
    H_P^n(x,y)=(Q_{0}^n H_P^n)(x,y)&=&\frac{1}{(2\pi)^2}\int_{\mathbb{R}^2}e^{-n((\eta+\zeta)^2/8+(\eta-\zeta)^4/16)}e^{-i(x\eta+y\zeta)}\,d\eta\,d\zeta\\
    &=&\frac{1}{n^{3/4}}\sqrt{\frac{2}{\pi}}\exp\left(-\frac{(x+y)^2}{2n}\right)h_4\left(\frac{x-y}{n^{1/4}}\right)\\
    &=&\frac{\sqrt{8}}{n^{3/4}}h_2\left(\frac{2(x+y)}{(2n)^{1/2}}\right)h_4\left(\frac{x-y}{n^{1/4}}\right)
\end{eqnarray*}
\begin{eqnarray*}
    (Q_{2}^nH_P^n)(x,y)&=&-\frac{n}{96}(\partial_1-\partial_2)^6 H_P^n(x,y)\\
    &=&-\frac{1}{n^{5/4}}\frac{2}{3}\sqrt{\frac{2}{\pi}}\exp\left(-\frac{(x+y)^2}{2n}\right)h_4^{(6)}\left(\frac{x-y}{n^{1/4}}\right)\\
    &=&-\frac{4\sqrt{2}}{3n^{5/4}}h_2\left(\frac{2(x+y)}{(2n)^{1/2}}\right)h_4^{(6)}\left(\frac{x-y}{n^{1/4}}\right),
\end{eqnarray*}
\begin{eqnarray}\label{eq:MisQH4}\nonumber
    \lefteqn{(Q_4^n H_P^n)(x,y)}\\\nonumber
    &=&\frac{n^2}{18432}(\partial_1-\partial_2)^{12}H_P^n(x,y)+\frac{n}{384}\left((\partial_1+\partial_2)^4+3(\partial_1+\partial_2)^2(\partial_1-\partial_2)^4-\frac{21}{20}(\partial_1-\partial_2)^8\right)H_P^n(x,y)\\\nonumber
        &=&\frac{\sqrt{2}}{n^{7/4}}\left(\frac{4}{9}h_2\left(\frac{2(x+y)}{(2n)^{1/2}}\right)h_4^{(12)}\left(\frac{x-y}{n^{1/4}}\right)+\frac{1}{3} h_2^{(4)}\left(\frac{2(x+y)}{(2n)^{1/2}}\right) h_4\left(\frac{x-y}{n^{1/4}}\right)\right.\\
        &&\left. \hspace{2cm}+2h_2^{(2)}\left(\frac{2(x+y)}{(2n)^{1/2}}\right)h_4^{(4)}\left(\frac{x-y}{n^{1/4}}\right)-\frac{7}{5} h_2\left(\frac{2(x+y)}{(2n)^{1/2}}\right) h_4^{(8)}\left(\frac{x-y}{n^{1/4}}\right)\right)
\end{eqnarray}
and
\begin{eqnarray}\label{eq:MisQH6}\nonumber
    \lefteqn{(Q_6^n H_P^n)(x,y)}\\\nonumber
    &=&\frac{n^3}{3^4\cdot 2^{16}}\left(\partial_1-\partial_2\right)^{18}H_P^n(x,y)\\\nonumber
    &&-\frac{ n^2}{ 3^2\cdot 2^{12}}\left((\partial_1+\partial_2)^4(\partial_1-\partial_2)^6+3(\partial_1+\partial_2)^2(\partial_1-\partial_2)^{10}-\frac{21}{20}(\partial_1-\partial_2)^{14}  \right)H_P^n(x,y)\\\nonumber
    &&\hspace{4.5cm}+\frac{n}{3 \cdot 2^8}\left((\partial_1+\partial_2)^2(\partial_1-\partial_2)^6-\frac{166}{315}(\partial_1-\partial_2)^{10} \right)H_P^n(x,y)\\\nonumber
    &=&\frac{\sqrt{2}}{3\cdot n^{9/4}}\left[\frac{8}{27}h_2\left(\frac{2(x+y)}{(2n)^{1/2}}\right)h_4^{(18)}\left(\frac{x-y}{n^{1/4}}\right)-\frac{2}{3}h_2^{(4)}\left(\frac{2(x+y)}{(2n)^{1/2}}\right)h_4^{(6)}\left(\frac{x-y}{n^{1/4}}\right)\right.\\\nonumber
    &&\hspace{1.6cm}-4h_2^{(2)}\left(\frac{2(x+y)}{(2n)^{1/2}}\right)h_4^{(10)}\left(\frac{x-y}{n^{1/4}}\right)+\frac{14}{5}h_2\left(\frac{2(x+y)}{(2n)^{1/2}}\right)h_4^{(14)}\left(\frac{x-y}{n^{1/4}}\right)\\
    &&\left.\hspace{1.8cm}+4h_2^{(2)}\left(\frac{2(x+y)}{(2n)^{1/2}}\right)h_4^{(6)}\left(\frac{x-y}{n^{1/4}}\right)-\frac{1328}{315}h_2\left(\frac{2(x+y)}{(2n)^{1/2}}\right)h_4^{(10)}\left(\frac{x-y}{n^{1/4}}\right)\right]
\end{eqnarray}
for $n\in\mathbb{N}_+$ and $(x,y)\in\mathbb{R}^2$; here, we have used the notation $h_2=h_2^1$ and $h_4=h_4^1$ for the single-variable heat kernel and bi-harmonic heat kernel, respectively, as defined in \eqref{eq:OneDHeatKer} of Remark \ref{rmk:OneDHeatKer}; though a slight abuse of notation, the superscript (in) $h^{(q)}$ indicates $q$-th (ordinary) derivative of these single-variable functions. Finally, let's note that, since $R=P$, a routine calculation (making use of Proposition \ref{prop:LFCompare}) shows that
\begin{equation*}
    R^{\#}(x,y)=\frac{(x+y)^2}{2}+3\frac{\abs{x-y}^{4/3}}{4^{4/3}}\asymp \abs{x+y}^2+\abs{x-y}^{4/3}
\end{equation*}
for $(x,y)\in\mathbb{R}^2$.

Let's put the above ingredients together to obtain local limit theorems for various orders of accuracy. First, as was noted in \cite{RSC17}, the phases in the sum of attractors (for $k=1,2$) combine in a single prefactor that describes the support\footnote{As shown in Section 7.6 of \cite{RSC17}, this phenomenon parallels that of the probabilistic setting (with $\phi\geq 0$) where prefactors describe the support/periodicity of associated random walks. This idea is developed further in \cite{RY25} and tied to the algebraic structure of $\Omega(\phi)$.} of $\phi^{(n)}$. We have
\begin{eqnarray*}
    \sum_{k=1}^2 e^{-i(x,y)\cdot\xi_k}\widehat{\phi}(\xi_k)^n H_{P_k}^n(x,y)&=&(1+\cos(\pi(x+y)))H_P^n(x,y)\\
    &=&\frac{\sqrt{8}\left(1+\cos(\pi(x+y))\right)}{n^{3/4}}h_2\left(\frac{2(x+y)}{(2n)^{1/2}}\right)h_4\left(\frac{x-y}{n^{1/4}}\right)
\end{eqnarray*}
for $(x,y)\in\mathbb{Z}^2$ and $n\in\mathbb{N}_+$. Since $\gamma=\gamma_1=\gamma_2=2$ so that $\mu+\gamma/2m=5/4$, an appeal to Corollary \ref{cor:LLTCor} yields the local limit theorem: There are constants $C$ and $M$ for which
\begin{eqnarray}\label{eq:MisR0LLT}\nonumber
\lefteqn{\hspace{-6cm}\abs{\phi^{(n)}(x,y)-\frac{\sqrt{8}\left(1+\cos(\pi(x+y))\right)}{n^{3/4}}h_2\left(\frac{2(x+y)}{(2n)^{1/2}}\right)h_4\left(\frac{x-y}{n^{1/4}}\right)}}\\\nonumber
\hspace{6cm}&=&\abs{\phi^{(n)}(x,y)-(1+\cos(\pi(x+y)))H_P^n(x,y)}\\\nonumber
    &\leq& \frac{C}{n^{5/4}}\exp\left(-n M R^{\#}\left(\frac{x}{n},\frac{y}{n}\right)\right)\\
    &=&\frac{C}{n^{5/4}}\exp\left(-M\left(\frac{\abs{x+y}^2}{2n}+\frac{3}{4}\frac{\abs{x-y}^{4/3}}{(4n)^{1/3}}\right)\right)
\end{eqnarray}
for all $n\in\mathbb{N}_+$ and $(x,y)\in\mathbb{Z}^2$. Figure \ref{fig:MisR0} illustrates this local limit theorem. We see easily that this improves significantly over Example 7.3 of \cite{RSC17} which reports only a local limit theorem uniform error of $o(n^{-3/4})$.

\begin{table}[!h]
  \centering
  \begin{tabular}{  |c| c | c | }
    \hline
     & $n=50$ & $n=500$ \\ \hline
     $\phi^{(n)}$ & 
    \begin{minipage}{.4\textwidth}
      \includegraphics[width=1\linewidth, trim = {0 6cm 0 6cm},clip]{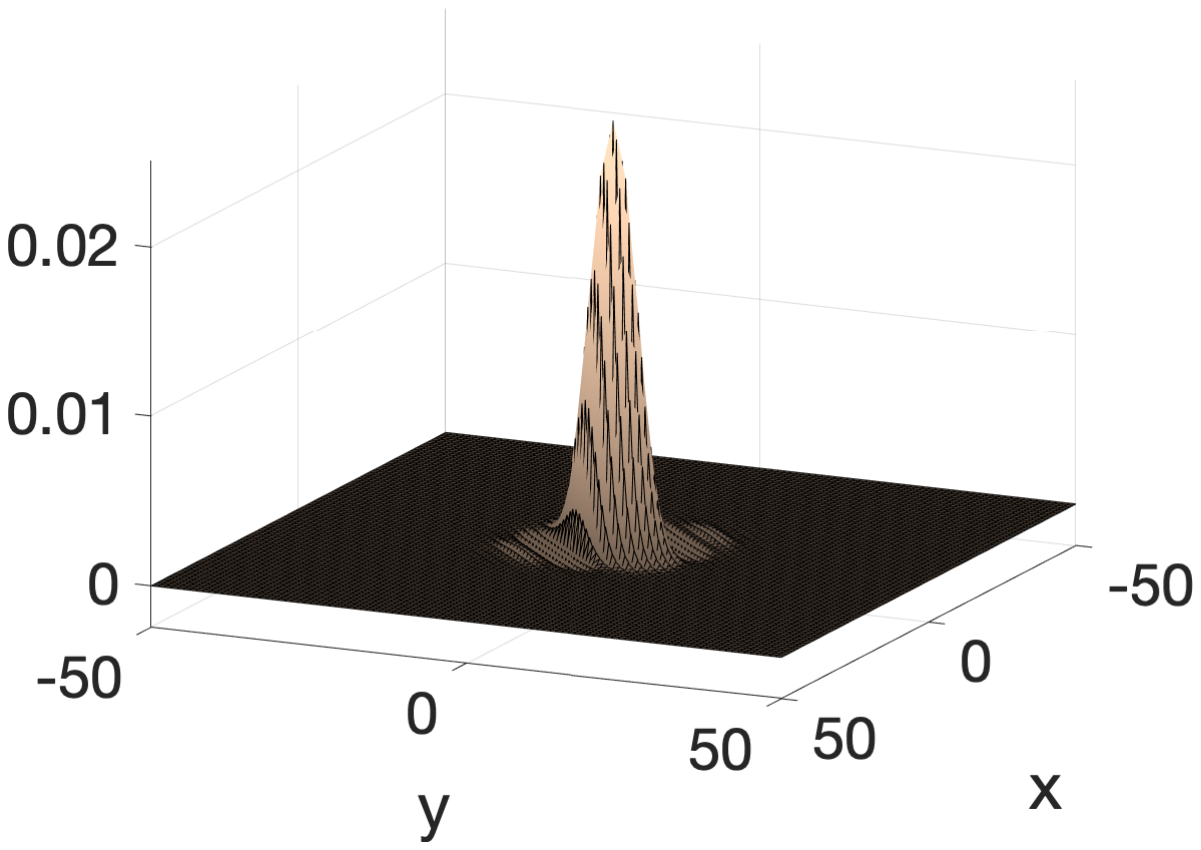}
    \end{minipage}
	&
      \begin{minipage}{.44\textwidth}
      \includegraphics[width=1\linewidth, trim = {0 6cm 0 6cm},clip]{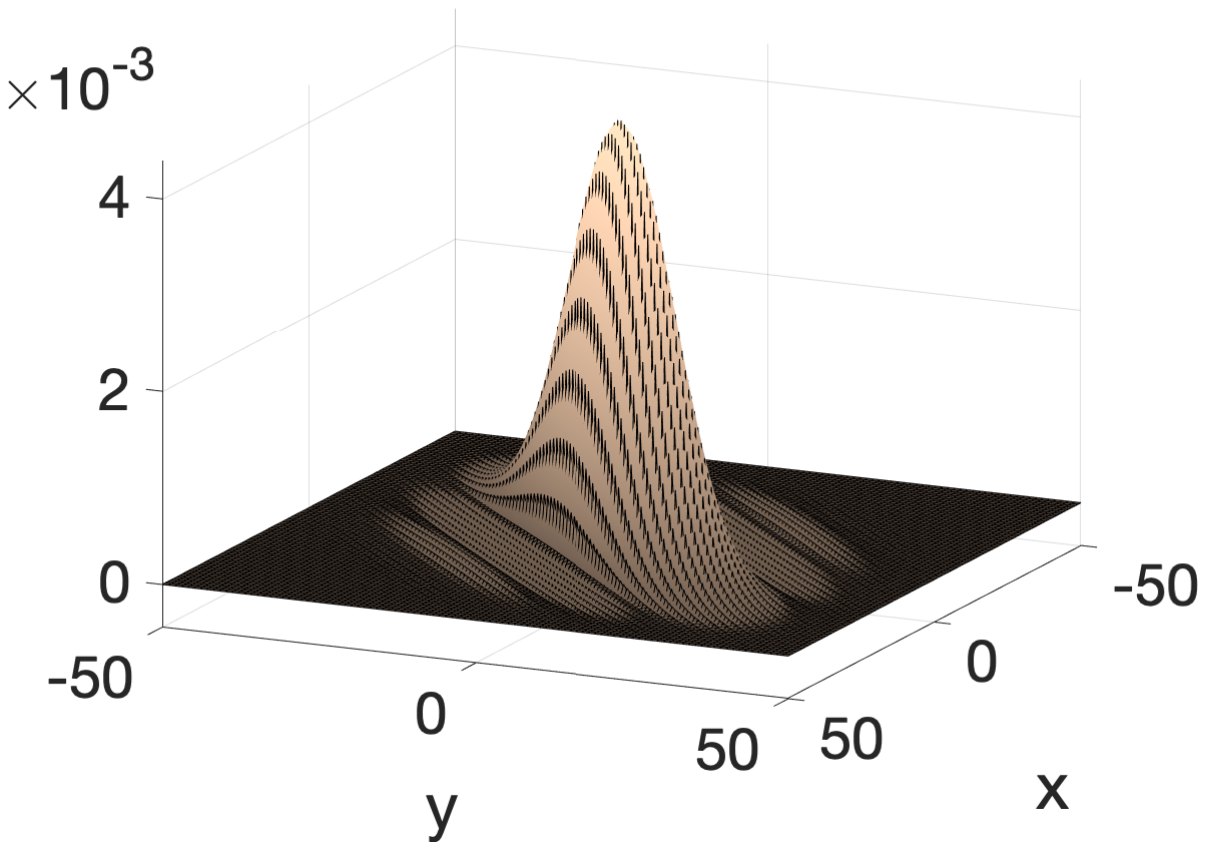}
    \end{minipage}\\ \hline
    $\mathcal{A}_1^n$ &
    \begin{minipage}{.4\textwidth}
     \includegraphics[width=1\linewidth, trim = {0 6cm 0 6cm},clip]{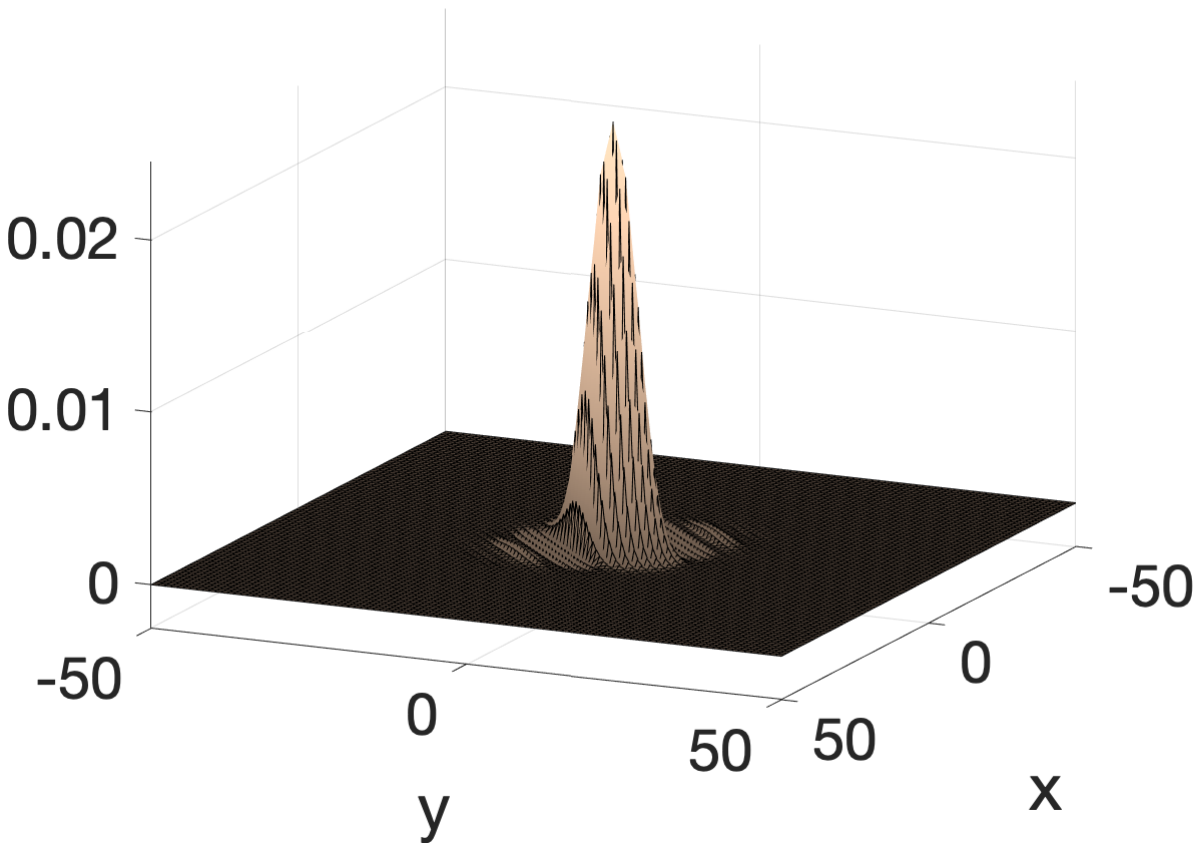}
    \end{minipage}
	&
     \begin{minipage}{.4\textwidth}
     \includegraphics[width=1\linewidth, trim = {0 6cm 0 6cm},clip]{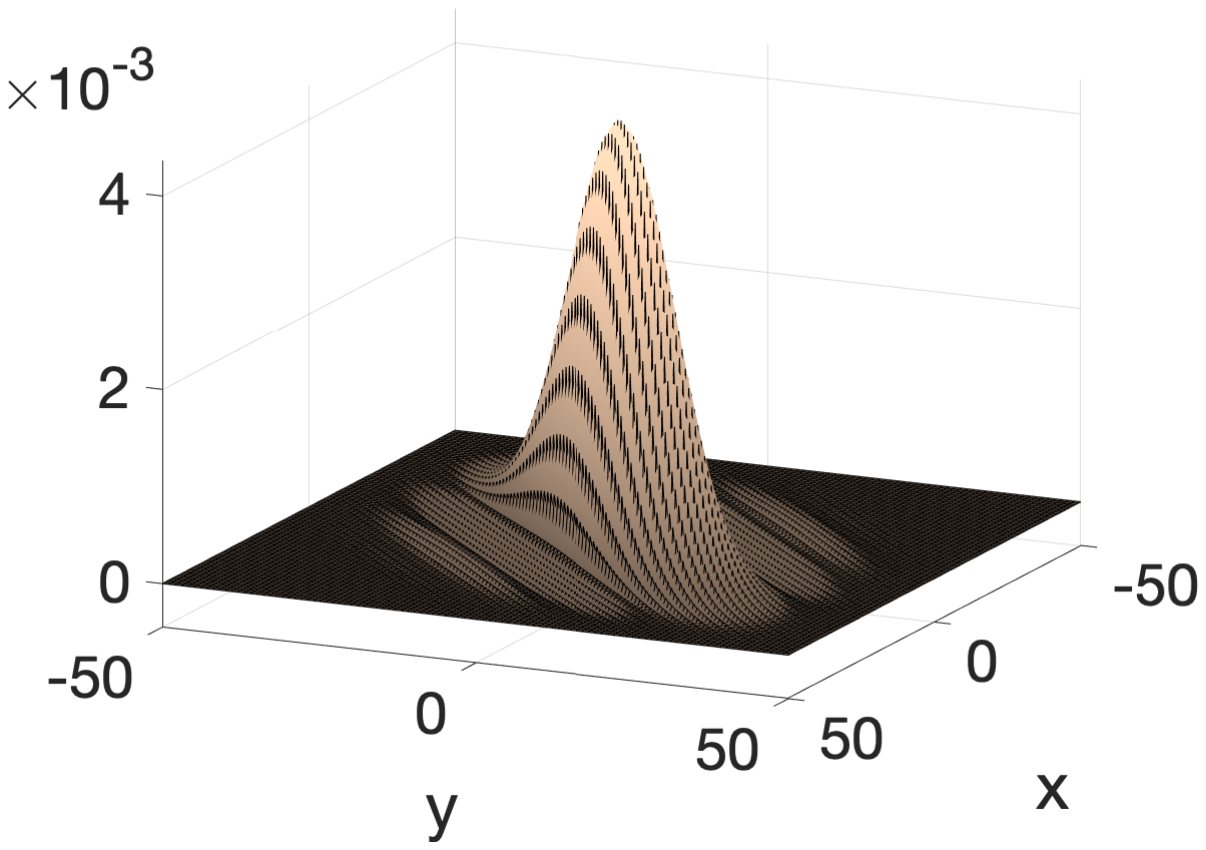}
    \end{minipage}\\ \hline
    \begin{turn}{270}\hspace{-0.5cm}Error\end{turn} & 
    \begin{minipage}{.4\textwidth}
      \includegraphics[width=1\linewidth, trim = {0 6cm 0 6cm},clip]{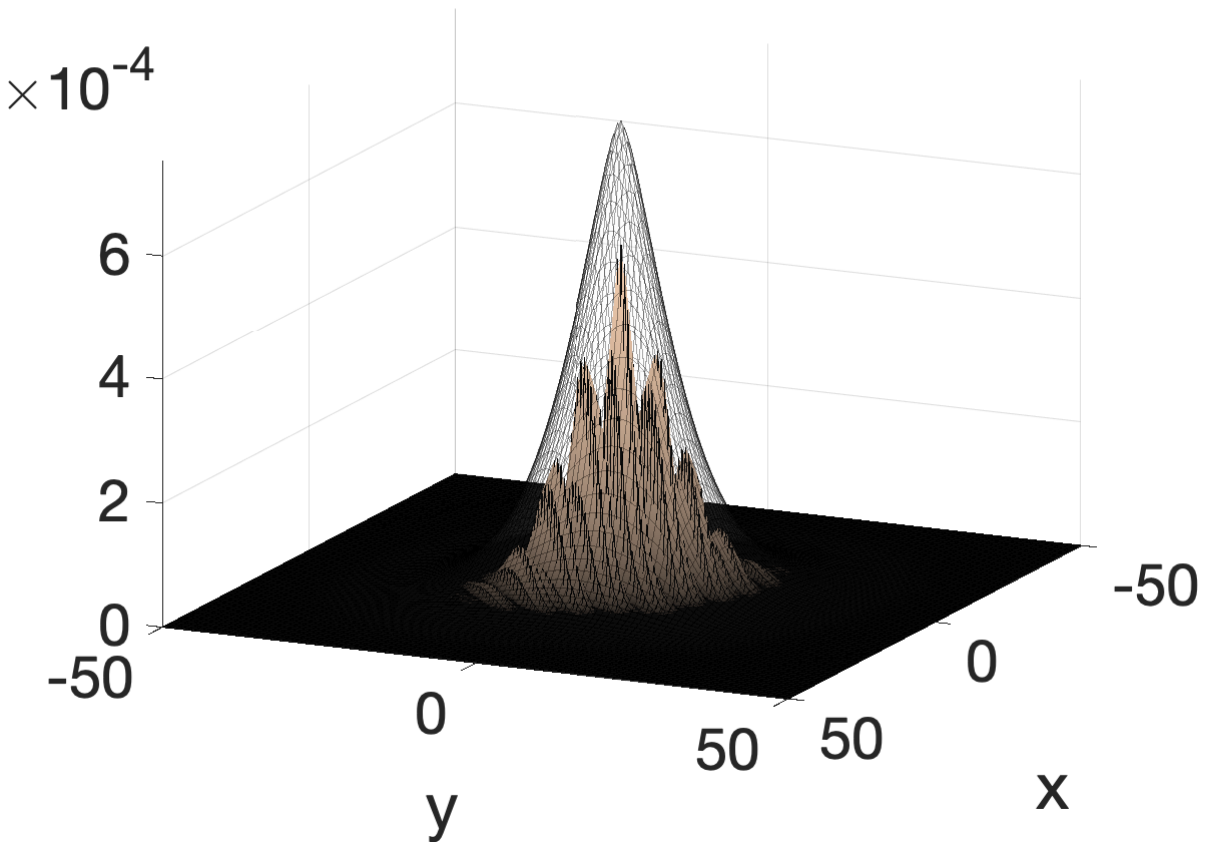}
    \end{minipage}
	&
      \begin{minipage}{.4\textwidth}
      \includegraphics[width=1\linewidth, trim = {0 6cm 0 6cm},clip]{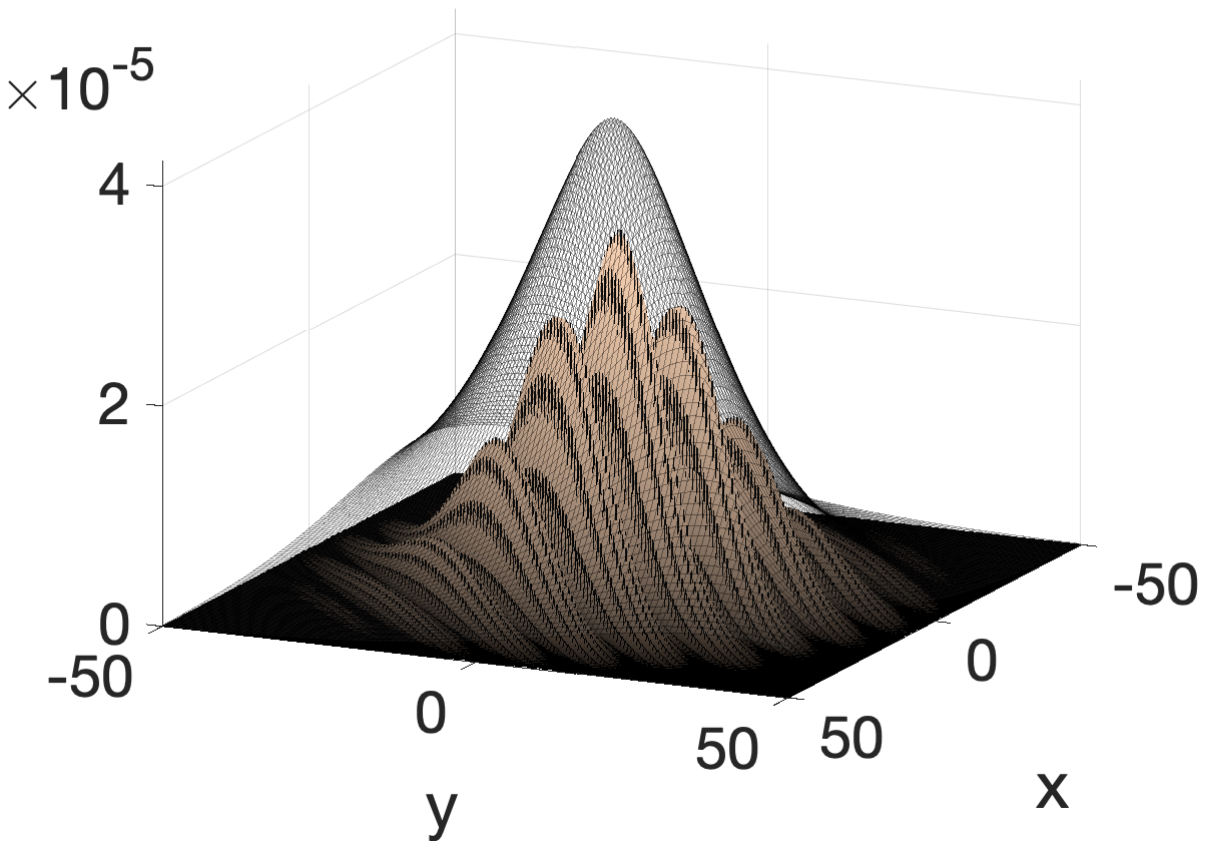}
    \end{minipage}\\ \hline
  \end{tabular}
  \captionof{figure}{The convolution powers $\phi^{(n)}$, the base attractors $\mathcal{A}_1^n=\phi^{(n)}-\mathcal{R}_1^n$, and the error are illustrated for $n=50,500$. In the third row, the real error $\mathcal{R}_1^n=\mathcal{R}_0^n$ is illustrated by the light brown surfaces and the Gaussian-type error \eqref{eq:MisR0LLT} is illustrated by the transparent nets above; here, $C=0.15$.}\label{fig:MisR0} 
\end{table}

\begin{table}[!h]
  \centering
  \begin{tabular}{  |c| c | c | }
    \hline
     $n=50$ & $n=500$ \\ \hline
    \begin{minipage}{.4\textwidth}
      \includegraphics[width=1\linewidth, trim = {0 6cm 0 6cm},clip]{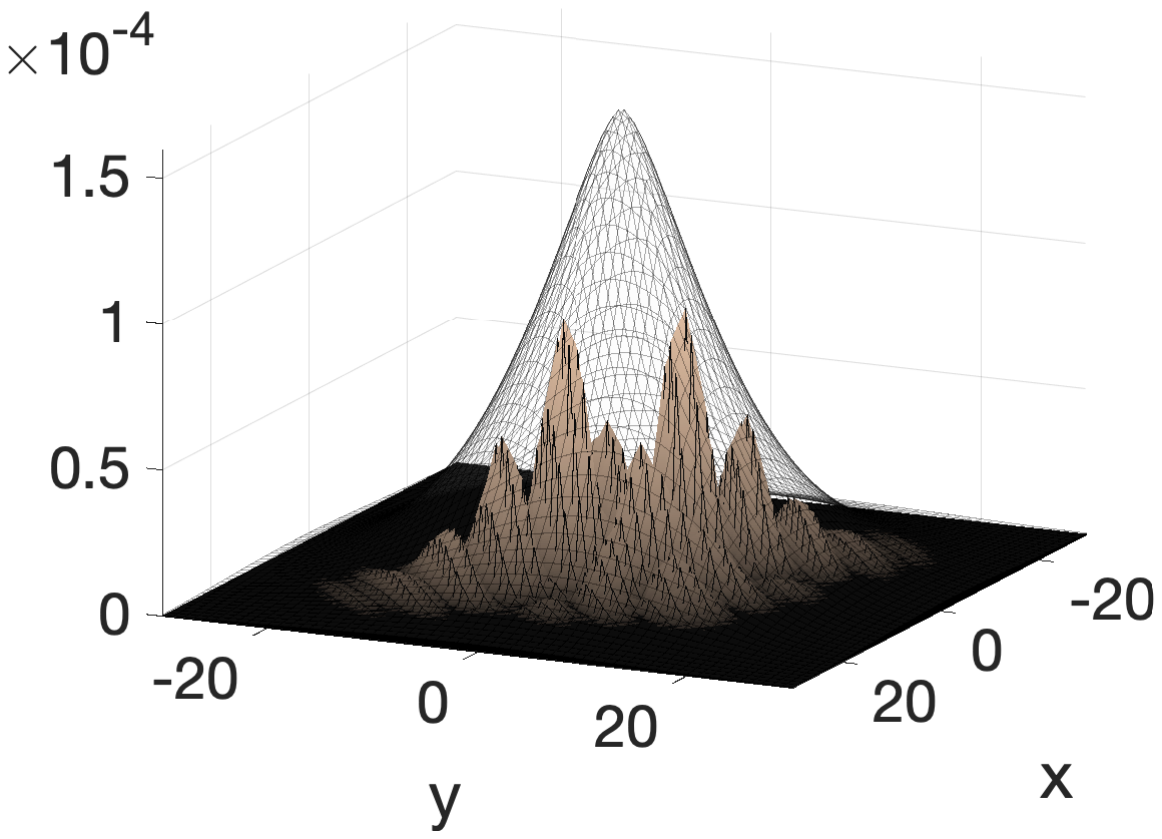}
    \end{minipage}
	&
      \begin{minipage}{.4\textwidth}
      \includegraphics[width=1\linewidth, trim = {0 6cm 0 6cm},clip]{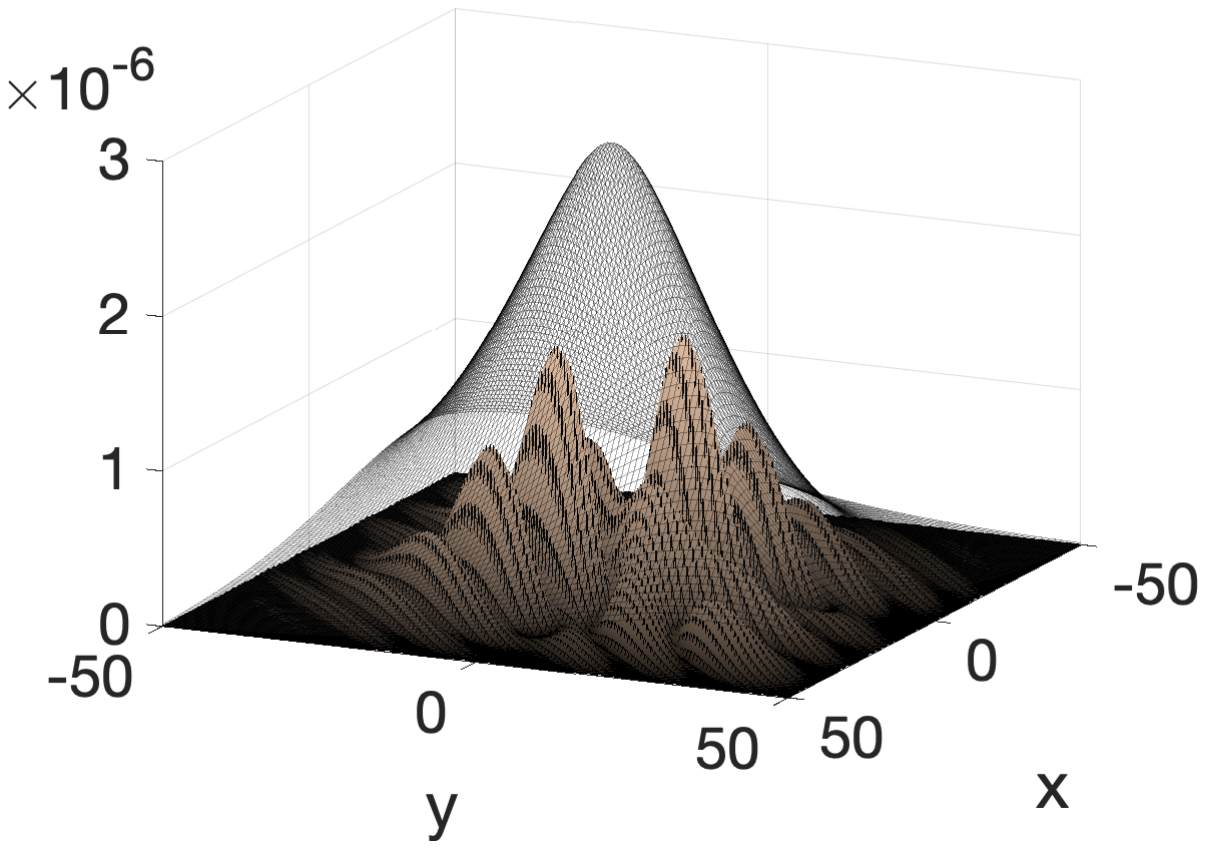}
    \end{minipage}\\ \hline
  \end{tabular}
  \captionof{figure}{This illustrates the real error $\mathcal{R}_3^n$ and Gaussian-type error of \eqref{eq:MisR3} for $n=50,500$; here, $C=0.15$ and $M = 0.2$.}\label{fig:MisR3} 
\end{table}

Let's now illustrate Theorem \ref{thm:LLT} for the order of accuracy $\lambda=\lambda_1=\lambda_2=3$. Upon recalling that $Q_1=Q_3=0$, we have
\begin{eqnarray*}\nonumber
    \mathcal{R}_3^n(x,y)&=&\phi^{(n)}(x,y)-\sum_{k=1}^2\sum_{\lambda=0}^3 e^{-i(x,y)\cdot\xi_k}\widehat{\phi}(\xi_k)^n (Q_{\lambda,k}^n H_{P_k}^n)(x,y)\\\nonumber
    &=&\phi^{(n)}(x,y)-(1+\cos(\pi(x+y)))\sum_{\lambda=0}^3 (Q_{\lambda}^n H_P^n)(x,y)\\\nonumber
    &=&\phi^{(n)}(x,y)-(1+\cos(\pi(x+y)))(H_P^n(x,y)+(Q_2^n H_P^n)(x,y))\\\nonumber
    &=&\phi^{(n)}(x,y)-\sqrt{8}\left(1+\cos(\pi(x+y))\right)\left[\frac{1}{n^{3/4}}h_2\left(\frac{2(x+y)}{(2n)^{1/2}}\right)h_4\left(\frac{x-y}{n^{1/4}}\right)\right.\\
    &&\hspace{6.5cm}\left.-\frac{2}{3 n^{5/4}}h_2\left(\frac{2(x+y)}{(2n)^{1/2}}\right)h_4^{(6)}\left(\frac{x-y}{n^{1/4}}\right)\right]
\end{eqnarray*}
 for $n\in\mathbb{N}_+$ and $(x,y)\in\mathbb{Z}^2$. For $\lambda=\lambda_1=\lambda_2=3$, Theorem \ref{thm:LLT} gives positive constants $C$ and $M$ for which
\begin{equation}\label{eq:MisR3}
\abs{\mathcal{R}_3^n(x,y)}\leq \frac{C}{n^{7/4}}\exp\left(-nMR^{\#}\left(\frac{x}{n},\frac{y}{n}\right)\right)=\frac{C}{n^{7/4}}\exp\left(-M\left(\frac{\abs{x+y}^2}{2n}+\frac{3}{4}\frac{\abs{x-y}^{4/3}}{(4n)^{1/3}}\right)\right)
\end{equation}
for $n\in\mathbb{N}_+$ and $(x,y)\in\mathbb{Z}^2$. This local limit theorem is illustrated in Figure \ref{fig:MisR3}. For the cases of $\lambda=5$ and $\lambda=7$, expressions for $\mathcal{R}_5$ and $\mathcal{R}_7$ are obtained analogously using \eqref{eq:MisQH4} and \eqref{eq:MisQH6}, respectively. For these cases, Theorem \ref{thm:LLT} gives positive constants $C$ and $M$ for which 
\begin{equation*}
    \abs{\mathcal{R}_5^n(x,y)}\leq \frac{C}{n^{9/4}}\exp\left(-M\left(\frac{\abs{x+y}^2}{2n}+\frac{3}{4}\frac{\abs{x-y}^{4/3}}{(4n)^{1/3}}\right)\right)
\end{equation*}
and
\begin{equation*}
    \abs{\mathcal{R}_7^n(x,y)}\leq \frac{C}{n^{11/4}}\exp\left(-M\left(\frac{\abs{x+y}^2}{2n}+\frac{3}{4}\frac{\abs{x-y}^{4/3}}{(4n)^{1/3}}\right)\right)
\end{equation*}
for $n\in\mathbb{N}_+$ and $(x,y)\in\mathbb{Z}^2$. In particular, we obtain a positive constant $C$ for which
\begin{equation*}
    \|\mathcal{R}_5^n\|_{\infty}\leq\frac{C}{n^{9/4}}\hspace{1cm}\mbox{and}\hspace{1cm}\|\mathcal{R}_7^n\|_\infty\leq \frac{C}{n^{11/4}}
\end{equation*}
for $n\in\mathbb{N}_+$. The $\ell^\infty$ decay of $\mathcal{R}_5^n$ and $\mathcal{R}_7^n$ are illustrated in Figures \ref{fig:MisR5_LogLog} and \ref{fig:MisR7_LogLog}, respectively.

\begin{figure}[h!]
    \centering
    \includegraphics[width=0.6\linewidth]{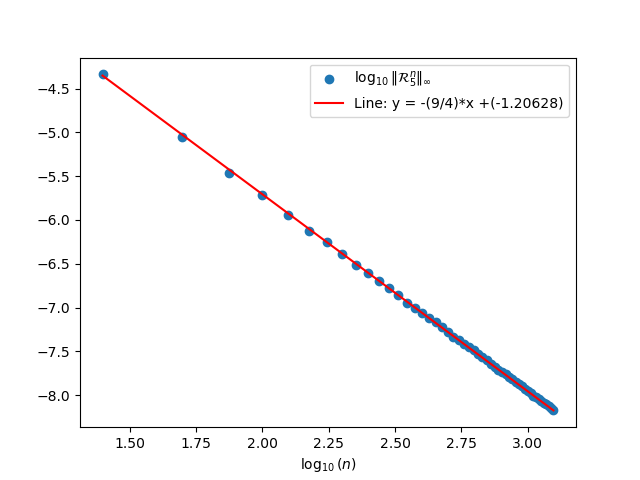}
    \caption{Graph illustrating the $\ell^{\infty}$ decay of $\mathcal{R}_5^n$. The points in blue are the values of $\Vert \mathcal{R}_{5}^n \Vert_{\infty}$ for values of $n$ ranging from $25$ to $1250$ with increments of $25$, in a $\log_{10}-\log_{10}$ scale. The red line is the best linear fit of the data having slope $-9/4$.}
    \label{fig:MisR5_LogLog}
\end{figure}

\begin{figure}[h!]
    \centering
    \includegraphics[width=0.6\linewidth]{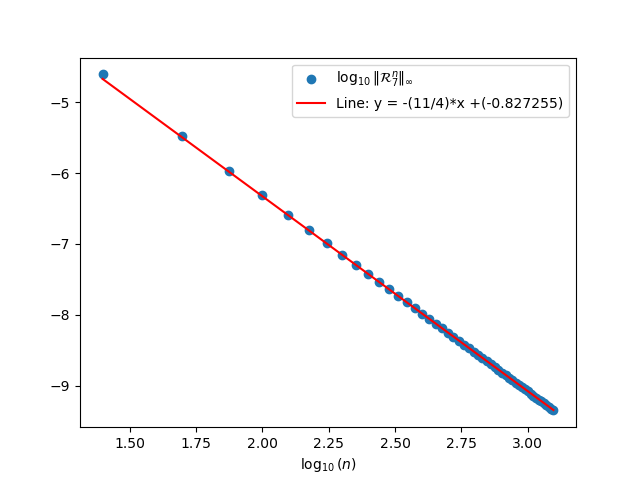}
    \caption{This illustrates the $\ell^{\infty}$ decay of $\mathcal{R}_{7}^{n}$: the values of $\log_{10}\Vert \mathcal{R}_{7}^{n} \Vert_{\infty}$ (blue points) as a function of $\log_{10}(n)$, for $n=25,50,75,\ldots, 1250$, and the best linear fit with slope $-11/4$ (red line) of the data.}
    \label{fig:MisR7_LogLog}
\end{figure}

\subsection{Revisiting Example 7.4 of \cite{RSC17} }\label{ssec:NonMinDecay}
In this final subsection, we revisit the ``contribution from non-minimal decay exponent" example of Subsection 7.4 of \cite{RSC17}. For the examples presented so far, it has been the case that when we consider better approximations of $\phi^{(n)}$
by taking into account more terms (higher-order corrections) in our local limit theorems, the rate at which the approximations become more accurate is the same for all $\xi_k\in \Omega(\phi)$. In other words, $\mu_k$ and $m_k$ have turned out to be same for all $k$. For the example we consider below, this is not the case, which allows us to play with different orders of accuracy for each point in $\Omega(\phi)$ along the lines of  Remark \ref{rmk:refRn}.

Let $\phi:\mathbb{Z}^2\rightarrow \mathbb{R}$ be defined by
\[
\phi(x,y) =  \begin{cases}
        19/128  & (x,y)=(0, 0)  \\
        19/256  &  (x,y) = (0, \pm 1) \\
         1/4  & (x,y)=(\pm 1,0) \\
         1/8 & (x,y)=(\pm 1, \pm 1) \\
         -5/64 & (x,y) = (\pm 2, 0) \\
         -5/128 & (x,y) = (\pm 2, \pm 1) \\
         1/256 & (x,y) = (\pm 4,0 ) \\
         1/512 & (x,y) = (\pm 4, \pm 1) \\
         \quad 0 & \text{otherwise}
    \end{cases}
\]
for $(x,y)\in\mathbb{Z}^2$. Routine calculations give the Fourier transform of $\phi$ as
\begin{eqnarray*}
\widehat{\phi}(\xi) &=& \frac{19}{128}+ \frac{19}{128}\cos(\zeta) + \frac{1}{2}\cos(\eta) + \frac{1}{2}\cos(\eta)\cos(\zeta) -\frac{5}{32}\cos(2\eta) -\frac{5}{32}\cos(2\eta)\cos(\zeta) \\ && + \frac{1}{128}\cos(4\eta)\cos(\zeta) + \frac{1}{128}\cos(4\eta),
\end{eqnarray*}
for $\xi=(\eta,\zeta)$ with which one can verify $\sup_{\xi}|\widehat{\phi}(\xi)|=1$ and
\begin{equation*}
\Omega(\phi) = \{ \xi_1, \xi_2 \}:= \{ (0,0), (\pi,0) \}
\end{equation*}
where $\widehat{\phi}(\xi_1)=1$ and $\widehat{\phi}(\xi_2)=-1$. Thus $\phi\in\mathcal{H}_2^*$. 

For $\xi_1$, the Taylor expansion of $\Gamma_{1}=\Log\left( \widehat{\phi}(\xi_1 + \cdot)/\widehat{\phi}(\xi_1) \right)$ around zero is
\begin{eqnarray*}
\Gamma_{1}(\xi) &=&  -\frac{1}{4}\zeta^2-\frac{1}{96}\zeta^4- \frac{1}{16}\eta^{6} - \frac{1}{1440}\zeta^6 + \frac{3}{128}\eta^8 -\frac{17}{322560}\zeta^{8} + O(|\xi|^{10}) \\ &  =& -P_{1}(\eta,\zeta) + \Upsilon_{1}(\eta, \zeta),
\end{eqnarray*}
for $\xi=(\eta,\zeta)$, where
\begin{equation*}
P_{1}(\eta, \zeta) = R_1(\eta,\zeta)=\frac{\eta^6}{16}+ \frac{\zeta^2}{4}
\end{equation*}
is positive semi-elliptic with $\mathbf{m}_1=(3,1)$, and 
\begin{eqnarray}\label{eq:ex3-Up1}\nonumber
    \Upsilon_1(\xi)&=& -\frac{1}{96}\zeta^4 - \frac{1}{1440}\zeta^6 + \frac{3}{128}\eta^8 -\frac{17}{322560}\zeta^{8} + O(|\xi|^{10}) \\
    &=&\underbrace{0 \vphantom{\frac{1}{2}}}_{S_{1,1}(\xi)/1!}+\underbrace{\frac{3}{128}\eta^8}_{S_{2,1}(\xi)/2!}+\underbrace{0 \vphantom{\frac{1}{2}}}_{S_{3,1}(\xi)/3!} +\cdots;
\end{eqnarray}
here, we have used that $m_1=\lcm(\mathbf{m}_1)=3$ and $\kappa_1=(1,3)$, so that $\Lambda_{1}(\beta)$ in the definition of $S_{1,\lambda}(\xi)$ is given by $\Lambda_{1}(\beta)=\beta_1+3\beta_2-6$ for $\beta=(\beta_1,\beta_2)\in \mathbb{N}_{+}^2$.
For $\xi_2$, we have
\begin{eqnarray*}
\Gamma_{2}(\xi) &=& -\eta^2 -\frac{1}{4}\zeta^2 - \frac{5}{12}\eta^4-\frac{1}{96}\zeta^4- \frac{137}{720}\eta^6-\frac{1}{1440}\zeta^6 + O(|\xi|^8) \\ &=& -P_{2}(\eta,\zeta) + \Upsilon_{2}(\eta,\zeta)
\end{eqnarray*}
for $\xi=(\eta,\zeta)$, where $P_{2}(\eta,\zeta)=R_2(\eta,\zeta)=\eta^2+\zeta^2/4$ is semi-elliptic with $\mathbf{m}_2=(1,1)$ and
\begin{equation}\label{eq:ex3-Up2}
    \Upsilon_2(\xi)= \underbrace{0 \vphantom{\frac{1}{2}}}_{S_{1,2}(\xi)/1!}  - \frac{5}{12}\eta^4-\frac{1}{96}\zeta^4- \frac{137}{720}\eta^6-\frac{1}{1440}\zeta^6 + O(|\xi|^8).
\end{equation}
 In this case, $m_2=\lcm(\mathbf{m}_2)=1$, $\kappa_2=(1,1)$, and $\Lambda_2(\beta)=\beta_1+\beta_2-2$ for $\beta=(\beta_1,\beta_2)\in \mathbb{N}_{+}^2$. Clearly, we have $\Upsilon_{k}(\xi)=\sum_{\abs{\beta:2\mathbf{m}_k}>1}b_{\beta,k}\xi^{\beta}$, so that $\Upsilon_k(\xi)=o(R_{k}(\xi))$ as $\xi\to 0$ in $\mathbb{R}^2$, $k=1,2$ (see \cite[Lemma A.5]{R23}). Thus $\xi_1$ and $\xi_2$ are of positive homogeneous type for $\widehat{\phi}$ with $\alpha_1=\alpha_2=(0,0)$ and so we can apply Theorem \ref{thm:LLT}.  By straightforward  computations, one gets
\begin{equation*}
R_{1}^{\#}(x,y)= \frac{5}{6}\left(\frac{8}{3} \right)^{1/5}\abs{x}^{6/5} + y^2\hspace{1cm}\mbox{and}\hspace{1cm} R_{2}^{\#}(x,y)=\frac{1}{4}x^2 + y^2
\end{equation*}
for $(x,y)\in\mathbb{R}^2$.

Before we state the local limit theorem for the present case, let's observe that $\mu_1=\abs{\mathbf{1}:2 \mathbf{m}_1}=2/3$, $m_{1}=3$, $\mu_2 =\abs{\mathbf{1}: 2 \mathbf{m}_2} = 1$, and $m_2 =1$, so that
\begin{equation*}
\mu_1\neq \mu_2\quad \text{and} \quad m_1\neq m_2.
\end{equation*}
Thus we conclude that the rate at which the asymptotic expansion becomes a better approximation of $\phi^{(n)}$ is slower for $\xi_1$ than for $\xi_2$. In other words, we need to consider more terms associated to $\xi_1$ than for $\xi_2$ to obtain a given order of accuracy. This will become more clear in the discussion that follows, where we are going to state our local limit theorem with degree of acuracy for the cases $\lambda_1=\lambda_2=1$ and $\lambda_1=\lambda_2=3$.

Let's first state the local limit theorem for the first case. For this, we need to compute $(Q^{n}_{\lambda,k}H^{n}_{P_k})(x,y)$ for $\lambda=0,1$ and $k=1,2$. Using $\Upsilon_{1}(\cdot)$ and $\Upsilon_2(\cdot)$ given, respectively, in \eqref{eq:ex3-Up1} and \eqref{eq:ex3-Up2}, as well as the expressions for $P_1(\cdot)$ and $P_2(\cdot)$ above, we get 
\begin{eqnarray*}
    H^n_{P_1}(x,y) &=&(Q^n_{0,1}H^{n}_{P_1})(x,y)  =\frac{1}{(2\pi)^2}\int_{\mathbb{R}^2} e^{-n(  \eta^6/16+ \zeta^2/4)}e^{-i(x\eta+y\zeta)}\, d\zeta d\eta \\ &=& \frac{1}{n^{2/3}} h_{6}^{1/16}\left( \frac{x}{n^{1/6}} \right)h_{2}^{1/4}\left( \frac{y}{n^{1/2}} \right) =  \frac{2^{5/3}}{n^{2/3}}h_6\left( \frac{2^{2/3}x}{n^{1/6}} \right)h_{2}\left( \frac{2y}{n^{1/2}} \right),
\end{eqnarray*}
\begin{eqnarray*}
   H^n_{P_2}(x,y)  &=& (Q^n_{0,2}H^{n}_{P_2})(x,y) = \frac{1}{(2\pi)^2}\int_{\mathbb{R}^2} e^{ -n(\eta^2+\zeta^2/4)}e^{-i(x\eta +y\zeta)}\, d\zeta d\eta\\ &=& \frac{1}{n} h_{2}\left(\frac{x}{n^{1/2}} \right)h_{2}^{1/4}\left( \frac{y}{n^{1/2}} \right) =  \frac{2}{n}h_2\left( \frac{x}{n^{1/2}} \right)h_{2}\left( \frac{2y}{n^{1/2}} \right),
\end{eqnarray*}
and, since $S_{1,1}(\xi)=S_{1,2}(\xi)=0$, 
\begin{equation*}
    (Q^{n}_{1,1}H^{n}_{P_1})(x,y)=(Q^{n}_{1,2}H^{n}_{P_2})(x,y)=0
\end{equation*}
for $(x,y)\in\mathbb{R}^2$ and $n\in\mathbb{N}_+$. Therefore
\begin{eqnarray*}
\mathcal{R}^{n}_{1}(x,y)&=& \phi^{(n)}(x,y) -\sum_{k=1}^{2}\sum_{\lambda=0}^{1} e^{-i(x,y)\cdot\xi_{k}}\widehat{\phi}(\xi_{k})^{n} (Q^{n}_{\lambda,k}H^{n}_{P_k})(x,y) \\ &=& \phi^{(n)}(x,y) -\sum_{k=1}^{2}e^{-i(x,y)\cdot\xi_{k}}\widehat{\phi}(\xi_{k})^{n} (Q^{n}_{0,k}H^{n}_{P_k})(x,y)  \\ &=& \phi^{(n)}(x,y)-\sum_{k=1}^{2}\frac{e^{-i(x,y)\cdot \xi_{k}}\widehat{\phi}(\xi_{k})^n}{n^{\mu_{k}}} (Q_{0,k}H_{P_k})\left(n^{-D_{k}}(x,y) \right) \\ &=& \phi^{(n)}(x,y) - \left[\frac{2^{5/3}}{n^{2/3}}h_6\left( \frac{2^{2/3}x}{n^{1/6}} \right)h_{2}\left( \frac{2y}{n^{1/2}} \right)+ (-1)^{x+n}  \frac{2}{n}h_2\left( \frac{x}{n^{1/2}} \right)h_{2}\left( \frac{2y}{n^{1/2}} \right)\right]
\end{eqnarray*}
for $n\in \mathbb{N}_+$ and $(x,y)\in \mathbb{Z}^2$. Here $D_1=\diag(1/6,1/2)$ and $D_2=\diag(1/2,1/2)$. Thus, Theorem \ref{thm:LLT} implies that, for some positive constants $C_1$, $C_2$, $M_1$, and $M_2$, there holds
\begin{eqnarray}\label{eq:ex3_R1_est}\nonumber
 \abs{ \mathcal{R}^{n}_{1}(x,y) }  &\leq& \frac{C}{n^{2/3+2/6}}\exp\left(-nMR_{1}^{\#}\left( \frac{x}{n},\frac{y}{n} \right) \right) +\frac{C}{n^{1+2/2}}\exp\left(-nMR_{2}^{\#}\left( \frac{x}{n},\frac{y}{n} \right) \right) \\\nonumber
 &=& \frac{C}{n}\exp \left( -M \left\lbrace \frac{5}{6}\left(\frac{8}{3}\right)^{1/5} \frac{\abs{x}^{6/5}}{n^{1/5}}  +  \frac{y^2}{n} \right\rbrace \right)+\frac{C}{n^{2}}\exp \left( -M \left\lbrace \frac{1}{4} \frac{x^2}{n} +  \frac{y^2}{n}  \right\rbrace \right) \\ &\leq& \frac{C}{n}e^{-My^2/n}\left[ \exp\left( -M  \frac{5}{6}\left(\frac{8}{3}\right)^{1/5} \frac{\abs{x}^{6/5}}{n^{1/5}}  \right)+\exp \left( -M  \frac{x^2}{4n}  \right)  \right]  
\end{eqnarray}
holds for all $n\in \mathbb{N}_+$ and $(x,y)\in \mathbb{Z}^2$. Figure \ref{fig:ex3_R1} illustrates this result. In the last line, we have taken $C=\max \{ C_1,C_2\}$ and $M=\min \{ M_1,M_2\}$, and considered the lowest decay rate of $n^{-1}$ or, in terms of the notation of Remark \ref{rmk:refRn}, $n^{-1}=n^{-\nu}$ where
\[
\nu = \min_{k}\{ \mu_k +(\lambda_k+1)/2m_k \}= \min_k \{ 2/3 +2/6, 1 + 2/2\}. 
\]
As $\mu_2\geq \nu$, we get $V_{\nu}=\{k: \mu_k \geq \nu\} = \{2 \}$, so that the previous estimate still holds if we neglect the terms associated to $k=2$ (see Remark \ref{rmk:refRn} for more details). In other words, the error
\begin{eqnarray*}
\widetilde{\mathcal{R}}^{n}_{1}(x,y)&:=& \phi^{(n)}(x,y) -\sum_{\lambda=0}^{1} e^{-i(x,y)\cdot\xi_{1}}\widehat{\phi}(\xi_{1})^{n} (Q^{n}_{\lambda,1}H^{n}_{P_1})(x,y) 
\\ &=& \phi^{(n)}(x,y) - \frac{2^{5/3}}{n^{2/3}}h_6\left( \frac{2^{2/3}x}{n^{1/6}} \right)h_{2}\left( \frac{2y}{n^{1/2}} \right)
\end{eqnarray*}
satisfies the estimate
\begin{equation}\label{eq:ex3_R1Tilde_est}
\vert \widetilde{\mathcal{R}}^{n}_{1}(x,y) \vert \leq \frac{C}{n}e^{-My^2/n}\left[ \exp\left( -M  \frac{5}{6}\left(\frac{8}{3}\right)^{1/5} \frac{\abs{x}^{6/5}}{n^{1/5}}  \right)+\exp \left( -M  \frac{x^2}{4n}  \right)  \right]
\end{equation}
for all $n\in \mathbb{N}_+$ and $(x,y)\in \mathbb{Z}^2$, where $C$ and $M$ are positive constants, as captured in Figure \ref{fig:ex3_R1_Tilde}. As this estimate holds in all $\mathbb{Z}^2$, we obtain
\begin{equation*}
\|\widetilde{\mathcal{R}}_{1}^{n}\|_{\infty} \leq \frac{C}{n},
\end{equation*}
for all $n\in \mathbb{N}_+$. Such decay is shown in Figure \ref{fig:ex3_R1_Tilde_LogLog}.

\begin{table}[!h]
  \centering
  \begin{tabular}{  |c| c | c | }
    \hline
     & $n=50$ & $n=500$ \\ \hline
     $\phi^{(n)}$ & 
    \begin{minipage}{.4\textwidth}
      \includegraphics[width=1\linewidth, trim = {0 6cm 0 6cm},clip]{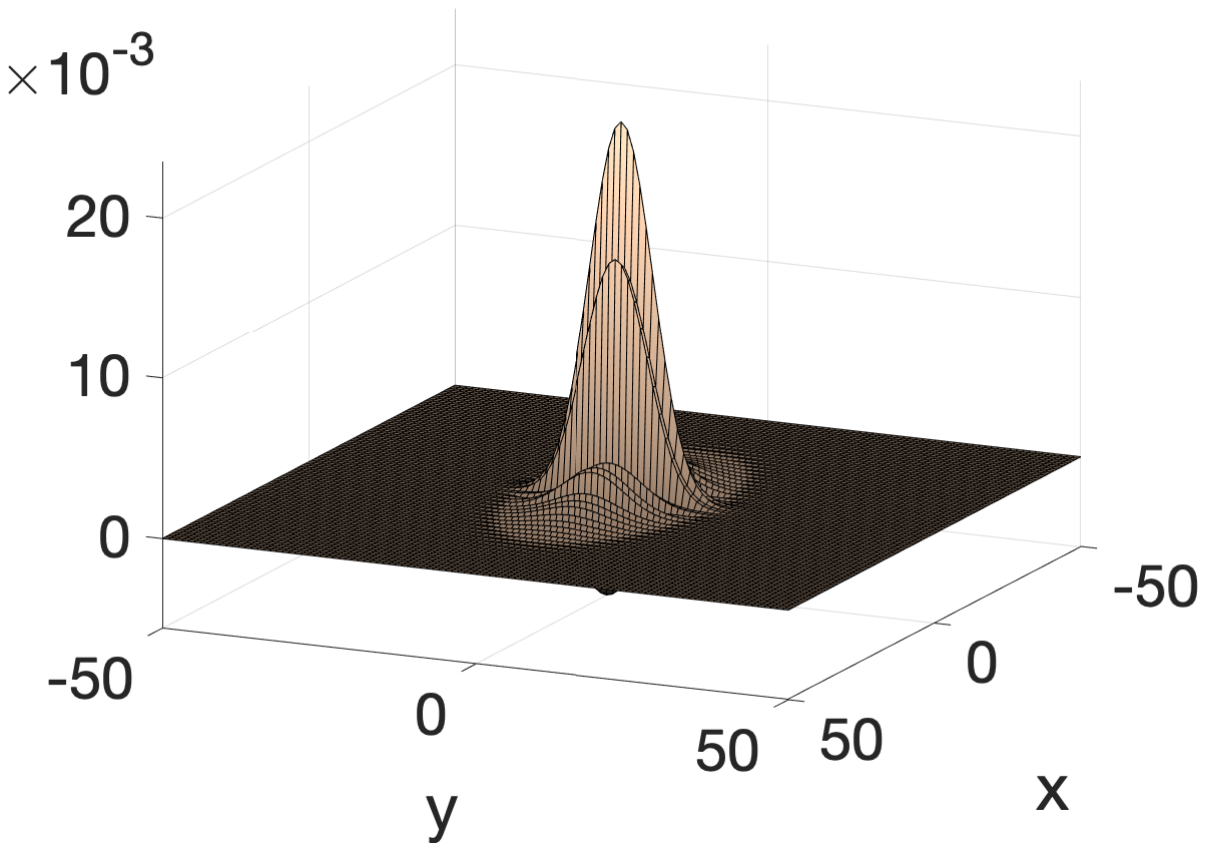}
    \end{minipage}
	&
      \begin{minipage}{.44\textwidth}
      \includegraphics[width=1\linewidth, trim = {0 6cm 0 6cm},clip]{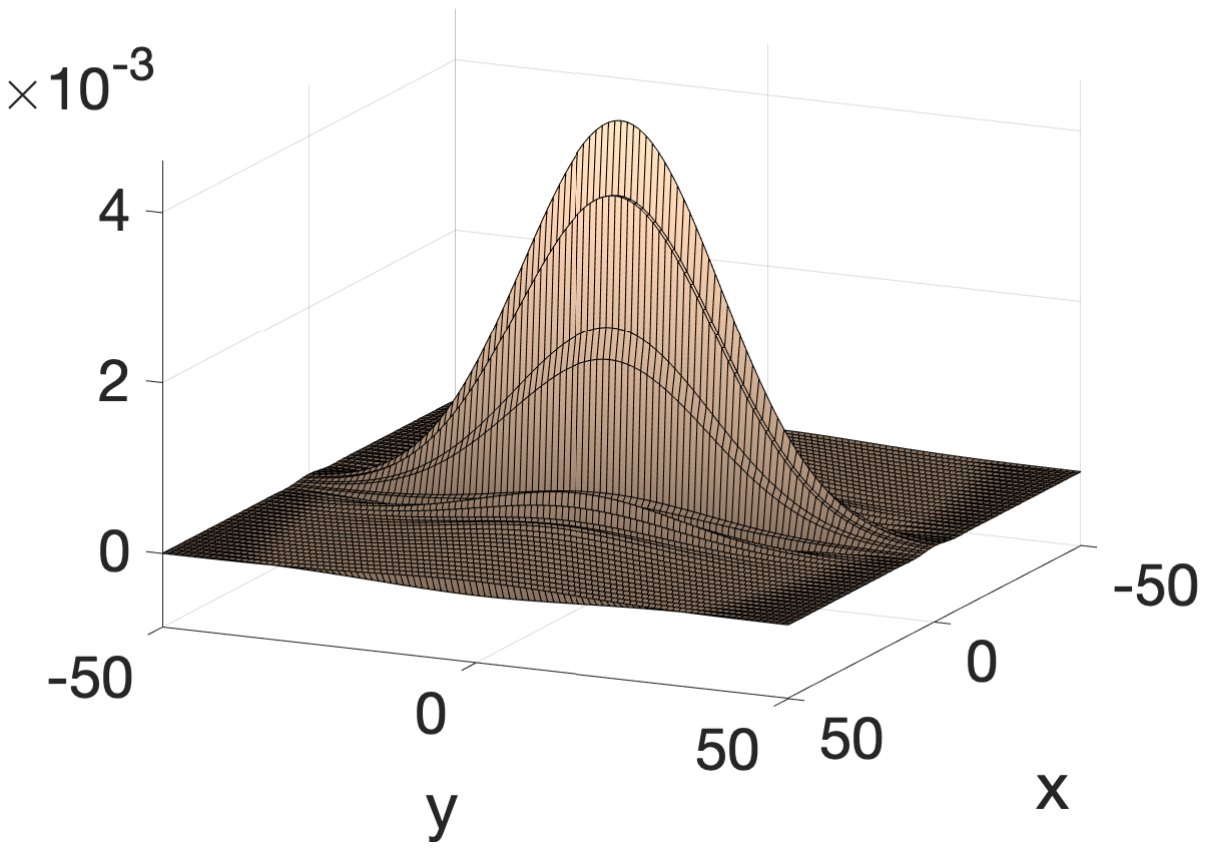}
    \end{minipage}\\ \hline
    $\mathcal{A}_1^n$ &
    \begin{minipage}{.4\textwidth}
     \includegraphics[width=1\linewidth, trim = {0 6cm 0 6cm},clip]{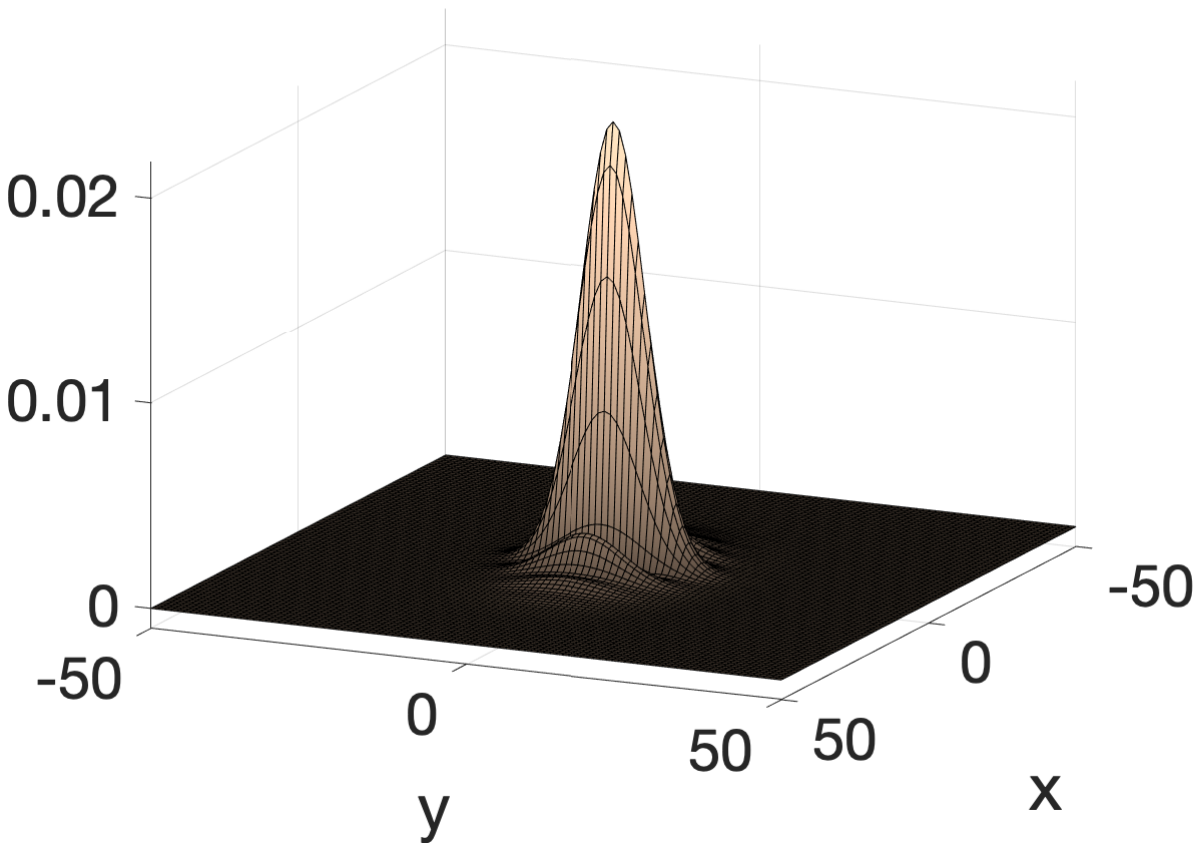}
    \end{minipage}
	&
     \begin{minipage}{.4\textwidth}
     \includegraphics[width=1\linewidth, trim = {0 6cm 0 6cm},clip]{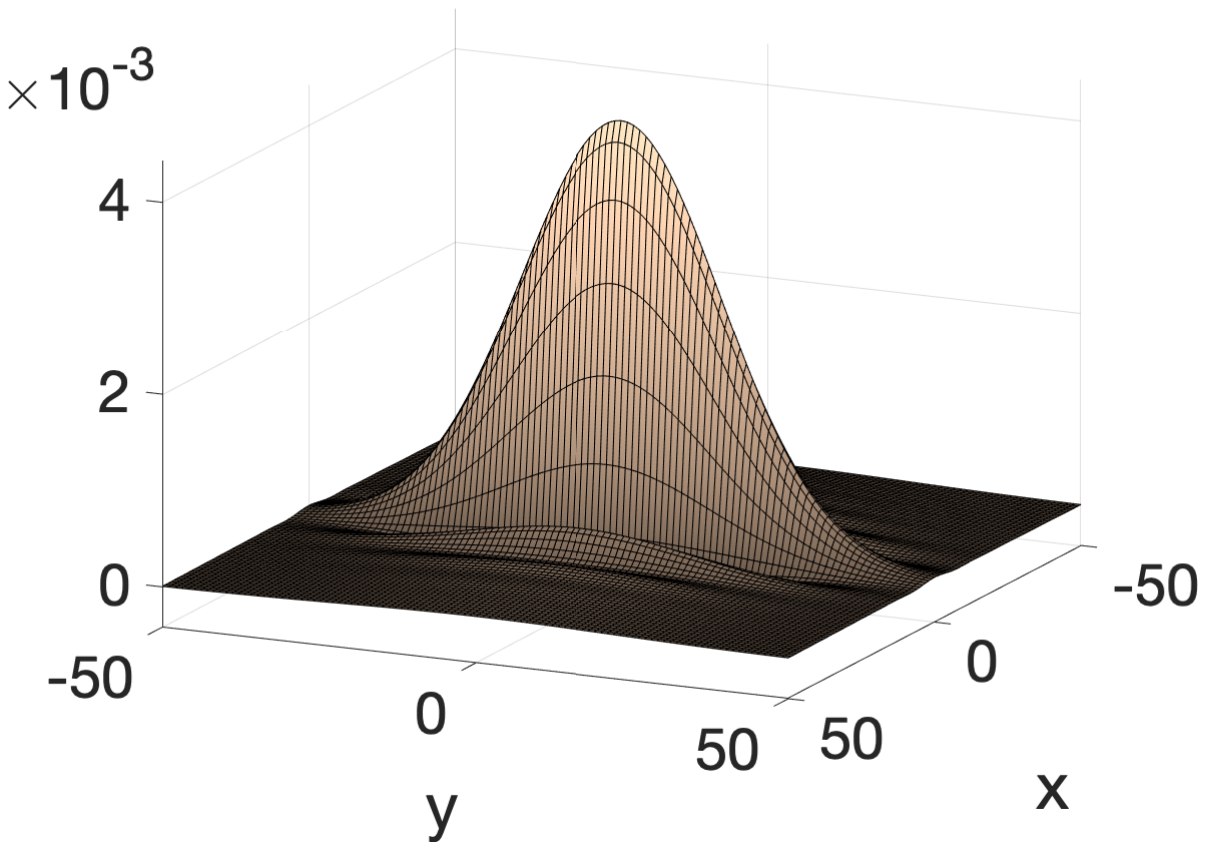}
    \end{minipage}\\ \hline
    \begin{turn}{270}\hspace{-0.5cm} Error \end{turn} & 
    \begin{minipage}{.4\textwidth}
      \includegraphics[width=1\linewidth, trim = {0 6cm 0 6cm},clip]{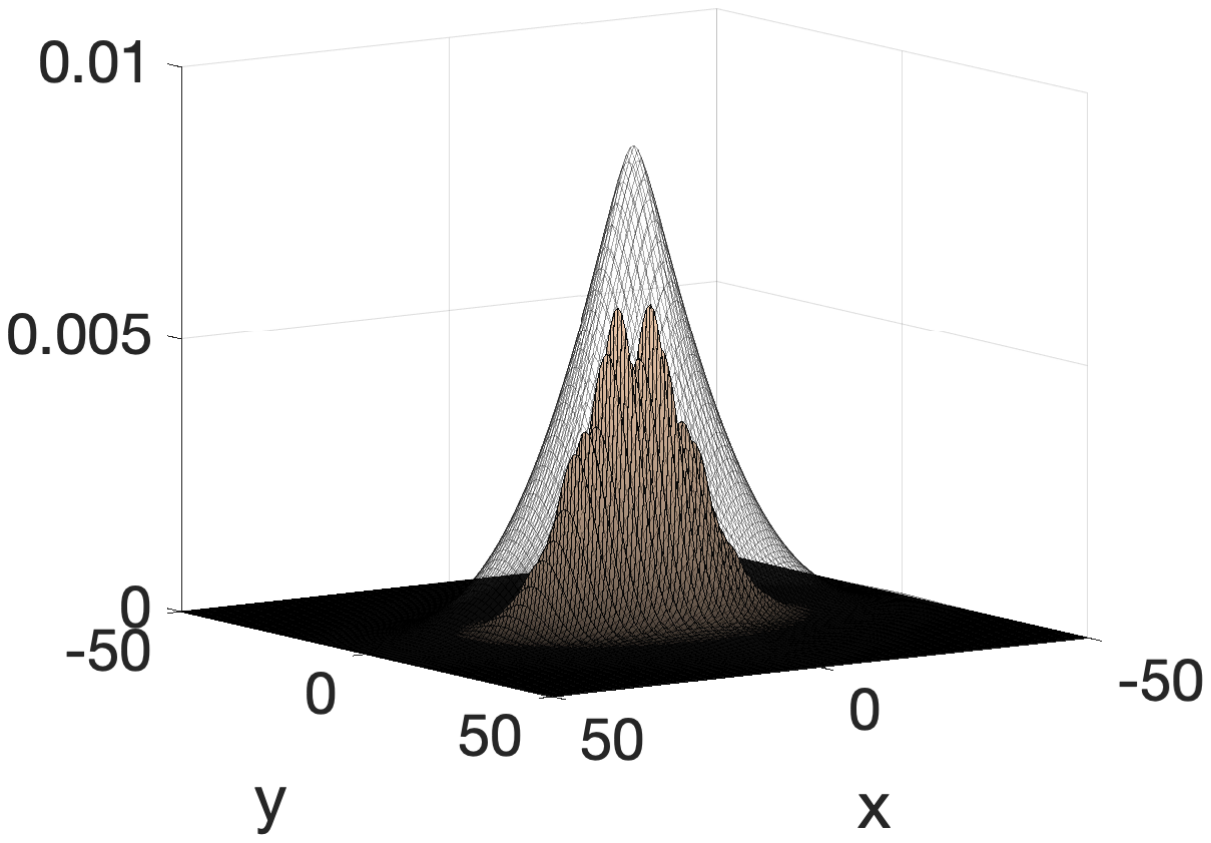}
    \end{minipage}
	&
      \begin{minipage}{.4\textwidth}
      \includegraphics[width=1\linewidth, trim = {0 6cm 0 6cm},clip]{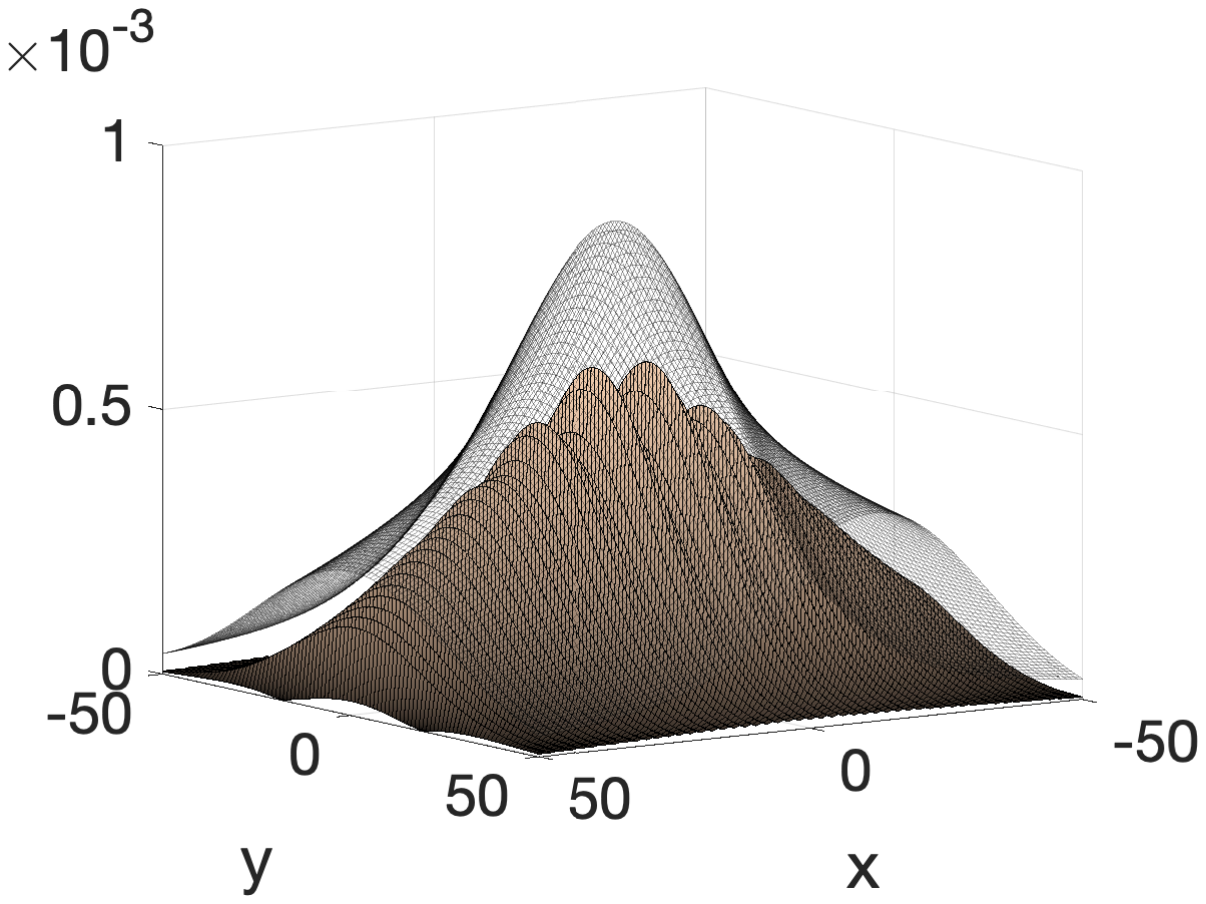}
    \end{minipage}\\ \hline
  \end{tabular}
  \captionof{figure}{The convolution powers $\phi^{(n)}$, the base attractors and first-order corrections $\mathcal{A}_1^n=\phi^{(n)}-\mathcal{R}_1^n$, and the error are illustrated for $n=50,500$. In the third row, the real error $\mathcal{R}_1^n$ is illustrated by the light brown surfaces and the Gaussian-type error in \eqref{eq:ex3_R1_est} is illustrated by the transparent nets above.}\label{fig:ex3_R1} 
\end{table}

\begin{table}[!h]
  \centering
  \begin{tabular}{  |c| c | c | }
    \hline
     $n=50$ & $n=500$ \\ \hline
    \begin{minipage}{.4\textwidth}
      \includegraphics[width=1\linewidth, trim = {0 6cm 0 6cm},clip]{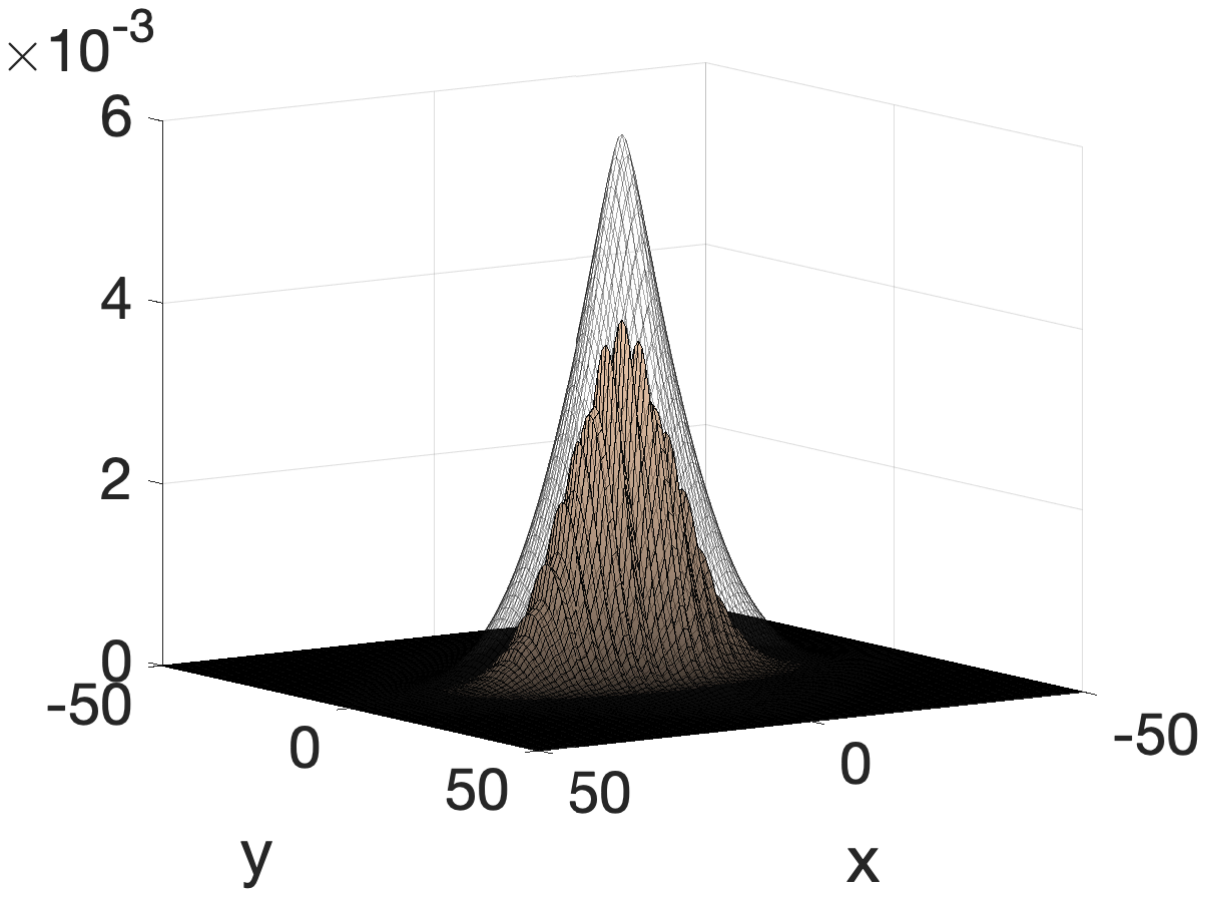}
    \end{minipage}
	&
      \begin{minipage}{.4\textwidth}
      \includegraphics[width=1\linewidth, trim = {0 6cm 0 6cm},clip]{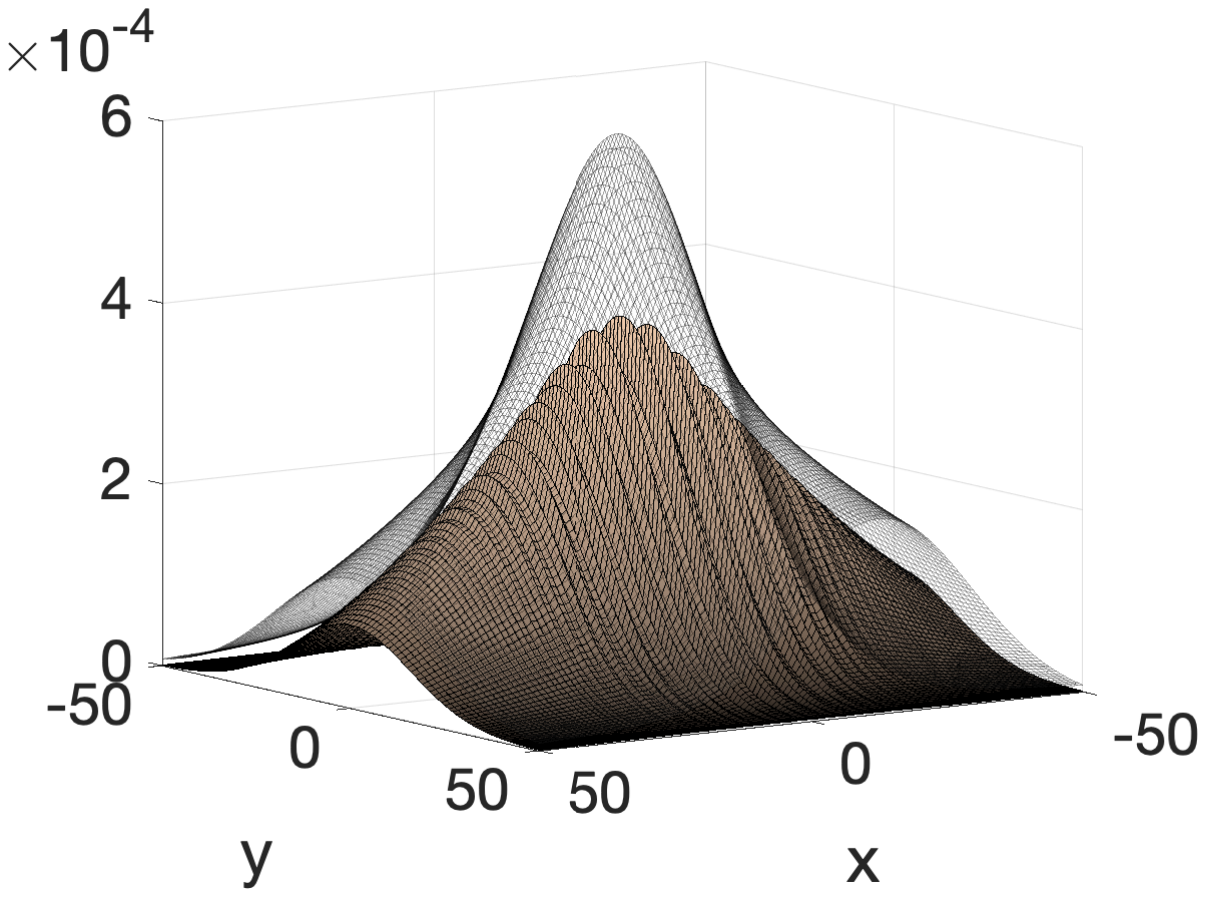}
    \end{minipage}\\ \hline
  \end{tabular}
  \captionof{figure}{The real error $\widetilde{\mathcal{R}}^n_1$ is illustrated by the light brown surfaces and the Gaussian-type error given in \eqref{eq:ex3_R1Tilde_est} is predicted by the transparent nets above for $n=50,500$.}\label{fig:ex3_R1_Tilde} 
\end{table}

\begin{figure}[h!]
    \centering
    \includegraphics[width=0.6\linewidth]{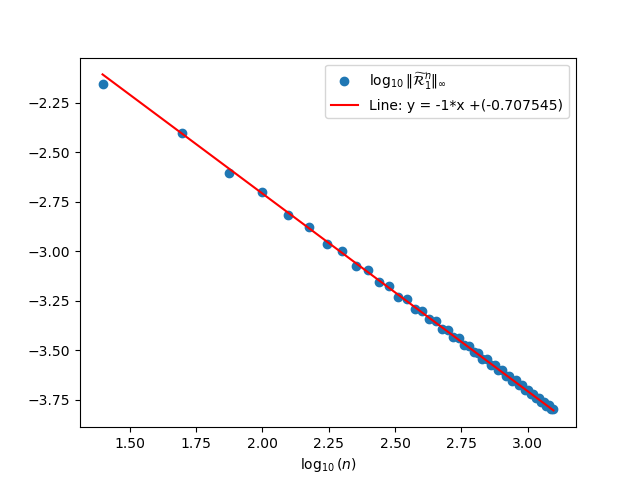}
    \caption{Graph of the values of $\log_{10} \Vert \widetilde{\mathcal{R}}^{n}_1\Vert_{\infty}$ (blue points) as a function of $\log_{10}(n)$. We considered values of $n$ from $25$ to $1,250$ in increments of $25$. The line in red is the best linear fit for the plotted points having slope $-1$.}
    \label{fig:ex3_R1_Tilde_LogLog}
\end{figure}

Next, we look at the case $\lambda_1=\lambda_2=3$. For that, in addition to the already computed $(Q^{n}_{\lambda,k}H^{n}_{P_k})(x,y)$ for $\lambda=0,1$ for $k=1,2$, we need to do the same for $\lambda=2,3$ . Before we move to this, and to avoid computing unnecessary terms, let us first see which terms can be neglected. If we consider the error 
\begin{equation*}
\mathcal{R}_{3}^{n}(x,y) =\phi^{(n)}(x,y)   -\sum_{k=1}^2\sum_{\lambda=0}^{3}e^{-i(x,y)\cdot \xi_{k}}\widehat{\phi}(\xi_{k})^n(Q^{n}_{\lambda,k}H^{n}_{P_k}) (x,y) ,
\end{equation*}
for $n\in \mathbb{N}_+$ and $(x,y)\in \mathbb{Z}^2$, we find
\begin{equation*}
\abs{\mathcal{R}^{n}_{3}(x,y)} \leq  \frac{C}{n^{4/3}}e^{-My^2/n}\left[ \exp\left( -M  \frac{5}{6}\left(\frac{8}{3}\right)^{1/5} \frac{\abs{x}^{6/5}}{n^{1/5}}  \right)+\exp \left( -M  \frac{x^2}{4n}  \right)  \right] 
\end{equation*}
for all $n\in \mathbb{N}_+$ and $(x,y)\in\mathbb{Z}^2$, for some positive constants $C$ and $M$. The decay $n^{-4/3}$ is given by 
\begin{equation*}
4/3= \nu= \min_{k}\{ \mu_k +(\lambda_k+1)/2m_k\} = \min_k\{2/3 + 4/6, 1+ 4/2 \}.
\end{equation*}
Now, as $\mu_2=1 < 4/3=\nu$, we can no longer neglect all the terms in $\mathcal{R}_3^{n}(\cdot)$ associated to $\xi_2$ and still have the above estimate. But we can neglect some of them. As discussed in Remark \ref{rmk:refRn}, for $k=2$ it's enough to consider the expansion up to $\underline{\lambda}_2$ satisfying 
\begin{equation*}
\mu_2+ \underline{\lambda}_2/2m_2<\nu, \,\, \text{and}\,\,\, \nu \leq \mu_2+ (\underline{\lambda}_2+1)/2m_2,
\end{equation*}
which is the case for $\underline{\lambda}_2=0$. Thus, the previous estimate remains to hold for the error considering only the purely attractor term for $k=2$, that is for
\begin{equation*}
 \widetilde{\mathcal{R}}_{3}^{n}(x,y) := \phi^{(n)}(x,y)  -\sum_{\lambda=0}^{3}e^{-i(x,y)\cdot \xi_{1}}\widehat{\phi}(\xi_{1})^n (Q^{n}_{\lambda,1}H^{n}_{P_1}) (x,y)- e^{-i(x,y)\cdot\xi_2} \widehat{\phi}(\xi_2)^{n}(Q^{n}_{0,2}H_{P_2})(x,y),
\end{equation*}
we have
\begin{equation}\label{eq:ex3_R3_Tilde_Est}
\vert \widetilde{\mathcal{R}}^{n}_{3}(x,y) \vert \leq  \frac{C}{n^{4/3}}e^{-My^2/n}\left[ \exp\left( -M  \frac{5}{6}\left(\frac{8}{3}\right)^{1/5} \frac{\abs{x}^{6/5}}{n^{1/5}}  \right)+\exp \left( -M  \frac{x^2}{4n}  \right)  \right]
\end{equation}
for all $n\in \mathbb{N}_+$ and $(x,y)\in\mathbb{Z}^2$, where $C$ and $M$ are positive constants. Based on information we have so far, to get the full expression for $\widetilde{\mathcal{R}}_3^{n}$ we are missing the terms of $k=1$ for $\lambda=2 $ and $\lambda=3$. For that, we use $S_{1,1}(\xi)$, $S_{2,1}(\xi)$, and $S_{3,1}(\xi)$ given in \eqref{eq:ex3-Up1} together with the expressions for the Bell polynomials given in Table \ref{tab:Bell}, to get 
\begin{equation*}
    Q^{n}_{2,1}=\frac{1}{2!}\left( 0^2 + nS_{2,1}(i\partial) \right) =\frac{3n}{128}\partial_1^8,\quad \text{and} \quad Q^{n}_{3,1}=0,
\end{equation*}
so that
\begin{equation*}
    (Q^{n}_{2,1}H^{n}_{P_1})(x,y)=\frac{3}{n}h_{6}^{(8)}\left( \frac{2^{2/3}x}{n^{1/6}} \right)h_{2}\left( \frac{2y}{n^{1/2}} \right),\quad \text{and} \quad (Q^{n}_{3,1}H^{n}_{P_1})(x,y)=0.
\end{equation*}
With this, we obtain
\begin{eqnarray*}
\lefteqn{\widetilde{\mathcal{R}}_{3}^{n}(x,y) =\phi^{(n)}(x,y)} \\ & &  - \left[\frac{2^{5/3}}{n^{2/3}}h_6\left( \frac{2^{2/3}x}{n^{1/6}} \right)h_{2}\left( \frac{2y}{n^{1/2}} \right)+  \frac{3}{n}h_{6}^{(8)}\left( \frac{2^{2/3}x}{n^{1/6}} \right)h_{2}\left( \frac{2y}{n^{1/2}} \right)+(-1)^{x+n}  \frac{2}{n}h_2\left( \frac{x}{n^{1/2}} \right)h_{2}\left( \frac{2y}{n^{1/2}} \right)\right]
\end{eqnarray*}
which, in view of our work above, satisfies \eqref{eq:ex3_R3_Tilde_Est} uniformly for $(x,y)\in\mathbb{Z}^2$ and $n\in\mathbb{N}_+$.

We conclude our discussion with one comment. Although for some order of accuracy we can neglect all the terms in the expansion associated with some of the $k'$s, as was the case when $\lambda_1=\lambda_2=1$, it is the case that when we increase the accuracy we will get to some point where we need to take into consideration terms from all the $k'$s, otherwise the order of accuracy we are expecting will not hold. Indeed, if in the expression for $\widetilde{\mathcal{R}}_{3}^{n}$, we had forgotten to take into account the term for $k=2$, that is if instead we had considered the error
\begin{eqnarray*}
_{*}\widetilde{\mathcal{R}}_{3}^{n}(x,y) &:=&\phi^{(n)}(x,y) -\sum_{\lambda=0}^{3}e^{-i(x,y)\cdot \xi_{1}}\widehat{\phi}(\xi_{1})^n (Q^{n}_{\lambda,1}H^{n}_{P_1}) (x,y)  \\ &=& \phi^{(n)}(x,y)- \left[\frac{2^{5/3}}{n^{2/3}}h_6\left( \frac{2^{2/3}x}{n^{1/6}} \right)h_{2}\left( \frac{2y}{n^{1/2}} \right)+  \frac{3}{n}h_{6}^{(8)}\left( \frac{2^{2/3}x}{n^{1/6}} \right)h_{2}\left( \frac{2y}{n^{1/2}} \right)\right].
\end{eqnarray*}
This error no longer satisfies estimate \eqref{eq:ex3_R3_Tilde_Est}, and instead satisfies
\begin{equation*}
\vert _{*}\widetilde{\mathcal{R}}^{n}_{3}(x,y) \vert \leq  \frac{C}{n}e^{-My^2/n}\left[ \exp\left( -M  \frac{5}{6}\left(\frac{8}{3}\right)^{1/5} \frac{\abs{x}^{6/5}}{n^{1/5}}  \right)+\exp \left( -M  \frac{x^2}{4n}  \right)  \right],
\end{equation*}
which is the same we obtained for $\widetilde{\mathcal{R}}_{1}^{n}$. In other words, $_{*}\widetilde{\mathcal{R}}_{3}^{n}$ is not a better approximation than $\widetilde{\mathcal{R}}_{1}^{n}$, and this is illustrated by the results of numerical simulations presented in Figure \ref{fig:ex3_R3_Star_Tilde_LogLog}.

\begin{figure}[h!]
    \centering
    \includegraphics[width=0.6\linewidth]{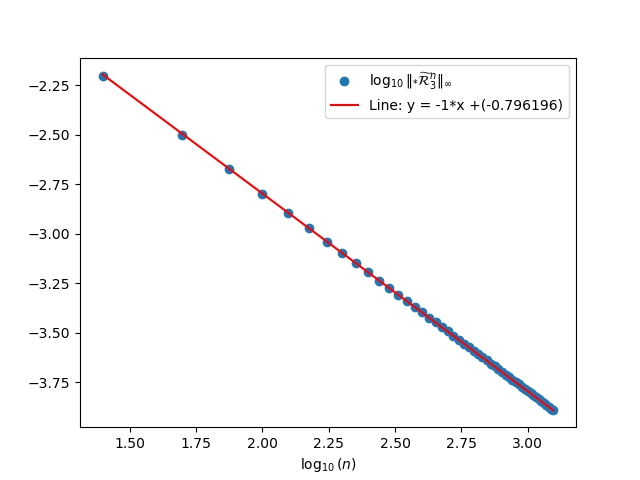}
    \caption{Graph showing the values of $\Vert _{*}\widetilde{\mathcal{R}}^{n}_3\Vert_{\infty}$ (blue points) as a function of $n$ in $\log_{10}-\log_{10}$ scale, for values of $n$ ranging from $25$ to $1,250$ in increments of $25$. The line in red is the best linear fit having slope $-1$ for the plotted points.}
    \label{fig:ex3_R3_Star_Tilde_LogLog}
\end{figure}

\section*{\large Acknowledgments:}The authors thank Jean-Fran\c{c}ois Coulombel and Gr\'{e}gory Faye for their feedback and several helpful suggestions that led to our presentation of local limit theorems in terms of Bell polynomials. We also thank Fernando Gouv\^{e}a and Leo Livshits for many helpful discussions.
\appendix
\section{Appendix}
\subsection{Appendix: Positive-Homogeneous Polynomials}\label{ssec:PHP}

In this subsection, we prove Propositions \ref{prop:PosHomAreSemiElliptic} and \ref{prop:QLmbdOpIsIndependentOfA}. We begin with a lemma used in the proof of the first proposition.

\begin{lemma}\label{lem:TripleSumDiagLemma}
Let $P$ be a positive-homogeneous polynomial in $d$ variables and let $D, N$, and $S$ be real matrices such that $D$ is diagonal, $S$ is skew-symmetric, $N$ is nilpotent and, further, we have the commutativity relations $[D+S,N]=[D,S]=0$; here, $[\cdot,\cdot]$ denotes the classical Lie bracket. If 
\begin{equation*}
    E:=D+S+N\in\Exp(P),
\end{equation*}
then
\begin{equation*}
    D = \diag\left(1/2m_1,1/2m_2,\dots,1/2m_d\right)
\end{equation*}
for some $\mathbf{m}=(m_1,m_2,\dots,m_d)\in\mathbb{N}_+^d$. Further, with this $\mathbf{m}$, $P$ is positive semi-elliptic with
\begin{equation*}
    P(\xi)=\sum_{\abs{\alpha:\mathbf{m}}=2}a_\alpha\xi^{\alpha}
\end{equation*}
for some complex coefficients $\{a_\alpha\}$ and, necessarily, $D\in\Exp(P)$.
\end{lemma}
\begin{proof} We write $D=\diag(\lambda)$ for $\lambda=(\lambda_1,\lambda_1,\dots,\lambda_d)\in\mathbb{R}^d$ and
\begin{equation*}
    P(\xi)=\sum_{\alpha\in \mathcal{I}}a_\alpha\xi^\alpha
\end{equation*}
where $\mathcal{I}$ is a finite subset of multi-indices (for which $a_\alpha\neq 0$). Our primary goal is to show that $\alpha\cdot\lambda=1$ for every $\alpha\in \mathcal{I}$. 

    To this end, given that $[D+S,N]=0$ and $[D,S]=0$, we have
    \begin{equation*}
        t^{E}=t^{N+(D+S)}=t^N t^{D+S}=t^Nt^{S+D}=t^N t^S t^D
    \end{equation*}
    for all $t>0$. Since $E\in\Exp(P)$, we have $tP(\xi)=P(t^Nt^St^D\xi)$ and equivalently
    \begin{equation}\label{eq:TripleSumDiagLemma1}
      P(t^Nt^S\xi)=tP(t^{-D}\xi)= \sum_{\alpha\in \mathcal{I}}a_\alpha t^{1-\alpha\cdot\lambda}\xi^\alpha
    \end{equation}
    for all $t>0$ and $\xi\in\mathbb{R}^d$. Given that $N$ is nilpotent (with, say, order $k+1$), 
    \begin{equation*}
        P(t^Nt^S\xi)=P\left(t^S\xi+(\ln t) N t^S\xi+\cdots+\frac{(\ln t)^k}{k!}N^kt^S\xi\right)=Q(\ln t,t^S\xi)
    \end{equation*}
    where $Q(s,\eta)$ for $(s,\eta)\in\mathbb{R}\times\mathbb{R}^d$ is a multivariate polynomial. In particular, for some finite degree $l$, we can write
    \begin{equation*}
        Q(s,\eta)=\sum_{j=0}^l b_j(\eta)s^j
    \end{equation*}
    where $b_j(\eta)$ for $j=0,1,\dots,l$ are polynomial functions of $\eta\in\mathbb{R}^d$. With this, \eqref{eq:TripleSumDiagLemma1} can be expressed equivalently as
    \begin{equation}\label{eq:TripleSumDiagLemma2}
        \sum_{\alpha\in \mathcal{I}} a_\alpha t^{1-\alpha\cdot\lambda}\xi^\alpha=\sum_{j=0}^l b_j(t^S\xi)(\ln t)^j
    \end{equation}
    for all $t>0$ and $\xi\in\mathbb{R}^d$. 
    
    With \eqref{eq:TripleSumDiagLemma2} in hand, let's assume to reach a contradiction that there is some $\alpha_0\in \mathcal{I}$ for which $\alpha_0\cdot\lambda<1$. Without loss of generality, we take this dot product to be minimal (among those $\alpha\in \mathcal{I}$) and define
    \begin{equation*}
        \mathcal{I}'=\{\alpha\in \mathcal{I}:\alpha\cdot\lambda=\alpha_0\cdot\lambda\}.
    \end{equation*}
    Given that distinct monomials are linearly independent and $a_{\alpha_0}\neq 0$, we may fix some $\xi\in\mathbb{R}^d$ for which
    \begin{equation*}
        \sum_{\alpha\in \mathcal{I}'}a_\alpha \xi^{\alpha}\neq 0. 
    \end{equation*}
    For $\delta:=1-\alpha_0\cdot\lambda>0$, \eqref{eq:TripleSumDiagLemma2} yields
    \begin{equation}\label{eq:TripleSumDiagLemma3}
    \sum_{s\in \mathcal{I}'}a_\alpha \xi^\alpha+\sum_{\alpha\in \mathcal{I}\setminus \mathcal{I}'}a_\alpha t^{1-\alpha\cdot\lambda-\delta}\xi^\alpha=\frac{1}{t^\delta}\sum_{\alpha\in \mathcal{I}}a_\alpha t^{1-\alpha\cdot\lambda}\xi^\alpha=\sum_{j=0}^l b_j(t^S\xi)\frac{(\ln t)^j}{t^\delta}
    \end{equation}
    for all $t>0$. Now, because $S$ is skew-symmetric, $\{t^S\}$ is a subgroup of $\OdR$ and so $t^S \xi$ is bounded for all $t>0$. Therefore
    \begin{equation*}
        \lim_{t\to\infty}\sum_{j=0}^l b_j(t^S\xi)\frac{(\ln t)^j}{t^\delta}=0
    \end{equation*}
    because the polynomial coefficients $b_j(\cdot)$ are bounded on bounded sets. Also, by the minimality of $\alpha_0$, $1-\alpha\cdot\lambda-\delta=\alpha_0\cdot\lambda-\alpha\cdot\lambda<0$ for every $\alpha\in \mathcal{I}\setminus \mathcal{I}'$ and so it follows that
    \begin{equation*}
        \lim_{t\to\infty}\sum_{\alpha\in \mathcal{I}\setminus \mathcal{I}'}a_\alpha t^{1-\alpha\cdot\lambda-\delta}\xi^\alpha=0. 
    \end{equation*}
    In view of the identity \eqref{eq:TripleSumDiagLemma3}, we obtain
    \begin{equation*}
        \sum_{\alpha\in \mathcal{I}'}a_\alpha\xi^\alpha=\lim_{t\to\infty}\left(\sum_{j=0}^l b_j(t^S\xi)\frac{(\ln t)^j}{t^\delta}-\sum_{\alpha\in \mathcal{I}\setminus \mathcal{I}'}a_\alpha t^{1-\alpha\cdot\lambda-\delta}\xi^\alpha\right)=0,
    \end{equation*}
     a contradiction. In the case that $\alpha_0\cdot\lambda>1$ for some $\alpha_0\in \mathcal{I}$, one can obtain a contradiction analogously by taking $t\to 0$. Thus $\alpha\cdot\lambda=1$ for all $\alpha\in \mathcal{I}$ and so we write
     \begin{equation*}
     P(\xi)=\sum_{\alpha\cdot\lambda=1}a_\alpha \xi^\alpha
     \end{equation*}
     (with this presentation allowing zero coefficients for potentially some $\alpha$ for which $\alpha\cdot\lambda=1$). To complete the proof, it remains to determine $\lambda$. For $k=1,2,\dots,d$, let's denote by $e_k$ the standard unit vector in $\mathbb{R}^d$ and observe that
     \begin{equation*}
         R(te_k)=\sum_{\alpha\cdot\lambda=1}\Re(a_\alpha)(te_k)^\alpha=\Re(a_{n_ke_k})t^{n_k}=c_kt^{n_k}
     \end{equation*}
     since all polynomial terms in $R$ vanish except (perhaps) a single $\alpha\in\mathbb{N}_+^d$ of the form $\alpha=n_k e_k$ for some $n_k\in\mathbb{N}$; necessarily, $n_k\lambda_k=\alpha\cdot\lambda=1$. Now, because $R$ is positive definite, it must be the case that $t\mapsto c_kt^{n_k}$ is also positive definite and therefore $\lambda_k=1/n_k=1/2m_k$ for some $m_k\in\mathbb{N}_+$. Consequently, $\alpha\cdot\lambda=1$ if and only if
     \begin{equation*}
        \abs{\alpha:\mathbf{m}}=\frac{\alpha_1}{m_1}+\frac{\alpha_2}{m_2}+\cdots+\frac{\alpha_d}{m_d}=2
     \end{equation*}
     for $\mathbf{m}=(m_1,m_2,\dots,m_d)\in\mathbb{N}_+^d$ and with this it is evident that
     \begin{equation*}
      D=\diag(1/2m_1,1/2m_2,\dots,1/2m_d)\in\Exp(P).
     \end{equation*}
\end{proof}

\begin{proof}[Proof of Proposition \ref{prop:PosHomAreSemiElliptic}]

    Let $F \in \Exp(P)$. Via the real Jordan Canonical Form theorem \cite[Theorem 3.4.1.5]{HJBook}, we know $F$ is similar to a matrix $G$ such that \begin{equation*}
        G = C_1 \oplus C_2 \cdots \oplus C_k \oplus J_1 \oplus \cdots \oplus J_{m},
    \end{equation*} where \begin{equation*}
        J_r = \lambda_r I_{n_r} + N_r \in M_{n_r}(\mathbb{R}),
    \end{equation*} with $\lambda_r \in \mathbb{R}$, $I_{n_r} \in M_{n_d}(\mathbb{R})$ is the $n_r \times n_r$ identity matrix, and $N_r$ is a strictly upper triangular whose entries are either $1$ or $0$ ($N_r = 0$ if $n_r = 1$). Further, \begin{equation*}
        C_r = a_r I_{2j_r} + S_r + N_r' \in M_{2j_r}(\mathbb{R}),
    \end{equation*} where $I_2$ is the $2\times 2$ identity matrix, $I_{2j_r}$ is the $2j_r\times 2j_r$ identity matrix, $a_r, b_r \in \mathbb{R}$, $S_r$ is a direct sum of $j_r$ matrices equal to \begin{equation*}
        B_{r} = \begin{pmatrix}
        0 & b_r \\
        -b_r & 0\\
    \end{pmatrix},
    \end{equation*} and $N_r'$ is a strictly upper triangular matrix whose entries are either $1$ or $0$ (and $N_r' = 0$ if $j_r = 2$). Here, the $\lambda_r$s denote the real eigenvalues of $F$ (and $G$) and the matrices $C_r$ correspond to complex eigenvalues $\mu_r = a_r \pm ib_r$ which necessarily appear in conjugate pairs. The eigenvalues $\lambda_r$ need not be distinct, and the number of times the value $\lambda_r$ occurs in the list $\lambda_{1}, \lambda_{2}, \cdots , \lambda_m$ is the geometric multiplicity of $\lambda_r$. Similarly, the $\mu_r$ need not be distinct and the number of times the block $B_r$ occurs in the list $B_1, B_2, \cdots , B_k$ is the geometric multiplicity of both $\mu_r = a_r + ib_r$ and $\overline{\mu}_r = a_r - ib_r$. It is also possible that no $C_r$ blocks or no $J_r$ blocks occur in $G$, should $G$ only have real or complex eigenvalues, respectively.

    Notice that $[\lambda_r I_{n_r}, N_r] = 0$, and $[a_rI_{2j_r}, N_r'] = [a_rI_{2j_r}, S_r] = 0$, since $I_{n_r}, I_{2j_r}$ are identity matrices. One can also check that $[a_rI_{2j_r} + S_r, N_r'] = 0$ by a quick computation. Then, setting $D_r' = a_rI_{2j_r}$ and $D_r = \lambda_r I_{n_r}$, we can write $G = D + S + N$ where
    \begin{equation}\nonumber
        \begin{split}
            D &= D_1' \oplus \cdots \oplus D_k' \oplus D_1 \oplus \cdots \oplus D_m\\
            N &= N_1' \oplus \cdots \oplus N_k' \oplus N_1 \oplus \cdots \oplus N_m \\
            S & = S_1 \oplus \cdots \oplus S_k \oplus O
        \end{split}
    \end{equation} where $O$ is a zero matrix of appropriate dimensions. Here, $D, S, N \in M_d(\mathbb{R})$, $D$ is diagonal, $N$ is strictly upper triangular and, so, nilpotent and $S = G-D-N \in \mathfrak{o}_d(\mathbb{R})$. By the commutativity relations for $D_r, N_r, D_r', N_r', S_r$ we also know that $[D+S, N] = [D, S] = 0$. 

    Since $F$ is similar to $G$ (via real similarity in the language of \cite{HJBook}), there exists $A\in\GldR$ for which $A^{-1}FA = G = D+N+S$. For $P_A(\xi) = P(A\xi)$, we observe that
    \begin{equation*}
        P_A(t^G\xi) = P_A(t^{A^{-1}FA}\xi) = P_A(A^{-1}t^FA\xi) = P(t^FA\xi) = tP(A\xi) = tP_A(\xi)
    \end{equation*} for all $t>0$ and $\xi\in\mathbb{R}^d$ on account of the fact that $F\in\Exp(P)$. Consequently, $G = D+S+N \in \Exp(P_A)$ with $[D+S, N] = [D, S] = 0$ and so we may invoke Lemma \ref{lem:TripleSumDiagLemma} to find that for some $\mathbf{m}=(m_1,m_2,\dots,m_d)\in\mathbb{N}_+^d$, we have $D =  \diag(1/2m_1, 1/2m_2, \cdots , 1/2m_d)\in \Exp(P_A)$ and \begin{equation*}
    P_A(\xi):=P(A\xi)=\sum_{\abs{\alpha:\mathbf{m}}=2}a_\alpha\xi^\alpha
    \end{equation*}
    Putting $E = ADA^{-1}$, we have \begin{equation*}
        P(t^{E}\xi) = P(t^{ADA^{-1}}\xi) = P(At^DA^{-1}\xi) = P_A(t^DA^{-1}\xi) = tP_A(A^{-1}\xi) = tP(\xi)
    \end{equation*} so $E\in \Exp(P)$ and the proof is complete.

\end{proof}

\noindent The remaining of this subsection is devoted to the proof of Proposition \ref{prop:QLmbdOpIsIndependentOfA}. We will first treat two lemmas.
\begin{lemma}\label{lem:Mon-PA-caract}
    Let $P$ be a positive-homogeneous polynomial and $A\in \GldR$ be such that $P_A$ is semi-elliptic with $\mathbf{m}=(m_1,m_2,\ldots,m_d)\in \mathbb{N}_{+}^{d}$. In view of Proposition \ref{prop:PosHomAreSemiElliptic}, $E=ADA^{-1}\in\Exp(P)$ where $D=\diag(1/2m_1,1/2m_2,\dots,1/2m_d)$ and $A=(v_1|v_2|\cdots|v_d)$ where each column $v_j$ is an eigenvector of $E$ with eigenvalue $1/2m_j$. If, for non-zero $\mathbf{b}\in\mathbb{R}^d$, $P_A(x\mathbf{b})$ is a monomial (in $x$) for $x\in\mathbb{R}$, the degree of this monomial is $2m_k$ for some $k=1,2,\dots,d$ and the only non-zero entries of $\mathbf{b}$ coincide with those of the columns in $A$ associated with the common eigenvalue $1/2m_k$ of $E$. Correspondingly, $A\mathbf{b}$ is an eigenvector of $E$ with eigenvalue $1/2m_k$. 
\end{lemma}
\begin{proof}
    We take $\mathbf{m}=(m_1,m_2,\dots,m_d)\in\mathbb{N}_+$ for which 
    \begin{equation*}
        P_A(\xi)=\sum_{\abs{\beta:2\mathbf{m}}=1}b_\beta \xi^\beta
    \end{equation*}
    and $D=\diag(1/2m_1,1/2m_2,\dots,1/2m_d)\in\Exp(P_A)$. In view of Proposition 8.7 of \cite{RSC17} (or Proposition \ref{prop:PureLambdaCompare} below), let $C,C'>0$ be such that
    \begin{equation*}
        C(\xi_1^{2m_1}+\xi_2^{2m_2}+\cdots+\xi_d^{2m_d})\leq P_A(\xi)\leq C'(\xi_1^{2m_1}+\xi_2^{2m_2}+\cdots+\xi_d^{2m_d})
    \end{equation*}
    for every $\xi=(\xi_1,\xi_2,\dots,\xi_d)\in\mathbb{R}^d$. Suppose that, for $\mathbf{b}=(b_1,b_2,\dots,b_d)\in\mathbb{R}^d\setminus\{0\}$, $P_A(x\mathbf{b})$ is a monomial in $x$ (for $x\in\mathbb{R}$). This means, for some $\delta\in\mathbb{N}$, $P_A(x\mathbf{b})=C_\mathbf{b} x^\delta$ for $x\in\mathbb{R}$ where $C_\mathbf{b}=P_A(\mathbf{b})>0$. Thus, we have
    \begin{equation*}
        C(x^{2m_1}b_1^{2m_1}+x^{2m_2}b_2^{2m_2}+\cdots x^{2m_d}b_d^{2m_d})\leq C_\mathbf{b} x^{\delta}\leq C'(x^{2m_1}b_1^{2m_1}+x^{2m_2}b_2^{2m_2}+\cdots x^{2m_d}b_d^{2m_d})
    \end{equation*}
for all $x\in\mathbb{R}$. Since $C_{\mathbf{b}}\neq 0$, a careful look at the asymptotic properties of the above inequality (as $x\to 0$ and $x\to\infty$) shows that $\delta=2m_k$ for some $k=1,2,\dots,d$ and, for any $j$ such that $m_j\neq m_k$, we must have $b_j=0$. In other words, $\delta=2m_k$ and $\mathbf{b}\in\mathbb{R}^d$ is such that its only non-zero entries coincide with those of the columns in $A$ associated to the eigenvalue $1/2m_k$ of $E$. 
\end{proof}

\begin{lemma}
    Let $P$ be a positive-homogeneous polynomial (in d variables) and let $A, A'\in \GldR$ be such that $P_A$ and $P_{A'}$ are semi-elliptic with $\mathbf{m}=(m_1,m_2,\ldots,m_d)$, and $\mathbf{m}'=(m'_1,m'_2,\ldots,m'_d)\in \mathbb{N}_{+}^d$, respectively. Let $E$ and $E'$ be given by $E=ADA^{-1}$  and $E'=A'D'(A')^{-1}$, both elements of $\Exp(P)$, with $D=\diag(1/2m_1,1/2m_2,\ldots,1/2m_d)$ and $D'=\diag(1/2m'_1,1/2m'_2,\ldots, 1/2m'_d)$. Then $\mathbf{m}'$ is a permutation of $\mathbf{m}$ and $E=E'$.
\end{lemma}

\begin{proof}
Let $A$ and $A'\in \GldR$ be as in the statement, and let's note that $P_A$ and $P_{A'}$ are related through the expression $P_{A'}(\cdot)=P_{A}(A^{-1}A'\cdot)$. \noindent For a fixed $k=1,2,\dots,d$ and any $l=1,2,\dots,d$ with $m'_l=m'_k$, we have
\begin{equation*}
    P_{A'}(xe_{l})=P_{A}(A^{-1}A'(xe_{l}))=P_{A}(xA^{-1}A'e_{l})=P_{A}(x\mathbf{b}),
\end{equation*}
where $e_l$ is the $l$th standard Euclidean coordinate vector and $\mathbf{b}=\mathbf{b}(l)=A^{-1}A'e_l\in \mathbb{R}^d$. The left-hand side of this identity is evidently a monomial in $x$ of degree $2m'_l=2m'_k$ so that $P_{A}(x\mathbf{b})$ is necessarily a monomial with this degree. Therefore, an appeal to Lemma \ref{lem:Mon-PA-caract} tells us that $A\mathbf{b}=A'e_l$ is an eigenvector of $E$ with eigenvalue $1/2m'_k$. Since, $A'e_l$ is also an eigenvector of $E'$ with eigenvalue $1/2m'_k$, we can exhaust every $l=1,2,\dots,d$ for which $m'_k=m'_l$ to see that the geometric multiplicity of $1/2m'_k$ as an eigenvalue of $E'$ is not more than its multiplicity as an eigenvalue of $E$. As this argument is completely reversible, we see that the spectra of $E$ and $E'$ coincide (with the same geometric multiplicity of eigenvalues). From this it follows immediately that $\mathbf{m}'$ is simply a permutation of $\mathbf{m}$. Further, this construction shows that the columns of $A'$ are a basis of $\mathbb{R}^d$ consisting of eigenvectors of both $E$ and $E'$ with matching eigenvalues (for $E$ and $E'$). Therefore, $E=E'$.
\end{proof}

\begin{proof}[Proof of Proposition \ref{prop:QLmbdOpIsIndependentOfA}]
Let's consider $A$ and $\widetilde{A}\in \GldR$ such that $P_A$ and $P_{\widetilde{A}}$ are semi-elliptic with $\mathbf{m}=(m_1,m_2,\ldots,m_d)$ and $\widetilde{\mathbf{m}}=(\widetilde{m}_1,\widetilde{m}_2,\ldots,\widetilde{m}_d)$, respectively. By the previous lemma, we have $E=ADA^{-1}=\widetilde{A}\widetilde{D}\widetilde{A}^{-1}\in\Exp(P)$ where $D=\diag(1/2m_1,1/2m_2,\dots,1/2m_d)$ and $\widetilde{D}=\diag(1/2\widetilde{m}_1,1/2\widetilde{m}_2,\dots,1/2\widetilde{m}_d)$. Consider the holomorphic functions $\psi_{A}(r,z)$ and $\psi_{\widetilde{A}}(r,z)$ defined by
\begin{equation*}
    \psi_{A}(r,z)=\sum_{\lambda=1}^{\infty}\frac{S_{\lambda}(z)}{\lambda!}r^{\lambda}\hspace{1cm}\mbox{and}\hspace{1cm} \psi_{\widetilde{A}}(r,z)=\sum_{\lambda=1}^{\infty}\frac{\widetilde{S}_{\lambda}(z)}{\lambda!}r^{\lambda}
\end{equation*}
which, in view of Lemma \ref{lem:PsiProperties}, converge and are holomorphic on some open neighborhood of $0$ in $\mathbb{C}^{d+1}$ (which, in particular, contains any $z\in\mathbb{C}^d$ so as long as $r$ is small). From \eqref{eq:PsiRelation}, we have
\begin{equation*}
    r^{2m}\psi_{A}(r,A^{-1}z)=\Upsilon_A(r^{2mD}A^{-1}z)=\Upsilon(Ar^{2mD}A^{-1}z)=\Upsilon(r^{2mE}z)
\end{equation*}
and, similarly, 
\begin{equation*}
    r^{2m}\psi_{\tilde{A}}(r,\widetilde{A}^{-1}z)=\Upsilon(r^{2m E}z)
\end{equation*}
for any $z\in\mathbb{C}^d$ and sufficiently small $r$, where $E=ADA^{-1}=\widetilde{A}\widetilde{D}\widetilde{A}^{-1}$. Thus, for any $z\in\mathbb{C}^d$ and sufficiently small $r$, we have
\begin{equation*}
    \sum_{\lambda=1}^\infty \frac{S_{\lambda}(A^{-1}z)}{\lambda!}r^\lambda=\psi_A(r,A^{-1}z)=r^{-2m}\Upsilon(r^{2mE}z)=\psi_{\widetilde{A}}(r,\widetilde{A}^{-1}z)=\sum_{\lambda=1}^\infty\frac{\widetilde{S}_\lambda(\widetilde{A}^{-1}z)}{\lambda!}r^\lambda.
\end{equation*}
Because power series coefficients are unique, we conclude that $S_\lambda(A^{-1}z)=\widetilde{S}_\lambda(\widetilde{A}^{-1}z)$ for each $z\in\mathbb{C}^d$ and $\lambda\in\mathbb{N}_+$, from which it immediately follws that $Q_{\lambda}^n=\widetilde{Q}_\lambda^n$ for every $n$ and $\lambda$.

\end{proof}

\subsection{Asymptotics for positive-homogeneous functions}\label{ssec:PHFAsymp}

\begin{proposition}\label{prop:ExpSetIntersectionAsympt}
    Let $P$ and $Q$ be positive-homogeneous functions with $E\in \Exp(P)\cap \Exp(Q)$. Then, $P\asymp Q$.
\end{proposition}
\begin{proof}
Let $S=\{\xi\in\mathbb{R}^d:P(\eta)=1\}$ and, since $S$ is compact thanks to Proposition \ref{prop:PosHomEquiv} and does not contain $0$, there are positive constants $c_1$ and $c_2$ for which $c_2 \leq Q(\eta) \leq  c_1$ for all $\eta\in S$. By the definition of $S$, we see that 
\begin{equation*}
    c_1\leq \frac{Q(\eta)}{P(\eta)}\leq c_2
\end{equation*}
for all $\eta\in S$. Now, for any $\xi\in\mathbb{R}^d\setminus \{0\}$, we can choose a $t>0$ for which $\xi=t^E\eta$ for $\eta\in S$ (simply take $t=P(\xi)$ and $\eta=P(\xi)^{-E}\xi$). With this, we see that
\begin{equation*}
    c_1\leq \frac{Q(\eta)}{P(\eta)}=\frac{tQ(\eta)}{tP(\eta)}=\frac{Q(t^E\eta)}{P(t^E\eta)}=\frac{Q(\xi)}{P(\xi)}=\frac{tQ(\eta)}{tP(\eta)}\leq c_2
\end{equation*}
giving our desired estimate on $\mathbb{R}^d\setminus\{0\}$. Recalling that $P(0) = Q(0) = 0$, the proof is complete.
\end{proof}

\noindent As an immediate corollary, we have the following.
\begin{proposition}\label{prop:PureLambdaCompare}
Given a positive-homogeneous function $R$, assume that $D=\diag(\lambda_1,\lambda_2,\dots,\dots,\lambda_d)\in \Exp(R)$ for some $\lambda_1,\lambda_2,\dots,\lambda_d$ which are all positive numbers. Then
\begin{equation*}
    R(x)\asymp \abs{x}^{\mathbf{1/\lambda}}:=\abs{x_1}^{1/\lambda_1}+\abs{x_2}^{1/\lambda_2}+\cdots+\abs{x_d}^{1/\lambda_d}
\end{equation*}
for $x\in\mathbb{R}^d$.
\end{proposition}

\noindent With the above proposition and Proposition \ref{prop:PosHomAreSemiElliptic} in hand, we are now above to prove Proposition \ref{prop:LFCompare}.

\begin{proof}[Proof of Proposition \ref{prop:LFCompare}]
    First, notice that for $A\in Gl_d(\mathbb{R})$ as given in Proposition \ref{prop:PosHomAreSemiElliptic}, we have \begin{equation*}
        R^{\#}(x) = \sup_{\xi} \{x\cdot \xi - R(\xi)\} = \sup_{\xi} \{x\cdot A\xi - R(A\xi)\} = \sup_{\xi} \{A^{\top}x\cdot \xi - R_A(\xi)\}  = R_A^{\#}(A^{\top}x).
    \end{equation*}
    for $x\in\mathbb{R}^d$. Now, by Proposition \ref{prop:PosHomAreSemiElliptic}, we know $D = \diag(1/2m_1, 1/2m_2, \cdots , 1/2m_d) \in \Exp(P_A) \subseteq \Exp(R_A)$, and so by Proposition \ref{prop:ExpSetLegFenchelTransform} we have $I-D \in \Exp(R_A^{\#})$. A quick computation also shows that for $Q(x) = \abs{x}^{\mathbf{1/2m}}$,  $I-D \in \Exp(Q)$. So, $I-D \in \Exp(Q) \cap \Exp(R_A^{\#})$, and the desired result follows by Proposition \ref{prop:ExpSetIntersectionAsympt}.
\end{proof}

\subsection{Characterizing subhomogeneity}\label{ssec:SubHom}

\noindent In this section, we discuss the notion of subhomogeneity (with respect to some $E\in\MdR$) and give a characterization for it in terms of positive-homogeneous functions (containing $E$ in their exponent sets). This material was introduced in \cite{BR22} and also appears in \cite{R23}. We present this material for completeness and note that our final result below, Lemma \ref{lem:SubHomLittleO}, appears as Proposition 1.7 in \cite{R23}.

\begin{definition}\label{def:SubHom}
    Let $f$ be a real-valued function defined on some open neighborhood $\mathcal{O}$ of $0$ in $\mathbb{R}^d$. Also, suppose that, for $E\in\MdR$, $\{t^E\}$ is contracting. We say that $f$ is subhomogeneous with respect to $E$ if, for every compact set $K\subseteq \mathbb{R}^d$ and every $\epsilon>0$, there is some $\tau$ for which
    \begin{equation*}
        \abs{f(t^Ex)} \leq \epsilon t
    \end{equation*}

    \noindent
    for all $x\in K$ and $0 < t \leq \tau$.
\end{definition}

\begin{lemma}\label{lem:SubHomLittleO}
Let $R$ be a positive-homogeneous function on $\mathbb{R}^d$ and let $f$ be defined on an open neighborhood $\mathcal{O}$ of $0$ in $\mathbb{R}^d$. Then $f(\xi)=o(R(\xi))$ as $\xi\to 0$ if and only if $f$ is subhomogeneous with respect to $E$ for every $E\in\Exp(P)$. 
\end{lemma}

\begin{proof}
We first suppose that $f(\xi)=o(R(\xi))$ as $\xi\to 0$ and take $E\in\Exp(P)$. Given $\epsilon>0$ and $K\subseteq\mathbb{R}^d$ be compact, let $M=\sup_{\xi\in K}R(\xi)<\infty$. Our hypotheses gives us a neighborhood $\mathcal{U}\subseteq \mathcal{O}$ on which  \begin{equation*}
    \abs{f(\xi)} \leq \frac{\epsilon}{M} R(\xi)
\end{equation*} for all $\xi\in \mathcal{U}$. Since $\{t^E\}$ is contracting by virtue of Proposition \ref{prop:PosHomEquiv}, we can pick $\tau >0$ for which $t^E(K) \subseteq \mathcal{U}$ for every $0<t\leq \tau$. Thus,
\begin{equation*}
    \abs{f(t^E\xi)} \leq \frac{\epsilon}{M} R(t^E\xi) =  \frac{\epsilon}{M} t R(\xi) \leq \epsilon t
\end{equation*} for all $0 < t \leq \tau$ and $\xi \in K$.

Conversely, suppose $f$ is subhomogeneous with respect to $E\in\Exp(P)$, let $\epsilon>0$, and define $S=S_R = \{\eta \in \mathbb{R}^d : R(\eta) = 1\}$, which is compact by Proposition \ref{prop:PosHomEquiv}. Then, there exists $\tau > 0$ such that \begin{equation*}
    \abs{f(t^E\eta)}\leq \epsilon t
\end{equation*} for every $\eta \in S$ and every $0<t\leq \tau$. For any $\xi$ in the open neighborhood $\mathcal{U}=\mathcal{O}\cap\{\xi\in\mathbb{R}^d:R(\xi)<\tau\}$ of $0$, $t=R(\xi)<\tau$ and $\eta=t^{-E}\xi\in S$ and so it follows that
 \begin{equation*}
    \abs{f(\xi)} = \abs{f(t^Et^{-E}\xi)}=\abs{f(t^E\eta)} \leq \epsilon t = \epsilon R(\xi).
\end{equation*}

\end{proof}

\end{document}